\documentclass[a4paper,11pt]{article}

\usepackage[T1]{fontenc}
\usepackage{lmodern}

\usepackage{dsfont}

\usepackage[latin1]{inputenc}
\usepackage{amsmath}
\usepackage{amsthm}
\usepackage{amssymb}
\usepackage{mathrsfs}
\usepackage{graphicx}
\usepackage{animate}
\usepackage[all]{xy}
\usepackage{hyperref}

\usepackage{makeidx}

\usepackage{stmaryrd}
\usepackage{caption}

\usepackage{abstract} 

\usepackage{lipsum}
\newcommand\blfootnote[1]{%
  \begingroup
  \renewcommand\thefootnote{}\footnote{#1}%
  \addtocounter{footnote}{-1}%
  \endgroup
}

\newtheorem{thm}{Theorem}[section]
\newtheorem{cor}[thm]{Corollary}
\newtheorem{claim}[thm]{Claim}
\newtheorem{fact}[thm]{Fact}

\newtheorem{lemma}[thm]{Lemma}
\newtheorem{prop}[thm]{Proposition}

\theoremstyle{definition}
\newtheorem{definition}[thm]{Definition}
\newtheorem{ex}[thm]{Example}
\newtheorem{exo}[thm]{Exercise}
\newtheorem{remark}[thm]{Remark}

\def\rquotient#1#2{%
	\makeatletter
	\raise.3ex\hbox{$#1$}/\lower.3ex\hbox{$#2$}%
	\makeatother
}	

\makeatletter
\newcommand{\subjclass}[2][2010]{%
	\let\@oldtitle\@title%
	\gdef\@title{\@oldtitle\footnotetext{#1 \emph{Mathematics subject classification.} #2}}%
}
\newcommand{\keywords}[1]{%
	\let\@@oldtitle\@title%
	\gdef\@title{\@@oldtitle\footnotetext{\emph{Key words and phrases.} #1.}}%
}
\makeatother

\newcommand{\Address}{{% additional braces for segregating \footnotesize
		\bigskip
		\small
		
\noindent \textsc{Universit\'e de Montpellier\\ 
Institut Math\'ematiques Alexander Grothendieck\\
Place Eug\`ene Bataillon\\
34090 Montpellier (France)}\par\nopagebreak
\noindent \textit{E-mail address}: \texttt{anthony.genevois@umontpellier.fr}
		
}}

\makeindex

\title{\Huge An illustrated introduction to the coarse topology of lamplighters}
\date{\today}
\author{Anthony Genevois}

\subjclass{Primary 20F65. Secondary 20F69.}
\keywords{Wreath product, lamplighters, quasi-isometry, coarse topology}

\begin{document}

\maketitle

\begin{minipage}{0.28\linewidth}
\includegraphics[width=\linewidth]{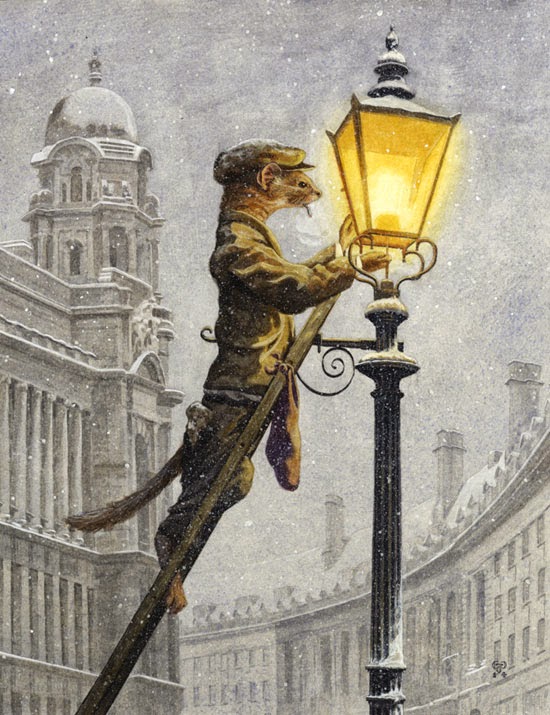}
\end{minipage}
\begin{minipage}{0.71\linewidth}
\begin{abstract}
Roughly speaking, lamplighter graphs encode the possible configurations of a lamplighter that moves along a given graph and that modifies the colours of lamps at vertices. This article is dedicated to the following delicate question: when do two lamplighter graphs have the same coarse geometry, i.e.\ when are they quasi-isometric? Inspired by elementary ideas from topology, which we will ``coarsify'', I will survey some techniques that allow us to compare efficiently lamplighter graphs (and more) up to quasi-isometry. Based on a minicourse given during the Journ\'ees de Topologie G\'eom\'etrique at the Institut Fourier in August 2025.
\end{abstract}
\end{minipage}
\blfootnote{Painting realised by \href{https://chrisdunnillustration.blogspot.com/2014/04/lamplighter.html}{Chris Dunn} (2014).}

\tableofcontents

\vspace{1cm}
\section{Introduction}

\noindent
Given a graph $X$ and an integer $n \geq 2$, the lamplighter graph $\mathcal{L}_n(X)$ encodes the possible configurations of a lamplighter that moves along $X$ and that modifies the $n$ possible colours of lamps at vertices. We refer to Section~\ref{section:CayleyGraphs} for a precise definition, which is illustrated by the following example:

\medskip
\begin{center}
\begin{tabular}{|c|c|c|} \hline
\includegraphics[trim=4cm 11cm 24cm 4cm,clip,width=0.29\linewidth]{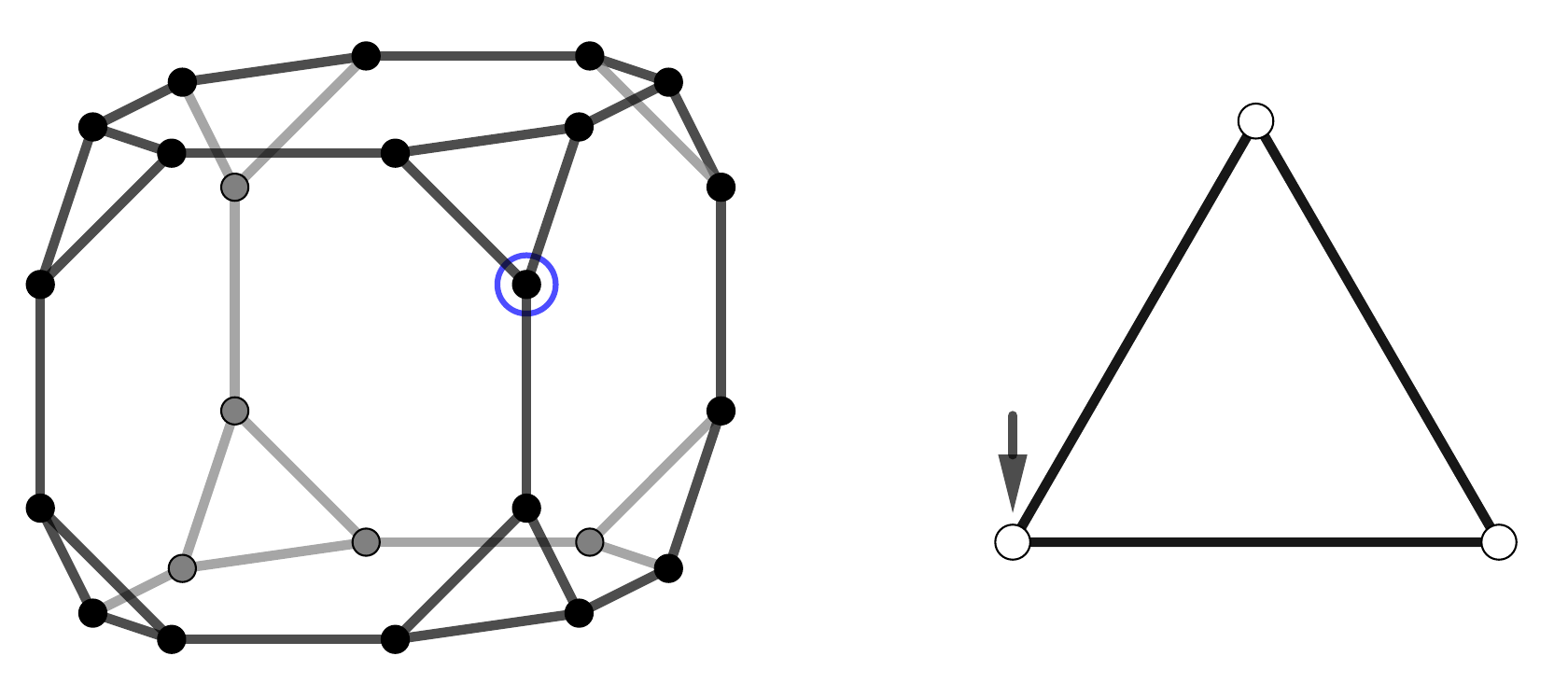} & \includegraphics[trim=4cm 11cm 24cm 4cm,clip,width=0.29\linewidth]{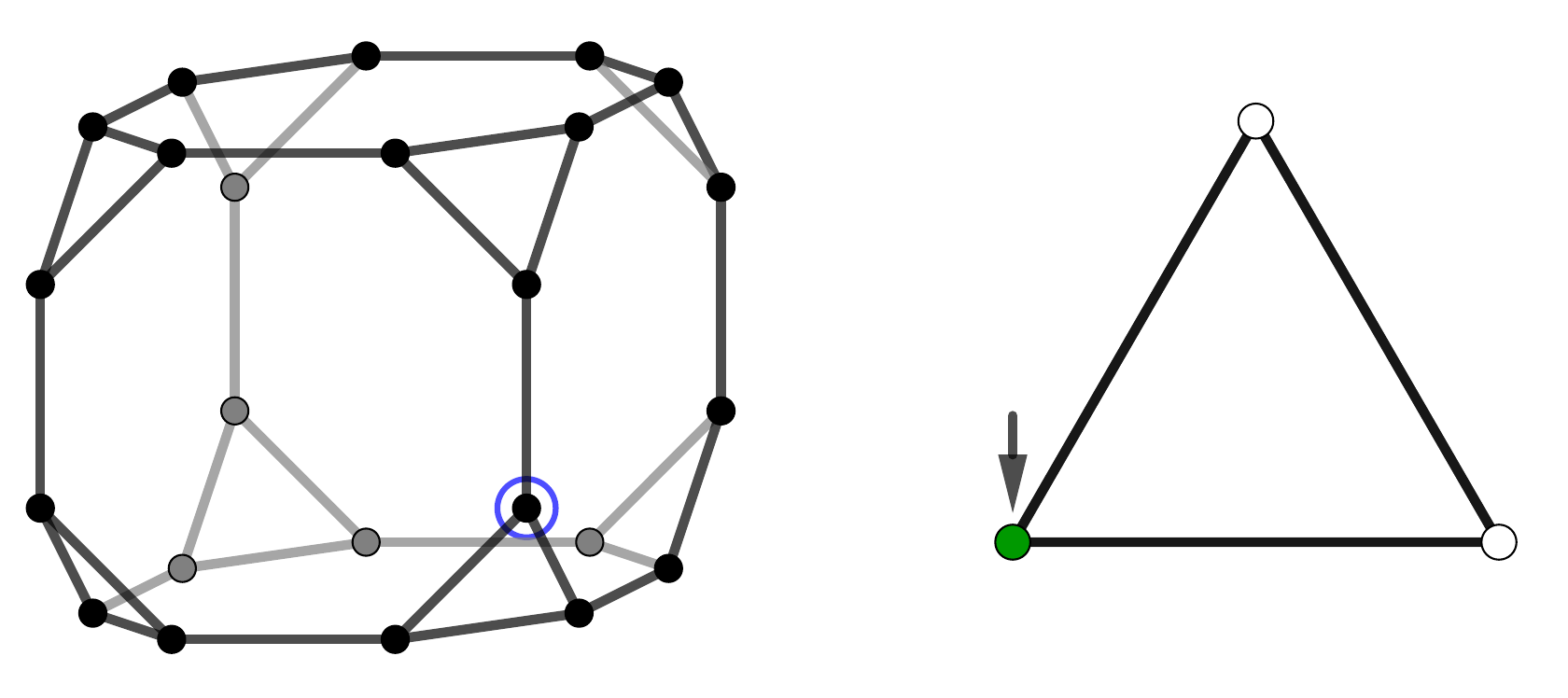} & \includegraphics[trim=4cm 11cm 24cm 4cm,clip,width=0.29\linewidth]{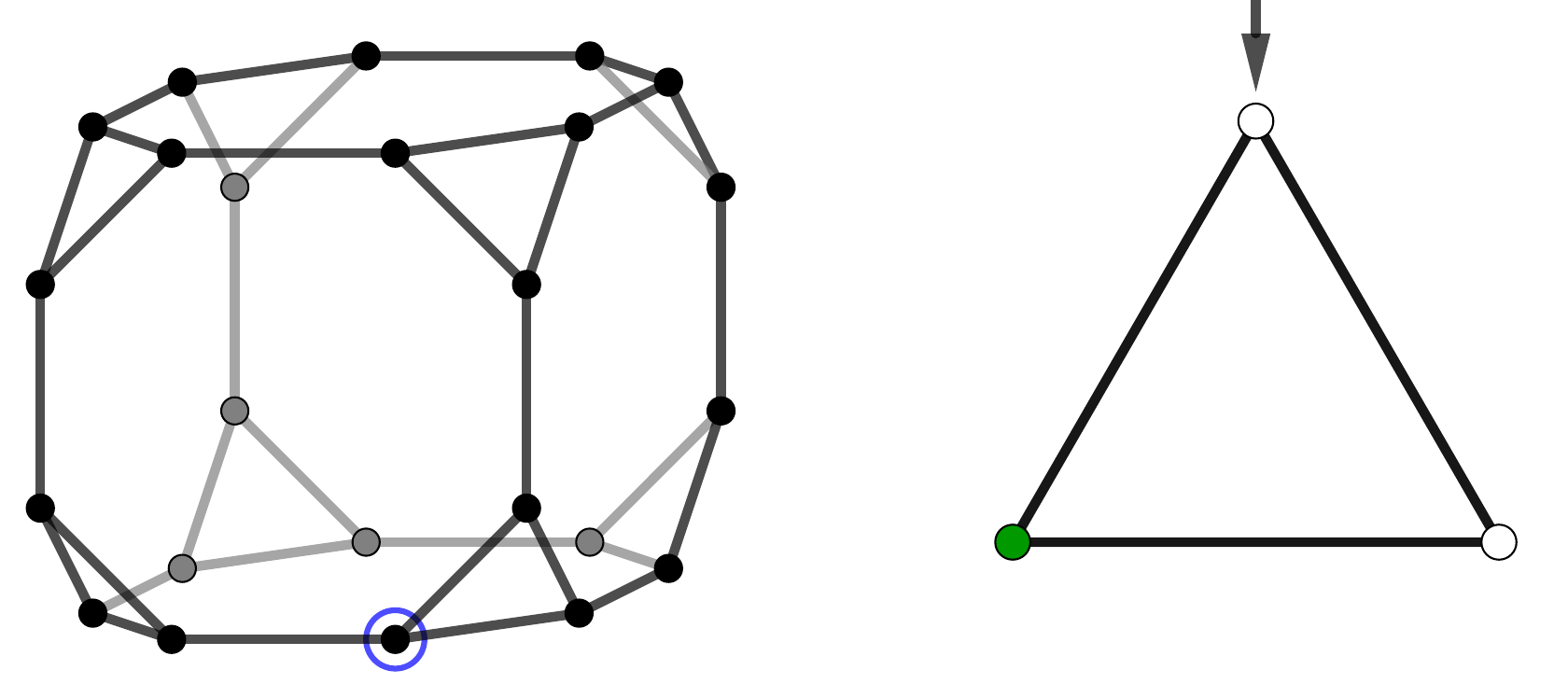}  \\ \hline
\includegraphics[trim=4cm 11cm 24cm 4cm,clip,width=0.29\linewidth]{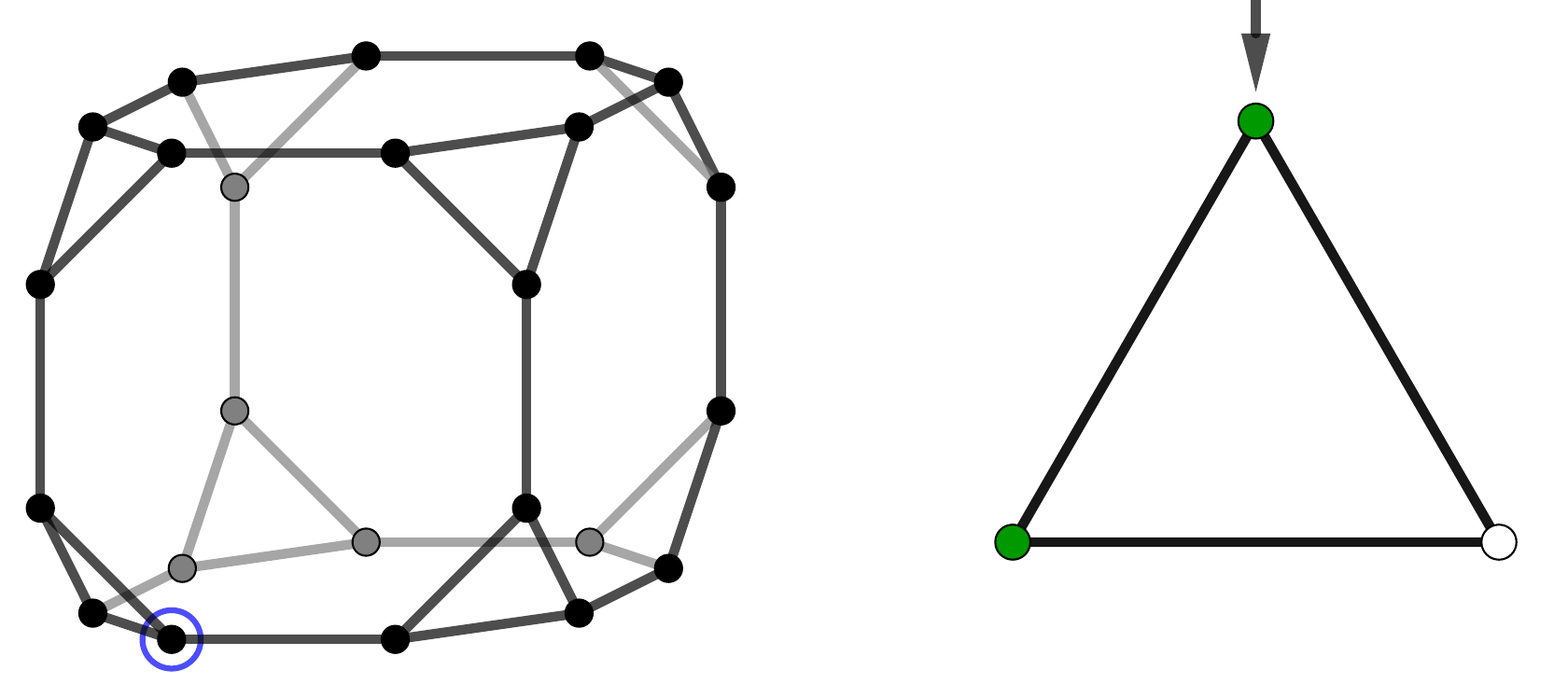} & \includegraphics[trim=4cm 11cm 24cm 4cm,clip,width=0.29\linewidth]{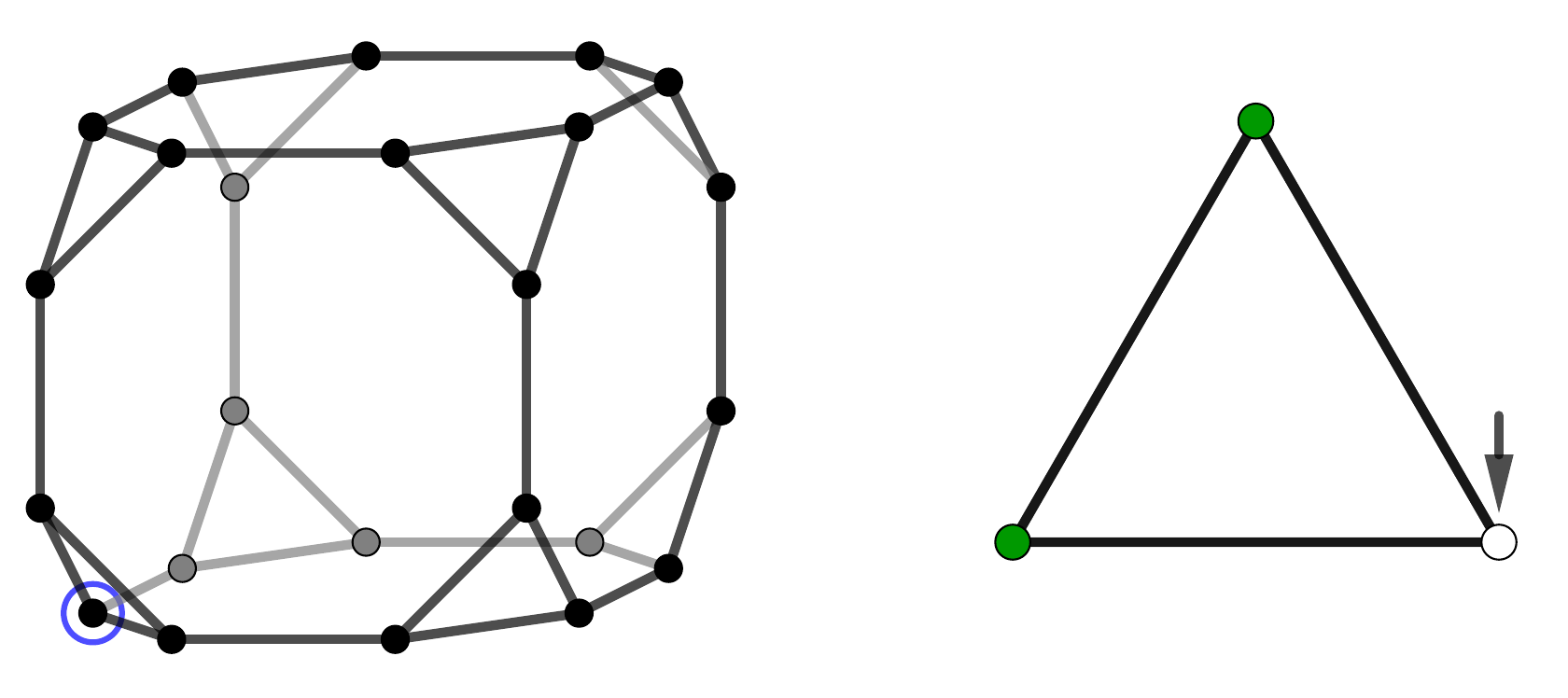} & \includegraphics[trim=4cm 11cm 24cm 4cm,clip,width=0.29\linewidth]{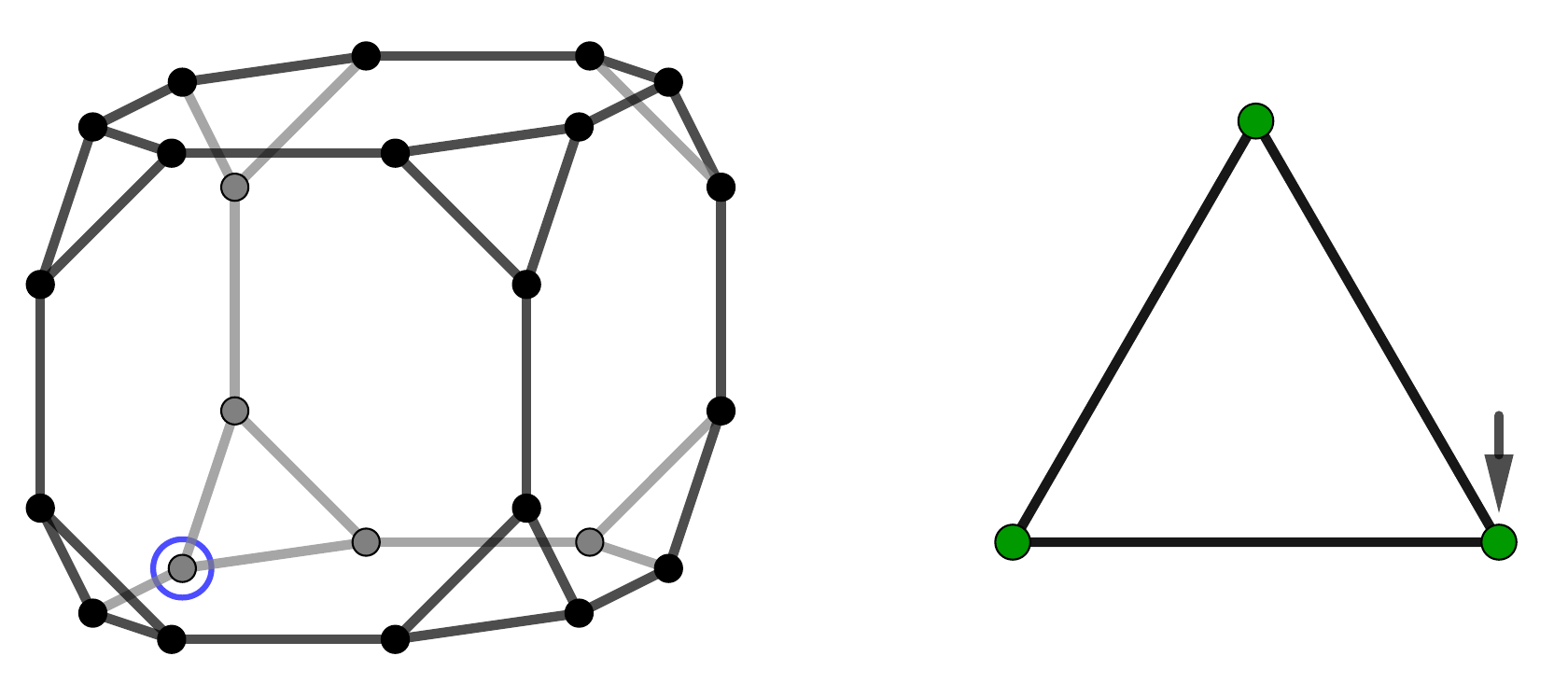}\\  \hline
\includegraphics[trim=4cm 11cm 24cm 4cm,clip,width=0.29\linewidth]{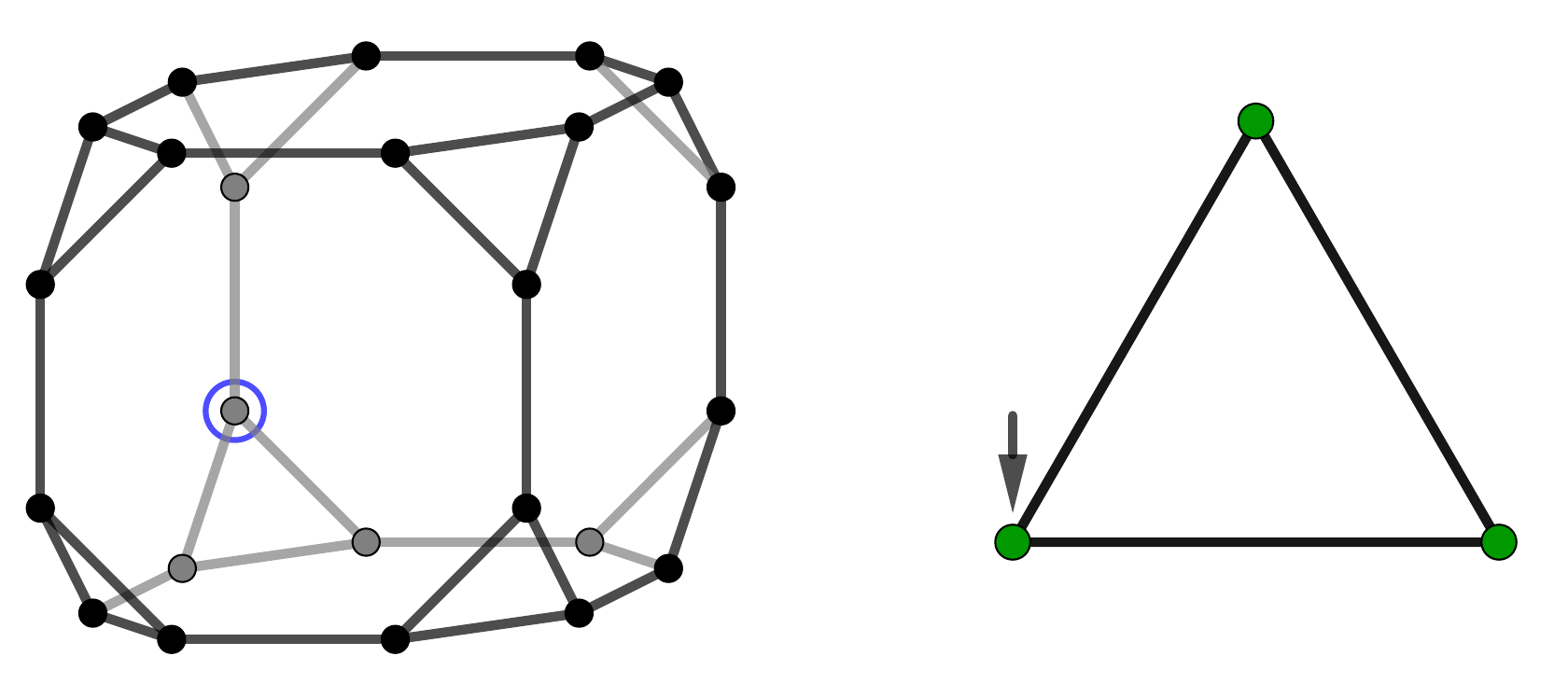} & \includegraphics[trim=4cm 11cm 24cm 4cm,clip,width=0.29\linewidth]{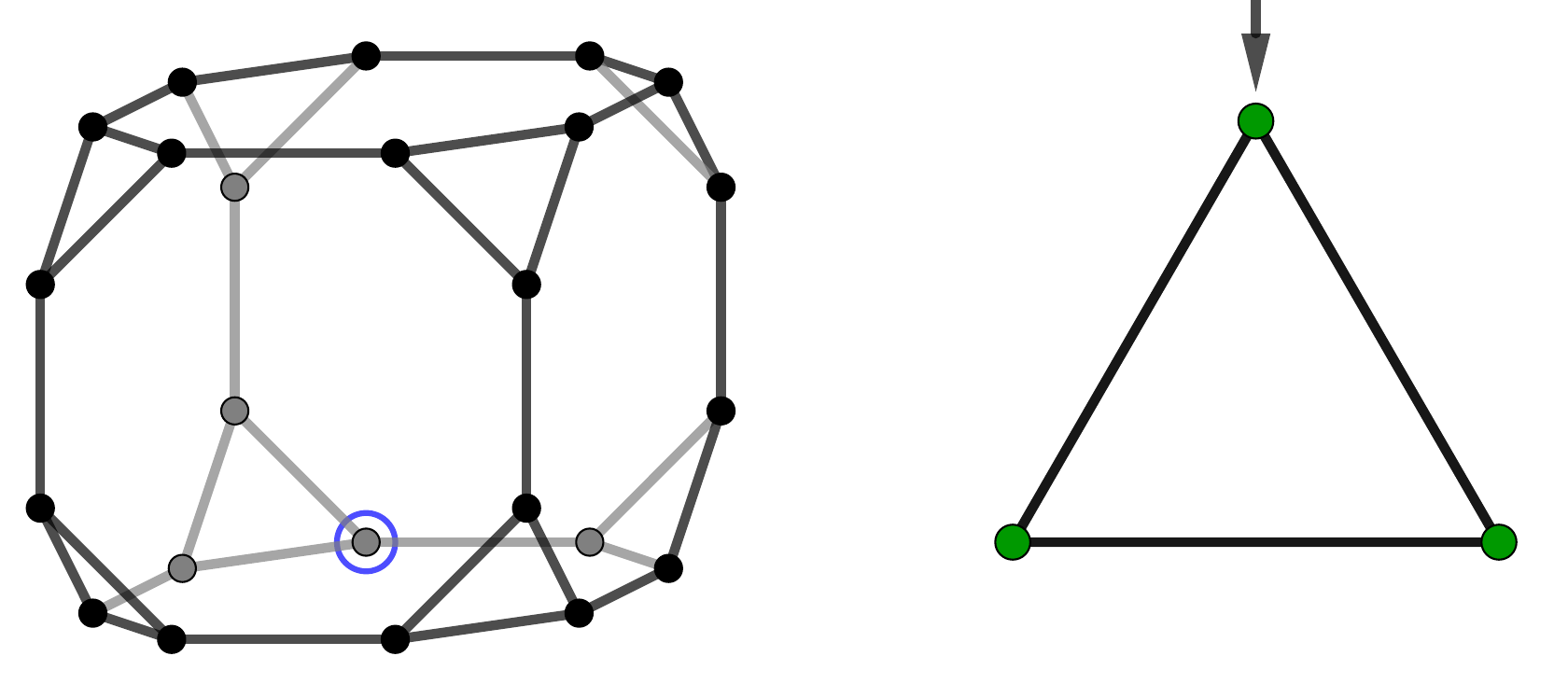} & \includegraphics[trim=4cm 11cm 24cm 4cm,clip,width=0.29\linewidth]{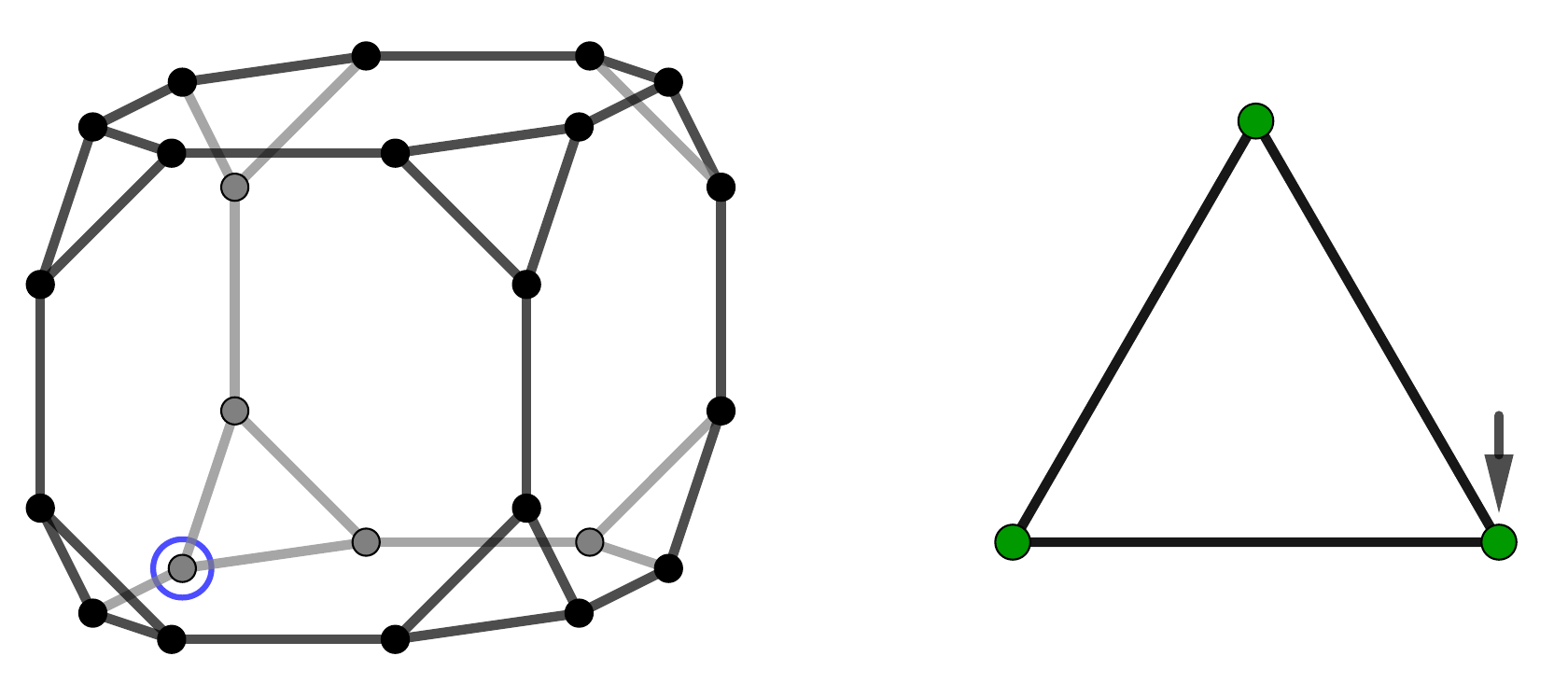} \\ \hline
\end{tabular}
\end{center}

\medskip \noindent
This mini-course is dedicated to the large-scale geometry of lamplighter graphs, with an emphasis on coarse topology. Roughly speaking, we use our intuition from general and algebraic topology in order to create various concepts that still make sense up to coarse equivalences of metrics. 

\medskip \noindent
Sections~\ref{section:QI} and~\ref{section:CoarseTopo} describe basic notions of coarse geometry and topology, which we illustrate with lamplighter graphs. Sections~\ref{section:BigEmbedding} and~\ref{section:Aptolic} present a(n almost self-contained and partly new) proof of the classification up to quasi-isometry of lamplighter graphs over one-ended coarsely simply connected graphs \cite{MR4794592,Halo}.

\section{Lamplighters and quasi-isometries}\label{section:QI}

\subsection{Algebraic definition}

\noindent
A \emph{wreath product}\footnote{Wreath product of permutation groups is a construction that can be found early in the history of group theory. For instance, it is used by Cauchy in 1844 in order to construct subgroups of symmetric groups with specific orders \cite[p.\ 194]{Cauchy}. Identifying a precise origin of the concept seems to be rather delicate and would require a rigorous historical investigation. However, the terminology can be traced back to Polya's monograph \cite{MR1577579} (see \cite{MR884155} for an English translation), where the German word \emph{Kranz} is used, later translated by \emph{wreath}.} $(G_1 \curvearrowright \Omega_1) \wr (G_2 \curvearrowright \Omega_2)$ is a natural way to combine two permutation groups $G_1 \curvearrowright \Omega_1$ and $G_2 \curvearrowright \Omega_2$ into a single permutation group over $\Omega_1 \times \Omega_2$.

\medskip \noindent
\begin{minipage}{0.45\linewidth}
\includegraphics[width=0.95\linewidth]{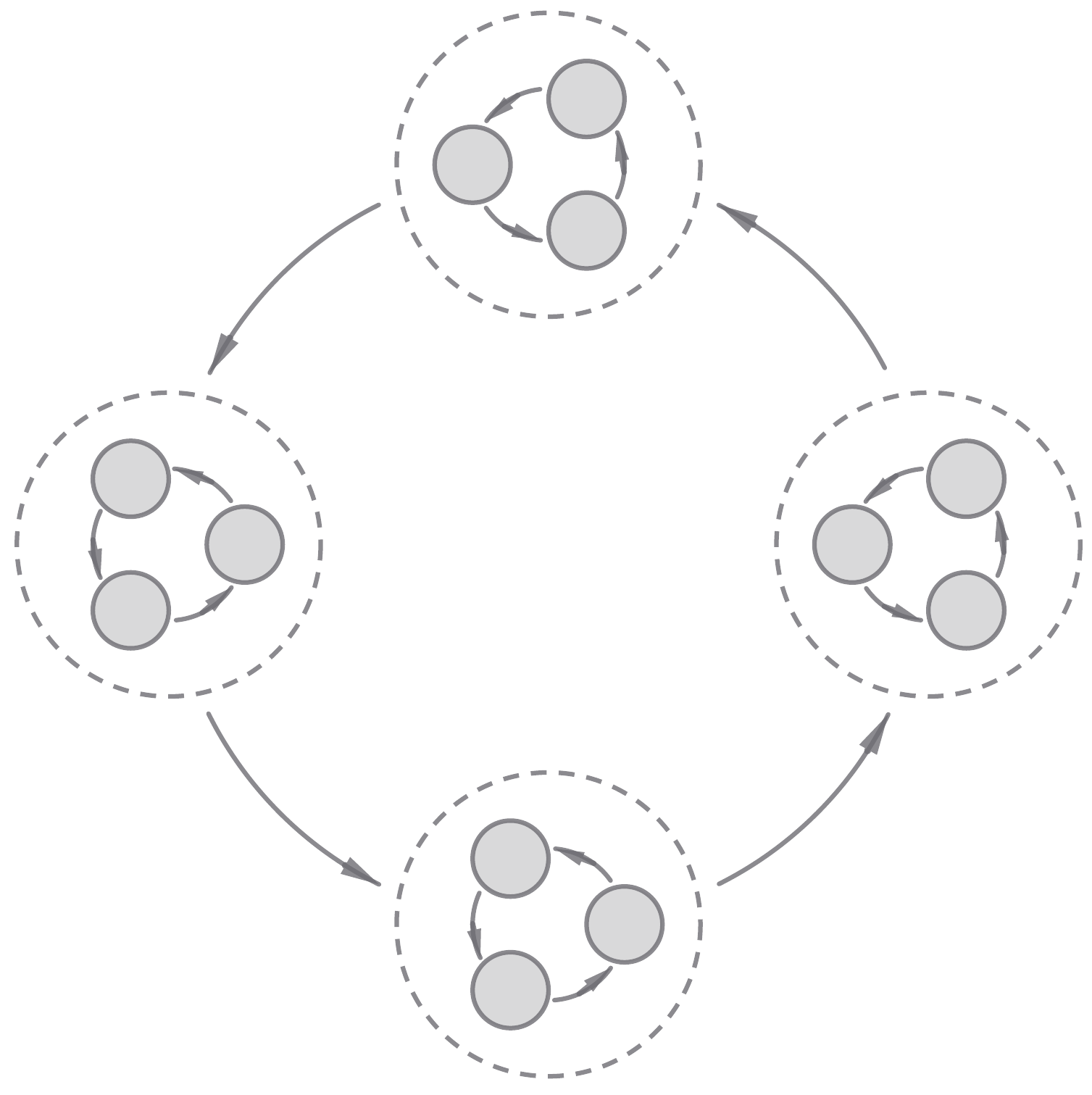}
\end{minipage}
\begin{minipage}{0.53\linewidth}
 Imagine that $\Omega_2$ is a set of bags, and that, in each bag, we set a copy of $\Omega_1$ (thought of as a collection of tokens). We get a set of tokens in bijection with $\Omega_1 \times \Omega_2$. Then, a permutation of $\Omega_1 \times \Omega_2$ belongs to the wreath product when it permutes the bags according to a permutation of $G_2$ and when the tokens in each bag are permuted according to a permutation of $G_1$. On the left, the figure illustrates a permutation of $(\mathbb{Z}_3 \curvearrowright \{0,1,2\}) \wr ( \mathbb{Z}_4 \curvearrowright \{0,1,2,3\})$. 
\end{minipage}

\medskip \noindent
Some wreath products naturally appear as symmetry groups of elementary objects. The most notable example is given by rooted trees. 

\medskip \noindent
\begin{minipage}{0.4\linewidth}
\includegraphics[width=\linewidth]{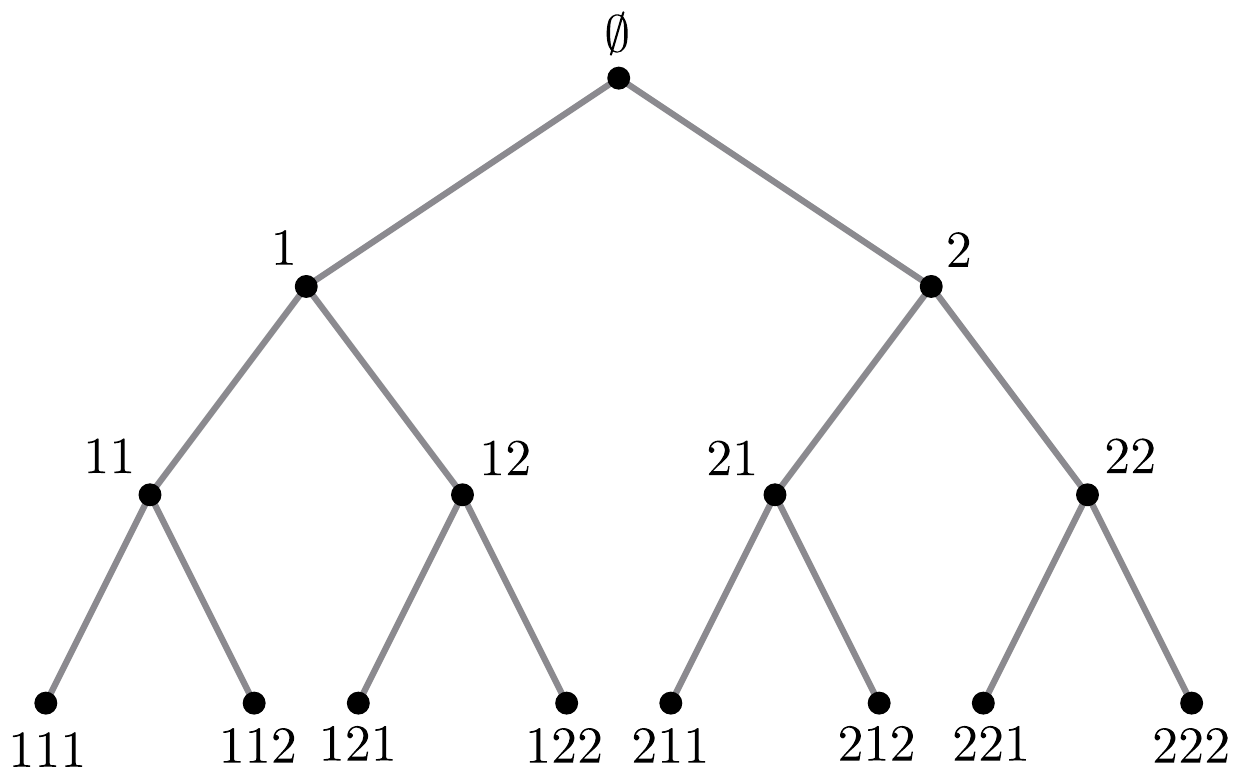}
\end{minipage}
\begin{minipage}{0.58\linewidth}
Let $T(n,k)$ denote the $n$-regular rooted tree of height $k$. Formally, $T(n,k)$ is the graph whose vertices are the words of length $\leq k$ written over $[n]:=\{1, \ldots, n\}$ and whose edges connect two words whenever one can be obtained from the other by adding a new letter at the end of the word. The figure on the left illustrates $T(2,3)$. 
\end{minipage}

\medskip \noindent
A \emph{leaf} of $T(n,k)$ refers to a vertex of degree one. Notice that the action of $\mathrm{Aut}(T(n,k))$ on the leaves of $T(n,k)$ preserves $n$ bags: the vertices below the vertex $1$, the vertices below the vertex $2$, $\ldots$, the vertices below the vertex $n$. The leaves in each bag can be thought of as the leaves of a copy of $T(n,k-1)$ in $T(n,k)$. This allows us to show that the permutation group $\mathrm{Aut}(T(n,k)) \curvearrowright \{ \text{leaves of } T(n,k)\}$ is isomorphic to 
$$(\mathrm{Aut}(T(n,k-1)) \curvearrowright \{\text{leaves of } T(n,k-1)\} ) \wr (\mathrm{Sym}(n) \curvearrowright [n] ),$$
and consequently to the iterated wreath product
$$(( ((\mathrm{Sym}(n) \curvearrowright [n] ) \wr (\mathrm{Sym}(n) \curvearrowright [n] )) \wr (\mathrm{Sym}(n) \curvearrowright [n] )) \wr \cdots ) \wr (\mathrm{Sym}(n) \curvearrowright [n] )$$
by iterating the decomposition. 

\medskip \noindent
It is worth mentioning that similar iterated wreath products can be used in order to described Sylow subgroups in symmetric groups. See for instance \cite[p. 176-177]{MR1307623}.

\medskip \noindent
As another interesting example, it can be shown that the symmetry group of the $n$-dimension cube $[0,1]^n$ can be realised as a permutation group
$$(\mathrm{Sym}(2) \curvearrowright \{1,2\}) \wr (\mathrm{Sym}(n) \curvearrowright \{1, \ldots, n\}).$$
See Exercise~\ref{exo:Hypercubes} for details. See also Exercise~\ref{exo:ChainSquares} for other examples of symmetry groups that can be realised as wreath products.

\medskip \noindent
Now, let us focus on the algebraic structure of our wreath product $(G_1 \curvearrowright \Omega_1) \wr (G_2 \curvearrowright \Omega_2)$. The permutation of the bags yields a morphism to $G_2$. And the kernel of this morphism corresponds to the permutations that fix each bag and that permute the tokens in each bag according to an element of $G_1$. Notice that the permutations in each bag are independent of each other. Thus, the kernel is isomorphic to a product of copies of $G_1$, one copy for each bag. More formally, we have the short exact sequence
$$1 \to G_1^{\Omega_2} \to (G_1 \curvearrowright \Omega_1) \wr (G_2 \curvearrowright \Omega_2) \to G_2 \to 1.$$
In fact, the sequence splits, i.e.\ the wreath product contains a copy of $G_2$, so we get the semidirect product decomposition
$$(G_1 \curvearrowright \Omega_1) \wr (G_2 \curvearrowright \Omega_2) = G_1^{\Omega_2} \rtimes G_2 \text{ as an abstract group}$$
where $G_2$ acts on $G_1^{\Omega_2}$ by permuting the coordinates according to its action on $\Omega_2$. Notice that, despite the fact that the realisation of $(G_1 \curvearrowright \Omega_1) \wr (G_2 \curvearrowright \Omega_2)$ as a permutation group over $\Omega_1 \times \Omega_2$ does depend on the permutation action $G_1 \curvearrowright \Omega_1$, the algebraic structure of the wreath product only depends on the algebraic structure of $G_1$. The permutation action $G_2 \curvearrowright \Omega_2$, on the other hand, does matter for the algebraic structure of the wreath product. 

\medskip \noindent
Since every abstract group can be realised as a permutation group, just by making it act on itself by left-multiplication, the wreath product of permutation groups can be used in order to define a product of abstract groups. 

\begin{definition}
The \emph{(unrestricted) wreath product} of two groups $A$ and $B$ is
$$A~\mathrm{Wr}~B:= A^B \rtimes B$$
where $B$ acts on the direct produt by permuting its coordinates according to its action by left-multiplication on itself, i.e.\ $g \ast (x_b)_{b \in B}=(x_{gb})_{b \in B}$ for all $g \in B$ and $(x_b)_{b \in B} \in A^B$. 
\end{definition}

\noindent
In other words, an element of $A ~\mathrm{Wr}~ B$ is a pair $((x_b)_{b \in B}, c)$ with $c \in B$ and $(x_b)_{b \in B} \in A^B$; and, for any two elements $((x_b)_{b \in B}, c)$ and $((y_b)_{b \in B}, d)$ of $A ~\mathrm{Wr}~ B$, their product is
$$((x_b)_{b \in B},c) \cdot ((y_b)_{b \in B}, d) := ((x_by_{cb})_{b \in B}, cd).$$
Notice that, when $A$ is non-trivial and $B$ infinite, the direct product $A^B$, and a fortiori the wreath product $A ~\mathrm{Wr}~ B$, is uncountable. In practice, we often focus on a smaller subgroup by replacing the direct product $A^B$ with the direct sum $A^{(B)}$. Recall that $A^{(B)}$ is the subgroup of $A^B$ given by the element all but finitely many of whose coordinates are trivial. 

\begin{definition}
The \emph{(restricted) wreath product} of two groups $A$ and $B$ is
$$A~\mathrm{wr}~B \text{ or } A \wr B := A^B \rtimes B$$
where $B$ acts on the direct sum by permuting its coordinates according to its action by left-multiplication on itself, i.e.\ $g \ast (x_b)_{b \in B}=(x_{gb})_{b \in B}$ for all $g \in B$ and $(x_b)_{b \in B} \in A^{(B)}$. 
\end{definition}

\medskip \noindent
Of course, when $B$ is a finite group, there is no difference between the restricted and unrestricted wreath products over $B$. But, when $B$ is infinite and $A$ non-trivial, $A \wr B$ is always a proper subgroup of $A~\mathrm{Wr}~B$, usually much smaller. We will see in the next section that the restricted wreath product of two finitely generated groups is finitely generated. 

\medskip \noindent
\textbf{In the rest of the course, a wreath product will always refers to a restricted wreath product.}

\subsection{Cayley graphs}\label{section:CayleyGraphs}

\noindent
Interestingly, every group can be pictured as a graph. This perspective will be developed further in the next section in order to think of (finitely generated) groups as geometric spaces. 

\begin{definition}
Let $G$ be a group and $S \subset G$ a set of elements. The \emph{Cayley graph} $\mathrm{Cayl}(G,S)$ is the graph 
\begin{itemize}
	\item whose vertices are the elements of $G$;
	\item and whose edges connect two distinct elements $a,b \in G$ whenever there exists $s \in S$ such that $b=as$. 
\end{itemize}
\end{definition}

\noindent
In other words, a Cayley graph $\mathrm{Cayl}(G,S)$ is a space whose points are the elements of our group $G$ and in which one travels by right-multiplications with elements in $S \cup S^{-1}$. Figure~\ref{Cayley} provides a few easy examples of Cayley graphs. In general, drawing a Cayley graph is a rather difficult task. We emphasize that Cayley graphs for the same group but for distinct subsets may be distinct. 

\begin{figure}[h!]
\begin{center}
\includegraphics[width=0.8\linewidth]{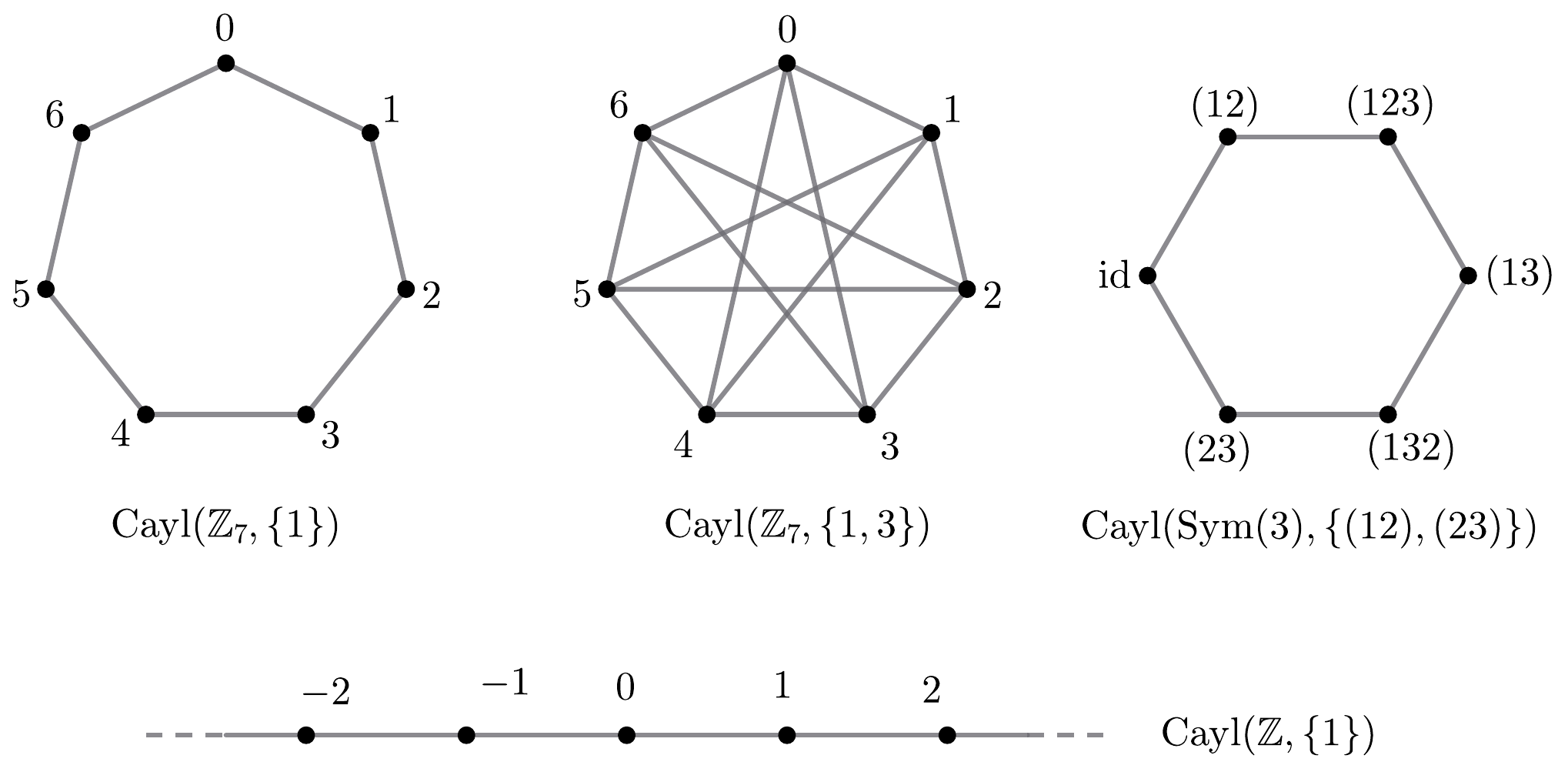}
\caption{A few examples of Cayley graphs.}
\label{Cayley}
\end{center}
\end{figure}

\noindent
Notice that our group $G$ naturally acts on $\mathrm{Cayl}(G,S)$ by left-multiplication, i.e.
$$\left\{ \begin{array}{ccc} G & \to & \mathrm{Aut}(\mathrm{Cayl}(G,S)) \\ g & \mapsto & (x \mapsto gx) \end{array} \right.$$
is an injective morphism. Thus, every group can be realised as a group of symmetries of a graph. In fact, of many graphs, since we have a lot of choices for $S$. In practice, we will usually assume that $S$ is a generating set (which assures that the Cayley graph is connected) and is finite (which is only possible if $G$ is finitely generated). 

\medskip \noindent
Endowing a given group with the structure of a connected graph allows us to endow our group with a metric, and consequently to think geometrically. Recall that, in a graph, the distance between two vertices is defined as the shortest length of a path connecting the two vertices.

\begin{prop}\label{prop:WordLength}
Let $G$ be a group and $S \subset G$ a subset. The graph $\mathrm{Cayl}(G,S)$ is connected if and only if $S$ generates $G$. If so, then
$$d_{\mathrm{Cayl}(G,S)}(g,h)  = \| g^{-1}h \|_S \text{ for all } g,h \in G$$
where $\|k\|_S:= \min \{ n \geq 0 \mid \exists s_1, \ldots, s_n \in S \cup S^{-1}, \ k=s_1 \cdots s_n\}$ denotes the \emph{word length} of an element $k \in G$.
\end{prop}

\begin{proof}
If $\mathrm{Cayl}(G,S)$ is connected, then, given an arbitrary element $g \in G$, there exists a path connecting $1$ to $g$. Such a path can be written as
$$1, \ s_1, \ s_1s_2, \ldots,\ s_1 s_2 \cdots s_{n-1}, \ s_1s_2 \cdots s_{n-1}s_n=g$$
where $s_1, \ldots, s_n \in S \cup S^{-1}$. Therefore, $S$ generates $G$. Conversely, If $G= \langle S \rangle$, then, given an arbitrary element $g \in G$, we can write $g=s_1 \cdots s_n$ for some $s_1, \ldots, s_n \in S \cup S^{-1}$, providing a path
$$1, \ s_1, \ s_1s_2, \ldots,\ s_1 s_2 \cdots s_{n-1}, \ s_1s_2 \cdots s_{n-1}s_n=g$$
in $\mathrm{Cayl}(G,S)$. Thus, every vertex can be connected by a path to $1$, proving that our Cayley graph is connected, as desired.

\medskip \noindent
From now on, assume that $G=\langle S \rangle$. A path in $\mathrm{Cayl}(G,S)$ connecting $g$ to $h$ amounts to right-multiplying $g$ with generators from $S \cup S^{-1}$ in order to get $h$. The product of these generators, then, represents $g^{-1}h$. Thus, the distance between $g$ and $h$ in $\mathrm{Cayl}(G,S)$ coincides with the smallest number of generators from $S \cup S^{-1}$ in a product representing $g^{-1}h$, namely $\|g^{-1}h\|_S$.
\end{proof}

\noindent
Now, we are interested in describing some Cayley graphs of wreath products. So, given two groups $A$ and $B$, let us consider the wreath product $A \wr B$. For convenience, we identify $A$ and $B$ with subgroups of $A \wr B:= A^{(B)} \rtimes B$, namely with the copy of $A$ indexed by $1 \in B$ and with the $B$-factor. Fixing two subsets $R \subset A$ and $S \subset B$, we want to describe  $\mathrm{Cayl}(A \wr B, R \cup S)$. 

\medskip \noindent
Formally, a vertex of our Cayley graph is an element of $A \wr B$, so a pair $( (a_b)_{b \in B}, p )$ where $(a_b)_{b \in B} \in A^{(B)}$ and $p \in B$. We think of first coordinate $(a_b)_{b \in B}$ as a colouring $b \mapsto a_b$ of the vertices of $\mathrm{Cayl}(B,S)$ by vertices of $\mathrm{Cayl}(A,R)$, and we think of the second coordinate $p$ as an arrow pointing to the vertex $p$ of $\mathrm{Cayl}(B,S)$.
\begin{figure}[h!]
\begin{center}
\includegraphics[width=0.7\linewidth]{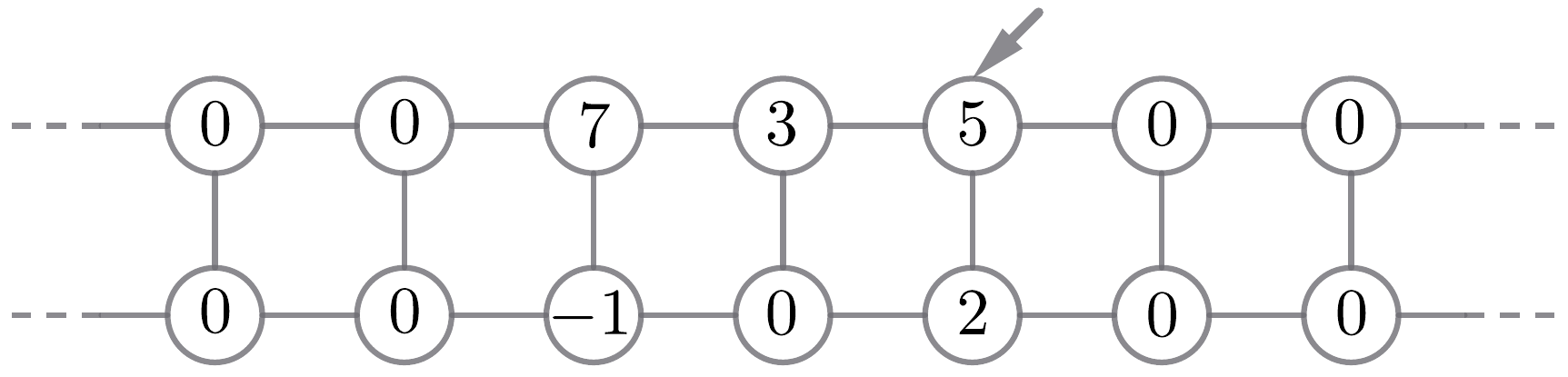}
\caption{An element of $\mathbb{Z} \wr \mathbb{D}_\infty$. }
\label {LampConf}
\end{center}
\end{figure}

\noindent
Now, we need to understand how this representation changes when one right-multiply with an element of $R \cup S$. Notice that, given an element $((a_b)_{b \in B}, p) \in A \wr B$, we have
$$\begin{array}{lcl} ((a_b)_{b \in B}, p ) \cdot r & = & ((a_b)_{b \in B}, p) \cdot \left( b \mapsto \left\{ \begin{array}{cl} r & \text{if } b=1 \\ 1 & \text{otherwise} \end{array} \right., 1 \right) \\ \\ & = & \left( b \mapsto \left\{ \begin{array}{cl} a_br & \text{if } b=p \\ a_b & \text{otherwise} \end{array} \right.,p \right) \end{array}$$
for every $r \in R$, and
$$((a_b)_{b \in B}, p ) \cdot s= ((a_b)_{b \in B}, p ) \cdot ( b \mapsto 1, s) =((a_b)_{b \in B} , ps)$$
for every $s \in S$. In the first case, the arrow does not move and the colour where the arrow points is right-multiplied by an element of $R$ (which amounts to saying that the new colour is a neighbour in $\mathrm{Cayl}(A,R)$); and, in the second case, the colouring is not modified but the arrow is moved to a neighbour in $\mathrm{Cayl}(B,S)$. 

\begin{figure}[h!]
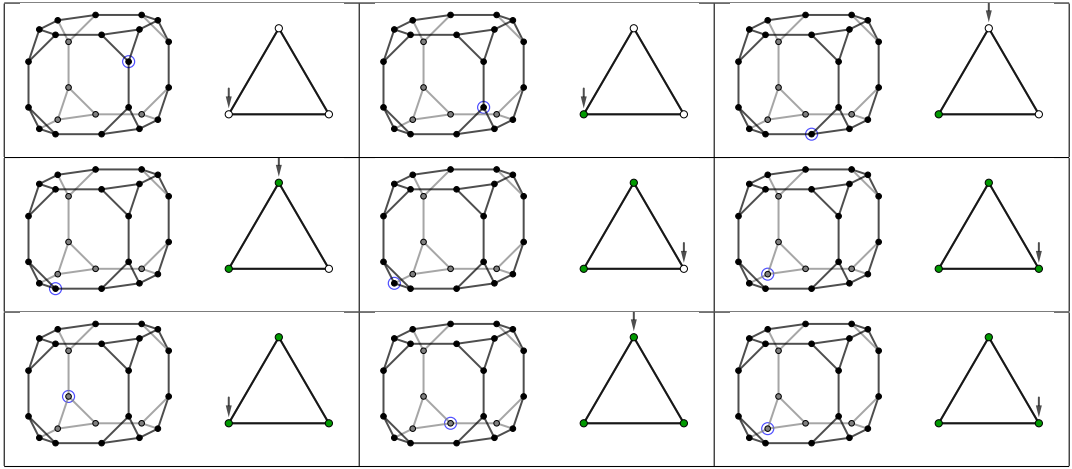

\begin{center}
\begin{tabular}{|c|c|c|} \hline
\includegraphics[trim=4cm 11cm 24cm 4cm,clip,width=0.29\linewidth]{L1} & \includegraphics[trim=4cm 11cm 24cm 4cm,clip,width=0.29\linewidth]{L21} & \includegraphics[trim=4cm 11cm 24cm 4cm,clip,width=0.29\linewidth]{L41}  \\ \hline
\includegraphics[trim=4cm 11cm 24cm 4cm,clip,width=0.29\linewidth]{L61} & \includegraphics[trim=4cm 11cm 24cm 4cm,clip,width=0.29\linewidth]{L81} & \includegraphics[trim=4cm 11cm 24cm 4cm,clip,width=0.29\linewidth]{L101}\\  \hline
\includegraphics[trim=4cm 11cm 24cm 4cm,clip,width=0.29\linewidth]{L121} & \includegraphics[trim=4cm 11cm 24cm 4cm,clip,width=0.29\linewidth]{L141} & \includegraphics[trim=4cm 11cm 24cm 4cm,clip,width=0.29\linewidth]{L161} \\ \hline
\end{tabular}
\caption{A Cayley graph of $\mathbb{Z}_2 \wr \mathbb{Z}_3$.}
\label{LampTriangle}
\end{center}
\end{figure}

\medskip \noindent
The picture to keep in mind is that we have an arrow $p$ pointing to some vertex of $\mathrm{Cayl}(B,S)$, which is allowed to move from vertices to adjacent vertices. The vertices of $\mathrm{Cayl}(B,S)$ are coloured with vertices of $\mathrm{Cayl}(A,R)$, with the restriction that all but finitely many colours are $1 \in A$. The arrow has the ability to modify the colour at the vertex where it is, modifying the colour to a neighbour in $\mathrm{Cayl}(A,R)$. Due to this interpretation, wreath products are also referred to as \emph{lamplighter groups}\footnote{As coined in \cite{MR1062874}, where the terminology is attributed to Jim Cannon. We warn the reader that, in the literature, ``lamplighter group'' may sometimes refer to $\mathbb{Z}_2 \wr \mathbb{Z}$ specifically, or to $\mathbb{Z}_n \wr \mathbb{Z}$.}.  

\medskip \noindent
Our description motivates the following definition.

\begin{definition}
Let $(X,o)$ be a pointed graph and $Y$ a graph. The \emph{wreath product} $(X,o) \wr Y$ is the graph 
\begin{itemize}
	\item whose vertices are the pairs $(\varphi : V(Y) \to V(X), p \in V(Y))$ where $\varphi(y)=o$ for all but finitely many $y \in V(Y)$;
	\item whose edges connect $(\varphi_1,p_1)$ and $(\varphi_2,p_2)$ if either $\varphi_1= \varphi_2$ and $\{p_1,p_2\} \in E(Y)$ or $\varphi_1,\varphi_2$ disagree only at $p_1=p_2$ where they take adjacent values.
\end{itemize}
\end{definition}

\noindent
Figure~\ref{WreathProductGraph} illustrates the wreath product of two edges, which just yields a cycle of length eight. Figure~\ref{LampTriangle} above shows that the wreath product between a triangle and an edge yields truncated $3$-cube. 
\begin{figure}
\begin{center}
\includegraphics[width=0.45\linewidth]{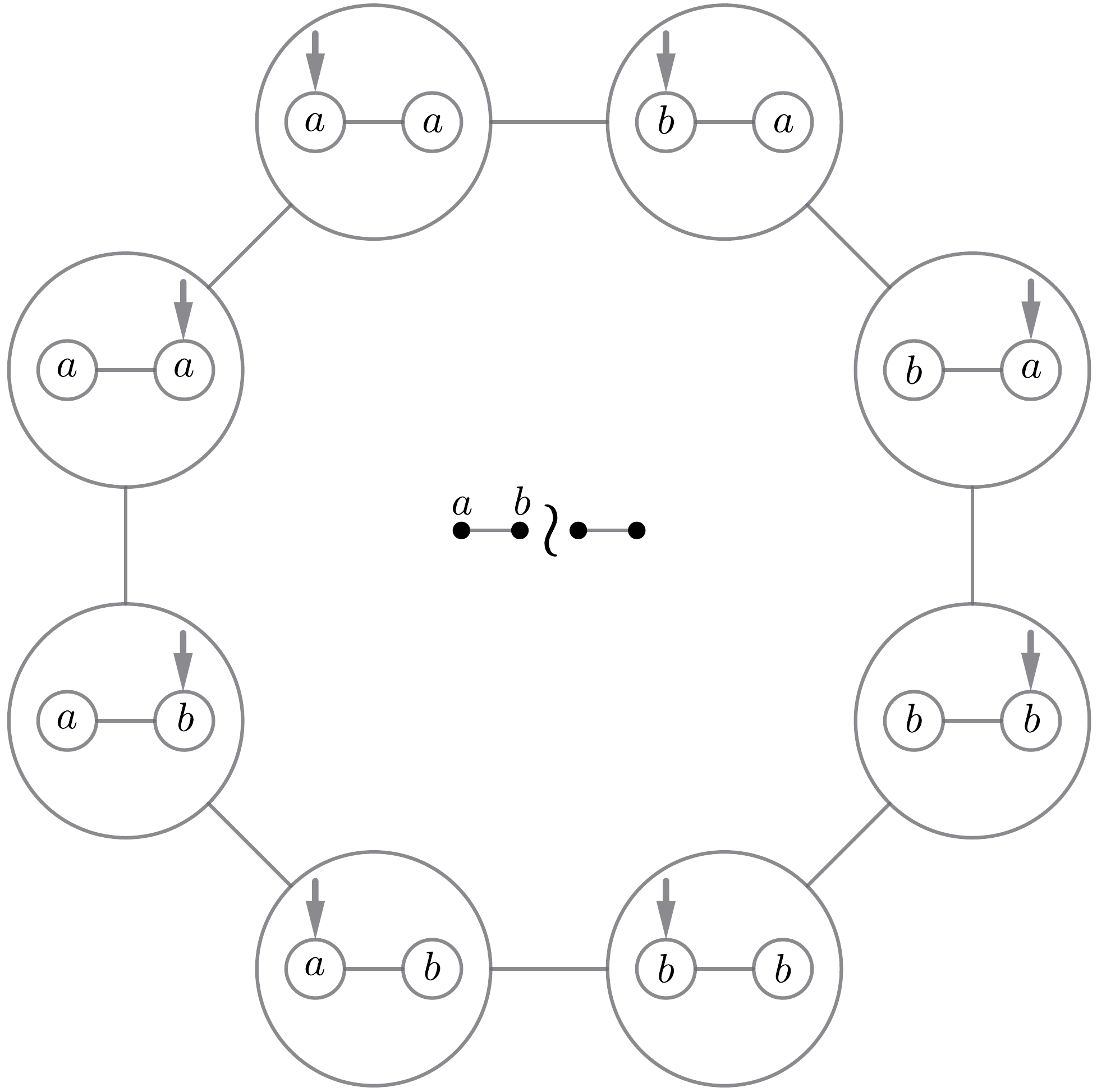}
\caption{Wreath product of two edges}
\label{WreathProductGraph}
\end{center}
\end{figure}

\medskip \noindent
As a consequence of the description given above of Cayley graphs of wreath products, we get the following statement:

\begin{prop}\label{prop:CaylWreath}
Let $A$ and $B$ be two groups. Given two subsets $R \subset A$ and $S \subset B$, 
$$\mathrm{Cayl}(A \wr B , R \cup S) \simeq (\mathrm{Cayl}(A, R), 1) \wr \mathrm{Cayl}(B,S).$$
\end{prop}

\noindent
In the rest of the course, we will be mainly interested in wreath products $A \wr B$ with $A$ finite. In this case, our description of Cayley graphs can be simplified by choosing convenient generating subsets. 

\begin{definition}
Let $X$ be a graph and $n \geq 2$ an integer. The \emph{lamplighter graph} $\mathcal{L}_n(X)$ is the graph
\begin{itemize}
	\item whose vertices are the pairs $(\varphi : V(X) \to \mathbb{Z}_n, p \in V(X))$ where $\varphi$ has finite support;
	\item whose edges connect $(\varphi_1,p_1)$ and $(\varphi_2,p_2)$ if either $\varphi_1=\varphi_2$ and $\{p_1,p_2\} \in E(X)$ or $\varphi_1,\varphi_2$ differ only at $p_1=p_2$.
\end{itemize}
\end{definition}

\medskip 
\begin{minipage}{0.35\linewidth}
\includegraphics[width=0.85\linewidth]{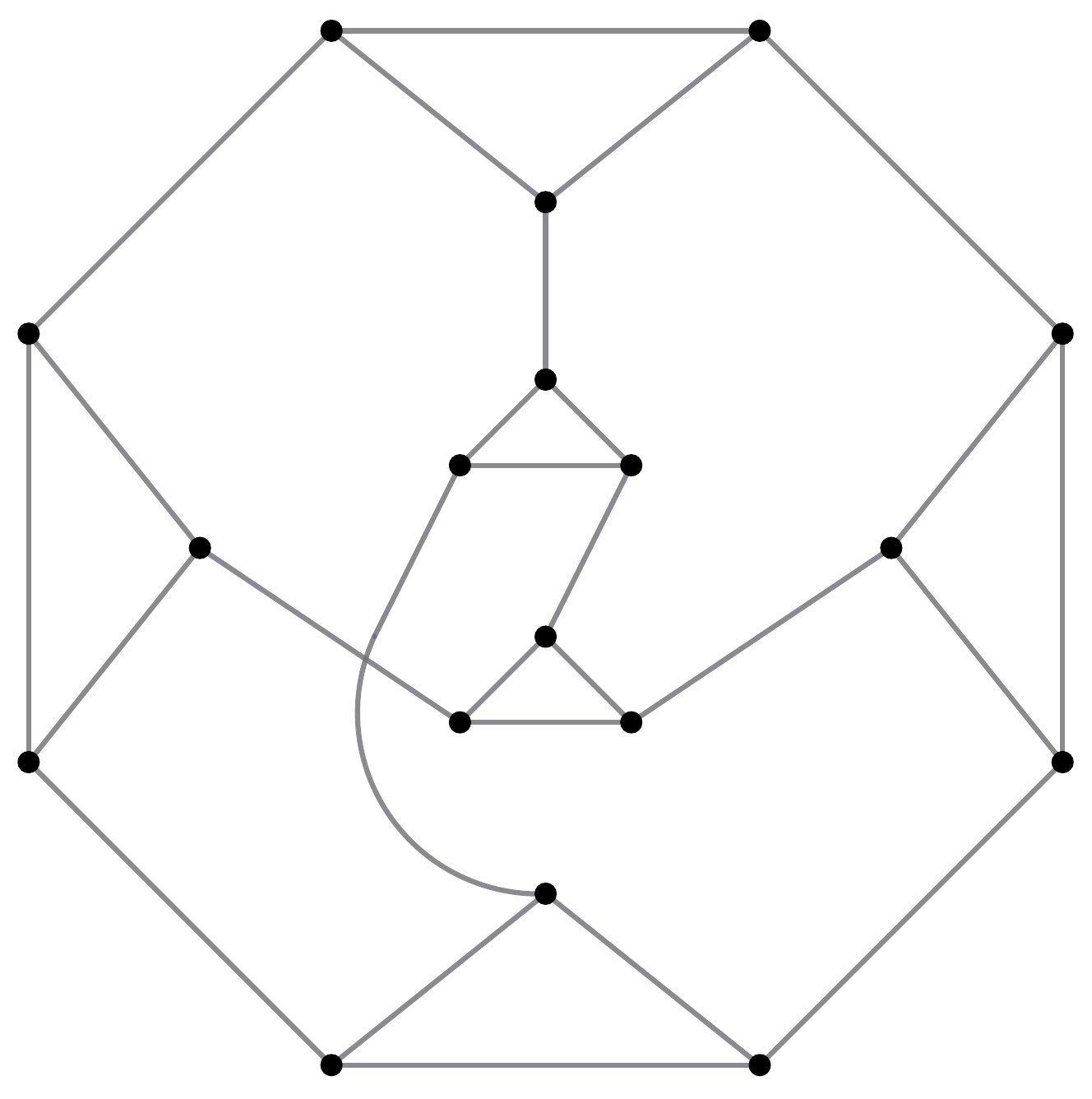}
\end{minipage}
\begin{minipage}{0.58\linewidth}
Notice that a lamplighter graph $\mathcal{L}_n(X)$ coincides with the wreath product $(K_n,o) \wr X$ where $K_n$ denotes the complete graph on $n$ vertices. So we know from the figures given above that $\mathcal{L}_2(\text{edge})$ is an $8$-cycle and that $\mathrm{L}_2(\text{triangle})$ is a truncated $3$-cube. The graph on the left illustrates $\mathcal{L}_3(\text{edge})$. We leave the details to the reader as an exercise.
\end{minipage}

\medskip \noindent
 As an immediate consequence of Proposition~\ref{prop:CaylWreath}, it follows that:

\begin{cor}
Let $H$ be a group and $n \geq 2$ an integer. Given a subset $S \subset H$,
$$\mathrm{Cayl}(\mathbb{Z}_n \wr H , \mathbb{Z}_n \cup S) \simeq \mathcal{L}_n(\mathrm{Cayl}(H,S)).$$
\end{cor}

\noindent
Section~\ref{section:CoarseHilbert} also contains a remarkable description of other Cayley graphs for the lamplighter groups $\mathbb{Z}_n \wr \mathbb{Z}$.

\medskip \noindent
Our definitions of wreath products and lamplighter graphs make sense for arbitrary graphs, but, since our goal is to do some geometry, we are mostly interested in connected graphs, so that the distance between any two vertices is finite.

\medskip \noindent
\textbf{In the rest of the course, unless otherwise stated a graph will be always connected.}

\subsection{Quasi-isometries and coarse embeddings}

\noindent
We saw in the previous section how one can think of a (finitely generated) group as a geometric space thanks to its Cayley graphs. However, different Cayley graphs lead to different geometries. Our next observation motivates the idea that these geometries are ``not too different''.

\begin{prop}
Let $G$ be a group and $R,S \subset G$ two finite generating sets. There exist constants $A,B>0$ such that
$$A \cdot d_{\mathrm{Cayl}(G,S)}(g,h) \leq d_{\mathrm{Cayl}(G,R)} (g,h) \leq B \cdot d_{\mathrm{Cayl}(G,S)} (g,h)$$
for all $g,h \in G$. 
\end{prop}

\begin{proof}
Let $g,h \in G$ be two elements. Write $g^{-1}h$ as a product $s_1 \cdots s_n$ with the smallest possible number of generators from $S \cup S^{-1}$. Notice that $n=\|g^{-1}h\|_S$. As a consequence of Proposition~\ref{prop:WordLength} and of the fact that the word length $\|\cdot\|_S$ satisfies the triangle inequality, we get:
$$d_{\mathrm{Cayl}(G,R)}(g,h) = \|g^{-1}h\|_R \leq \sum\limits_{i=1}^n \|s_i\|_R \leq \underset{=:B}{\underbrace{\max \{ \|s\|_R \mid s \in S \cup S^{-1} \} }} \cdot d_{\mathrm{Cayl}(G,S)}(g,h).$$
The second inequality of our proposition is proved symmetrically. 
\end{proof}

\noindent
In other words, the metrics induced on a group by its Cayley graphs with respect to finite generating sets are all \emph{biLipschitz equivalent}. More generally, let us introduce the following more general definition:

\begin{definition}
A map $\eta : X \to Y$ between two metric spaces is a \emph{coarse embedding} if there exist two maps $\rho_-,\rho_+ : [0,+ \infty) \to [0,+ \infty)$ satisfying $\rho_-(t),\rho_+(t) \to + \infty$ as $t \to + \infty$ such that
$$\rho_-(d(a,b)) \leq d(\eta(a), \eta(b)) \leq \rho_+(d(a,b))$$
for all $a,b \in X$. It is a \emph{coarse equivalence} if moreover there exists some constant $C \geq 0$ such that every point of $Y$ lies at distance $\leq C$ from $\eta(X)$.
\end{definition}

\noindent
When $\rho_-$ and $\rho_+$ are linear, one says that $\eta$ is \emph{biLipschitz}. A \emph{biLipschitz equivalence} is a surjective biLipschitz map. When $\rho_-$ and $\rho_+$ are affine, one says that $\eta$ is \emph{quasi-isometric}. Quasi-isometries are fundamental examples, especially due to the fact that a coarse equivalence between two graphs is automatically a quasi-isometry (see Exercise~\ref{ex:CoarseQI}). In particular, coarsely equivalent finitely generated groups are always quasi-isometric. 

\medskip \noindent
Let us mention a few elementary examples of quasi-isometries and coarse embeddings. See also Exercise~\ref{exo:ExQI} for an elementary example. 

\begin{ex}
It is standard that any two norms on a given finite-dimensional vector space are biLipschitz equivalent. In particular, the metric spaces $(\mathbb{R}^n, \| \cdot \|_p)$ and $(\mathbb{R}_n,\|\cdot\|_q)$ are biLipschitz equivalent for all $n \geq 1$ and $p,q \in [1, + \infty]$. 
\end{ex}

\begin{ex}
Given an $n \geq 1$, the inclusion map $\iota : \mathbb{Z}^n \hookrightarrow \mathbb{R}^n$ is a quasi-isometry, where $\mathbb{R}^n$ is endowed with any of its norms. In fact, if we endow $\mathbb{Z}^n$ with the word-length given by its coordinate vectors and if we endow $\mathbb{R}^n$ with the norm $\| \cdot\|_1$, then $\iota$ is an isometric embedding. A quasi-inverse (see Exercise~\ref{exo:QuasiInverse}) of $\iota$ is
$$\left\{ \begin{array}{ccc} \mathbb{R}^n & \to & \mathbb{Z}^n \\ (x_i)_{1 \leq i \leq n} & \mapsto & \left( \lfloor x_i \rfloor \right)_{1 \leq i \leq n} \end{array} \right..$$
\end{ex}

\begin{ex}
As shown on the figure below, tile the plane using squares whose bottom edges is subdivided. One gets a graph $\mathbb{H}$, in which we fix an horizontal line $L$. As suggested by the figure below, two vertices at distance $2^n$ in $L$ are at distance $\leq 2n+1$ in $\mathbb{H}$. Thus, the inclusion map $L \hookrightarrow \mathbb{H}$ is a coarse embedding but not a quasi-isometric embedding.
\begin{center}
\includegraphics[trim=0 16cm 0 0,clip,width=\linewidth]{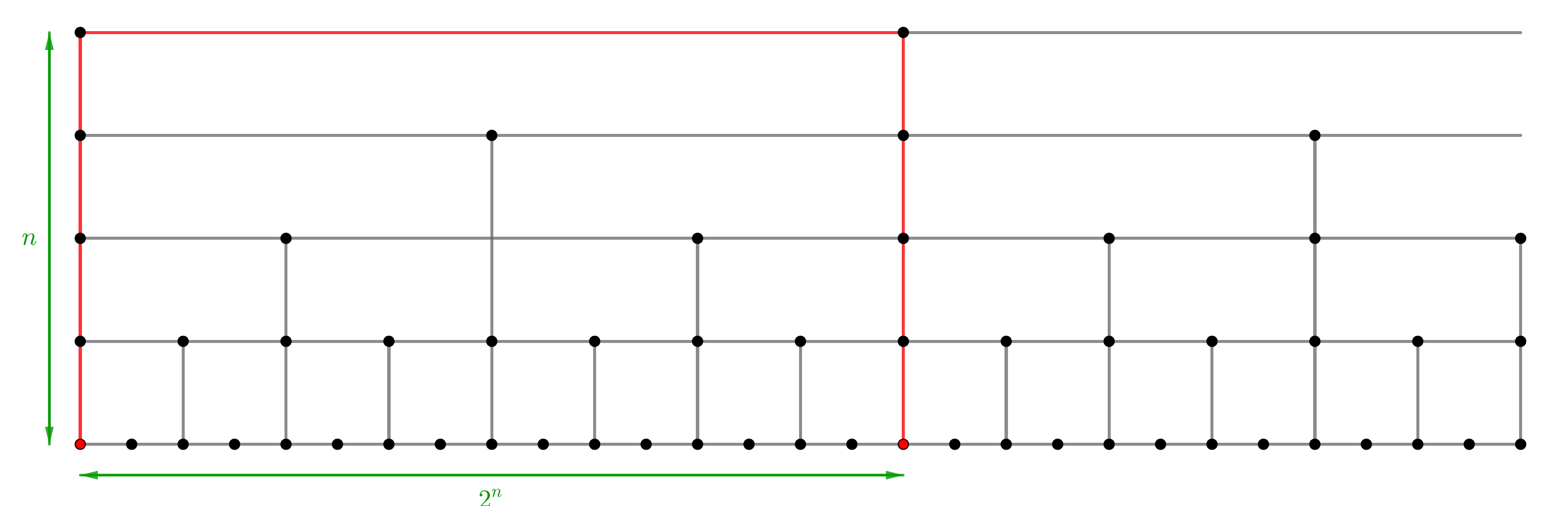}
\end{center}
In fact, $\mathbb{H}$ turns out to be quasi-isometric to the hyperbolic plane $\mathbb{H}^2$ and $L$ corresponds to an horocycle. Our graph exemplifies the fact that, in a hyperbolic space $\mathbb{H}^n$, horospheres yield coarse embeddings $\mathbb{Z}^n \hookrightarrow \mathbb{H}^n$ that are not quasi-isometric embeddings.
\end{ex}

\medskip \noindent
Examples of quasi-isometries between finitely generated groups include:

\begin{lemma}\label{lem:QIComm}
Let $G$ be a finitely generated group.
\begin{itemize}
	\item[(i)] If $H \leq G$ is a finite-index subgroup, then $H$ is finitely generated and the inclusion map $H \hookrightarrow G$ is a quasi-isometry.
	\item[(ii)] If $N \lhd G$ is a finite normal subgroup, then $G/N$ is finitely generated and the quotient map $G \twoheadrightarrow G/N$ is a quasi-isometry. 
\end{itemize}
\end{lemma}

\begin{proof}
Fix a finite generating set $S \subset G$. We start by proving (i). So let $H \leq G$ be a finite-index subgroup. Fix a set of representatives $H=Ha_1, \ldots, Ha_k$ of $G$ modulo $H$ and note $L:= \max \{ \|a_i\|_S \mid 1 \leq i \leq k\}$. We claim that the finite subset
$$R:= \{ h \in H \mid \|h\|_S \leq 1+2L \}$$
generates $H$. So let $h \in H$ be an arbitrary element. Write $h$ as a product $s_1 \cdots s_n$ with $s_1, \ldots, s_n \in S \cup S^{-1}$ and $n=\|h\|_S$. For every $0 \leq i \leq n$, there exists $1 \leq r(i) \leq k$ such that $s_1 \cdots s_i \in Ha_{r(i)}$, say $s_1 \cdots s_i = h_ia_{r(i)}$ for some $h_i \in H$. 
\begin{center}
\includegraphics[width=0.6\linewidth]{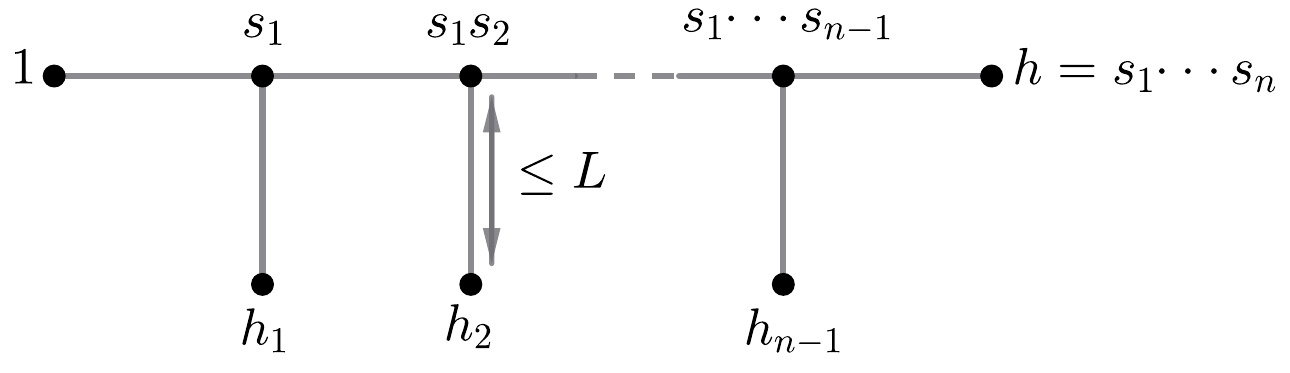}
\end{center}
Notice that $h_0=1$. Since $s_1 \cdots s_n \in H$, we can assume that $h_n=h$. Then
$$h= (h_0^{-1}h_1) \cdot (h_1^{-1}h_2) \cdot (h_2^{-1}h_3) \cdots (h_{n-1}^{-1}h_n)$$
where, for every $0 \leq i \leq n-1$, we have
$$\| h_i^{-1} h_i\|_S \leq \| h_i^{-1}s_1 \cdots s_i \|_S + \|s_{i+1}\|_S + \|s_{i+1}^{-1} \cdots s_1^{-1} h_{i+1}\|_S \leq 1+ 2 L,$$
hence $h_i^{-1}h_i \in R$. Thus, it is possible write $h$ as a product of $n= \|h\|_S$ elements from $R$. We conclude that $R$ indeed generates $H$, but also that $\|\cdot\|_R \leq \| \cdot \|_S$. Next, given an element $h \in H$, which we write as $r_1 \cdots r_m$ where $r_1, \ldots, r_m \in R$ and $m=\|h\|_R$, we have
$$\|h\|_S \leq \sum\limits_{i=1}^m \|r_i\|_S \leq (1+2L)m = (1+2L) \cdot \|h\|_R$$
by definition of $R$. Thus, the word length $\|\cdot\|_R$ and $\|\cdot\|_S$ are biLipschitz equivalent on $H$, proving that the inclusion map $H \hookrightarrow G$ is a quasi-isometric embedding. Notice that, given an arbitrary element $g \in G$, we can write $g=ha_i$ for some $h \in H$ and $1 \leq i \leq k$, hence 
$$d_{\mathrm{Cayl}(G,S)}(g,H) \leq d_{\mathrm{Cayl}(G,S)}(ha_i,h) = \|a_i\|_S \leq L.$$
Therefore, the inclusion map $H \hookrightarrow G$ is actually a quasi-isometric embedding. This concludes the proof of (i).

\medskip \noindent
Now, let us prove (ii). Let $N \lhd G$ be a finite normal subgroup and let $\pi : G \twoheadrightarrow G/N$ denote the quotient map. Notice that, given a finite generating set $S \subset G$, clearly $\pi(S)$ is a finite generating set of $G/N$. Since $\pi$ is surjective, it only remains to verify that $\pi$ is a quasi-isometric embedding. So let $g \in G$ be an arbitrary element. First, write $g=s_1 \cdots s_n$ with $s_1, \ldots, s_n \in S \cup S^{-1}$ and $n= \|g\|_S$. We have
$$\| \pi(g) \|_{\pi(S)} \leq \sum\limits_{i=1}^n \| \pi(s_i)\|_{\pi(S)} = n = \|g\|_S.$$
Next, write $\pi(g)= \pi(r_1) \cdots \pi(r_m)$ with $r_1, \ldots, r_m \in S \cup S^{-1}$ and $m= \|\pi(g)\|_{\pi(S)}$. Since $\pi(g)= \pi(r_1 \cdots r_m)$, $g$ must belong to $r_1 \cdots r_m N$, say $g=r_1 \cdots r_m h$ for some $h \in N$. Hence
$$\|g\|_S \leq \sum\limits_{i=1}^m \|r_i\|_S  + \|h \|_S \leq \|\pi(g)\|_{\pi(S)} + D$$
where $D:= \max \{ \|k\|_S \mid k \in N\}$, which is finite as $N$ is finite. This concludes the proof of (ii). 
\end{proof}

\noindent
However, coarse embeddings between groups may not be quasi-isometric embeddings, as justified by the following example. (See also Proposition~\ref{prop:DistortedZwrZ} below.)

\begin{ex}\label{ex:DistortionBS}
Given an integer $n \geq 2$, consider the Baumslag-Solitar group
$$\mathrm{BS}(1,2):= \langle a,t \mid tat^{-1}= a^2 \rangle.$$
We claim that the inclusion map $\langle a \rangle \hookrightarrow \mathrm{BS}(1,2)$ is not a quasi-isometric embedding. Notice that, for every $n \geq 1$,
$$\begin{array}{lcl} t^nat^{-n} & = & \displaystyle t^{n-1} a^2 t^{-n+1} = \left( t^{n-1}a t^{-n+1} \right)^2 = \left( t^{n-2}a^2t^{-n+2} \right)^2 \\ \\ & = & \displaystyle \left( t^{n-2} a t^{-n+2} \right)^4 = \cdots = a^{2^n}, \end{array}$$
hence $\|a^{2^n}\|_{\{a\}} = 2^n$ but $\|a^{2^n}\|_{\{a,t\}} \leq 1+2n$. So there is no constants $A$ and $B$ such that $\|\cdot\|_{\{a\}} \leq A \cdot \|\cdot\|_{\{a,t\}} +B$ on $\langle a \rangle$. 
\end{ex}

\noindent
Despite the fact that subgroups may not be quasi-isometrically embedded, they are always coarsely embedded.

\begin{lemma}
Let $G$ be a finitely generated group and $H \leq G$ an arbitrary finitely generated subgroup. The inclusion map $H \hookrightarrow G$ is a coarse embedding.
\end{lemma}

\begin{proof}
Fix two finite generating sets $R \subset H$ and $S \subset G$. Since we can choose our generating set without modifying the metric up to biLipschitz equivalence, we can assume that $R \subset S$ (up to replacing $S$ with $R \cup S$). We have
$$\rho(d_H(a,b)) \leq d_G(a,b) \leq d_H(a,b) \text{ for all } a,b \in H$$
where $d_G$ (resp.\ $d_H$) denotes the word-metric with respect to $S$ (resp.\ $R$) and where
$$\rho(k):= \min \{ d_G(a,b) \mid a,b \in H \text{ such that } d_H(a,b)=k \} \text{ for every } k \geq 0.$$
In order to conclude that the inclusion map $H \hookrightarrow G$ is a coarse embedding, it suffices to verify that $\rho(k) \to + \infty$ as $k \to + \infty$. If it is not the case, then we can find a constant $D \geq 0$ and two sequences $(a_n)_{n \geq 0}, (b_n)_{n \geq 0}$ of elements in $H$ such that, for every $ n \geq 0$, $d_G(a_n,b_n) \leq D$ but $d_H(a_n,b_n) \geq n$. Since $\|a_n^{-1}b_n\|_S \leq D$ for every $n \geq 0$, necessarily $a_n^{-1}b_n$ can take only finitely many values (as balls in $G$ are finite). A fortiori, $d_H(a_n,b_n) = \| a_n^{-1}b_n\|_R$ also takes finitely many values, a contradiction with the fact that $d_H(a_n,b_n) \geq n$ for every $n \geq 0$. 
\end{proof}

\noindent
But let us come back to lamplighters. The central question addressed by this mini-course is: when are two finitely generated wreath products quasi-isometric? A first step in this direction is the following observation:

\begin{prop}[\cite{MR1800990}]\label{prop:biLipWreath}
Let $(X_1,o_1), (X_2,o_2)$ be two pointed graphs and $Y_1,Y_2$ two graphs. If $X_1$ and $Y_1$ are respectively biLipschitz equivalent to $X_2$ and $Y_2$, then $(X_1,o) \wr Y_1$ and $(X_2,o) \wr Y_2$ are biLipschitz equivalent as well. 
\end{prop}

\noindent
It is worth mentioning that we already know that, given a finitely generated group $H$ and two finite groups $F_1, F_2$, the two wreath products $F_1 \wr H$ and $F_2 \wr H$ are biLipschitz equivalent whenever $|F_1|=|F_2|$. Indeed, fixing a finite generating set $S \subset H$, it follows from Proposition~\ref{prop:CaylWreath} that the Cayley graphs $\mathrm{Cayl}(F_1 \wr H, F_1 \cup S)$ and $\mathrm{Cayl}(F_2 \wr H, F_2 \cup S)$ are both isomorphic to a wreath product $(K_n,o) \wr \mathrm{Cayl}(H,S)$ where $K_n$ denotes the complete graph of size $|F_1|=|F_2|$, or equivalently to the lamplighter graph $\mathcal{L}_n(\mathrm{Cayl}(H,S))$. So $F_1 \wr H$ and $F_2 \wr H$ turn out to share a common Cayley graph. 

\medskip \noindent
In order to prove the general case of our proposition, we first need to understand distances in wreath products of graphs. This is the purpose of our next lemma.

\begin{lemma}\label{lem:DistWreath}
Let $(X,o)$ be a pointed graph and $Y$ a graph. The equality
$$d((c_1,p_1),(c_2,p_2)) = \mathrm{TS}(p_1, c_1 \triangle c_2, p_2) + \sum\limits_{q \in c_1 \triangle c_2} d(c_1(q),c_2(q))$$
holds for all vertices $(c_1,p_1),(c_2,p_2) \in (X,o) \wr Y$. 
\end{lemma}

\noindent
Here, given two colourings $c_1,c_2 : V(Y) \to V(X)$, $c_1 \triangle c_2$ refers to the set of vertices where $c_1$ and $c_2$ differ; and, given two vertices $p_1,p_2 \in V(Y)$ and a set of vertices $S \subset V(Y)$, $\mathrm{TS}(p_1,S,p_2)$ refers to the shortest length of path that starts from $p_1$, ends at $p_2$, and visits all the vertices in $S$ in between. (In other words, $\mathrm{TS}$ corresponds to the solution to \emph{travelling salesman problem}.) 

\medskip \noindent
The intuition behind Lemma~\ref{lem:DistWreath} should be clear. Moving in $(X,o) \wr Y$ amounts to moving an arrow in $Y$ that is allowed to modify the colours of the lamps it visits. Consequently, in order to go from $(c_1,p_1)$ to $(c_2,p_2)$, one has to move the arrow from $p_1$ to $p_2$, visiting along the way all vertices where $c_1$ and $c_2$ differ, where the arrow has to stop in order to modify the colour of $c_1$ into the colour of $c_2$. See Figure~\ref{GeodesicEx}.

\begin{figure}
\begin{center}
\begin{tabular}{|c|c|c|} \hline
\includegraphics[trim=0 18cm 23cm 0,clip,width=0.29\linewidth]{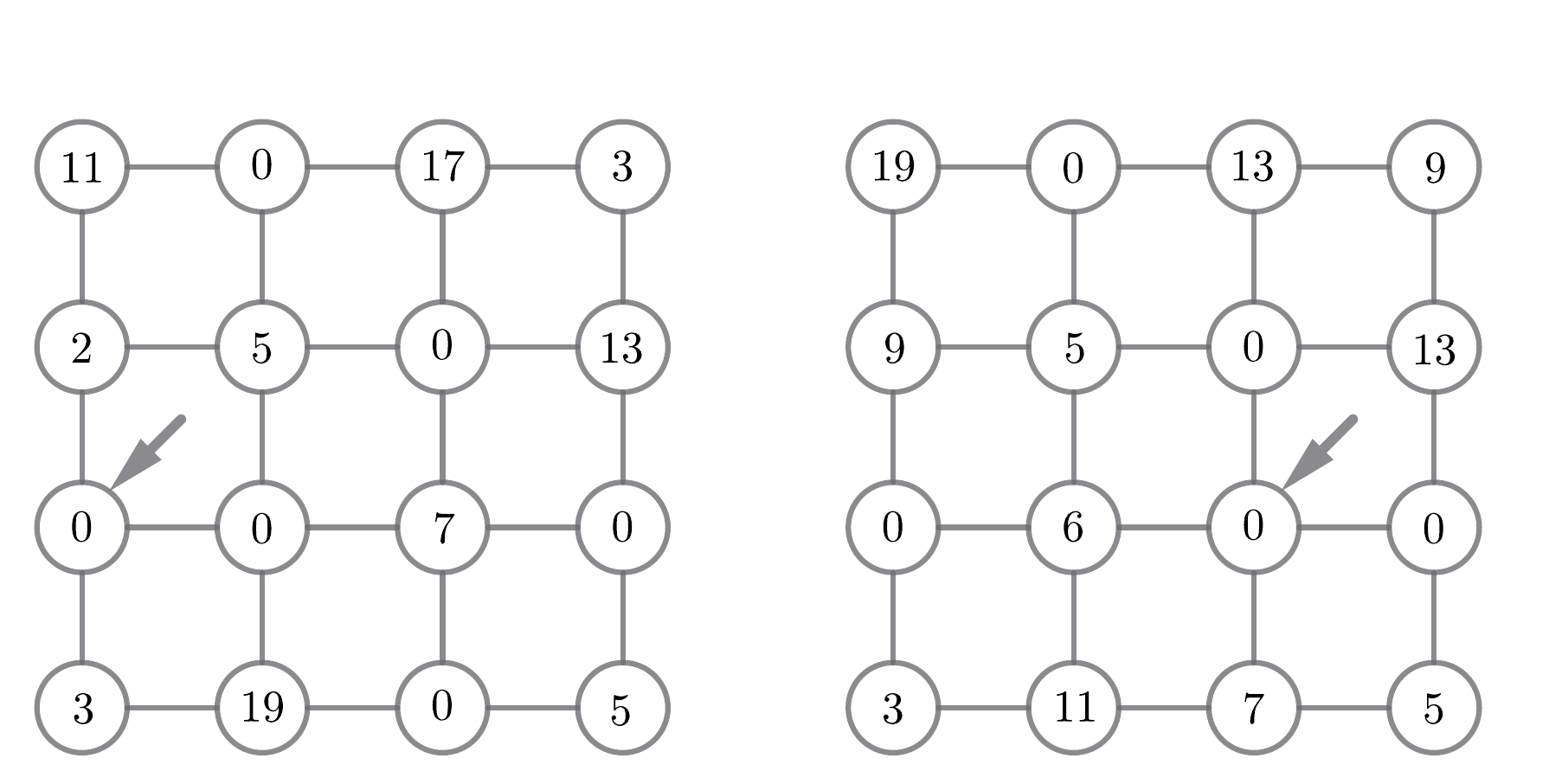} &
\includegraphics[trim=0 18cm 23cm 0,clip,width=0.29\linewidth]{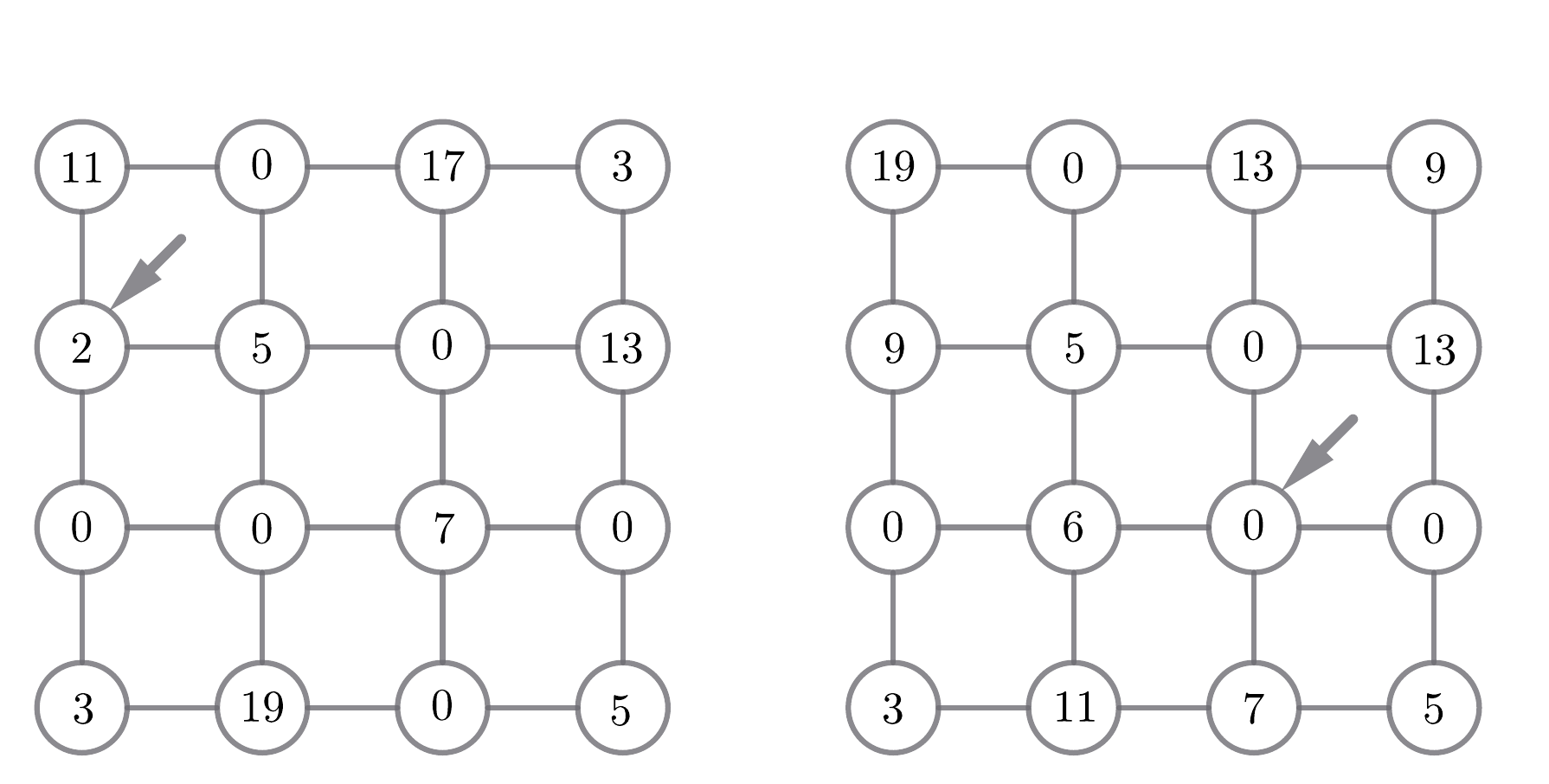} &
\includegraphics[trim=0 18cm 23cm 0,clip,width=0.29\linewidth]{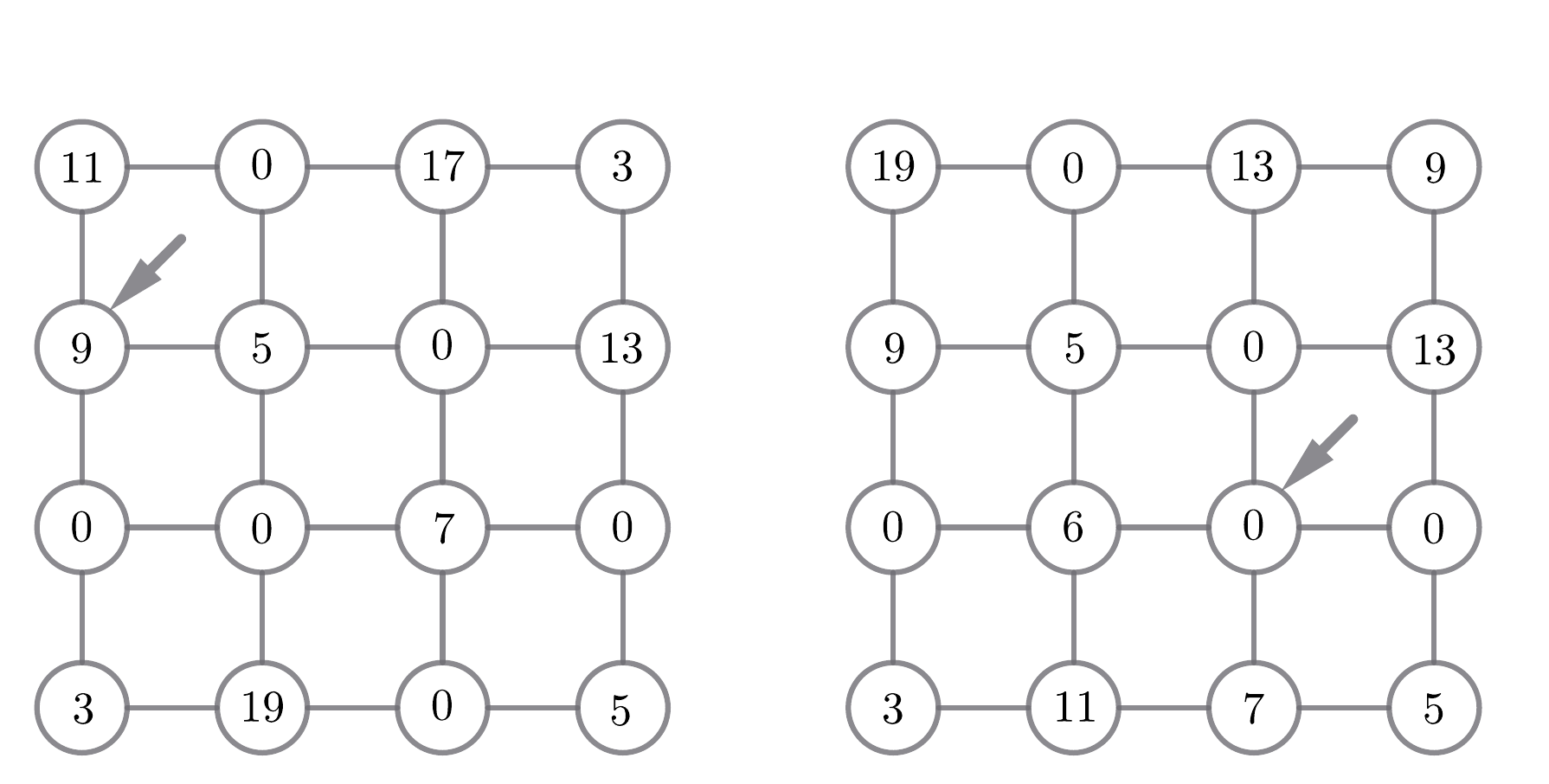} \\ \hline
\includegraphics[trim=0 18cm 23cm 0,clip,width=0.29\linewidth]{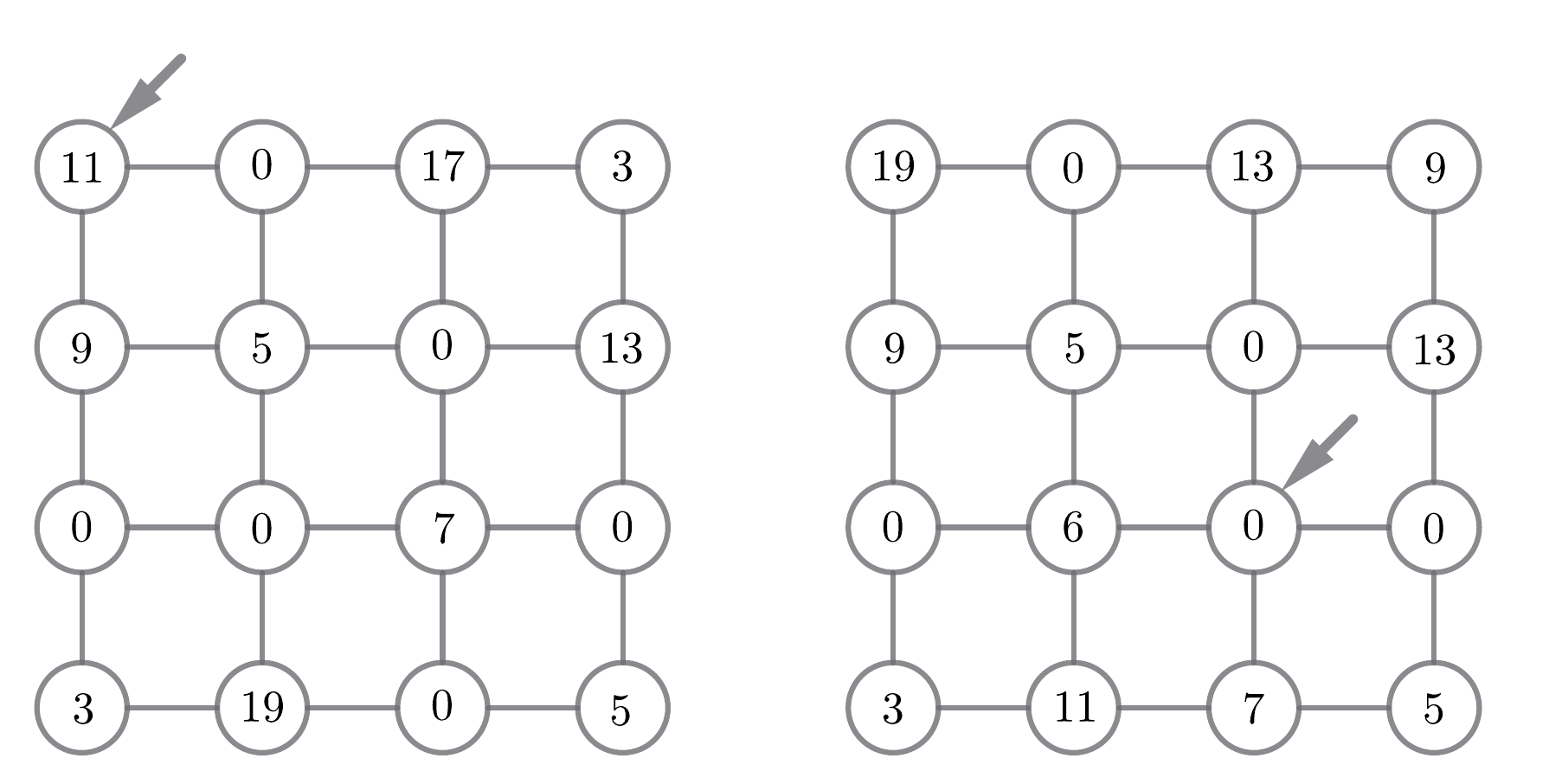} &
\includegraphics[trim=0 18cm 23cm 0,clip,width=0.29\linewidth]{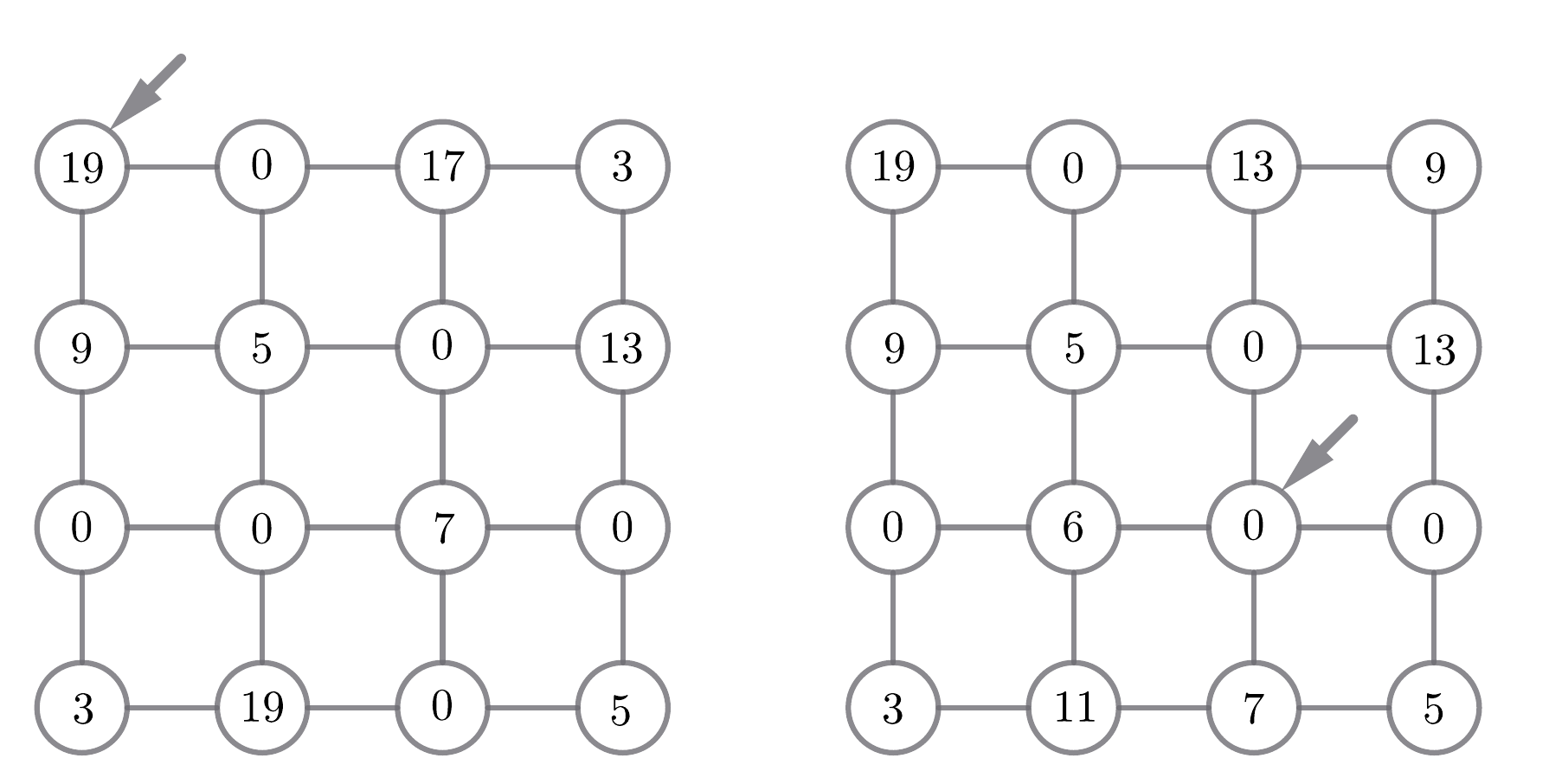} &
\includegraphics[trim=0 18cm 23cm 0,clip,width=0.29\linewidth]{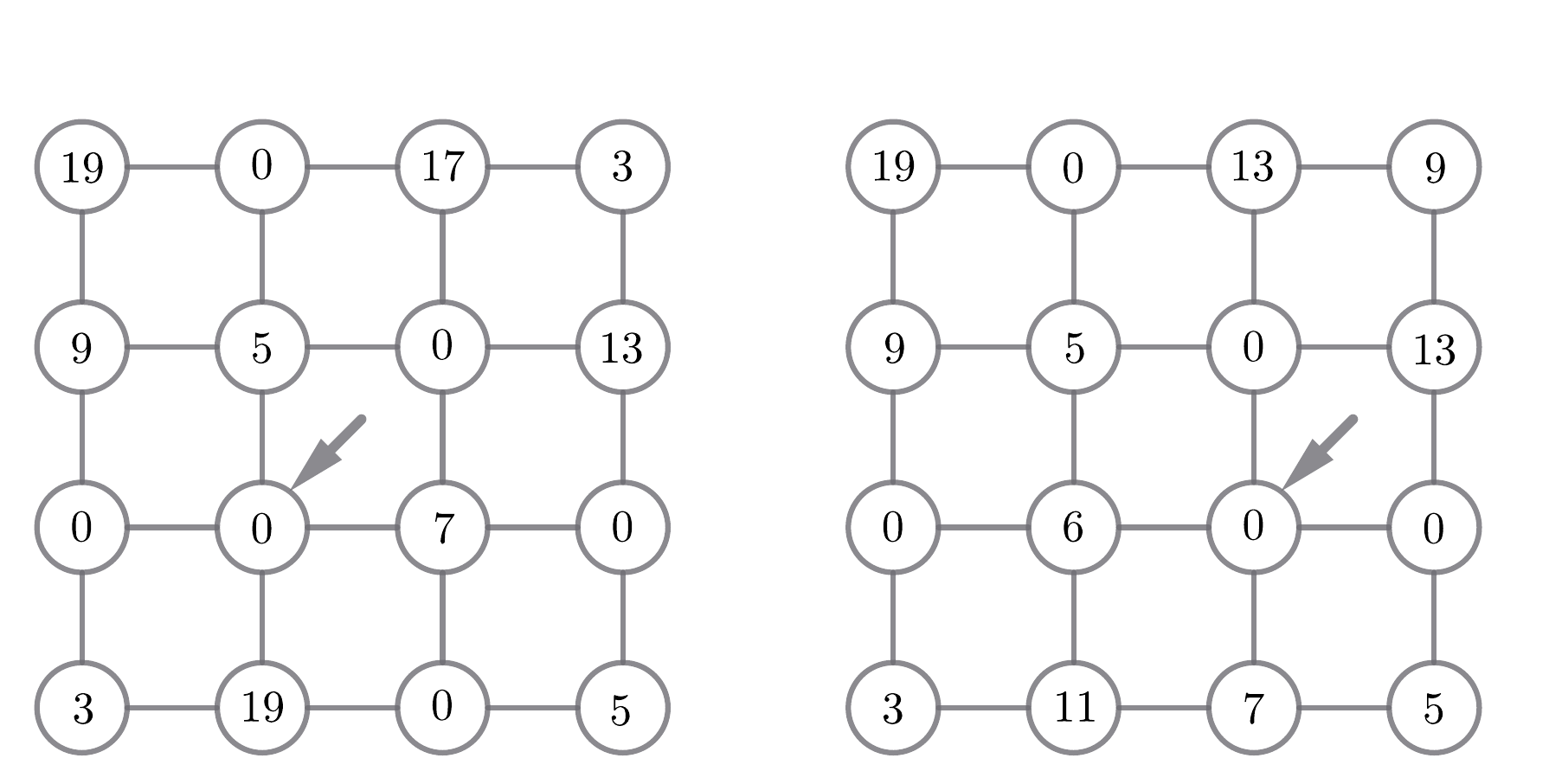} \\ \hline\includegraphics[trim=0 18cm 23cm 0,clip,width=0.29\linewidth]{G1} &
\includegraphics[trim=0 18cm 23cm 0,clip,width=0.29\linewidth]{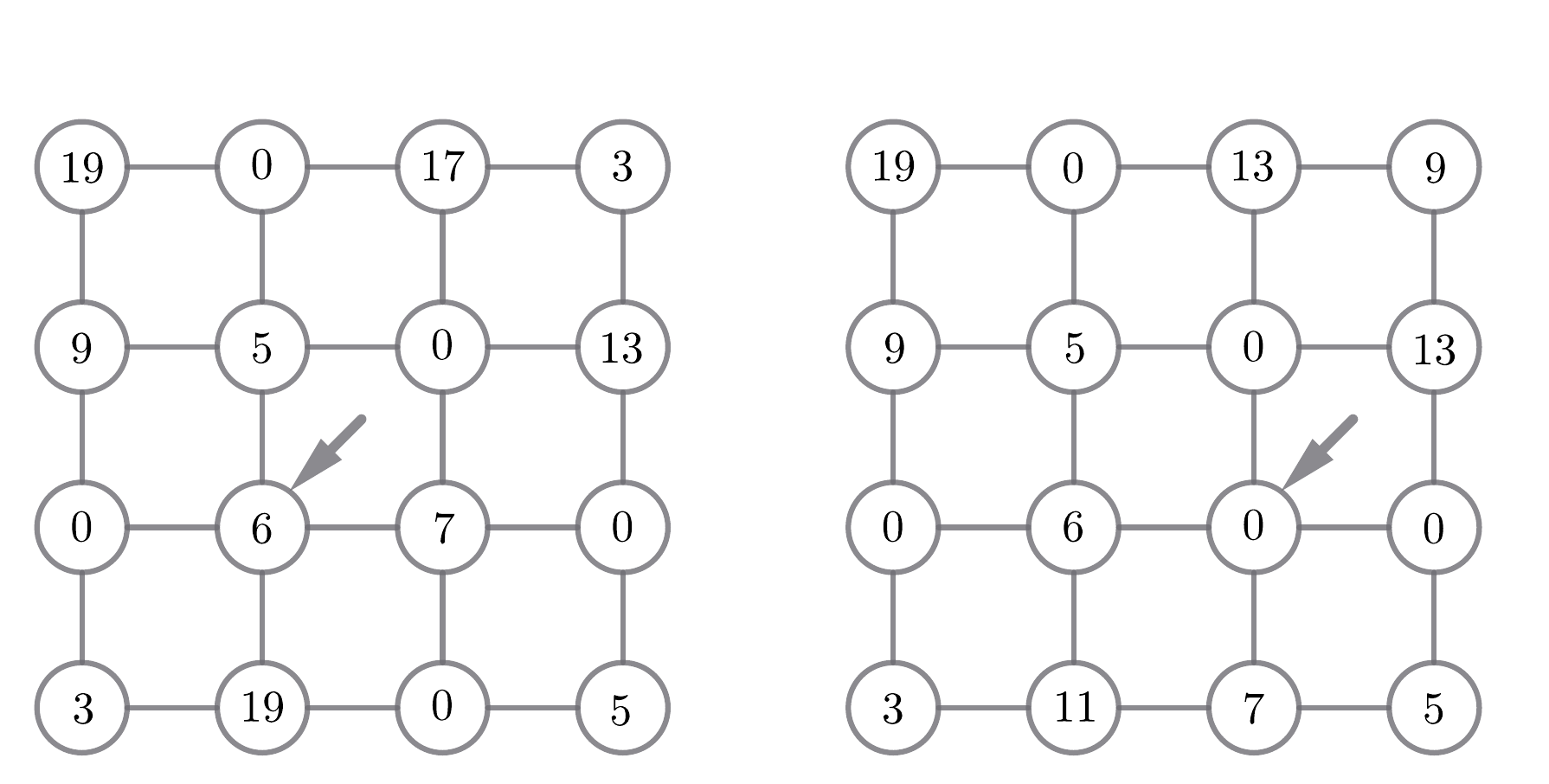} &
\includegraphics[trim=0 18cm 23cm 0,clip,width=0.29\linewidth]{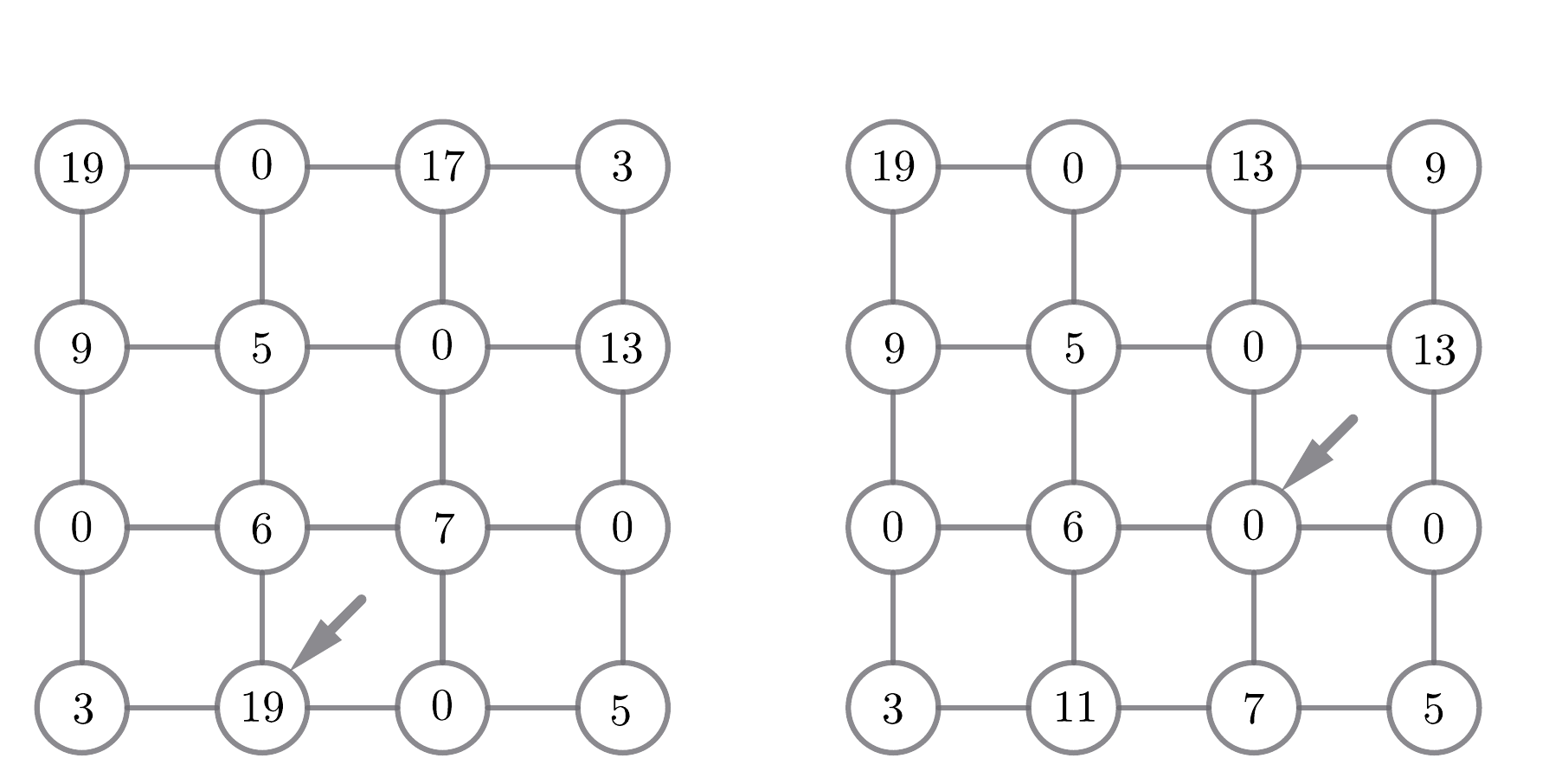} \\ \hline
\includegraphics[trim=0 18cm 23cm 0,clip,width=0.29\linewidth]{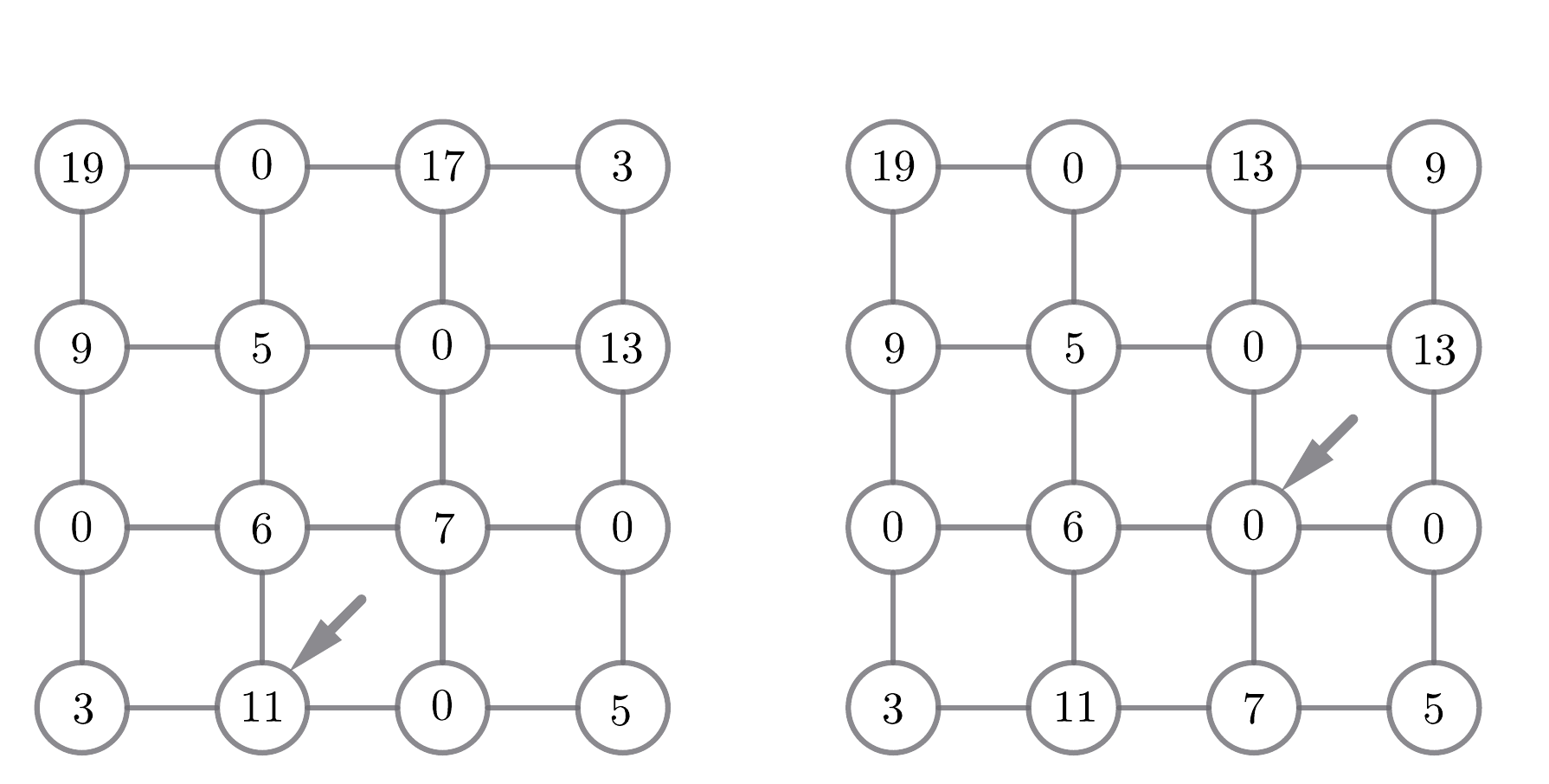} &
\includegraphics[trim=0 18cm 23cm 0,clip,width=0.29\linewidth]{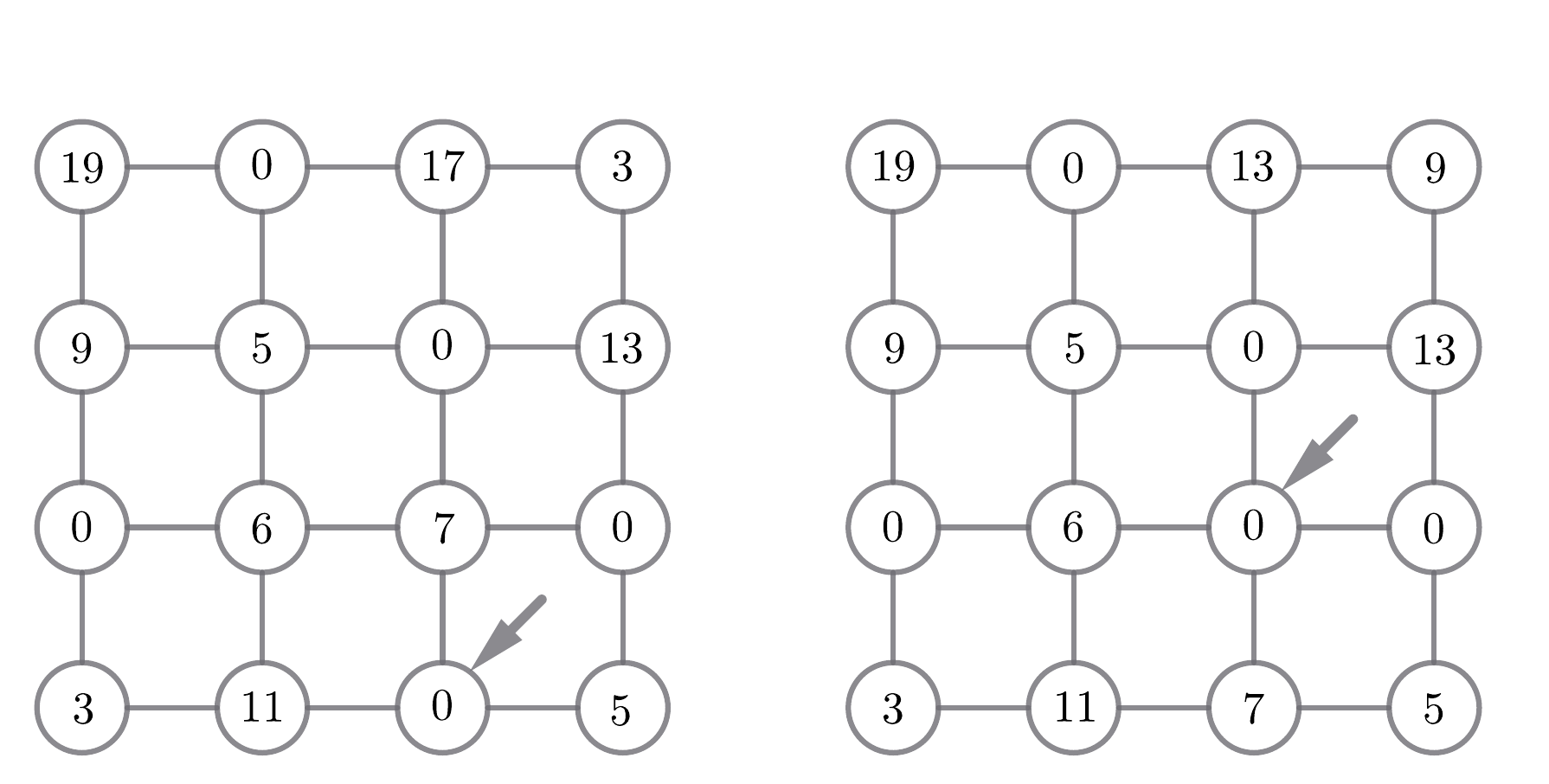} &
\includegraphics[trim=0 18cm 23cm 0,clip,width=0.29\linewidth]{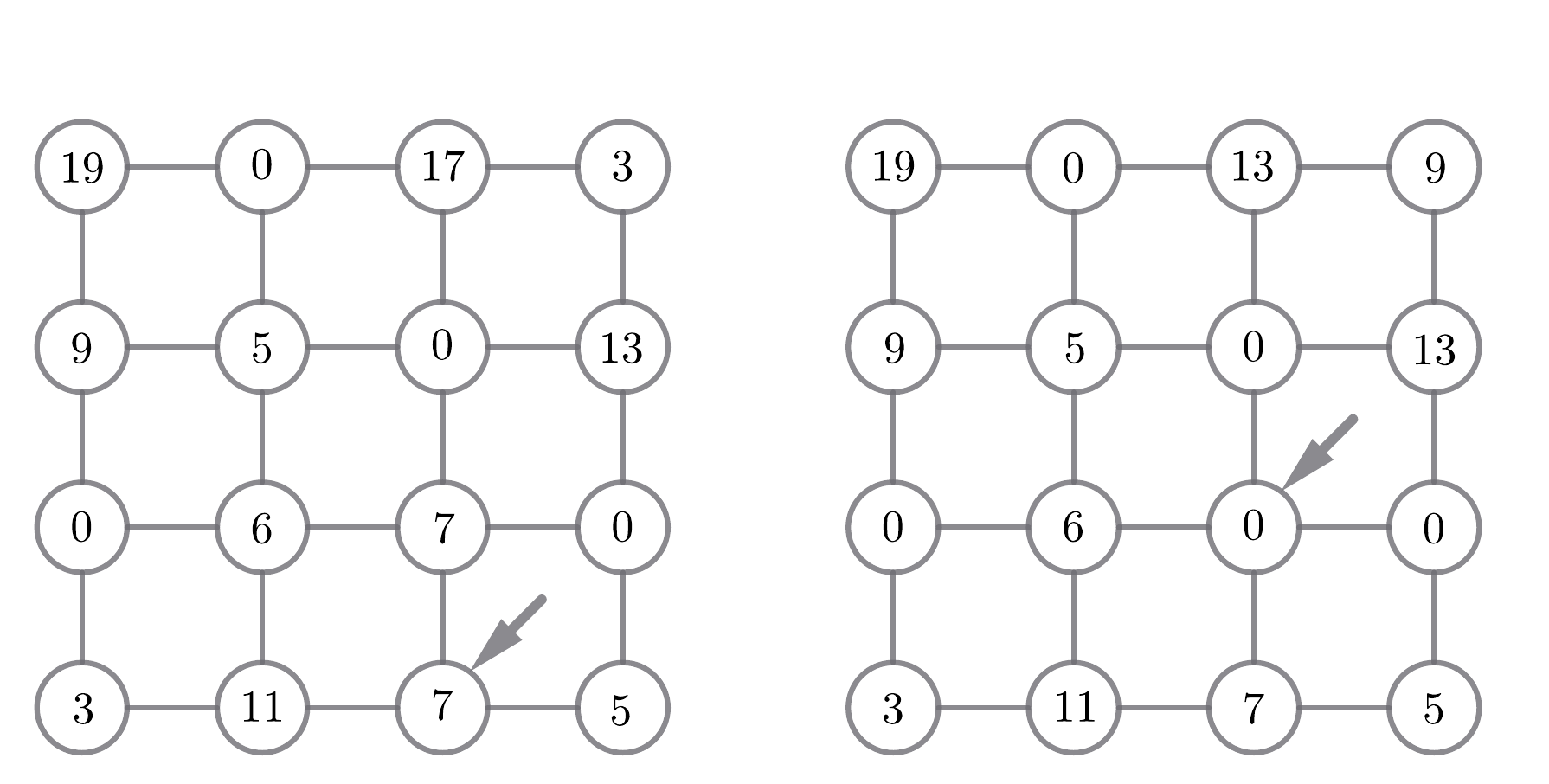} \\ \hline
\includegraphics[trim=0 18cm 23cm 0,clip,width=0.29\linewidth]{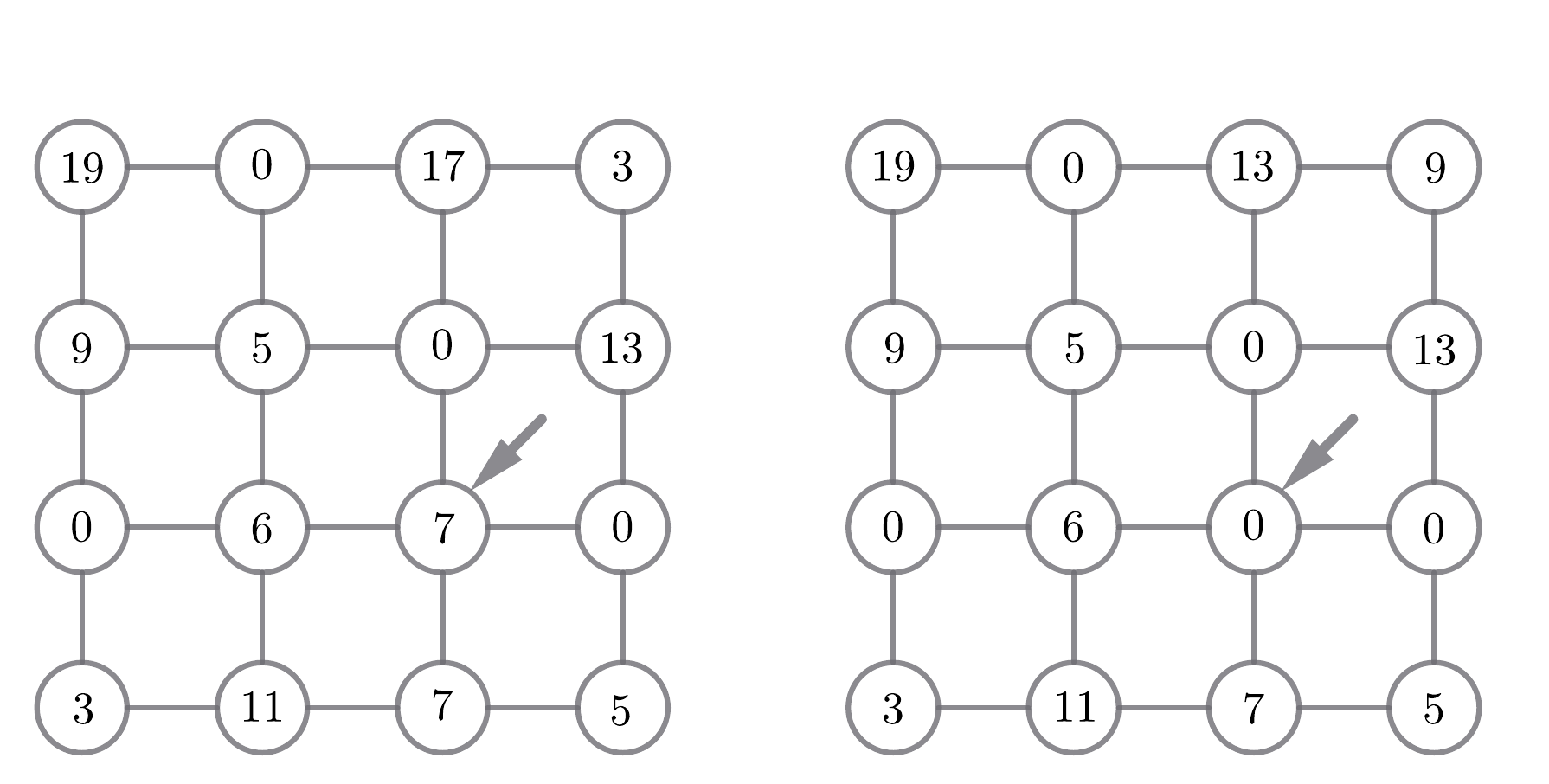} &
\includegraphics[trim=0 18cm 23cm 0,clip,width=0.29\linewidth]{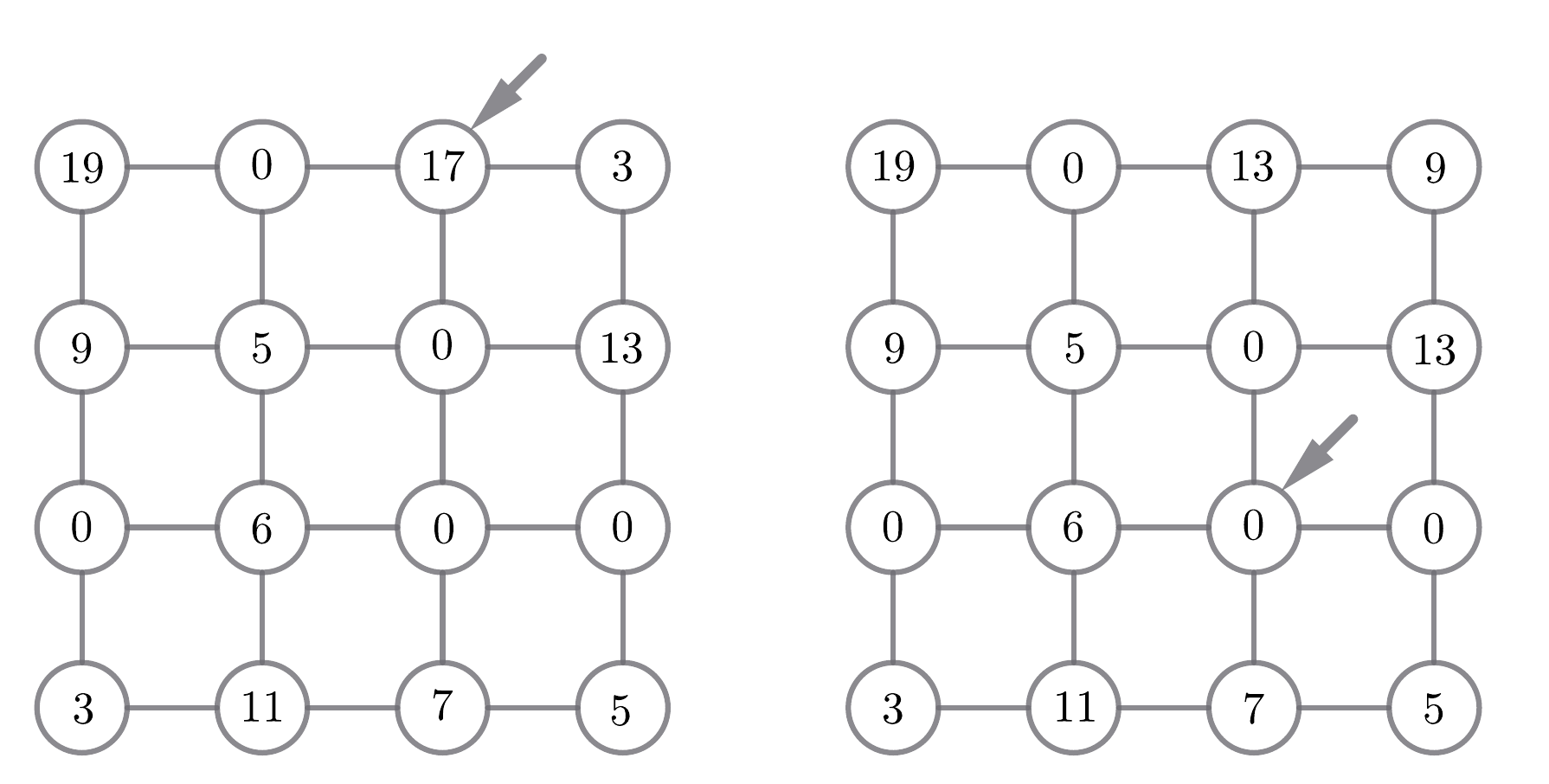} &
\includegraphics[trim=0 18cm 23cm 0,clip,width=0.29\linewidth]{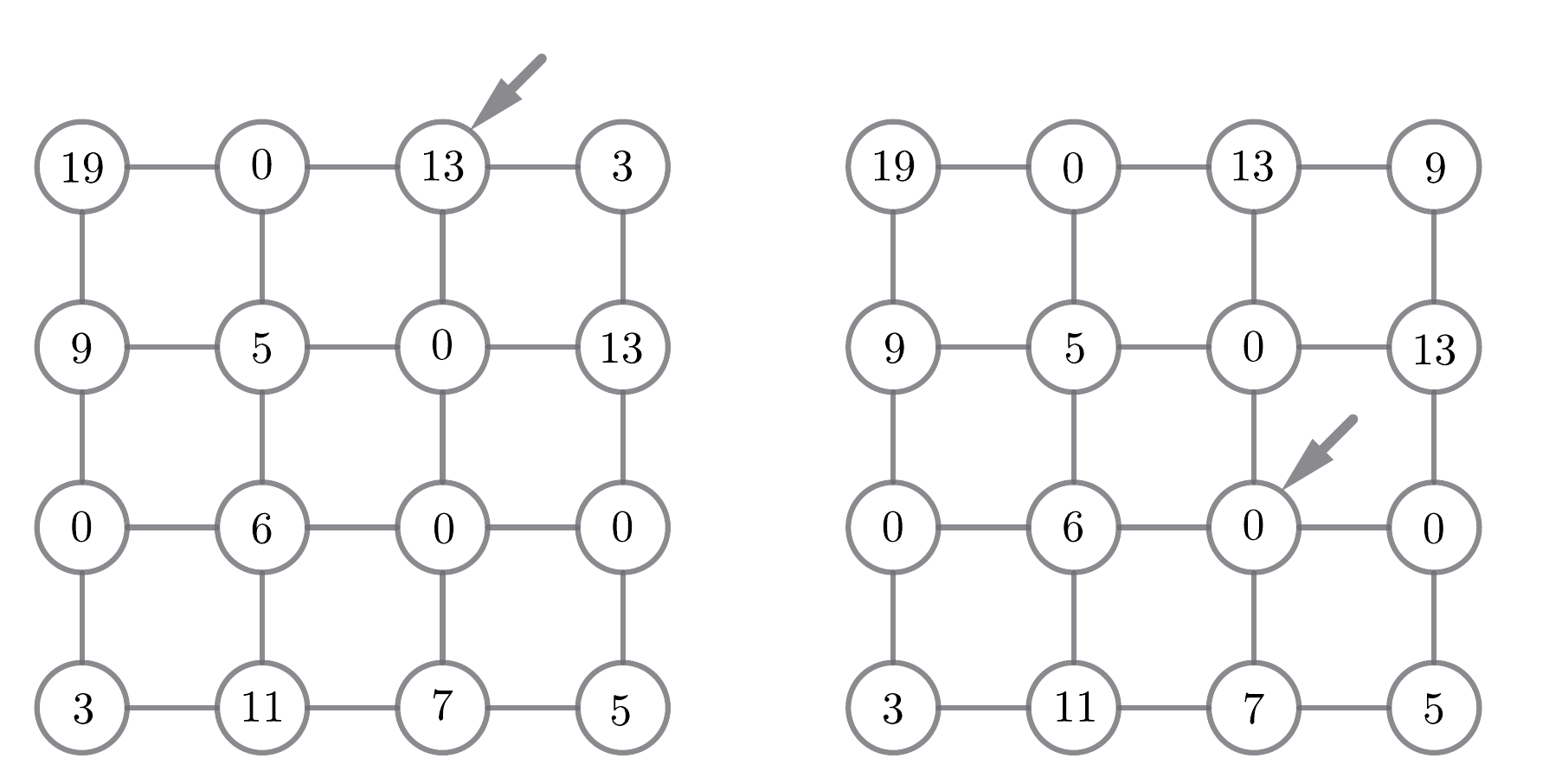} \\ \hline
\includegraphics[trim=0 18cm 23cm 0,clip,width=0.29\linewidth]{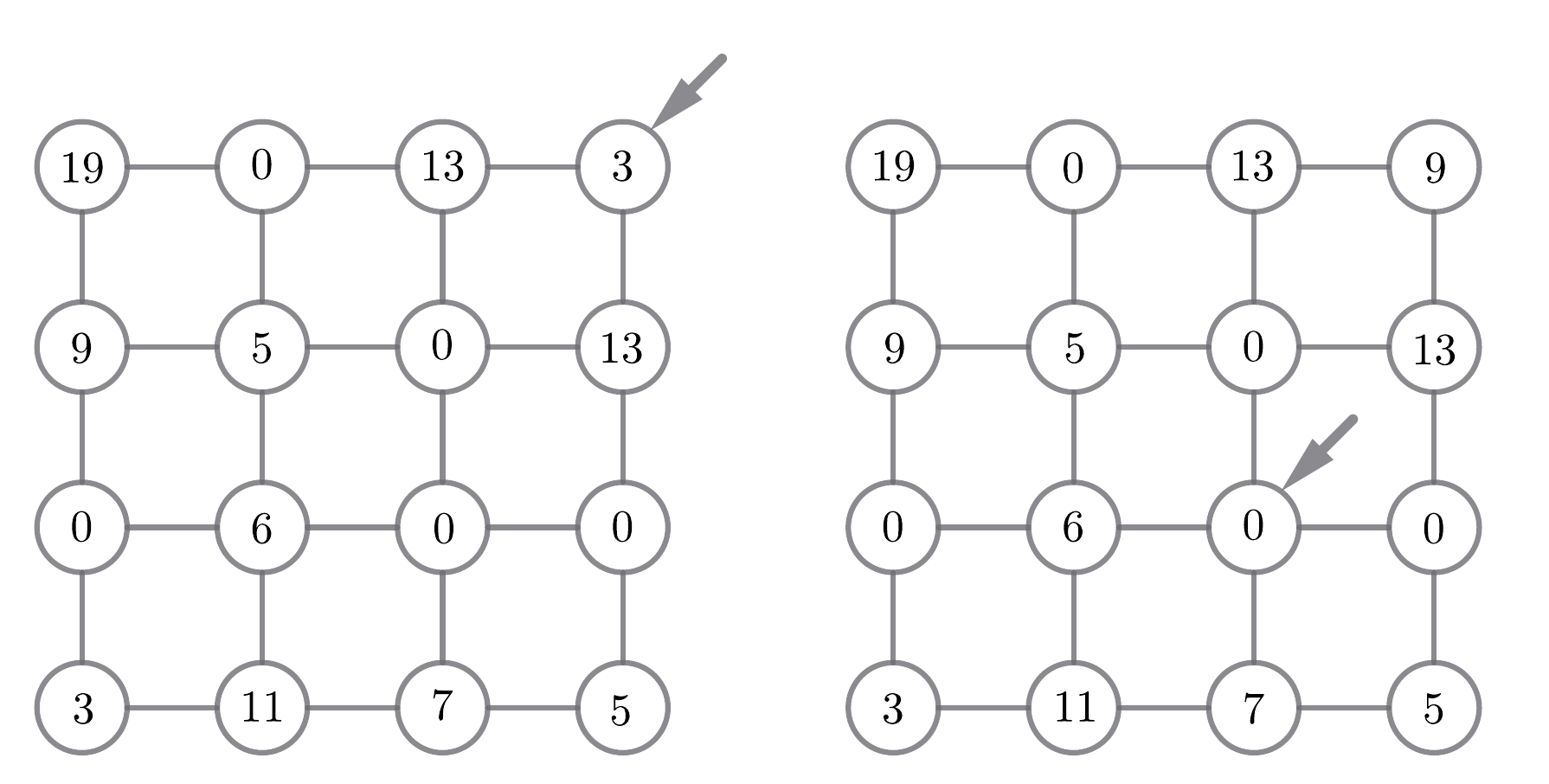} &
\includegraphics[trim=0 18cm 23cm 0,clip,width=0.29\linewidth]{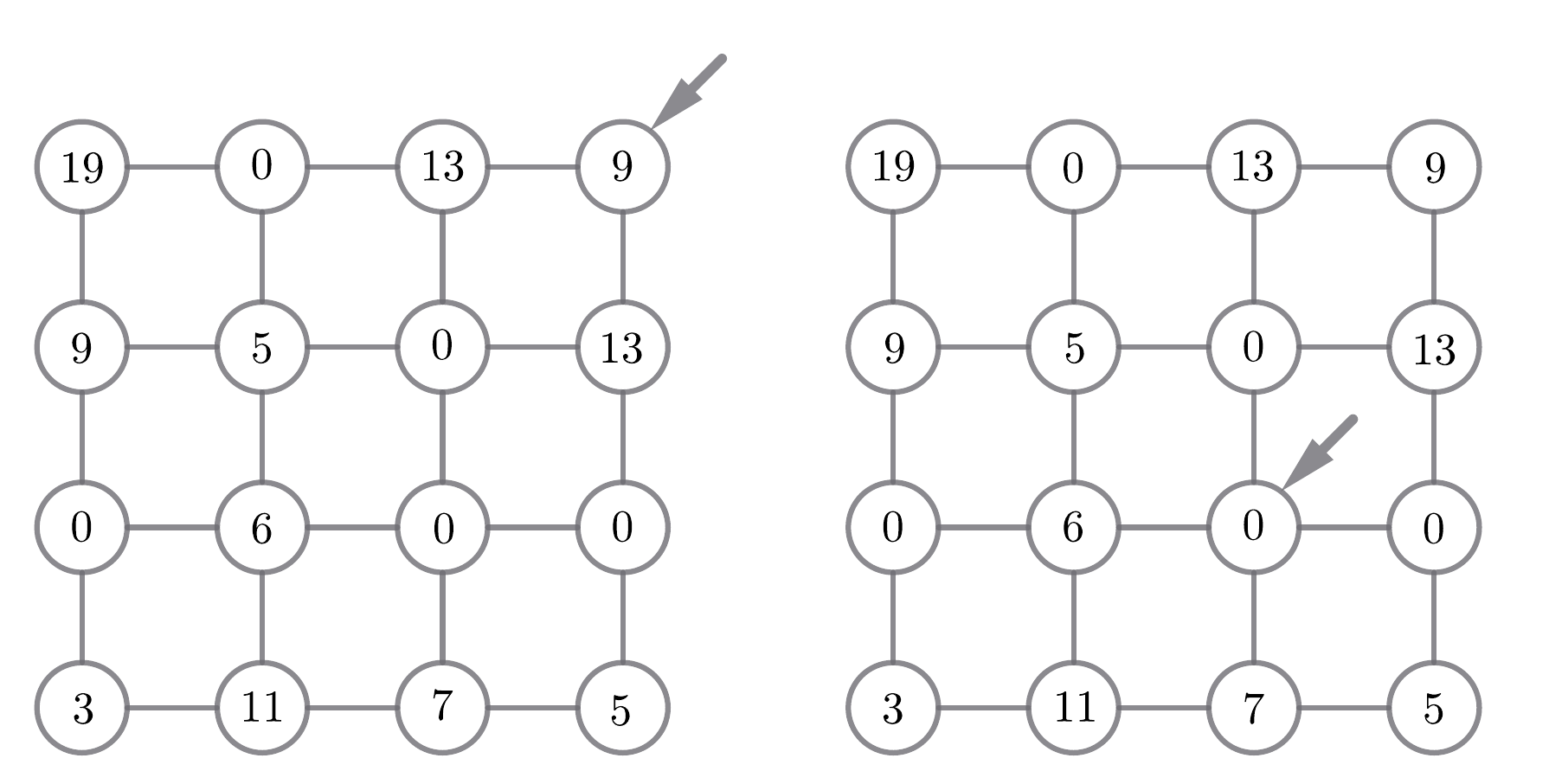} &
\includegraphics[trim=0 18cm 23cm 0,clip,width=0.29\linewidth]{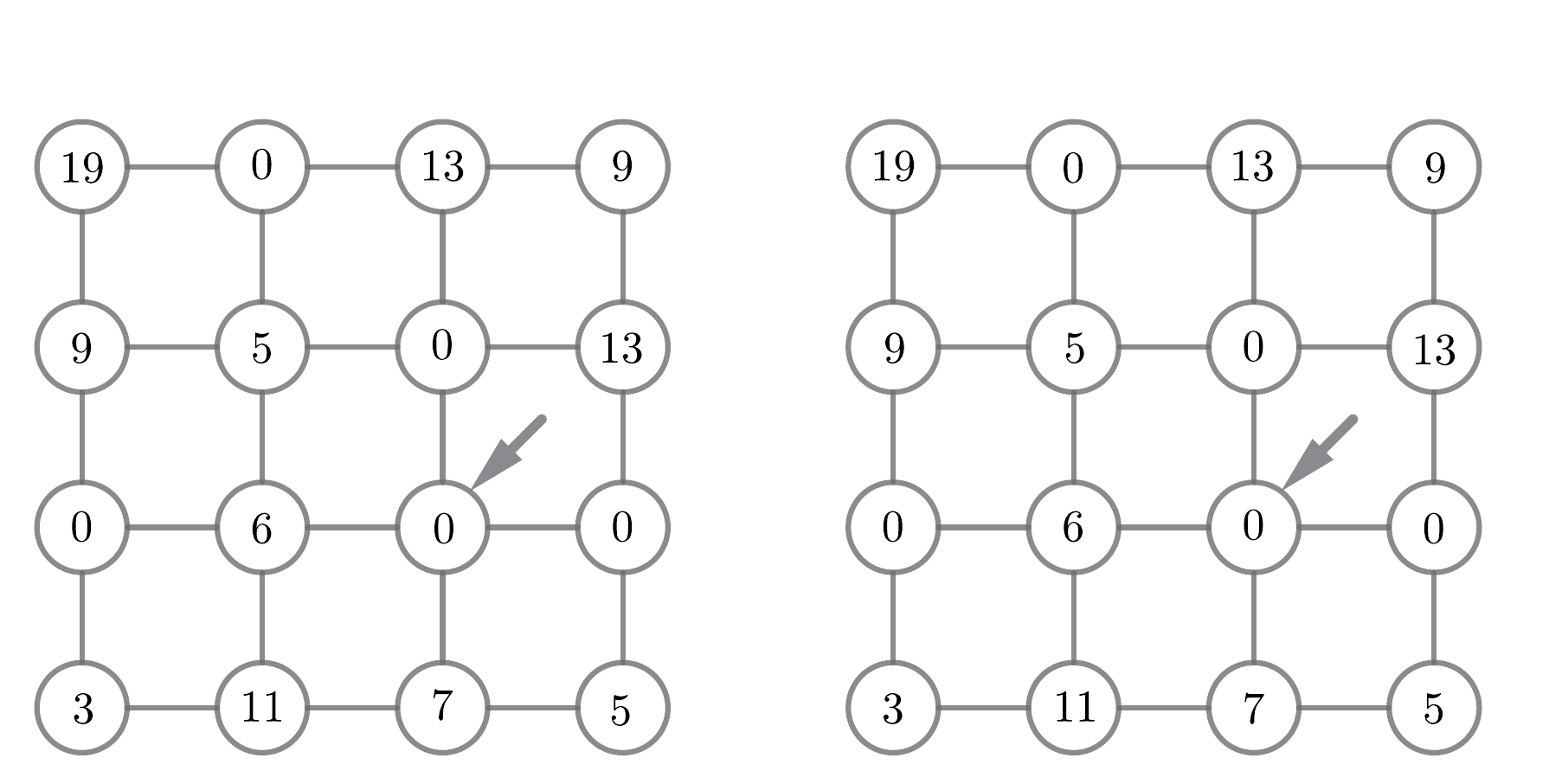} \\ \hline
\end{tabular}
\caption{A geodesic in the lamplighter graph over a grid.}
\label{GeodesicEx}
\end{center}
\end{figure}

\medskip \noindent
The two terms in the sum in Lemma~\ref{lem:DistWreath} refers to the cost of moving the arrow and to the cost of modifying the colouring. Below is a formal argument.

\begin{proof}[Proof of Lemma~\ref{lem:DistWreath}.]
Fix two vertices $(c_1,p_1),(c_2,p_2) \in (X,o) \wr Y$. A path $\gamma$ in $(X,o) \wr Y$ connecting $(c_1,p_1)$ to $(c_2,p_2)$ amounts to move the arrow from $p_1$ to $p_2$ in $Y$ and, along the way, to modify the colours of the lamps where $c_1$ and $c_2$ differ. Thus, we can distinguish different types of edges along $\gamma$: edges of lamp-type, corresponding to a move of the arrow; and, for every $q \in c_1 \triangle c_2$, edges of type $q$, corresponding to a modification of the colour of the lamp at $q$. Clearly, there must be at least $\mathrm{TS}(p_1, c_1 \triangle c_2,p_2)$ edges of lamp-type and, for every $q \in c_1 \triangle c_2$, at least $d(c_1(q),c_2(q))$ edges of type $q$, hence
$$d((c_1,p_1),(c_2,p_2)) \geq \mathrm{TS}(p_1, c_1 \triangle c_2, p_2) + \sum\limits_{q \in c_1 \triangle c_2} d(c_1(q),c_2(q)).$$
In order to prove the reverse equality, it suffices to construct a path of length 
$$\mathrm{TS}(p_1, c_1 \triangle c_2, p_2) + \sum_{q \in c_1 \triangle c_2} d(c_1(q),c_2(q))$$ 
connecting $(c_1,p_1)$ to $(c_2,p_2)$. Fix an enumeration $c_1 \triangle c_2 = \{v_1, \ldots, v_n\}$ such that 
$$\mathrm{TS}(p_1, c_1 \triangle c_2, p_2) = d(p_1,v_1) + \sum\limits_{i=1}^{n-1} d(v_i,v_{i+1}) + d(v_n,p_2).$$
Now, we define a path $\alpha$ connecting $(c_1,p_1)$ to $(c_2,p_2)$ in $(X,o) \wr Y$ thanks to the following sequence of operations:
\begin{itemize}
	\item we move the arrow from $p_1$ to $v_1$ along some geodesic in $Y$, and we modify the colour at $v_1$ from $c_1(v_1)$ to $c_2(v_2)$ by following a geodesic in $X$;
	\item we move the arrow from $v_1$ to $v_2$ along some geodesic in $Y$, and we modify the colour at $v_2$ from $c_1(v_2)$ to $c_2(v_2)$ by following a geodesic in $X$;
	\item etc.
	\item we move the arrow from $v_{n-1}$ to $v_n$ along some geodesic in $Y$, and we modify the colour at $v_n$ from $c_1(v_n)$ to $c_2(v_n)$ by following a geodesic in $X$;
	\item we move the arrow from $v_n$ to $p_2$.
\end{itemize}
Then, our path $\alpha$ has the desired length, concluding the proof of our lemma. 
\end{proof}

\begin{proof}[Proof of Proposition~\ref{prop:biLipWreath}.]
Fix two biLipschitz equivalences $\alpha : X_1 \to X_2$ and $\beta : Y_1 \to Y_2$, and define
$$\Psi : \left\{ \begin{array}{ccc} (X_1,o_1) \wr Y_1 & \to & (X_2,o_2) \wr Y_2 \\ (c,p) & \mapsto  & (\alpha \circ c \circ \beta^{-1}, \beta(p)) \end{array} \right..$$
Checking that $\Psi$ defines a biLipschitz equivalence is now just a computation. Given two vertices $(c_1,p_1), (c_2,p_2) \in (X_1,o_1) \wr B_1$, we know from Lemma~\ref{lem:DistWreath} that
$$\begin{array}{lcl} d(\Psi(c_1,p_1),\Psi(c_2,p_2)) & = & \mathrm{TS} \left(\beta(p_1), \alpha \circ c_1 \circ \beta^{-1} \triangle \alpha \circ c_2 \circ \beta^{-1}, \beta(p_2) \right) \\ \\ & & \displaystyle  + \sum\limits_{q \in \alpha \circ c_1 \circ \beta^{-1} \triangle \alpha \circ c_2 \circ \beta^{-1}} d(\alpha \circ c_1 \circ \beta^{-1} (q), \alpha \circ c_2 \circ \beta^{-1} (q)). \end{array}$$
Noticing that 
$$\alpha \circ c_1 \circ \beta^{-1} \triangle \alpha \circ c_2 \circ \beta^{-1} = \beta(c_1 \triangle c_2),$$
our equality simplifies as
\begin{equation}\label{eq:TSzero}
d(\Psi(c_1,p_1), \Psi(c_2,p_2)) = \mathrm{TS}\left( \beta(p_1), \beta(c_1 \triangle c_2), \beta(p_2) \right) + \sum\limits_{q \in c_1 \triangle c_2} d(\alpha(c_1(q)), \alpha(c_2(q))).
\end{equation}
First, if $L >0$ denotes a constant such that $\alpha$ and $\beta$ are $L$-biLipschitz, then we clearly have
\begin{equation}\label{eq:TSone}
\frac{1}{L} \sum\limits_{q \in c_1 \triangle c_2} d(c_1(q),c_2(q)) \leq \sum\limits_{q \in c_1 \triangle c_2} d(\alpha(c_1(q)), \alpha(c_2(q))) \leq L \sum\limits_{q \in c_1 \triangle c_2} d(c_1(q),c_2(q)).
\end{equation}
Next, there exists an enumeration $c_1 \triangle c_2 = \{r_1, \ldots, r_n\}$ such that 
$$\mathrm{TS}(\beta(p_1), \beta (c_1 \triangle c_2), \beta(p_2)) = d(\beta(p_1),\beta(r_1)) + \sum\limits_{i=1}^{n-1} d(\beta(r_i),\beta(r_{i+1})) + d(\beta(r_n),\beta(p_2)),$$
which implies that
\begin{equation}\label{eq:TStwo}
\begin{array}{lcl} \mathrm{TS}(\beta(p_1), \beta (c_1 \triangle c_2), \beta(p_2)) & \geq & \displaystyle \frac{1}{L} \left( d(p_1,r_1) + \sum\limits_{i=1}^{n-1} d(r_i,r_{i+1}) + d(r_n,p_2) \right) \\ \\ & \geq & \frac{1}{L} \mathrm{TS}(p_1, c_1 \triangle c_2,p_2) \end{array}
\end{equation}
Similarly, there exists an enumeration $c_1 \triangle c_2 = \{s_1, \ldots, s_n\}$ such that
$$\mathrm{TS}(p_1,c_1 \triangle c_2, p_2) = d(p_1,s_1) + \sum\limits_{i=1}^{n-1} d(s_i,s_{i+1}) + d(s_n, p_2),$$
from which we deduce that
\begin{equation}\label{eq:TSthree}
\begin{array}{lcl} \mathrm{TS}(p_1, c_1 \triangle c_2, p_2) & \geq  & \displaystyle \frac{1}{L} \left( d(\beta(p_1), \beta(s_1)) + \sum\limits_{i=1}^{n-1} d(\beta(s_i), \beta(s_{i+1})) + d(\beta(s_n), \beta(p_2)) \right) \\ \\ & \geq & \frac{1}{L} \mathrm{TS} (\beta(p_1), \beta(c_1 \triangle c_2), \beta(p_2)). \end{array}
\end{equation}
By combining Equations~\ref{eq:TSone}, \ref{eq:TStwo}, and~\ref{eq:TSthree} with Equation~\ref{eq:TSzero}, we deduce thanks to Lemma~\ref{lem:DistWreath} that
$$\frac{1}{L} \cdot d((c_1,p_1),(c_2,p_2)) \leq d(\Psi(c_1,p_1), \Psi(c_2,p_2)) \leq L \cdot d((c_1,p_1),(c_2,p_2)).$$
Since $\Psi$ is clearly a bijection, we conclude that $\Psi$ is an $L$-biLipschitz equivalence. 
\end{proof}

\noindent
As an application of Proposition~\ref{prop:biLipWreath}, we can exhibit various interesting examples:

\begin{cor}\label{cor:NotQIpreserved}
The following assertions hold. 
\begin{itemize}
	\item The groups $\mathbb{Z}_6 \wr \mathbb{Z}$ and $\mathrm{Sym}(3) \wr \mathbb{Z}$ are quasi-isometric but not commensurable (even up to finite kernel).
	\item The solvable group $\mathbb{Z}_{60} \wr \mathbb{Z}$ is quasi-isometric to $\mathrm{Alt}(5) \wr \mathbb{Z}$, which is not virtually solvable.
	\item The torsion-free group $\mathbb{Z} \wr \mathbb{Z}$ is quasi-isometric to $\mathbb{D}_\infty \wr \mathbb{Z}$, which is not virtually torsion-free.
\end{itemize}
\end{cor}

\noindent
Recall that two groups $G$ and $H$ are (abstractly) commensurable if they contain isomorphic finite-index subgroups; and they are commensurable up to finite kernel if one can find two finite normal subgroups $G_0 \lhd G$ and $H_0 \lhd H$ such that $G/G_0$ and $H/H_0$ are commensurable. It follows from Lemma~\ref{lem:QIComm} that two finitely generated groups that commensurable up to finite kernel are necessarily quasi-isometric. The first item provides an elementary counterexample to the converse of this assertion. It is also worth mentioning that the second assertion contrasts with Gromov's theorem about groups of polynomial growth, which implies that being virtually nilpotent is preserved by quasi-isometries. As another application of Proposition~\ref{prop:biLipWreath}, it can be shown that being residually finite is not preserved by quasi-isometries either; see Exercise~\ref{ex:WreathRF}. We leave the proof of Corollary~\ref{cor:NotQIpreserved} as an exercise for the interested reader. 

\medskip \noindent
We saw with Example~\ref{ex:DistortionBS} that subgroup may not be quasi-isometrically embedded, i.e.\ they may be \emph{distorted}. We conclude this section by providing another such example, but in the family of lamplighter groups. 

\begin{prop}[\cite{MR2811580}]\label{prop:DistortedZwrZ}
The group $\mathbb{Z} \wr \mathbb{Z}$ contains distorted copy of itself.
\end{prop}

\begin{proof}[Sketch of proof.] 
Let $a$ (resp.\ $t$) denote a generator of the left (resp.\ right) $\mathbb{Z}$-factor of $\mathbb{Z} \wr \mathbb{Z}$. Thus, thinking of elements of $\mathbb{Z}$ as pairs of a $\mathbb{Z}$-colourin of $\mathcal{Z}$ with an arrow pointing at some integer, right-multiplying with $a$ amounts to adding $1$ in the lamp where the arrow is and right-multiplying with $t$ amounts to shifting to the right the arrow. Consider the subgroup $H := \langle ata^{-1}t^{-1}, t \rangle$. The group $H$, which turns out to be isomorphic to $\mathbb{Z} \wr \mathbb{Z}$, has a lamplighter-like interpretation as $\mathbb{Z} \wr \mathbb{Z}$, the difference being that the element $atat^{-1}$, instead of adding $1$ in the lamp where the arrow is, add $1$ in the lamp where the arrow is and add $-1$ in the lamp at the right from the arrow. Now, given an $n \geq 1$, consider the element $(c,n)$ where $c$ is the colouring that takes the value $n$ at $0$, $-n$ at $n$, and $0$ elsewhere. It belongs to our subgroup $H$, and we can compare its distance from $(0,0)$ both in $H$ (with respect to the generating set $\{ata^{-1}t^{-1},t\}$) and in $\mathbb{Z} \wr \mathbb{Z}$ (with respect to the generating set $\{a,t\}$).
\begin{center}
\begin{tabular}{|c|c|c|}\hline
\includegraphics[trim=8cm 24cm 8cm 0,clip,width=0.29\linewidth]{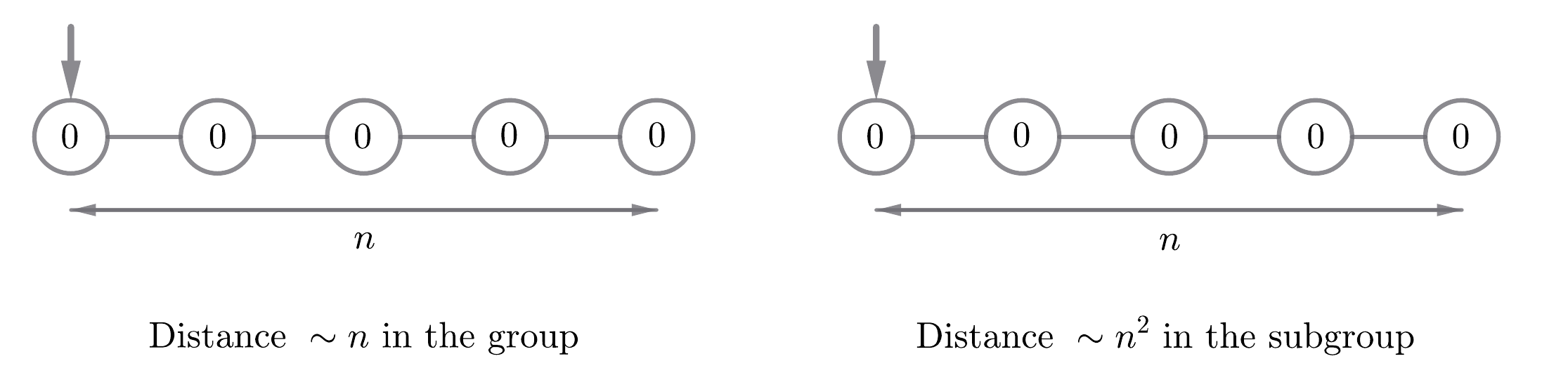} & 
\includegraphics[trim=8cm 24cm 8cm 0,clip,width=0.29\linewidth]{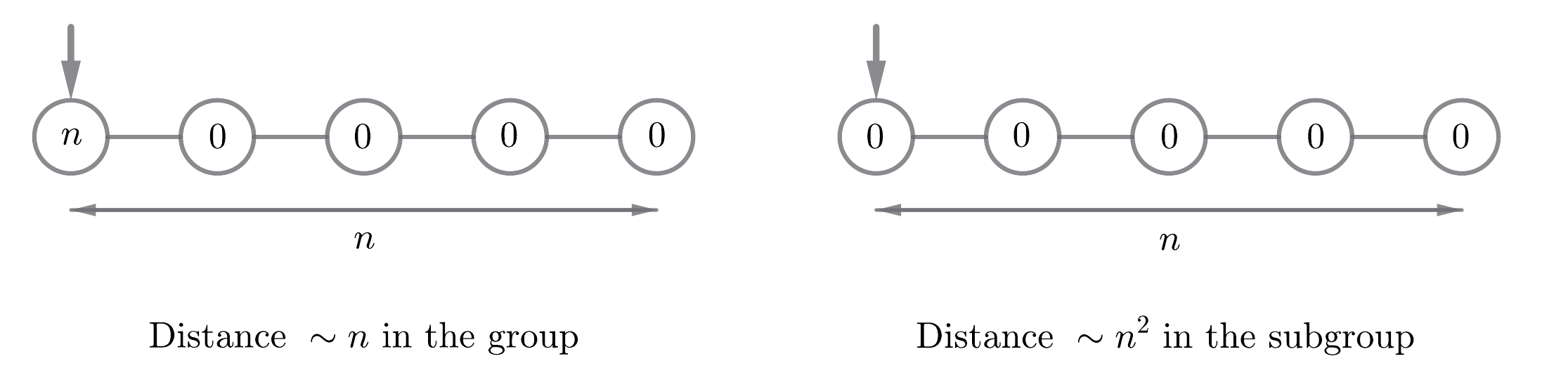} & 
\includegraphics[trim=8cm 24cm 8cm 0,clip,width=0.29\linewidth]{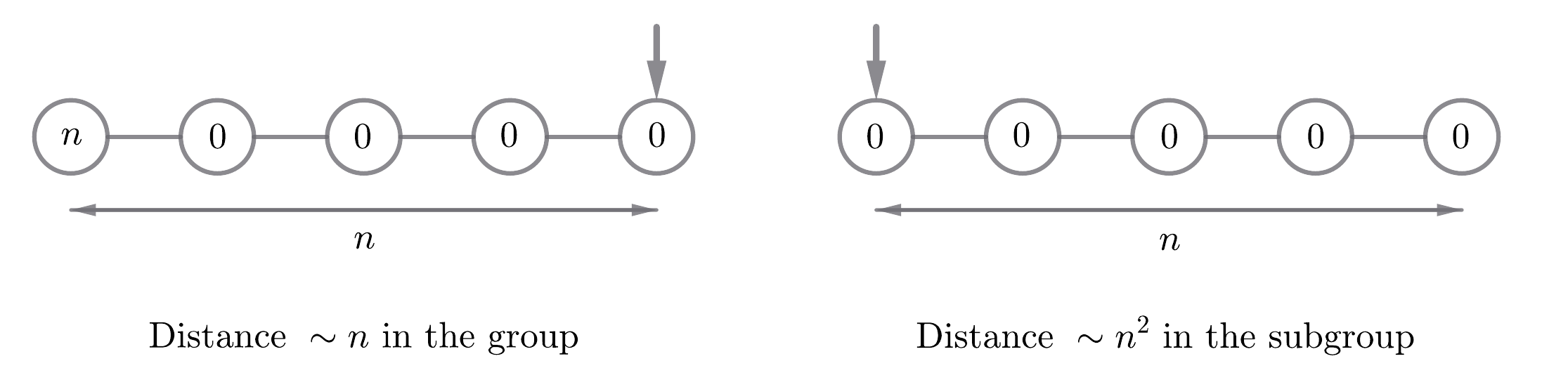} \\ \hline
\includegraphics[trim=8cm 24cm 8cm 0,clip,width=0.29\linewidth]{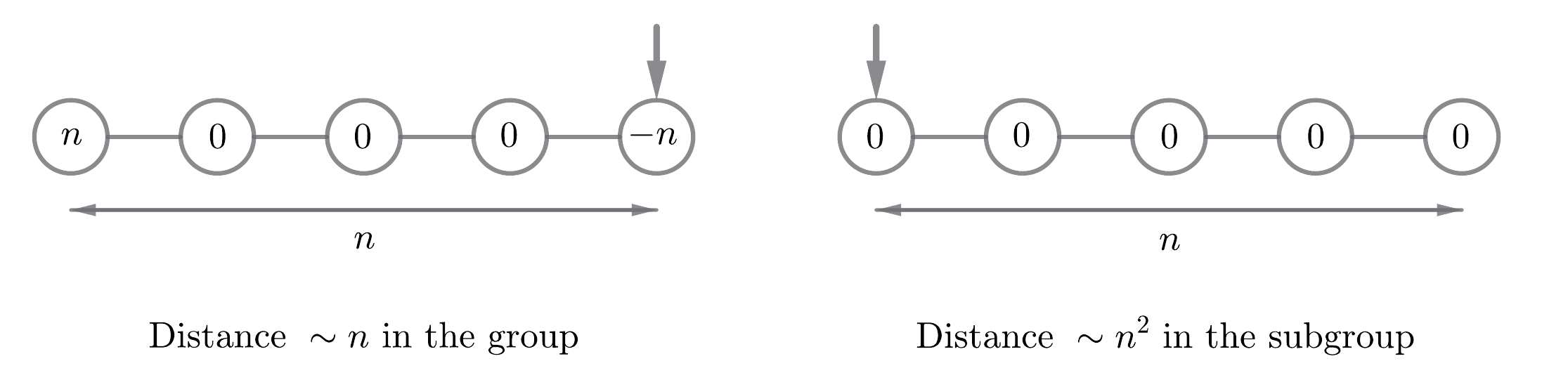} & 
\includegraphics[trim=8cm 24cm 8cm 0,clip,width=0.29\linewidth]{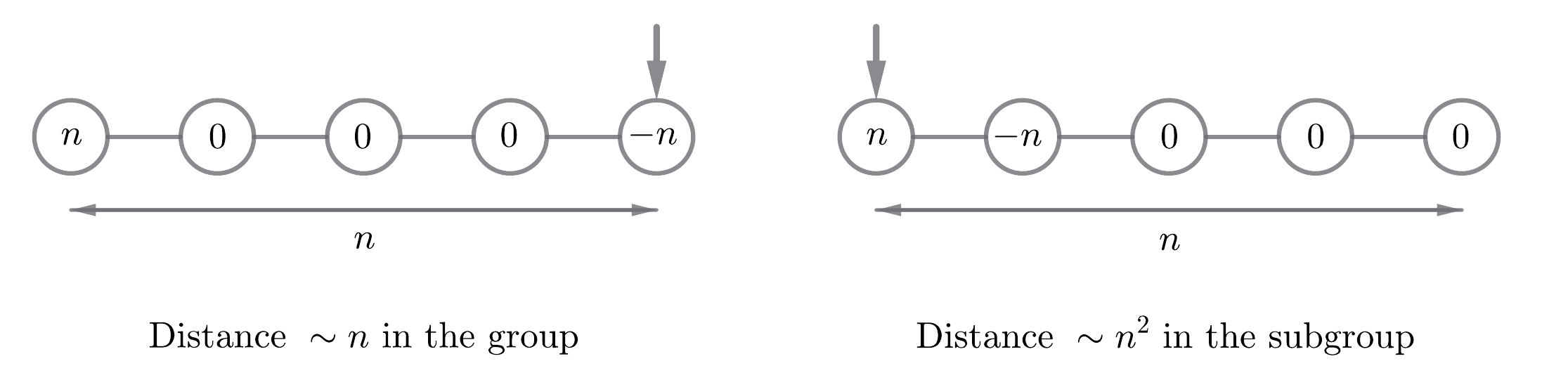} & 
\includegraphics[trim=8cm 24cm 8cm 0,clip,width=0.29\linewidth]{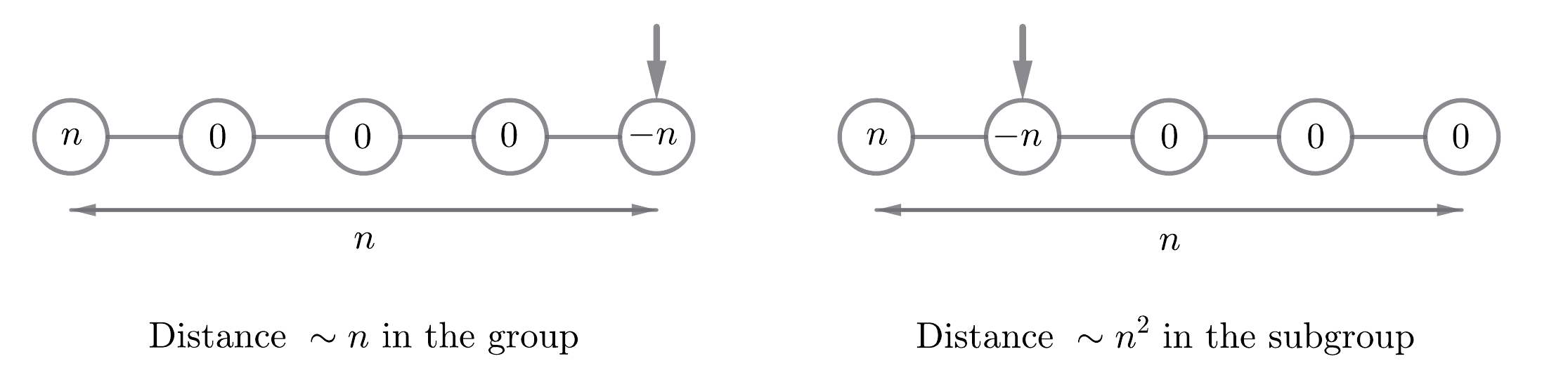} \\ \hline
\includegraphics[trim=8cm 24cm 8cm 0,clip,width=0.29\linewidth]{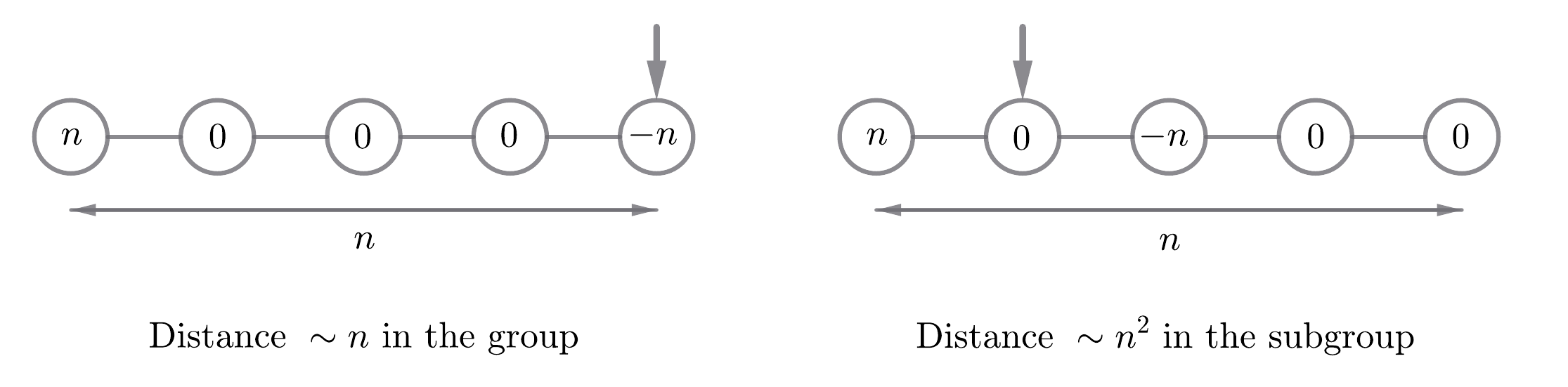} & 
\includegraphics[trim=8cm 24cm 8cm 0,clip,width=0.29\linewidth]{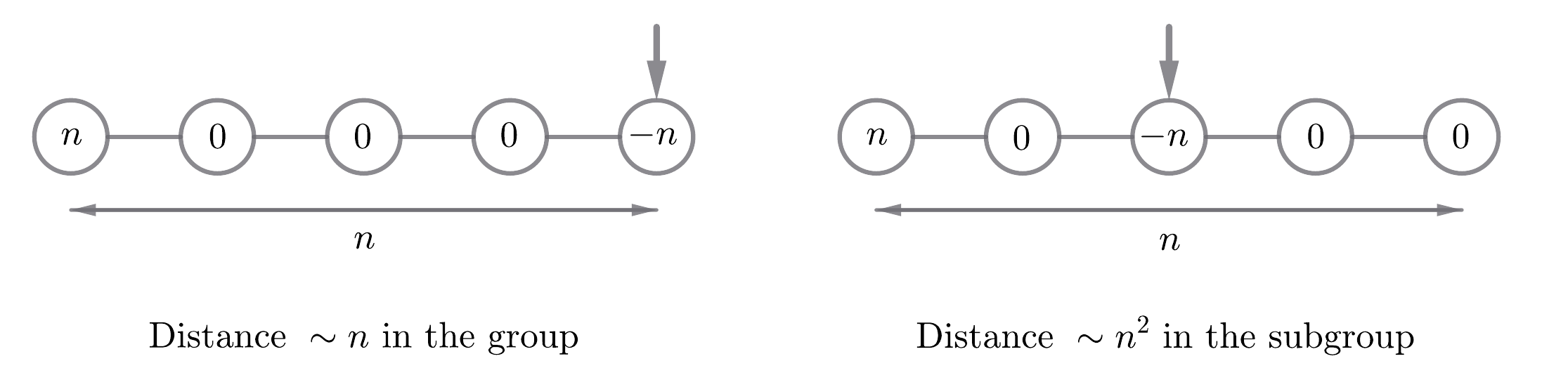} & 
\includegraphics[trim=8cm 24cm 8cm 0,clip,width=0.29\linewidth]{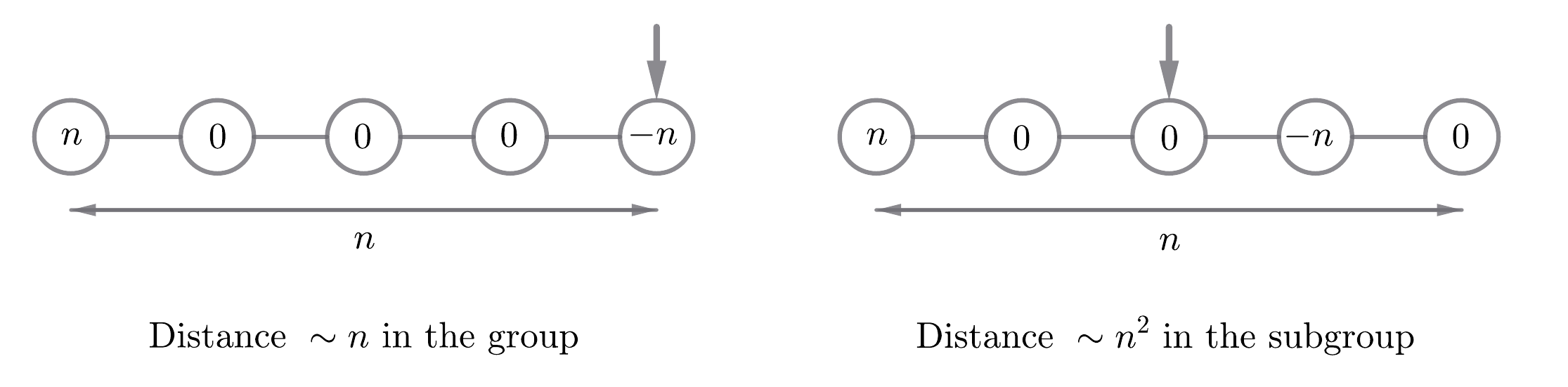} \\ \hline
\includegraphics[trim=8cm 24cm 8cm 0,clip,width=0.29\linewidth]{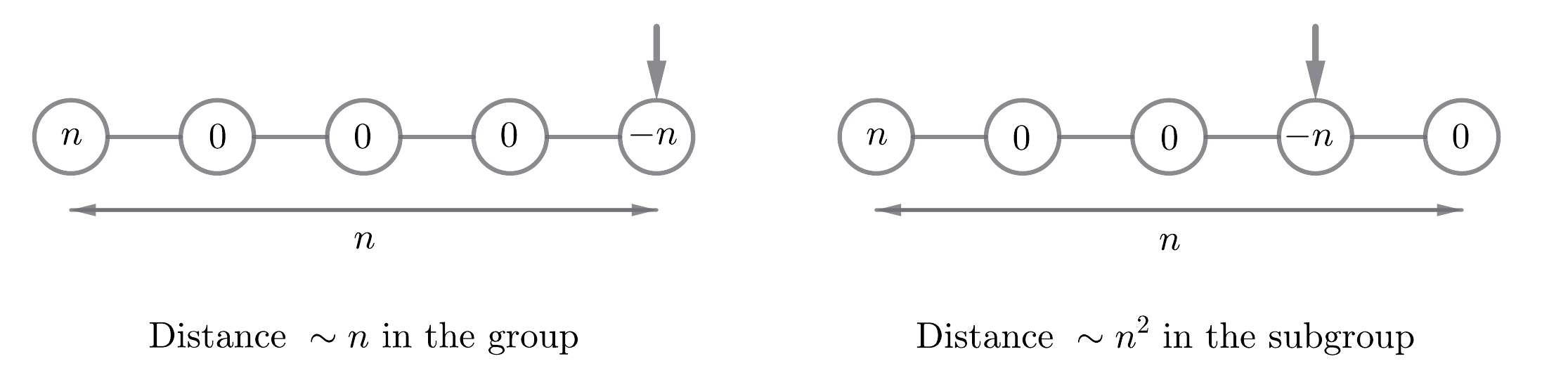} & 
\includegraphics[trim=8cm 24cm 8cm 0,clip,width=0.29\linewidth]{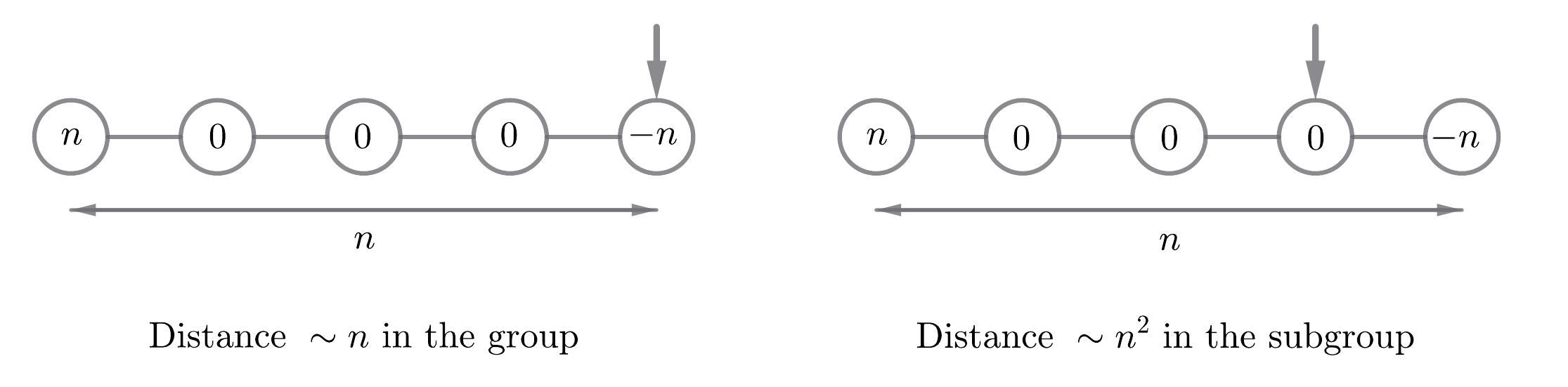} & 
\includegraphics[trim=8cm 24cm 8cm 0,clip,width=0.29\linewidth]{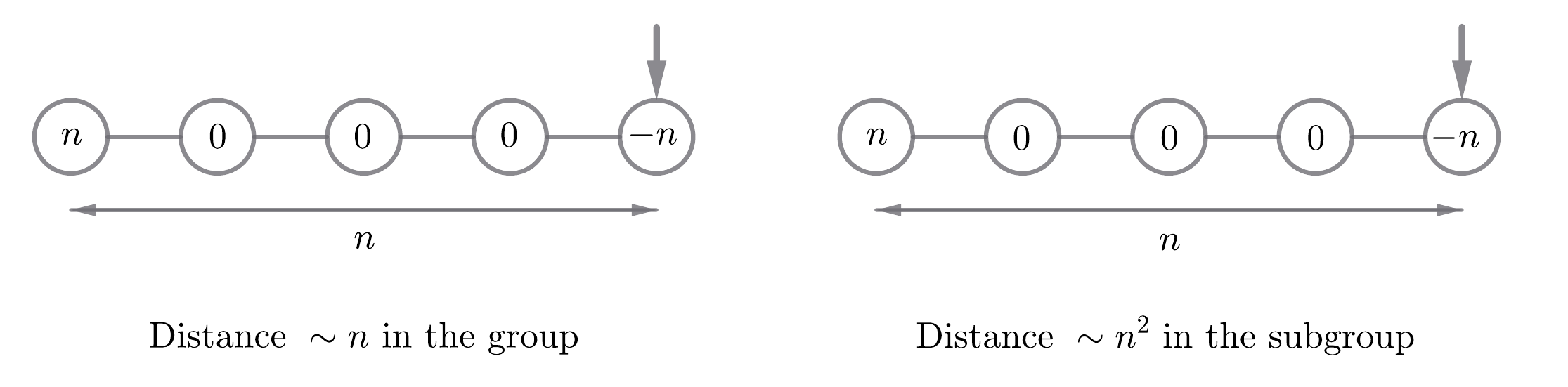} \\ \hline
\end{tabular}
\end{center}
As suggested by the figure above, the distance between $(c,n)$ and $(0,0)$ is about $n$ in $\mathbb{Z} \wr \mathbb{Z}$ but about $n^2$ in $H$. Thus, the inclusion map $H \hookrightarrow \mathbb{Z} \wr \mathbb{Z}$ cannot be a quasi-isometric embedding. 
\end{proof}

\subsection{What is known}

\noindent
Let us review in a few words what is known about the comparison of wreath products of groups up to quasi-isometry. First of all, various quasi-isometric invariants have been successively computed for wreath products, including for instance the number of ends (see Section~\ref{section:CoarseSeparation}), finite presentability (see Section~\ref{section:CoarseSimplyConnected}), asymptotic dimension (see Section~\ref{section:CoarseDim}), isoperimetric profile \cite{MR2011120}, Hilbert space compression (see Section~\ref{section:CoarseHilbert}). However, such invariants are usually not able to distinguish up to quasi-isometry to wreath products $F_1 \wr H$ and $F_2 \wr H$ for some finitely generated group $H$ and two finite groups $F_1,F_2$ of distinct cardinality. 

\medskip \noindent
In the opposite direction, as already mentioned (Proposition~\ref{prop:biLipWreath}), there exist many examples of non-isomorphic (or even non-commensurable) wreath products that turn out to be quasi-isometric:

\begin{prop}[\cite{MR1800990}]\label{prop:BiLipWreathGroups}
Let $A_1,A_2,B_1,B_2$ be four finitely generated groups. If $A_1$ (resp.\ $B_2$) is biLipschitz equivalent to $A_2$ (resp.\ $B_2$), then the wreath products $A_1 \wr B_1$ and $A_2\wr B_2$ are biLipschitz equivalent.
\end{prop}

\noindent
The first notable classification of a family of wreath products up to quasi-isometry is the following:

\begin{thm}[\cite{MR2925383, MR3034290}]\label{thm:EFW}
For all non-trivial finite groups $F_1,F_2$, the lamplighters $F_1 \wr \mathbb{Z}$ and $F_2 \wr \mathbb{Z}$ are quasi-isometric if and only if $|F_1|$ and $|F_2|$ are powers of a common number. 
\end{thm}

\noindent
We already saw that the quasi-isometry type of $F \wr \mathbb{Z}$ only depends on $|F|$. And it is not difficult to realise $\mathbb{Z}_{n^k} \wr \mathbb{Z}$ as a finite-index subgroup of $\mathbb{Z}_n \wr \mathbb{Z}$ for all $n,k \geq 2$. Thus, it is clear that $F_1 \wr \mathbb{Z}$ and $F_2 \wr \mathbb{Z}$ are quasi-isometric if $|F_1|$ and $|F_2|$ are powers of a common number. The converse, on the other hand, is much more difficult. 

\medskip \noindent
In fact, more information is found in \cite{MR2925383, MR3034290}: lamplighter groups are not only compared between them, but also to arbitrary finitely generated groups. Indeed, it is proved that an arbitrary finitely generated group is quasi-isometric to the lamplighter group $\mathbb{Z}_n \wr \mathbb{Z}$ if and only if it can be realised as a uniform lattice in the Diestel-Leader graph $\mathrm{DL}(n)$, a specific Cayley graph of $\mathbb{Z}_n \wr \mathbb{Z}$ that we will define in Section~\ref{section:CoarseHilbert}. As a consequence, some information can be extracted on the algebraic structure of groups quasi-isometric to lamplighters over $\mathbb{Z}$, see \cite{MR2991419}. 

\medskip \noindent
As a notable application of \cite{MR2925383, MR3034290}, it can be proved that there exist finitely generated groups that are quasi-isometric but not bijectively quasi-isometric (or equivalently, biLipschitz equivalent): 

\begin{thm}[\cite{MR2730576}]\label{thm:Dymarz}
Let $F_1,F_2$ be two non-trivial finite groups such that $|F_2|=|F_1|^k$ for some $k \geq 2$. There does not exist any bijective quasi-isometry between $F_1 \wr \mathbb{Z}$ and $F_2 \wr \mathbb{Z}$ if $k$ is not a product of prime factors appearing in $|F_1|$. 
\end{thm} 

\noindent
See Sections~\ref{section:MeasureScaling} and~\ref{section:ApplicationsTwo} for more details and other examples. 

\medskip \noindent
A second notable classification of wreath products up to quasi-isometry is given by the next statement. It is one the goal of this course to give a(n almost self-contained) proof of this classification (from a perspective different from the original paper).

\begin{thm}[\cite{MR4794592}]\label{thm:QIclassification}
Let $F_1,F_2$ be two finite groups and $H_1,H_2$ two one-ended finitely presented groups. The wreath products $F_1 \wr H_1$ and $F_2 \wr H_2$ are quasi-isometric if and only if one of the following two conditions occurs:
\begin{itemize}
	\item $H_1,H_2$ are non-amenable and quasi-isometric, and $|F_1|,|F_2|$ have the same prime divisors;
	\item $H_1,H_2$ are amenable, $|F_1|,|F_2|$ are powers of a common number, say $|F_1|=q^a$ and $|F_2|=q^b$, and there exists a quasi-$\frac{b}{a}$-to-one quasi-isometry $H_1 \to H_2$. 
\end{itemize}
\end{thm}

\noindent
For instance, given integers $n,m,p,q \geq 2$, the wreath products $\mathbb{Z}_n \wr \mathbb{Z}^p$ and $\mathbb{Z}_m \wr \mathbb{Z}^q$ are quasi-isometric if and only if $p=q$ and $n,m$ are powers of a common number. And, given two integers $n,m \geq 2$ and a non-abelian free group $\mathbb{F}$ of finite rank, the wreath products $\mathbb{Z}_n \wr (\mathbb{F} \times \mathbb{Z})$ and $\mathbb{Z}_m \wr (\mathbb{F} \times \mathbb{Z})$ are quasi-isometric if and only if $n$ and $m$ have the same prime divisors. 

\medskip \noindent
Some information about arbitrary finitely generated groups that are quasi-isometric to lamplighters over one-ended finitely presented groups can also be extracted, but the global picture is far from being understood, even in simple cases such as $\mathbb{Z}_n \wr \mathbb{Z}^2$. See \cite{Halo} for more details.

\medskip \noindent
In view of Theorems~\ref{thm:EFW} and~\ref{thm:QIclassification}, lamplighters over infinitely ended groups (such as non-abelian free groups) remain to be investigated. Given a non-abelian free group $\mathbb{F}$, we know that $\mathbb{Z}_n \wr \mathbb{F}$ and $\mathbb{Z}_m \wr \mathbb{F}$ are quasi-isometric when $n$ and $m$ have the same prime divisors (see Section~\ref{section:LampNonAmenable}), but it is not known whether the converse holds. Moreover, there exist rather unexpected finitely generated groups that turn out to be quasi-isometric to lamplighters over free groups:

\begin{thm}[\cite{MR4240607}]
For every $n \geq 2$ and every non-abelian free group of finite rank $\mathbb{F}$, there exists a simple group quasi-isometric to $\mathbb{Z}_n \wr \mathbb{F}$. 
\end{thm}

\noindent
It is expected that a similar phenomenon also occurs for lamplighters over some one-ended groups, but we do not know whether a simple group can be quasi-isometric to a wreath product such as $\mathbb{Z}_n \wr \mathbb{Z}^m$ for some $n,m \geq 2$.

\subsection{Exercises}

\begin{exo}\label{exo:Hypercubes}
Let $S$ be a set. Define the hypercube $Q(S)$ as the graph whose vertices are the finite subsets of $S$ and whose edges connect two subsets whenever their symmetric difference has size $1$. Prove that $\mathrm{Aut}(Q(S))$ is isomorphic to the permutation group $(\mathrm{Sym}(2) \curvearrowright \{1,2\} )  \wr (\mathrm{Sym}(S) \curvearrowright S)$. 
\end{exo}

\begin{exo}\label{exo:ChainSquares}
In this exercise, we realise some wreath products as automorphism groups of graphs.
\begin{enumerate}
	\item Let $X$ denote the bi-infinite chain of squares \includegraphics[width=0.15\linewidth]{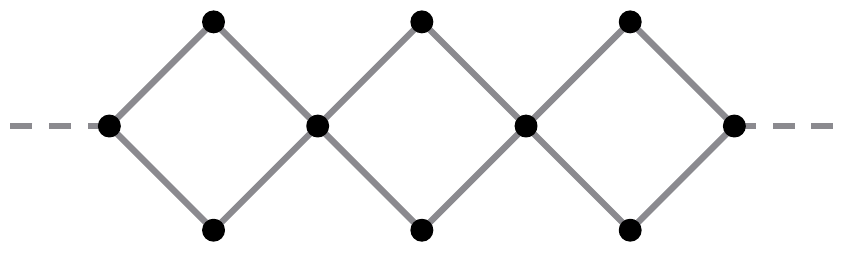}. Show that $\mathrm{Aut}(X)$ is isomorphic to $\mathbb{Z}_2 \wr \mathbb{D}_\infty$. 
	\item Given an integer $n \geq 1$, let $Y_n$ denote the graph \includegraphics[width=0.15\linewidth]{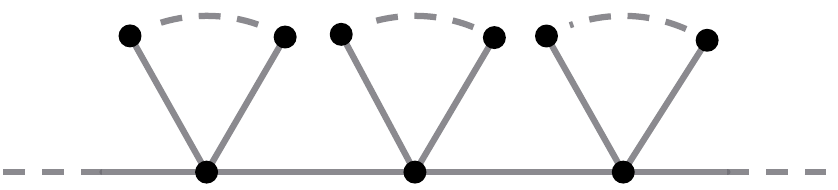} obtained from a bi-infinite line by adding $n$ spikes at each vertex. Prove that $\mathrm{Aut}(Y_n)$ is isomorphic to $\mathrm{Sym}(n) \wr \mathbb{D}_\infty$.  
	\item Given a graph $X$ and an integer $n \geq 1$, constructed a graph $X(n)$ whose automorphism groups is isomorphic to the permutation group $(\mathrm{Sym}(n) \curvearrowright \{1, \ldots, n\} ) \wr  (\mathrm{Aut}(X) \curvearrowright V(X))$.
\end{enumerate}
\end{exo}

\begin{exo}
Let $X$ and $Y$ be two graphs. Define the \emph{lexicographic product} $X[Y]$ as the graph obtained from disjoint copies $Y_x$ of $Y$ indexed by $V(X)$ by adding an edge between every vertex of $Y_a$ and every vertex of $Y_b$ whenever $a,b \in V(X)$ are adjacent in~$X$. 
\begin{enumerate}
	\item Draw $C_5[C_3]$, where $C_n$ denotes the cycle of length $n$.
	\item Realise $\mathrm{Aut}(Y) \wr \mathrm{Aut}(X)$ as a subgroup of $\mathrm{Aut}(X[Y])$.
	\item Let $K_{n,m}$ denote the complete bipartite graph with $n+m$ vertices and $\overline{K_s}$ the graph that has $s$ vertices but no edges. Show that $K_{n,m}[K_s]$ is isomorphic to $K_{sn,sm}$.
	\item Deduce that the inclusion $\mathrm{Aut}(Y) \wr \mathrm{Aut}(X) \subset \mathrm{Aut}(X[Y])$ previously constructed is not always an equality. 
\end{enumerate}
\end{exo}

\begin{exo}\label{exo:AssociativeWreath}
The goal here is to show that the wreath product is highly non-associative. 
\begin{enumerate}
	\item Show that, given two finite groups $A$ and $B$, the cardinality of $A \wr B$ is $|A|^{|B|} |B|$.
	\item Deduce that, given three finite groups $A,B,C$, the wreath products $A \wr (B \wr C)$ and $(A \wr B) \wr C$ are isomorphic if and only if one of $A,B,C$ is trivial.
	\item Verify that the groups $A = \cdots (((\mathbb{Z} \wr \mathbb{Z}) \wr \mathbb{Z}) \wr \mathbb{Z}) \cdots$ and $C= \cdots (\mathbb{Z} \wr (\mathbb{Z} \wr ( \mathbb{Z} \wr \mathbb{Z}))) \cdots$ are well-defined and prove that $A \wr (\mathbb{Z} \wr C) \simeq A \wr C \simeq (A \wr \mathbb{Z}) \wr B$. Thus, the previous item does not generalise to arbitrary infinite groups. 
\end{enumerate}
\end{exo}

\begin{exo}\label{exo:QuasiInverse}
Let $X,Y$ be two metric spaces and $\varphi : X \to Y$ a quasi-isometry. Fix $A\geq 1$ and $B \geq 0$ such that $\varphi$ is an $(A,B)$-quasi-isometry, i.e.\ 
$$\frac{1}{A} d(x,y) - B \leq d(\varphi(x), \varphi(y)) \leq A d(x,y) + B \text{ for all } x,y \in X$$
and every point in $Y$ lies at distance $\leq B$ from the image of $\varphi$. Prove $\varphi$ admits a \emph{quasi-inverse}, i.e.\ a quasi-isometry $\bar{\varphi} : Y \to X$ such that $\varphi \circ \bar{\varphi}$ (resp.\ $\bar{\varphi}\circ \varphi$) lies at finite distance from $\mathrm{Id}_Y$ (resp.\ $\mathrm{Id}_X$). Moreover, verify that $\bar{\varphi}$ can be chosen as a $(A,3AB)$-quasi-isometry.
\end{exo}

\begin{exo}\label{exo:ExQI}
Prove that the map
\begin{center}
\includegraphics[width=0.8\linewidth]{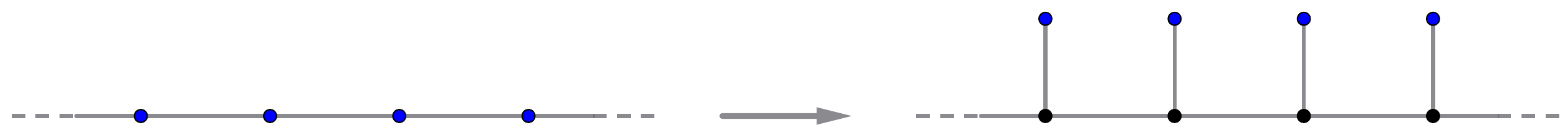}
\end{center}
is a $(1,1)$-quasi-isometry.
\end{exo}

\begin{exo}\label{ex:CoarseQI}
Let $\varphi : X \to Y$ be a map between two graphs.
\begin{enumerate}
	\item If $\varphi$ is a coarse embedding, show that $\varphi$ is Lipschitz. (\emph{Hint: Apply the triangle inequality along geodesics.})
	\item Deduce that, if $\varphi$ is a coarse equivalence, then it must be a quasi-isometry.
\end{enumerate}
\end{exo}

\begin{exo}
Let $X$ and $Y$ be two graphs, which we do not assume connected. Fix a vertex $o \in X$. 
\begin{enumerate}
	\item Show that the wreath product $(X,o) \wr Y$ is connected if and only if $X$ and $Y$ are both connected.
	\item Deduce that, if $A$ and $B$ are two groups generated respectively by $R \subset A$ and $S \subset B$, then $A \wr B$ is generated by $R \cup S$.
	\item Reprove this assertion algebraically.
\end{enumerate}
\end{exo}

\begin{exo}
Let $X$ and $Y$ be two graphs and $o \in V(X)$ a basepoint.
\begin{enumerate}
	\item Compute the girth of $(X,o) \wr Y$. 
	\item Assuming that $X$ and $Y$ are finite, compute the diameter of $(X,o) \wr Y$.
	\item The chromatic number of $\mathcal{L}_2(X)$ is the same as the chromatic number of $X$.
\end{enumerate}
\end{exo}

\begin{exo}
Let $\mathbb{E}^2$ denote the infinite grid. For every $n \geq 2$, describe the automorphism group of $\mathcal{L}_n(\mathbb{E}^2)$. (\emph{Hint: Verify that every automorphism is aptolic.})
\end{exo}

\begin{exo}
Let $k \geq 2$ be an integer.
\begin{enumerate}
	\item Prove that the multiplicative group $\mathbb{Z}_k[\mathbb{Z}]^\times$ of the invertible elements of the ring $\mathbb{Z}_k[\mathbb{Z}]$ is finitely generated if and only if $k$ is square-free (i.e.\ $1$ is the only square dividing $k$).
	\item By identifying $\bigoplus_\mathbb{Z} \mathbb{Z}_k$ with $\mathbb{Z}_k[\mathbb{Z}]$, realise $\mathbb{Z}_k[\mathbb{Z}]^\times$ as a subgroup of $\mathrm{Aut}(\mathbb{Z}_k \wr \mathbb{Z})$.
	\item Show that $\mathrm{Aut}(\mathbb{Z}_k \wr \mathbb{Z})$ is finitely generated if and only if $k$ is square-free.
\end{enumerate}
\end{exo}

\begin{exo}\label{ex:WreathRF}
Recall that a group $G$ is residually finite if, for every non-trivial $g \in G$, there exists a finite-index subgroup that does not contain $g$. 
\begin{enumerate}
	\item Prove that a group containing a finite-index subgroup that is residually finite must be residually finite.
	\item Show that, if $A$ is non-abelian and $B$ infinite, then $A \wr B$ is not residually finite. (\emph{Hint: Notice that, if $a,b \in A$ do not commute, we know that $a$ and $tbt^{-1}$ commute for every non-trivial $t \in B$.})
	\item Deduce that $A \wr B$ is residually finite if and only if either $A$ is residually finite and $B$ finite or $A$ is abelian residually finite and $B$ is infinite but residually finite.
	\item Prove that being residually finite is not preserved by quasi-isometries. 
\end{enumerate}
\end{exo}

\begin{exo}\label{exo:NotQIpreserved}
Let $A$ and $B$ be two groups.
\begin{enumerate}
	\item Prove that, if $A$ is non-trivial and $B$ infinite, then $A \wr B$ does not contain a finite normal subgroup.
	\item Denote by $K(A)$ the kernel of the morphism $A \times A \times \cdots \twoheadrightarrow A$ that sends each factor identically to $A$. Show that, if $A$ contains a non-trivial finite-order element (resp.\ two non-trivial finite-order elements that commute), then so does $K(A)$.
	\item Show that, if $B$ is infinite, every finite-index subgroup of $A \wr B$ contains a subgroup isomorphic to $K(A)$.
	\item Prove Corollary~\ref{cor:NotQIpreserved}. 
\end{enumerate}
\end{exo}

\begin{exo}
In a graph $X$, a \emph{dead end of depth $D$} is the data of two vertices $u,v \in V(X)$ such that $d(u,w)<d(u,v)$ for all $w \in B(v,D)$. By considering the vertices $u:=(0,0)$ and $v_n:= (c_n,0)$ in $\mathcal{L}_2(\mathbb{Z})$, where $c_n$ denotes the colouring that takes the value $1$ on $[-n,n]$ and $0$ elsewhere, prove that $\mathcal{L}_2(\mathbb{Z})$ contains dead ends of arbitrarily large depth. 
\end{exo}

\begin{exo}
Our goal here is to construct two quasi-isometric groups respectively with a solvable and an unsolvable word problem. 
\begin{enumerate}
	\item Prove that $\mathbb{Z}_2 \wr \mathbb{Z}$ admits $$\langle a,t \mid a^2=1, [a, t^nat^{-n}]=1 \ (n \in \mathbb{Z}) \rangle$$ as a presentation. 
	\item Show that $\mathbb{Z}_2 \wr \mathbb{Z}$ has a solvable word problem, i.e.\ there exists an algorithm that takes as an input a word written over $\{a, t^{\pm 1}\}$ and outputs whether or not it represents the trivial element in $\mathbb{Z}_2 \wr \mathbb{Z}$. 
	\item For every $I \subset \mathbb{Z}$, let $L(I)$ denote the group given by the presentation $$\langle a,t,z \mid z^2 = 1, z \text{ central}, [a,t^nat^{-n}]= z \text{ if } n \in I \text{ and } 1 \text{ otherwise} \rangle.$$Prove that $z$ is non-trivial in $L(I)$.
	\item Show that, for every $I \subset \mathbb{Z}$, $L(I)$ is quasi-isometric to $\mathbb{Z}_2 \wr \mathbb{Z}$. 
	\item Prove that, for all $I,J \subset \mathbb{Z}$, the groups $L(I)$ and $L(J)$ are isomorphic if and only if $I=J$.
	\item Deduce that there exists some $I \subset \mathbb{Z}$ such that $L(I)$ has an unsolvable word problem. (\emph{Hint: there exists only countably many algorithms.})
\end{enumerate}
\end{exo}

\begin{exo}[\cite{MR1610411}]
Let $a$ (resp.\ $t$) be a generator of left (resp.\ right) $\mathbb{Z}$-factor in $\mathbb{Z} \wr \mathbb{Z}$.
\begin{enumerate}
	\item Prove that an element $g \in \mathbb{Z} \wr \mathbb{Z}$ belongs to $\bigoplus_\mathbb{Z} \mathbb{Z}$ if and only if $$\exists x, \ \left[ g,xax^{-1} \right]=1.$$
	\item For all integers $n,m \geq 1$, show that $n$ divides $m$ if and only if $$\exists u \in \bigoplus_\mathbb{Z} \mathbb{Z}, \ u^{-1} b^n u b^{-n} = a^{-1} b^m a b^{-m}.$$
\end{enumerate}
The combination of these two observations shows that the arithmetic can be interpreted in the first-order logic of the group $\mathbb{Z} \wr \mathbb{Z}$. As a consequence of G\"{o}del's theorem, the first-order theory of $\mathbb{Z} \wr \mathbb{Z}$ is undecidable. 
\end{exo}

\section{Some ideas from coarse topology}\label{section:CoarseTopo}

\subsection{Coarse simple connectedness}\label{section:CoarseSimplyConnected}

\noindent
In topology, a space is simply connected if every loop can be continuously shrunk to a single point. In order to ``coarsify'' this idea, we replace, roughly speaking, continuous deformations with arbitrary modifications at small scales. More formally:

\begin{definition}
Let $X$ be a graph. Given a \emph{scale} $E \geq 0$, two paths $\alpha,\beta$ with the same endpoints are \emph{$E$-coarsely homotopy equivalent} (\emph{relative to their endpoints}) if there exists a sequence of paths
$$\gamma_1=\alpha, \ \gamma_2, \ldots, \ \gamma_{n-1}, \ \gamma_n=\beta$$
with the same endpoints such that, for every $1 \leq i \leq n-1$, the symmetric difference $\gamma_{i+1} \triangle \gamma_i$ has diameter $\leq E$. 
\end{definition}
\begin{center}
\begin{tabular}{|c|c|c|c|} \hline
\includegraphics[trim=0 20cm 37cm 0,clip,width=0.20\linewidth]{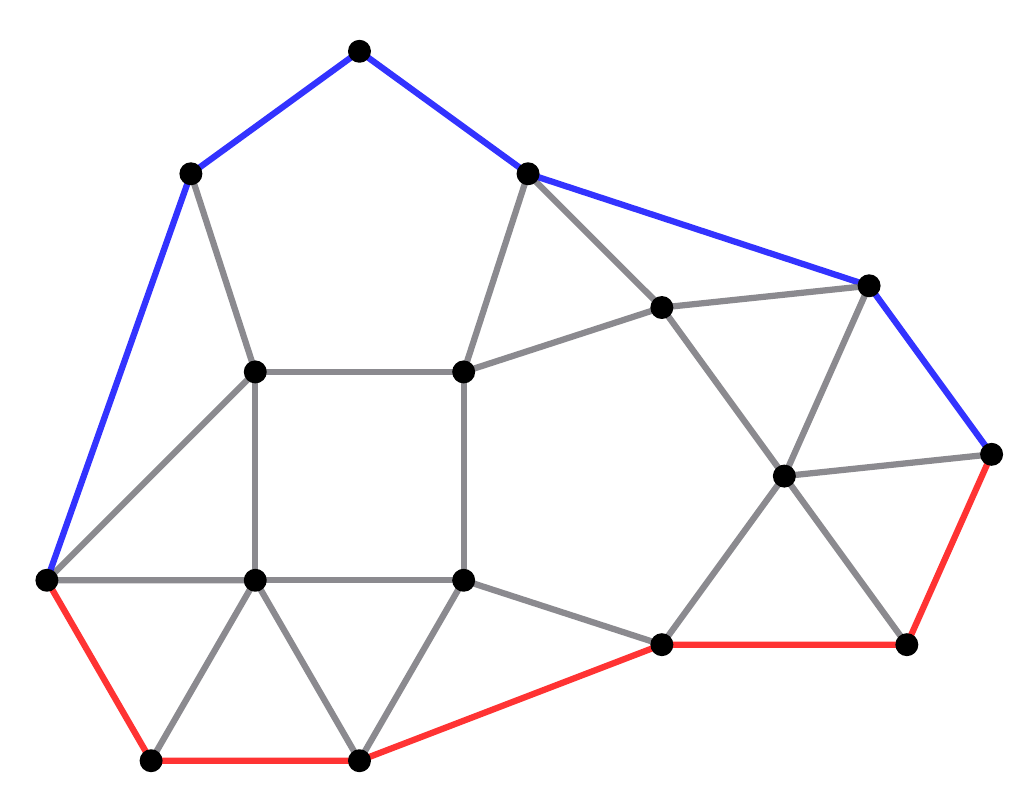} &
\includegraphics[trim=0 20cm 37cm 0,clip,width=0.20\linewidth]{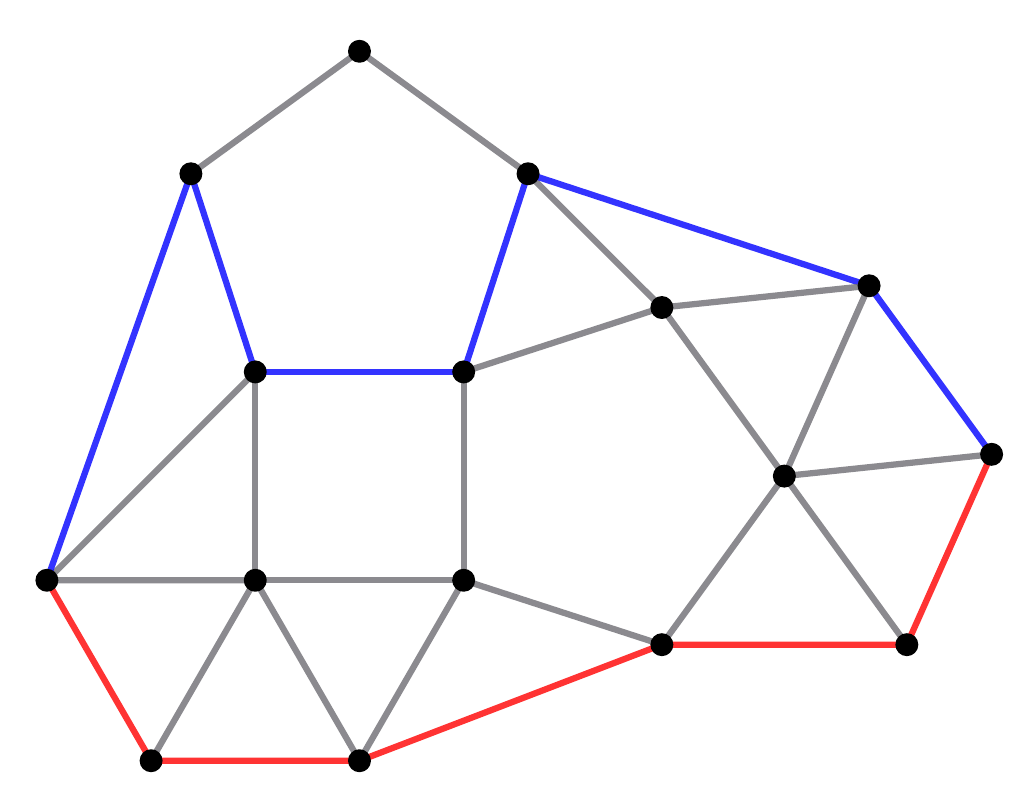} &
\includegraphics[trim=0 20cm 37cm 0,clip,width=0.20\linewidth]{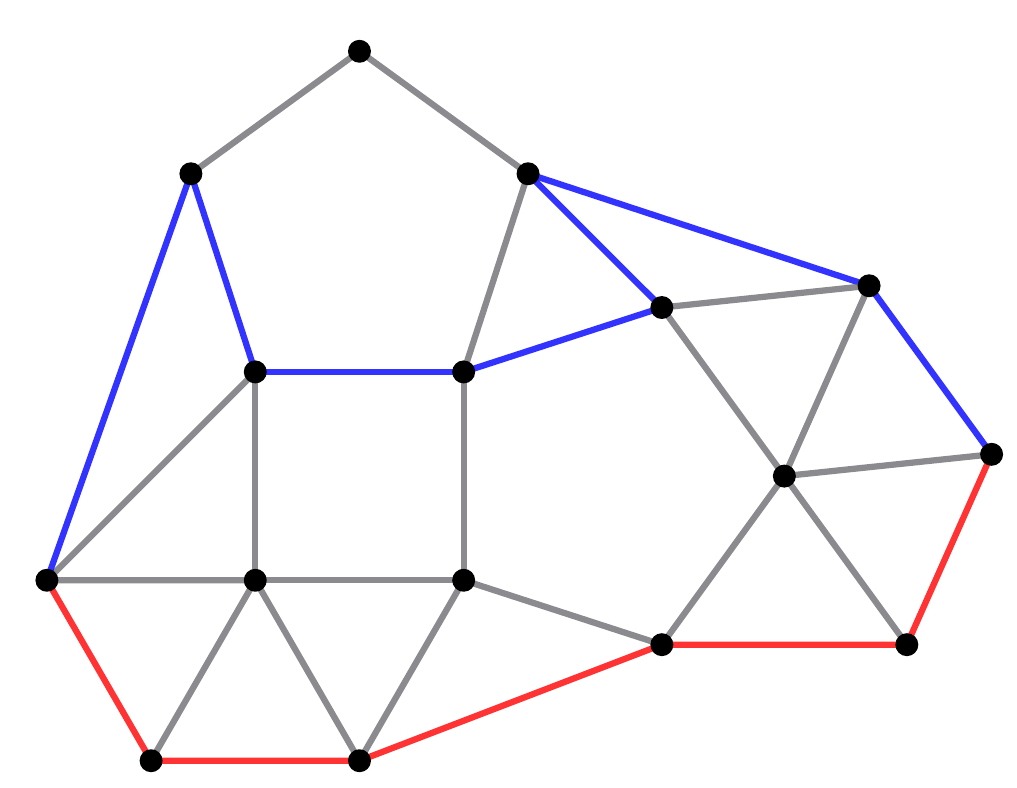} &
\includegraphics[trim=0 20cm 37cm 0,clip,width=0.20\linewidth]{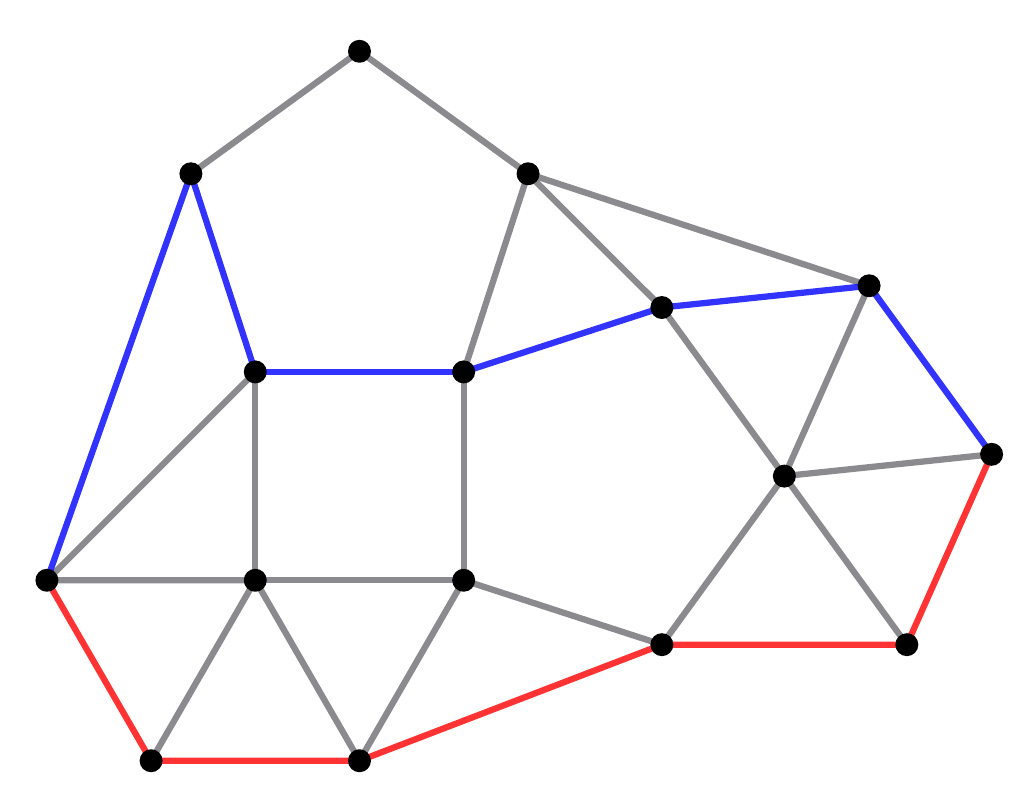} \\ \hline
\includegraphics[trim=0 20cm 37cm 0,clip,width=0.20\linewidth]{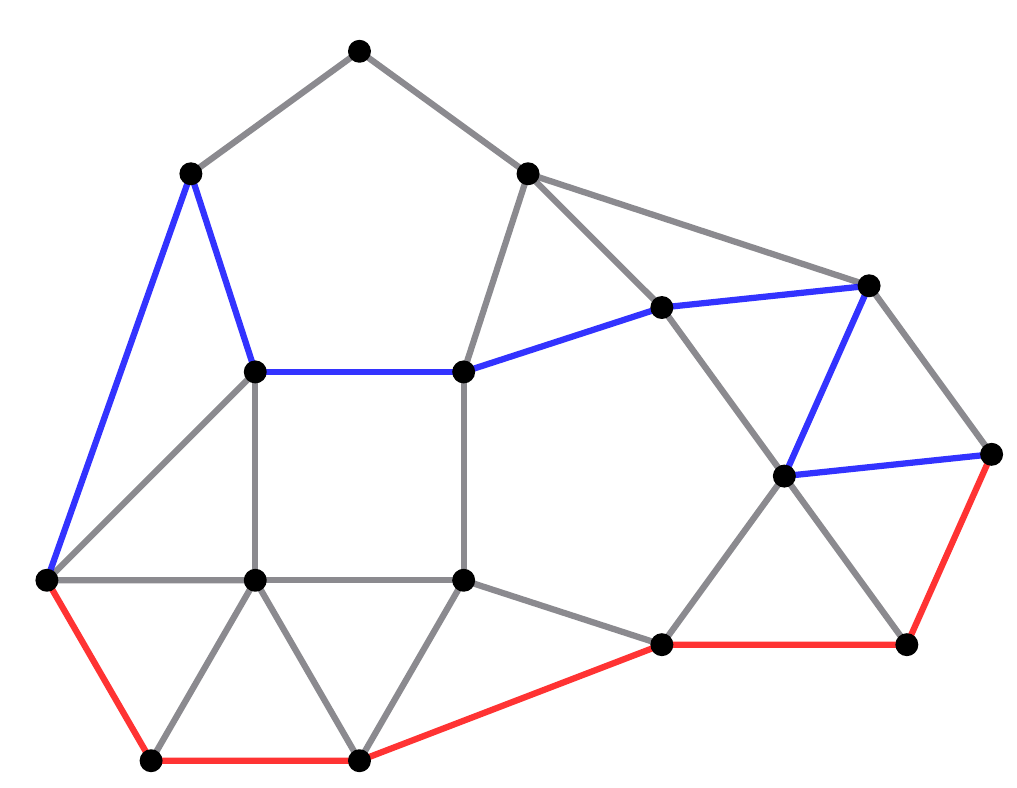} &
\includegraphics[trim=0 20cm 37cm 0,clip,width=0.20\linewidth]{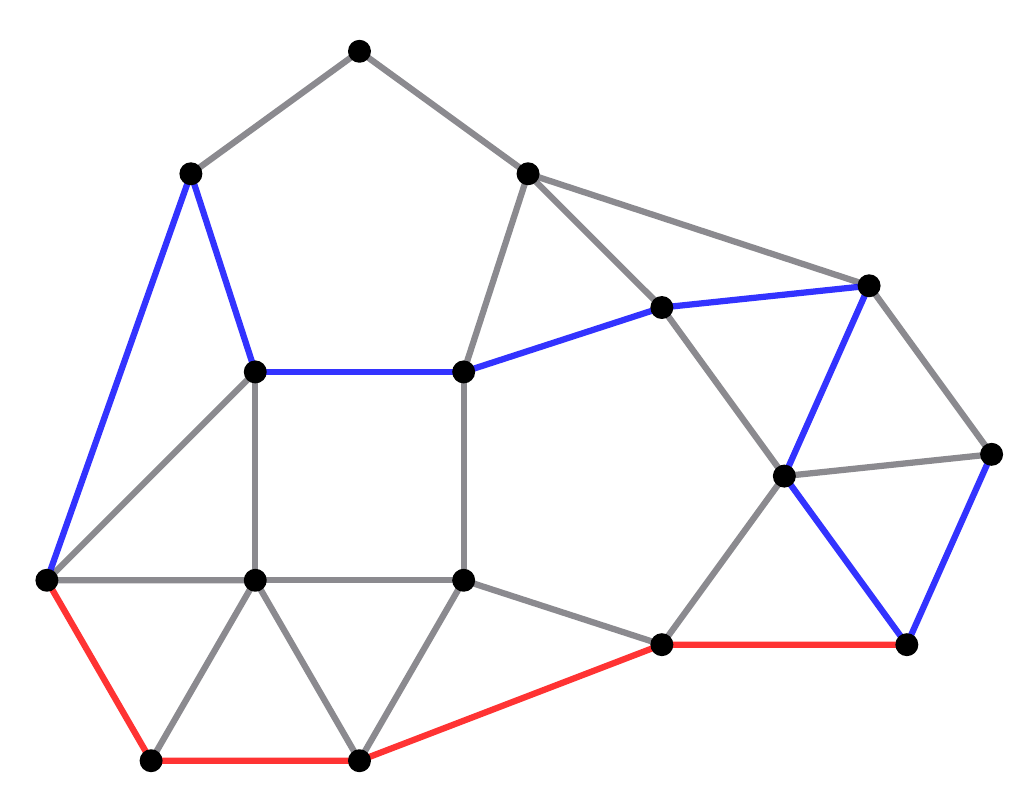} &
\includegraphics[trim=0 20cm 37cm 0,clip,width=0.20\linewidth]{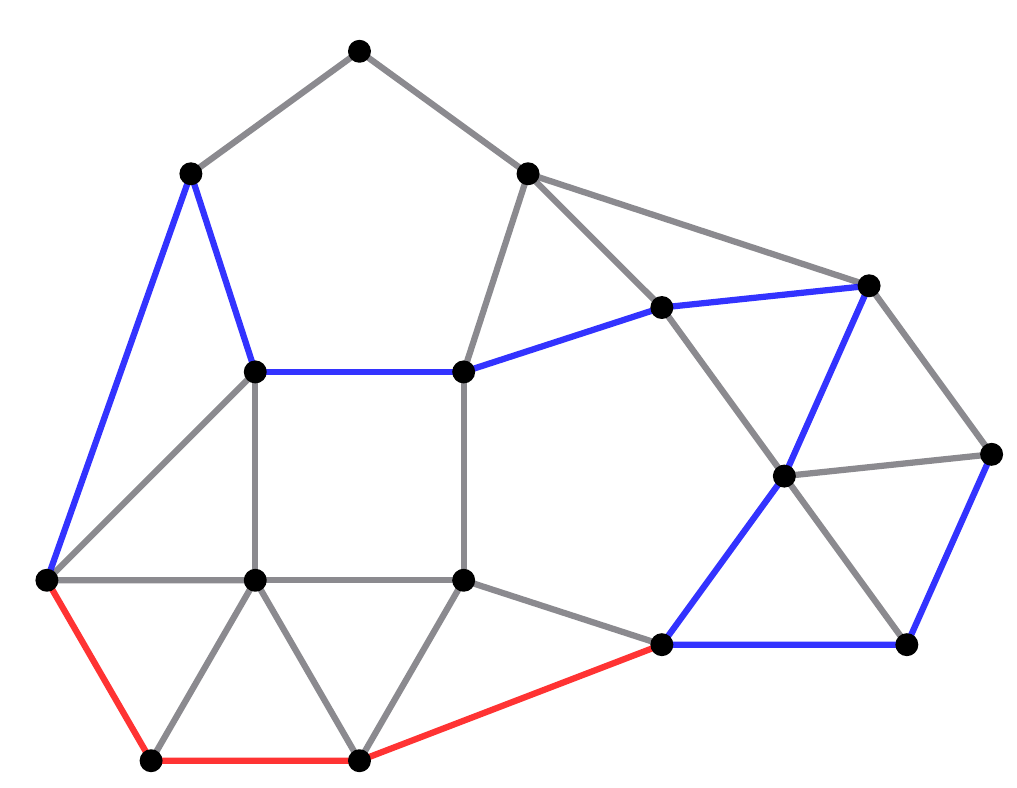} &
\includegraphics[trim=0 20cm 37cm 0,clip,width=0.20\linewidth]{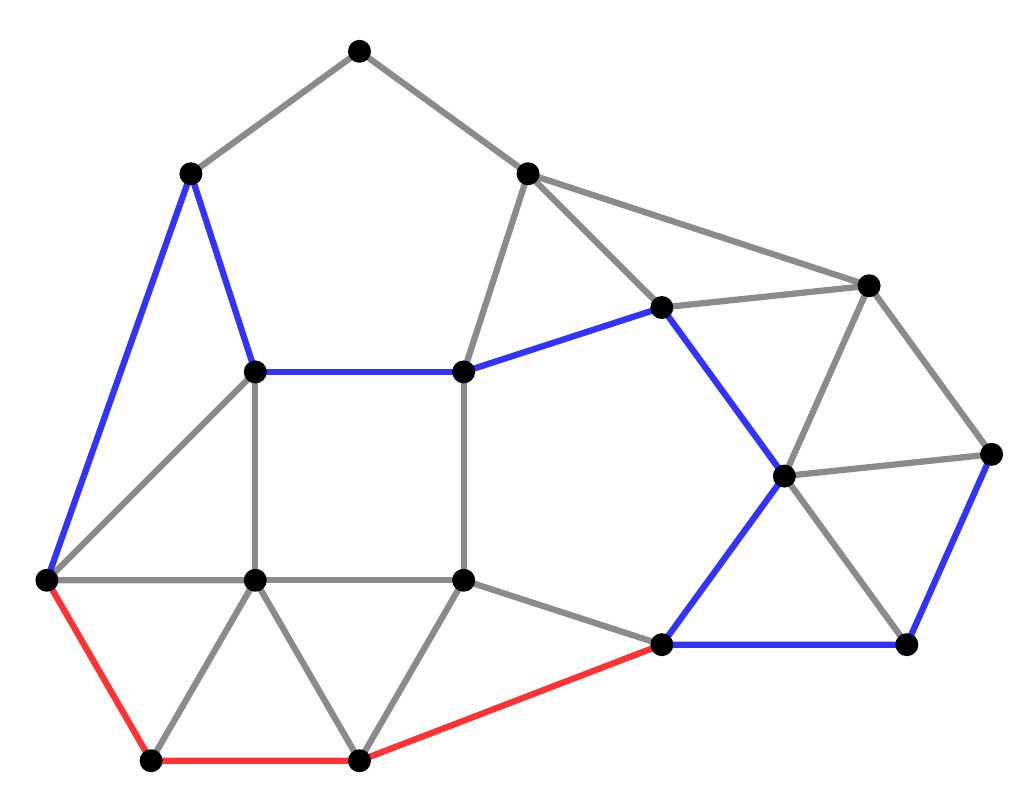} \\ \hline
\includegraphics[trim=0 20cm 37cm 0,clip,width=0.20\linewidth]{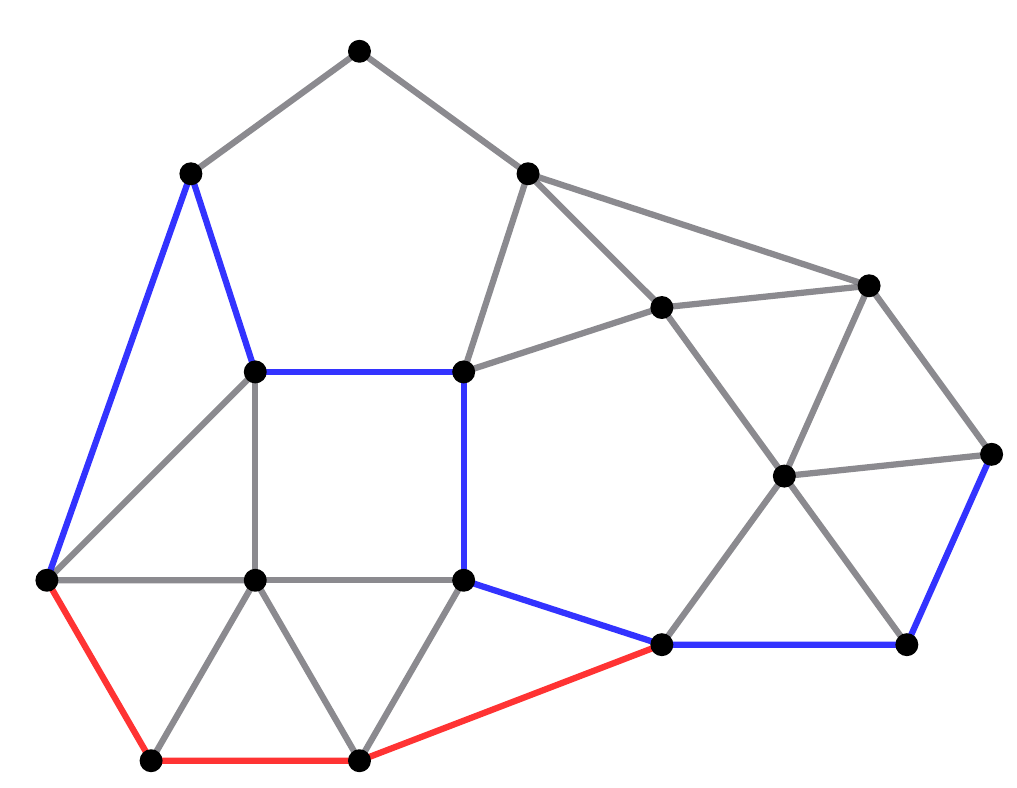} &
\includegraphics[trim=0 20cm 37cm 0,clip,width=0.20\linewidth]{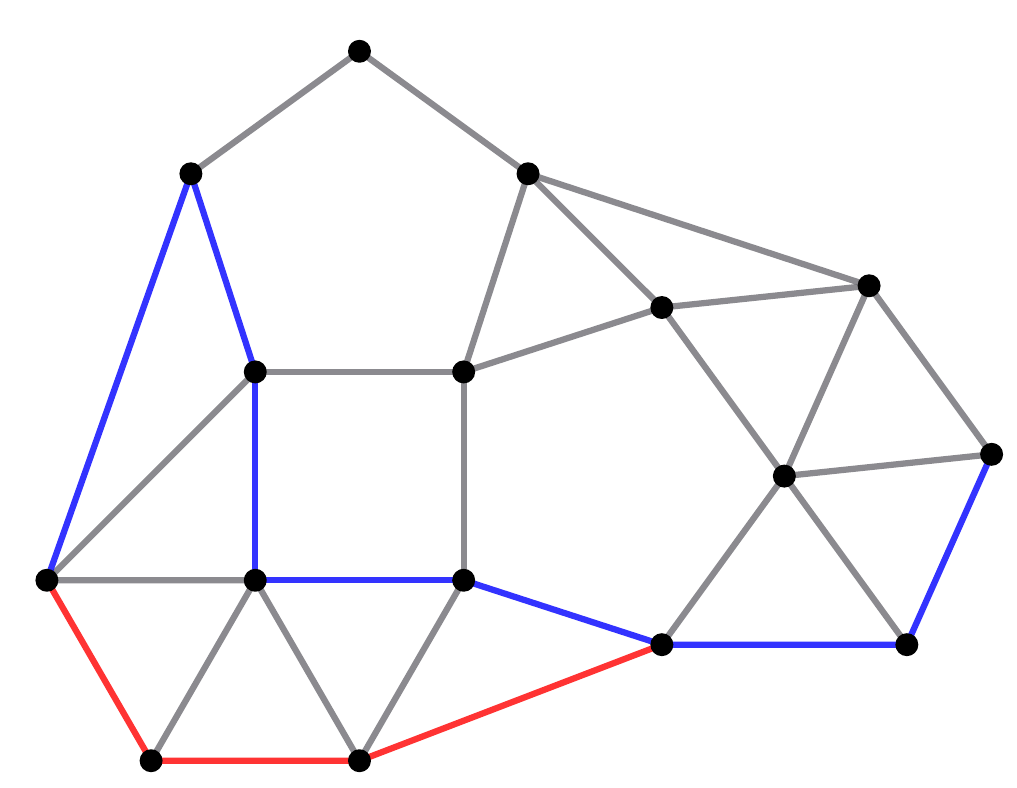} &
\includegraphics[trim=0 20cm 37cm 0,clip,width=0.20\linewidth]{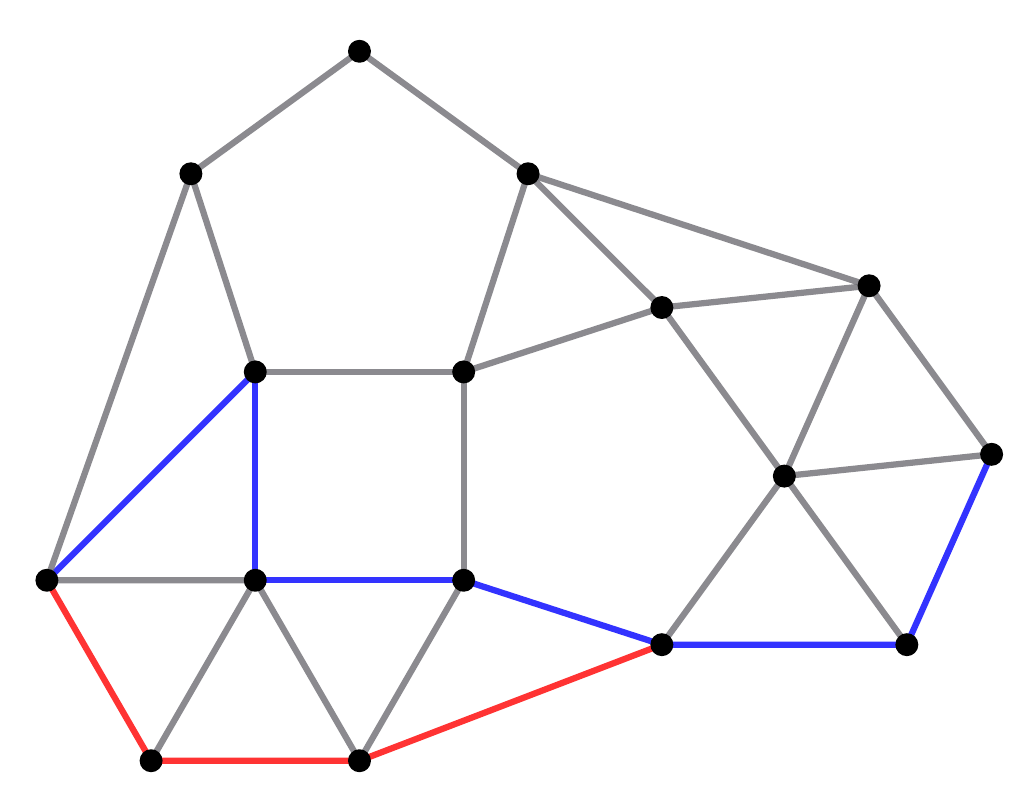} &
\includegraphics[trim=0 20cm 37cm 0,clip,width=0.20\linewidth]{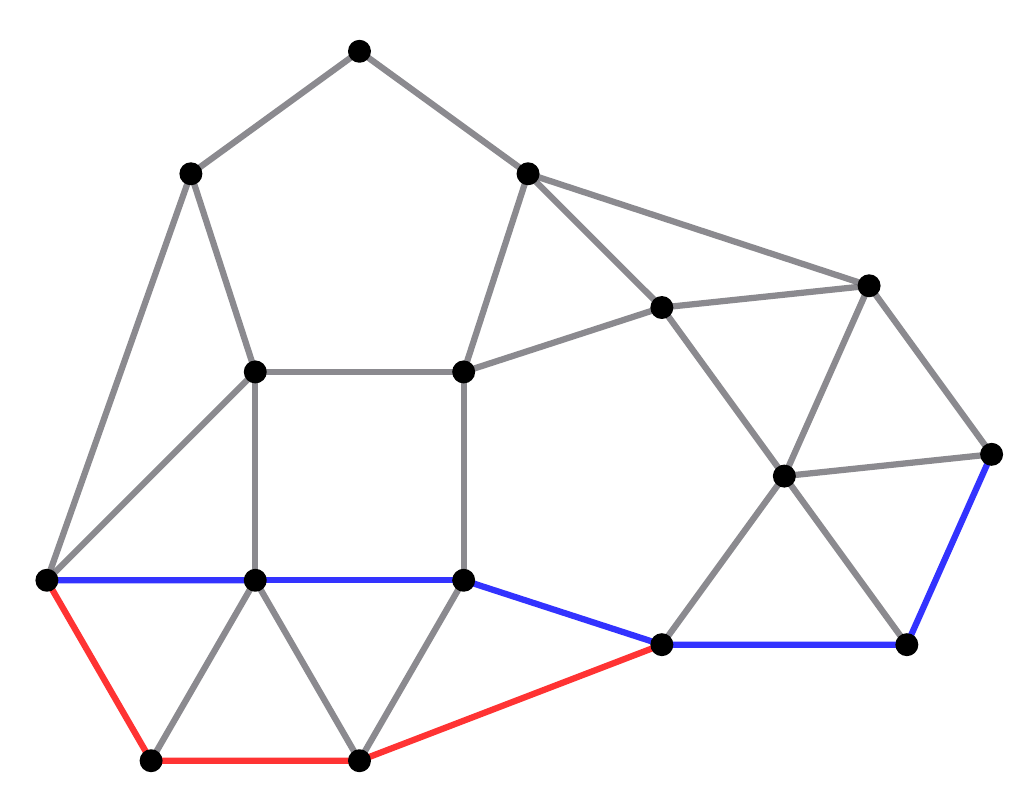} \\ \hline
\includegraphics[trim=0 20cm 37cm 0,clip,width=0.20\linewidth]{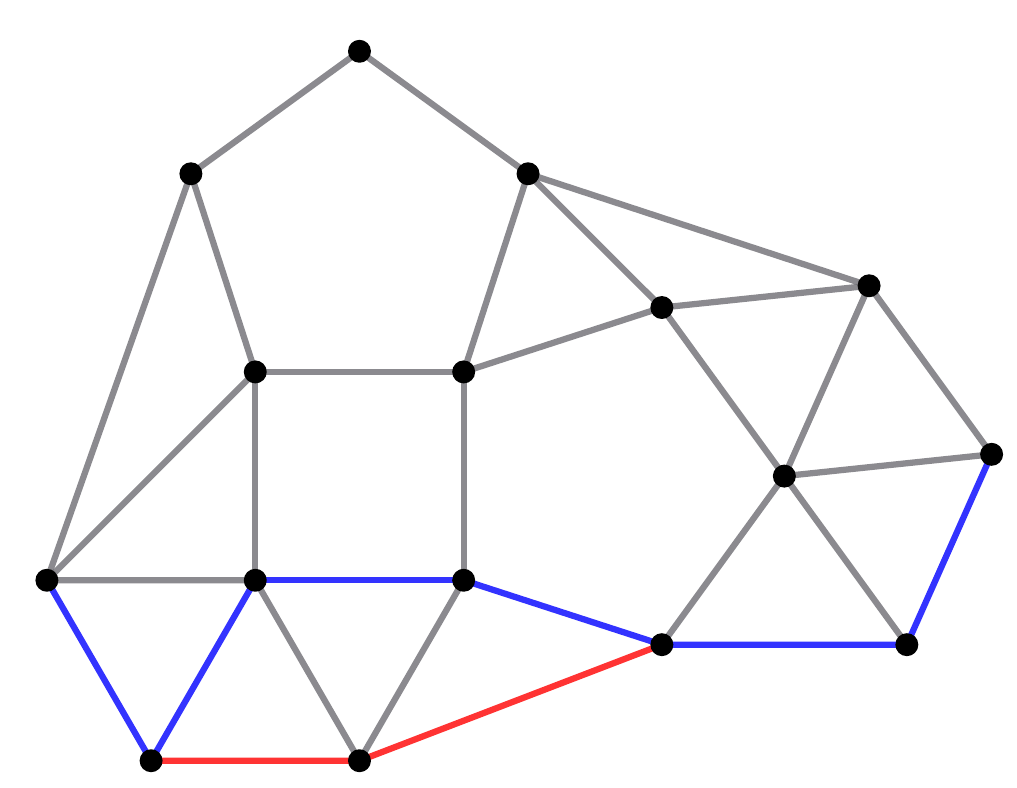} &
\includegraphics[trim=0 20cm 37cm 0,clip,width=0.20\linewidth]{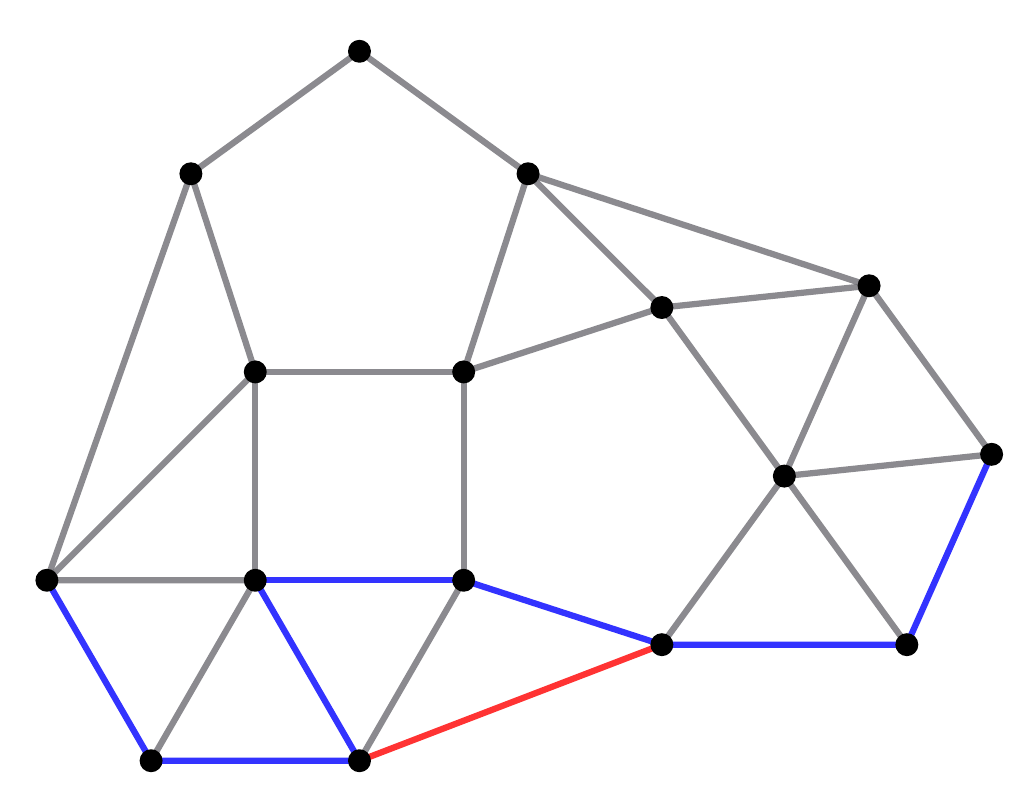} &
\includegraphics[trim=0 20cm 37cm 0,clip,width=0.20\linewidth]{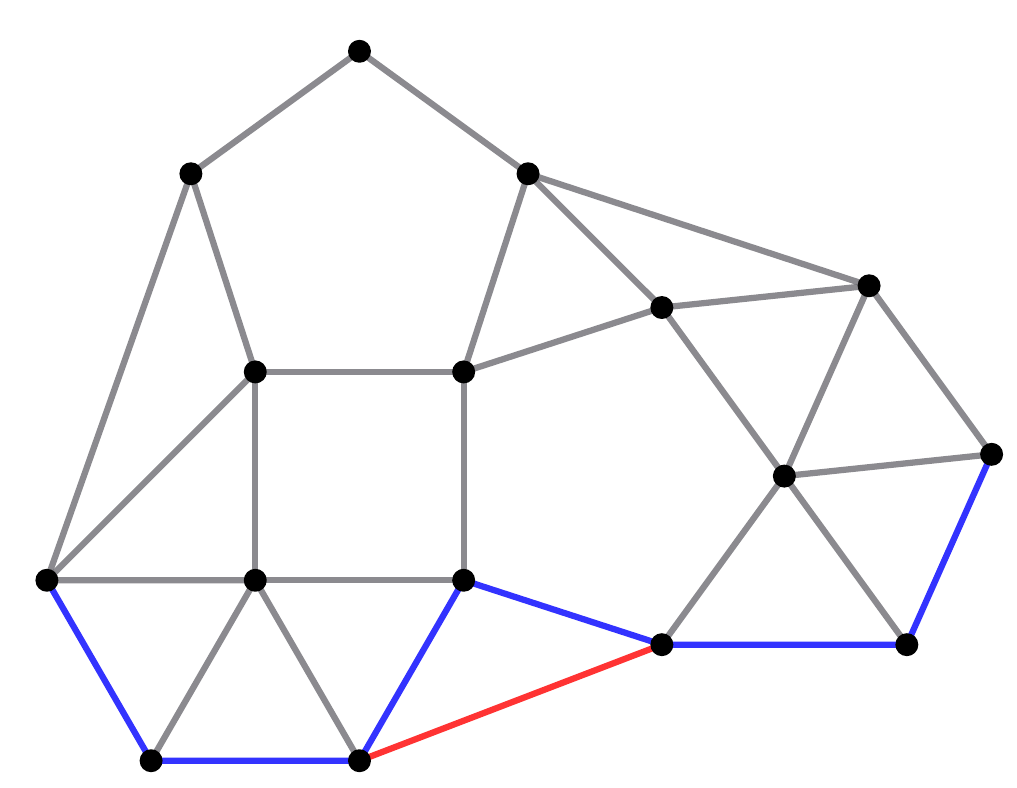} &
\includegraphics[trim=0 20cm 37cm 0,clip,width=0.20\linewidth]{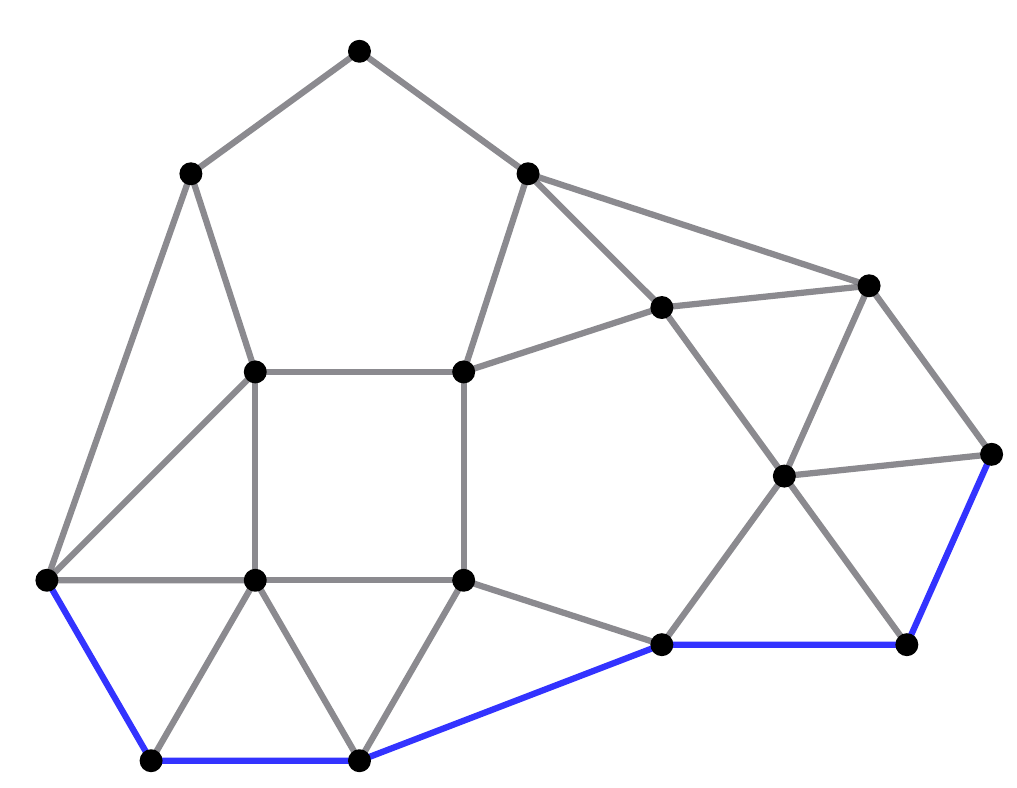} \\ \hline
\end{tabular}
\end{center}

\noindent
Then, by using this notion of homotopy equivalence, we can mimic the definition of the fundamental group of a topological space:

\begin{definition}
Let $X$ be a connected graph. Given a basepoint $o \in V(X)$ and a scale $E \geq 0$, the \emph{$E$-coarse fundamental group} $\pi_1(X,o,E)$ is the set of loops based at $o$ up to $E$-coarse homotopy equivalence endowed with concatenation. If $\pi_1(X,o,E)$ is trivial, then $X$ is \emph{$E$-coarsely simply connected}. A graph is \emph{coarsely simply connected} if it is $E$-coarsely simply connected for some $E \geq 0$. 
\end{definition}

\noindent
For instance, the square and hexagonal tilings of the plane define graphs that are coarsely simply connected. An example of a graph that is not coarsely simply connected is:
\begin{center}
\includegraphics[width=\linewidth]{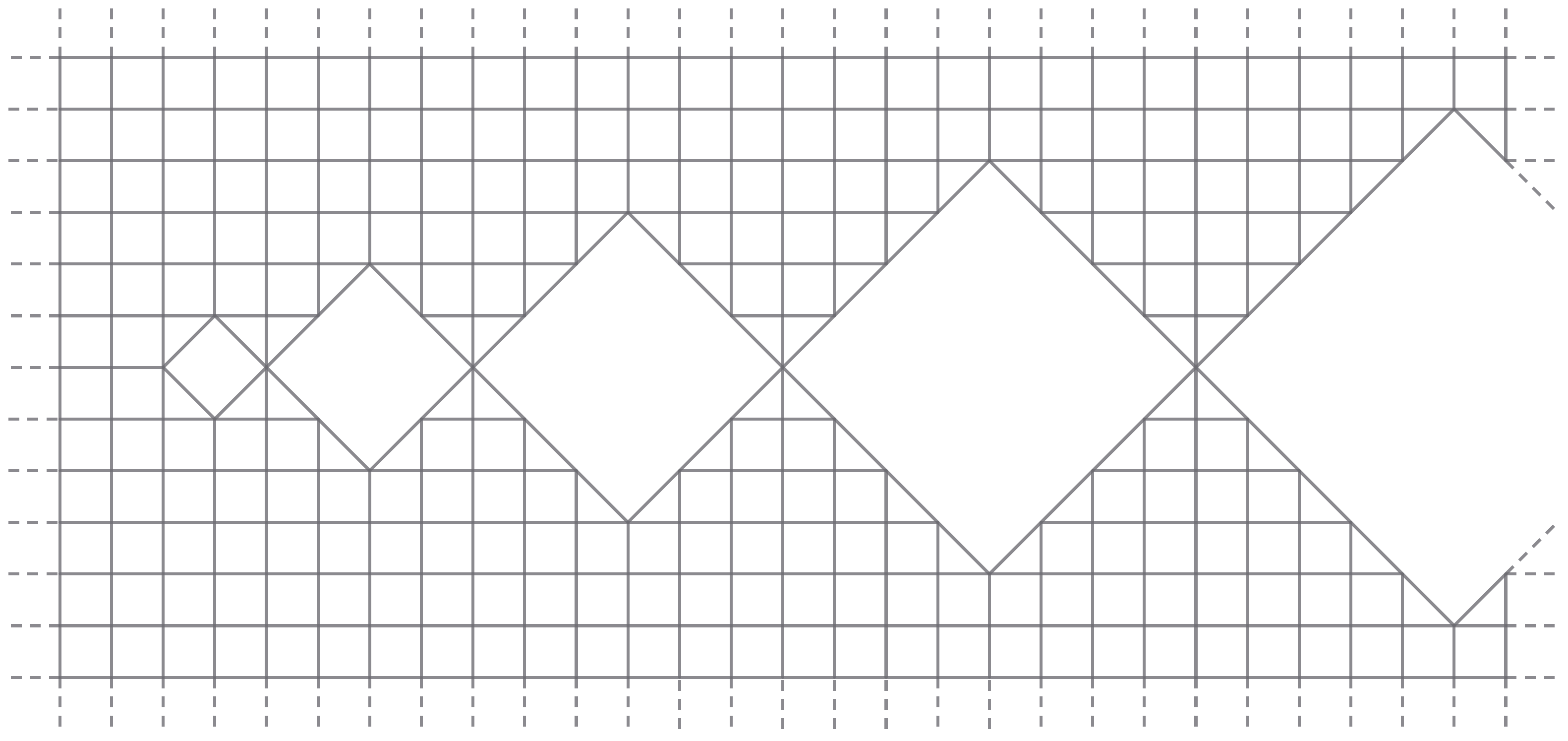}
\end{center}

\noindent
It is worth mentioning that being coarsely simply connected can be characterised through a genuine simply connectedness thanks to the so-called Rips' complexes. See Exercise~\ref{exo:Rips}. 

\medskip \noindent
Now, let us verify that being coarsely simply connected is indeed stable under quasi-isometries, as desired.

\begin{prop}\label{prop:CoarselySimplyConnected}
If $X$ and $Y$ are two quasi-isometric graphs, then $X$ is coarsely simply connected if and only if so is $Y$.
\end{prop}

\begin{proof}
Let $\varphi : X \to Y$ be a quasi-isometry. Fix a quasi-inverse $\bar{\varphi} : Y \to X$ (as defined in Exercise~\ref{exo:QuasiInverse}). Let $A>0$ and $B \geq 0$ be sufficiently large so that $\varphi,\bar{\varphi}$ are $(A,B)$-quasi-isometries and $\varphi \circ \bar{\varphi}, \bar{\varphi} \circ \varphi$ lie at distance $\leq B$ from $\mathrm{Id}_Y, \mathrm{Id}_X$ respectively. Assume that $Y$ is coarsely simply connected, say $E$-coarsely simply connected for some $E \geq 0$. Set $T:=\max ((A+B)E, A(A+B)+3B+1)$. We claim that $X$ is $T$-coarsely simply connected.

\medskip \noindent
Let $\gamma$ be a loop in $X$. We know that $\varphi(\gamma)$ is a sequence of vertices in $Y$ such that any two consecutive vertices lie at distance $\leq A+B$. Let $\gamma'$ denote the loop obtained from $\varphi(\gamma)$ by connecting successive vertices by geodesics. There exists a sequence of paths
$$\gamma_0'=\gamma', \ \gamma_1', \ldots, \ \gamma_{n-1}', \ \gamma_n' = \{ \mathrm{pt}\}$$
such that, for every $0 \leq i \leq n-1$, $\gamma_{i+1}'$ is obtained from $\gamma_i'$ by replacing some subpath $\zeta$ with a path $\xi$ such that $\zeta \cup \xi$ has diameter $\leq E$. For every $0 \leq i \leq n$, $\bar{\varphi}$ sends the loop $\gamma_i'$ to a sequence of vertices successively at distance $\leq A+B$. Let $\gamma_i''$ denote the loop obtained from $\bar{\varphi}(\gamma_i')$ by connecting successive vertices with geodesics. Then
$$\gamma_0'', \ \gamma_1'', \ldots, \ \gamma_{n-1}'', \ \gamma_n''=\{\mathrm{pt}\}$$
yields a sequence of loops such that, for every $0 \leq i \leq n-1$, $\gamma_{i+1}''$ can be obtained from $\gamma_i''$ by replacing some subpath $\zeta$ with some path $\xi$ such that $\zeta \cup \xi$ is a cycle of diameter $\leq (A+B)E$. It follows that $\gamma_0''$ is $T$-coarsely homotopically trivial because $T \geq (A+B)E$.

\begin{center}
\includegraphics[width=0.7\linewidth]{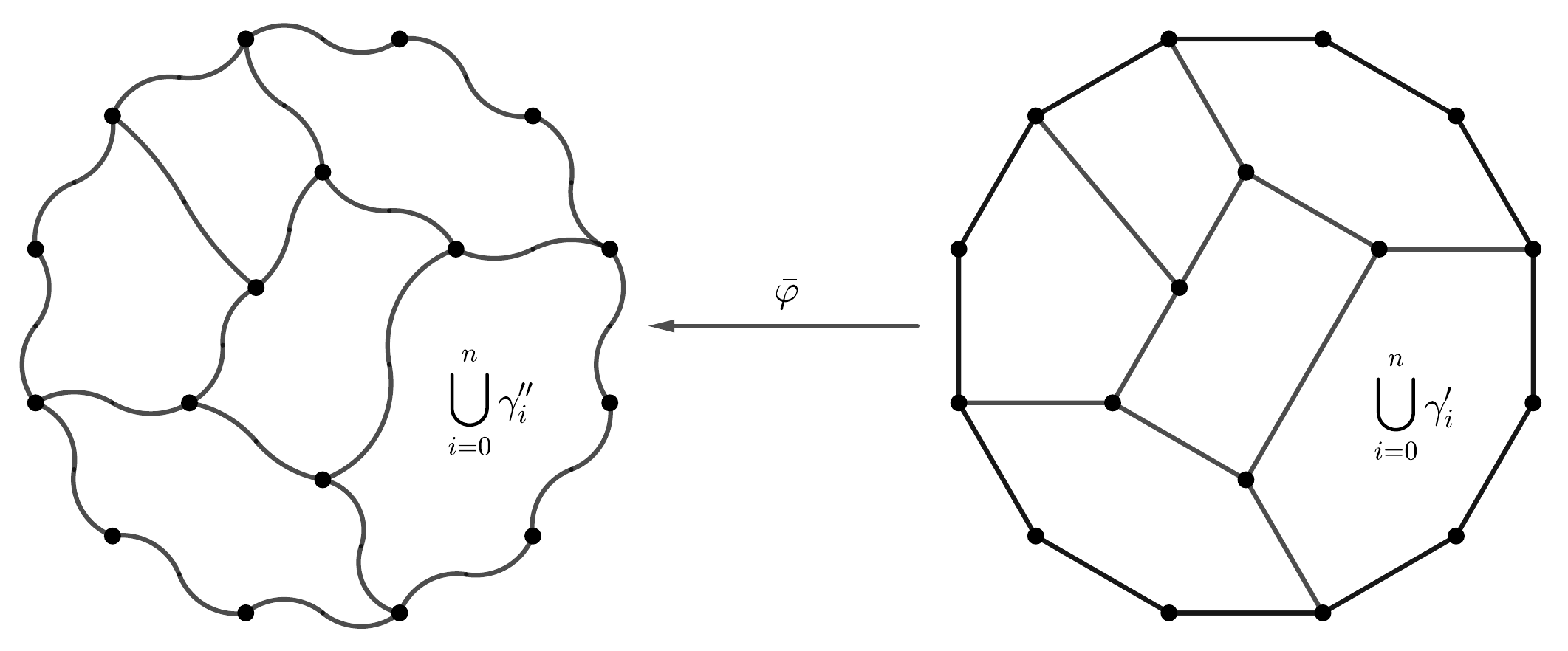}
\end{center}

\noindent
Notice that, by construction, $\gamma_0''$ passes through the vertices of $\bar{\varphi} \circ \varphi(\gamma_0)$. Moreover, the length of a subsegment of $\gamma_0''$ between two such vertices is $\leq A(A+B)+B$. 

\begin{center}
\includegraphics[width=\linewidth]{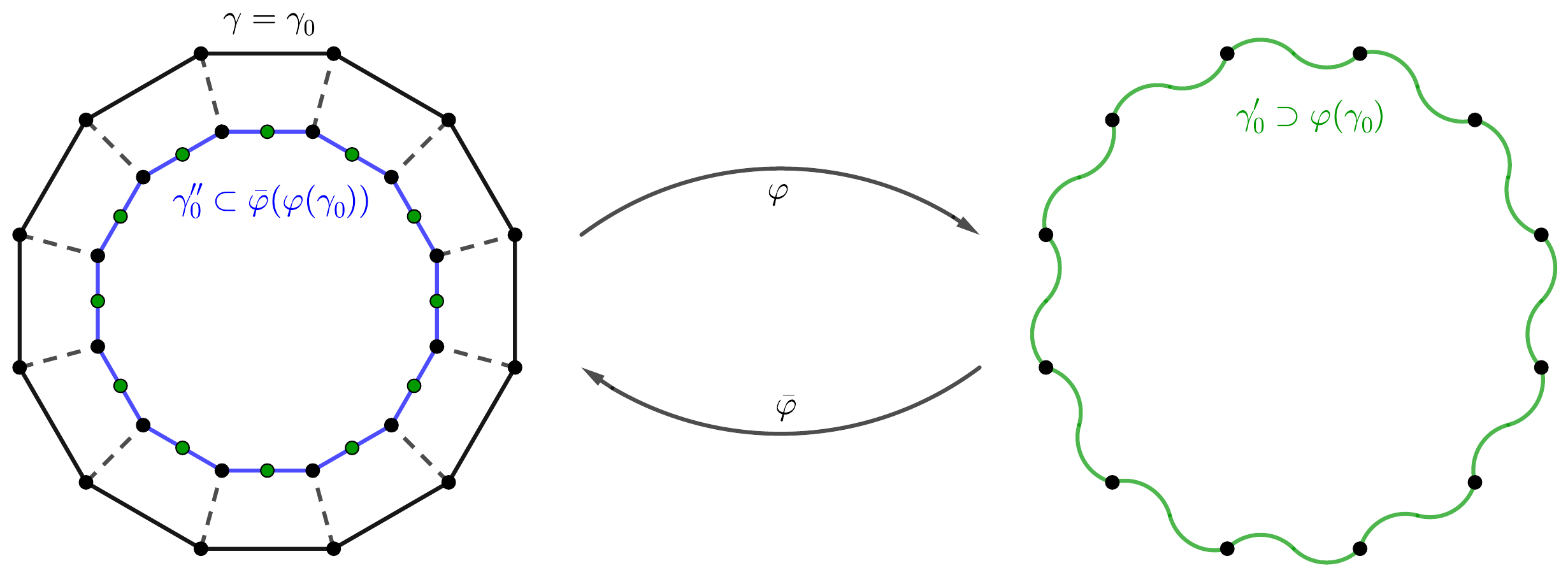}
\end{center}

\noindent
Because $T \geq A(A+B)+3B+1$, it follows that $\gamma_0''$ and $\gamma_0=\gamma$ are $T$-coarsely homotopically equivalent. Since we have already proved that $\gamma_0''$ is $T$-coarsely homotopically trivial, we conclude that $\gamma$ is $T$-coarsely homotopically trivial. 
\end{proof}

\noindent
Interestingly, for (Cayley graphs of) finitely generated groups, being coarsely simply connected can be characterised purely algebraically. 

\begin{prop}\label{prop:WhenFP}
A finitely generated group is coarsely simply connected if and only if it admits a finite presentation.
\end{prop}

\begin{proof}
Let $G$ be a finitely generated group. First, assume that $G$ admits a finite presentation, say $\langle s_1, \ldots, s_n \mid r_1, \ldots, r_m \rangle$. Setting $M:=\max \{ \|r_i\|_S \mid 1 \leq i \leq m\}$ where $S:= \{s_1, \ldots, s_n\}$, we claim that $\mathrm{Cayl}(G,S)$ is $M$-coarsely simply connected. Indeed, let $\gamma$ be a loop in $\mathrm{Cayl}(G,S)$ based at $1$. Reading the generators labelling the edges of $\gamma$, we get a word $w$ written over $S \cup S^{-1}$. Notice that, if $w$ is not freely reduced (i.e.\ if there exists a subword $ss^{-1}$ or $s^{-1}s$ in $w$), then $\gamma$ contains an edge-backtrack. Since our goal is to prove that $\gamma$ is coarsely homotopically trivial, we can assume that it is no backtrack, and therefore that $w$ is freely reduced. The fact that $\gamma$ is a loop imposes that $w$ represents the trivial element in $G$. Therefore, $w$ coincides with the free reduction of a word of the form 
$$g_1t_1g_1^{-1} g_2t_2g_2^{-1} \cdots g_k t_k g_k^{-1} \text{ for some } t_1, \ldots, t_k \in R \cup R^{-1}$$
where $R:= \{r_1, \ldots, r_m\}$. Let $\alpha$ denote the path in $\mathrm{Cayl}(G,S)$ starting from $1$ and labelled by this word. 
\begin{center}
\includegraphics[width=0.5\linewidth]{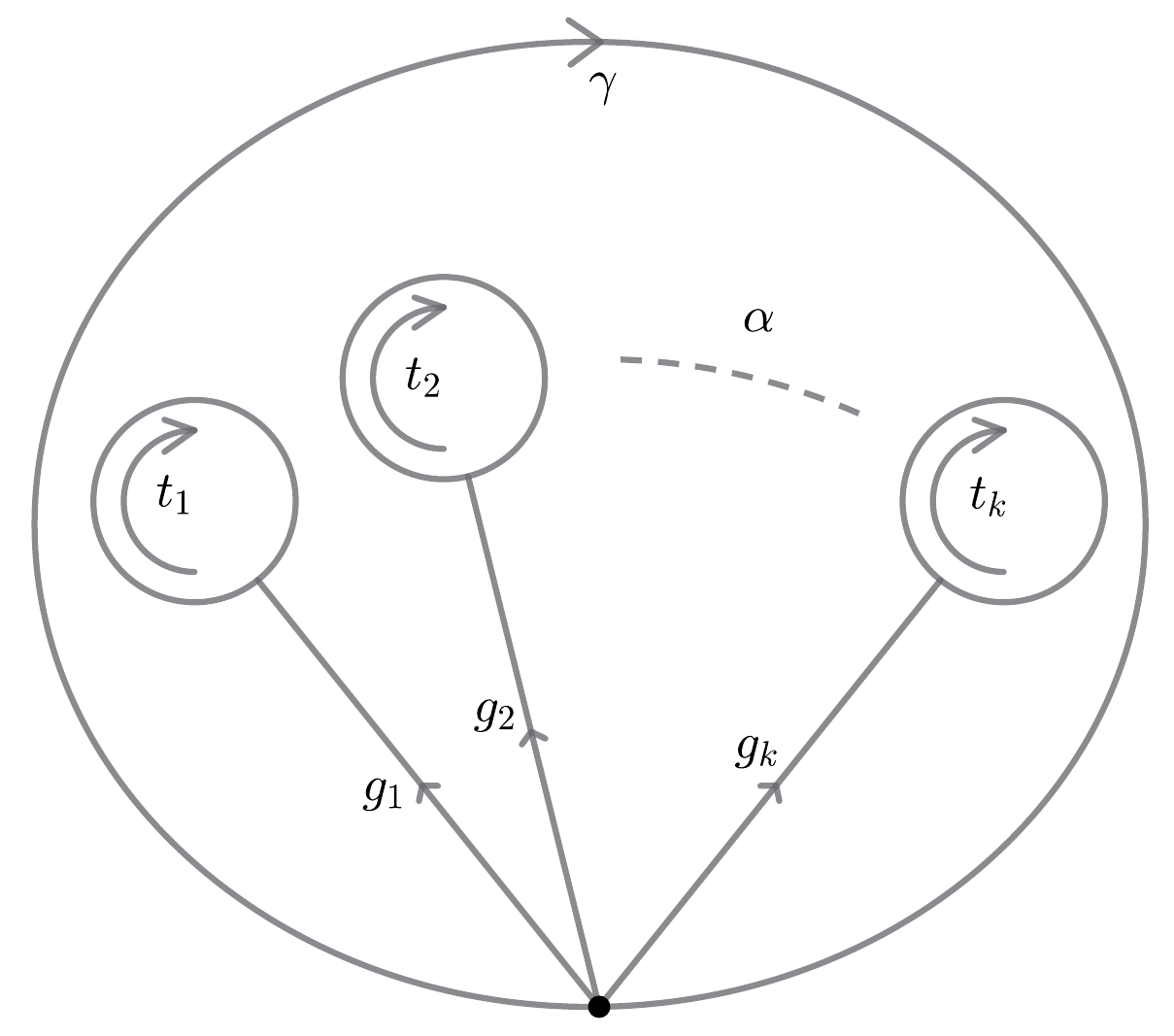}
\end{center}
This is a loop because the word represents a trivial element of $G$. As already mentioned, a subword of the form $ss^{-1}$ or $s^{-1}s$ corresponds to a backtrack in $\alpha$. Thus, reducing our word amounts to removing backtracks from $\alpha$, which of course does not modify the coarse homotopy class of our loop. But $w$ is the word we obtain after reduction, so $\gamma$ can be obtained from $\alpha$ by removing backtrack. Consequently, in order to conclude that $\gamma$ is $M$-coarsely homotopically trivial, it suffices to show that $\alpha$ is $M$-coarsely homotopically trivial. But $\alpha$ can be described as the concatenation of the loops in $\mathrm{Cayl}(G,S)$ based at $1$ labelled by $g_1t_1g^{-1}, \ldots, g_kt_kg_k^{-1}$. For every $1 \leq i \leq k$, the loop associated to $g_it_ig_i^{-1}$ follows a path labelled by $g_i$, next a loop labelled by $t_i$ (which as diameter $\leq \|t_i\|_S \leq M$), and finally backtracks along our previous path labelled by $g_i$. Consequently, this loop is $M$-coarsely homotopically trivial in $\mathrm{Cayl}(G,S)$, as desired.

\medskip \noindent
Conversely, assume that $G$ is coarsely simply connected. Fix a finite generating set $S \subset G$ and let $M \geq 0$ be a constant such that $\mathrm{Cayl}(G,S)$ is $M$-coarsely simply connected. Let $R$ denote the set of the words labelling all the loops of diameters $\leq M$ based at $1$ in $\mathrm{Cayl}(G,S)$. We claim that $\langle S \mid R \rangle$ is a presentation of $G$. In order words, if $w$ is a words written over $S \cup S^{-1}$ representing a trivial element in $G$, we claim that this can be deduced from the relations given by $R$. Because $w$ is trivial in $G$, the path $\gamma$ in $\mathrm{Cayl}(G,S)$ starting from $1$ and labelled by $w$ is a loop. There exists a sequence of loops
$$\gamma_0=\gamma, \ \gamma_1, \ldots, \ \gamma_{n-1}, \ \gamma_n= \{ \mathrm{pt}\}$$
such that, for every $0 \leq i \leq n-1$, $\gamma_{i+1}$ is obtained from $\gamma_i$ by replacing some subpath $\zeta$ with a new subpath $\xi$ satisfying $\mathrm{diam}(\zeta \cup \xi) \leq M$. The labels of these loops yield a sequence of words
$$w_0=w, \ w_1, \ldots, \ w_{n-1}, \ w_n=1$$
written over $S \cup S^{-1}$ such that, for every $0 \leq i \leq n-1$, $w_{i+1}$ is obtained from $w_i$ by replacing a subword $r_1$ with a subword $r_2$ where $r_1r_2^{-1}$ belongs to $R$. Thus, the equality $w=1$ in $G$ can be deduced from the relations given by $R$, as desired.
\end{proof}

\noindent
Thus, being finitely presented, which is a purely algebraic property, only depends on the coarse topology of a finitely generated group. As an immediate consequence:

\begin{cor}\label{cor:FPQI}
Let $G_1,G_2$ be two quasi-isometric finitely generated groups. If $G_1$ admits a finite presentation, then so does $G_2$. 
\end{cor}

\begin{proof}
This is an immediate consequence of Propositions~\ref{prop:WhenFP} and~\ref{prop:CoarselySimplyConnected}.
\end{proof}

\noindent
Now, let us focus on wreath products of graphs and groups. As justified by our next theorem, wreath products of graphs are essentially never coarsely simply connected. 

\begin{thm}\label{thm:NotCoarselySimplyConnected}
Let $(X,o)$ be a pointed graph and $Y$ a graph. If $X$ contains at least two vertices and if $Y$ is unbounded, then $(X,o) \wr Y$ is not coarsely simply connected. 
\end{thm}

\noindent
In order to prove the theorem, the following general observation will be needed:

\begin{lemma}\label{lem:NotCoarselySimplConnected}
Let $X$ be a graph, $Y \leq X$ a subgraph, and $L,R \geq 0$ two constants. If $X$ is $R$-coarsely simply connected and if there exists an $L$-Lipschitz retract $X \to Y$, then $Y$ is $(L+1)R$-coarsely simply connected. 
\end{lemma}

\begin{proof}
Let $\gamma$ be a loop in $Y$. Since $X$ is $R$-coarsely simply connected, there exists a sequence of loops in $X$
$$\gamma_0=\gamma, \gamma_1, \ldots, \gamma_{n-1}, \gamma_n = \{ \mathrm{pt}\}$$
such that, for every $0 \leq i \leq n-1$, $\gamma_{i+1}$ is obtained from $\gamma_i$ by replacing some subpath $\zeta_i$ with a new path $\xi_i$ satisfying $\mathrm{diam}(\zeta_i \cup  \xi_i) \leq R$. By projecting everything to $Y$ thanks to our $L$-Lipschitz retract $r : X \to Y$, we get a sequence of subgraphs in $Y$
$$r(\gamma_0)=\gamma, r(\gamma_1), \ldots, r(\gamma_{n-1}), r(\gamma_n) = \{ \mathrm{pt}\}$$
where, for every $0 \leq i \leq n-1$, $r(\gamma_{i+1})$ is obtained from $r(\gamma_i)$ by replacing the subgraph $r(\zeta_i)$ with the subgraph $r(\xi_i)$. Notice that $\mathrm{diam}(r(\zeta_i) \cup r(\xi_i)) \leq LR$. Each $r(\gamma_i)$ is a sequence of vertices successively at distance $\leq L$, so we can turn $r(\gamma_i)$ into a genuine loop $\gamma_i'$ by connecting any two successive vertices with some geodesics in $Y$. Then, we find a sequence of loops in $Y$
$$\gamma_0'=\gamma, \gamma_1', \ldots, \gamma_{n-1}', \gamma_n' = \{ \mathrm{pt}\}$$
where, for every $0 \leq i \leq n-1$, $\gamma_{i+1}'$ is obtained from $\gamma_i'$ by replacing some subpath $\zeta_i'$ with a new path $\xi_i'$ satisfying $\mathrm{diam}(\zeta_i' \cup \xi') \leq (L+1)R$. Thus, $\gamma$ is $(L+1)R$-coarsely homotopically trivial. 
\end{proof}

\begin{proof}[Proof of Theorem~\ref{thm:NotCoarselySimplyConnected}.]
Fix a vertex $x \in V(X)$ distinct from $o$ and $p,q \in V(Y)$ two vertices at distance $\geq R$ for some arbitrary $R \geq 0$ that we fix. Consider the subgraph $Z$ of $(X,o) \wr Y$ induced by the vertices
$$\left\{ (c, w) \mid c(y)=o \text{ for every } y \notin \{p,q\} \text{ and } c(p),c(q) \in \{o,x\} \right\}.$$
Now, consider the retraction $\rho : (X,o) \wr Y \to Z$ given by 
$$((c, w) \mapsto \left( y \mapsto \left\{ \begin{array}{cl} o & \text{if } y \notin \{p,q\} \\ o & \text{if } y \in \{p,q\} \text{ but } x_y \notin \{o,x \} \\ c(y) & \text{otherwise} \end{array} \right., \ w \right).$$
In other words, $\rho$ switches off the lamps not in $\{p,q\}$, switches off the lamps in $\{o,x\}$ if they have a colour not in $\{o,x\}$, and does not modify a lamp in $\{p,q\}$ taking $o$ or $x$ as a colour. 

\medskip \noindent
Notice that $\rho$ is $1$-Lipschitz. Indeed, given two vertices $(a, u)$ and $(b, v)$, if we denote by $(a', u)$ and $(b', v)$ as their respective image under $\rho$, then, using Lemma~\ref{lem:DistWreath}, we find that
$$\begin{array}{lcl} d(\rho(a,u), \rho(b,v)) & = & d((a',u), (b',v)) = \mathrm{TS}(u, a' \triangle b', v) \\ & \leq & \mathrm{TS}(u, a \triangle b, v) = d((a,u),(b,v)) \end{array}$$
where the inequality is justified by the inclusion $a' \triangle b' \subset a \triangle b$. 

\medskip \noindent
Then, we deduce from Lemma~\ref{lem:NotCoarselySimplConnected} that, if $(X,o) \wr Y$ is $D$-coarsely simply connected, then $Z$ must be $2D$-coarsely simply connected. But $Z$ is just four copies of $Y$ cyclically connected by edges. 
\begin{center}
\includegraphics[width=0.8\linewidth]{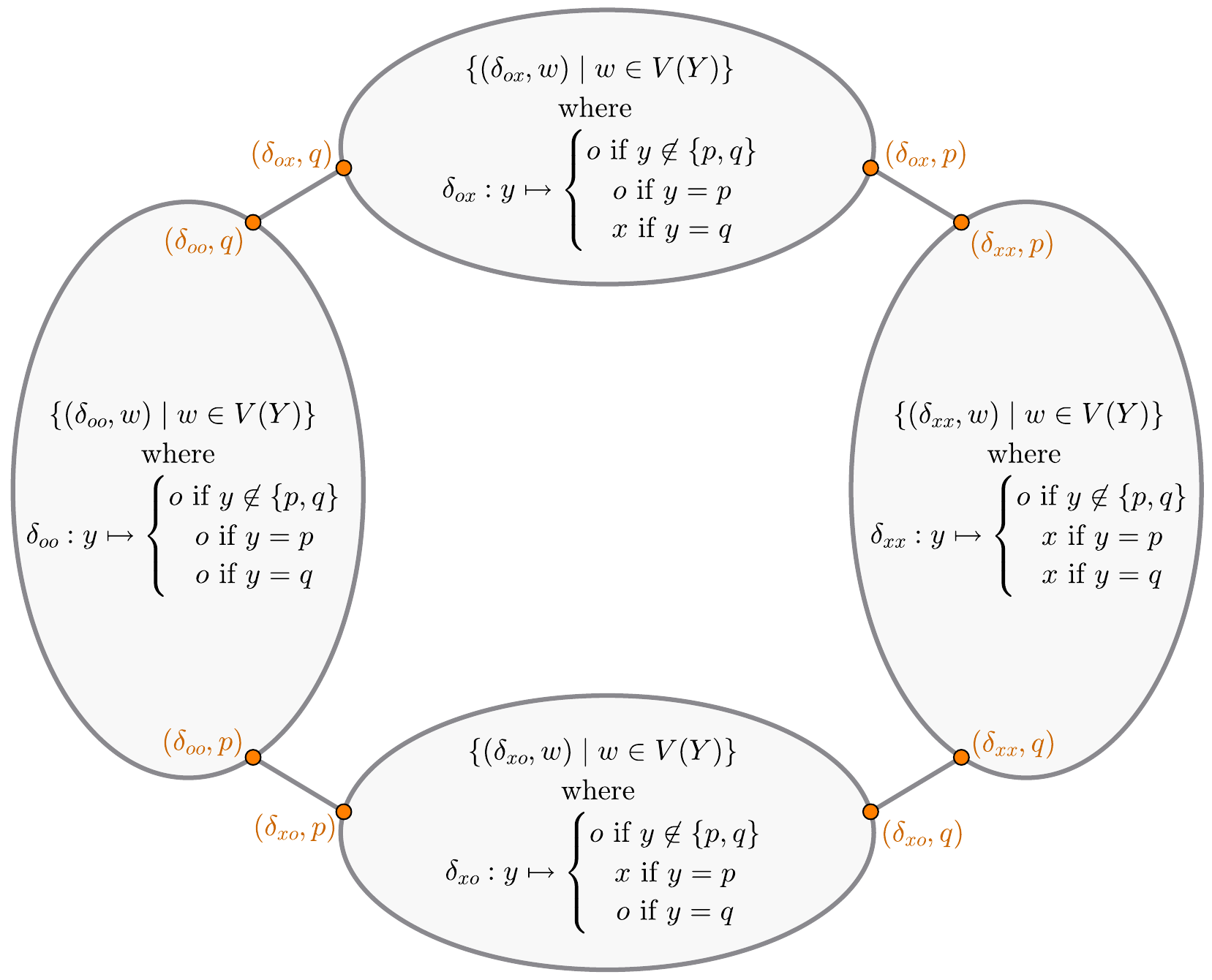}
\end{center}
Clearly, $Z$ contains a loop that is not $(4R+3)$-coarsely homotopically trivial. Therefore, $(X,o) \wr Y$ is not $(4R-3)/2$-coarsely simply connected. Since $R$ can be chosen arbitrarily large, we conclude that $(X,o) \wr Y$ is not coarsely simply connected.
\end{proof}

\noindent
As a consequence, it follows that most wreath products of groups are not finitely presentable. (See Exercise~\ref{exo:WreathFP} for a more precise characterisation.)

\begin{cor}
Let $A$ and $B$ be two finitely generated groups. If $A$ is non-trivial and $B$ infinite, then $A \wr B$ is not finitely presentable. 
\end{cor}

\begin{proof}
Since $A \wr B$ has a Cayley graph that is a wreath product of Cayley graphs of $A$ and $B$ (according to Proposition~\ref{prop:CaylWreath}), it follows from Theorem~\ref{thm:NotCoarselySimplyConnected} that $A \wr B$ is not coarsely simply connected. We conclude from Proposition~\ref{prop:WhenFP} that $A \wr B$ is not finitely presentable. 
\end{proof}

\subsection{Coarse separation}\label{section:CoarseSeparation}

\noindent
In topology, it might be useful to understand which subspaces can separate a given space. For instance, one can distinguish $\mathbb{R}$, $\mathbb{S}^1$, and $\mathbb{R}^2$ up to homeomorphism by noticing that $\mathbb{R}$ can be separated by removing a single point, that $\mathbb{S}^1$ can be separated by removing two points but not a single point, and that $\mathbb{R}^2$ cannot be separated by removing only finitely many points. Topological separation can be coarsified as follows:

\begin{definition}
Let $X$ be a graph. A subgraph $Z$ \emph{coarsely separates} $X$ if there exists $F \geq 0$ such that, for every $L \geq 0$, $X \backslash Z^{+L}$ contains at least two connected components with vertices at distance $\geq L$ from $Z$. 
\end{definition}

\noindent
For instance, the line $\mathbb{E}^1$ is coarsely separated by any of its vertices. On the other hand, the infinite grid $\mathbb{E}^2$ cannot be coarsely separated by a single vertex. But it can be coarsely separated by a line. An instructive example is the half-line $\mathbb{N}$. Can it be coarsely separated by a single vertex? On the one hand, removing a vertex from $\mathbb{N}$ always yields an unbounded connected component and bounded connected component. On the other hand, this bounded component can have an arbitrarily large diameter if one chooses carefully the vertex we remove. This property distinguishes $\mathbb{N}$ from both $\mathbb{E}$ and $\mathbb{E}^2$ up to quasi-isometry, but, formally, $\mathbb{N}$ is not coarsely separated by a vertex. Loosely speaking, the vertices of $\mathbb{N}$ do not coarsely separate ``individually'' but they do coarsely separate ``collectively''. This phenomenon motivates the following definition:

\begin{definition}
Let $X$ be a graph. A collection of subgraphs $\mathcal{Z}$ \emph{coarsely separates} $X$ if there exists $F \geq 0$ such that, for every $L \geq 0$, one can find some $Z \in \mathcal{Z}$ for which $X \backslash Z^{+F}$ contains at least two connected components with vertices at distance $\geq L$ from~$Z$. 
\end{definition}

\noindent
For instance, in our previous example, $\mathbb{N}$ is not coarsely separated by $\{v\}$ for every vertex $v$ but it is coarsely separated by the family $\{\{v\} \mid v \text{ vertex}\}$. 

\medskip \noindent
\begin{minipage}{0.35\linewidth}
\includegraphics[width=0.95\linewidth]{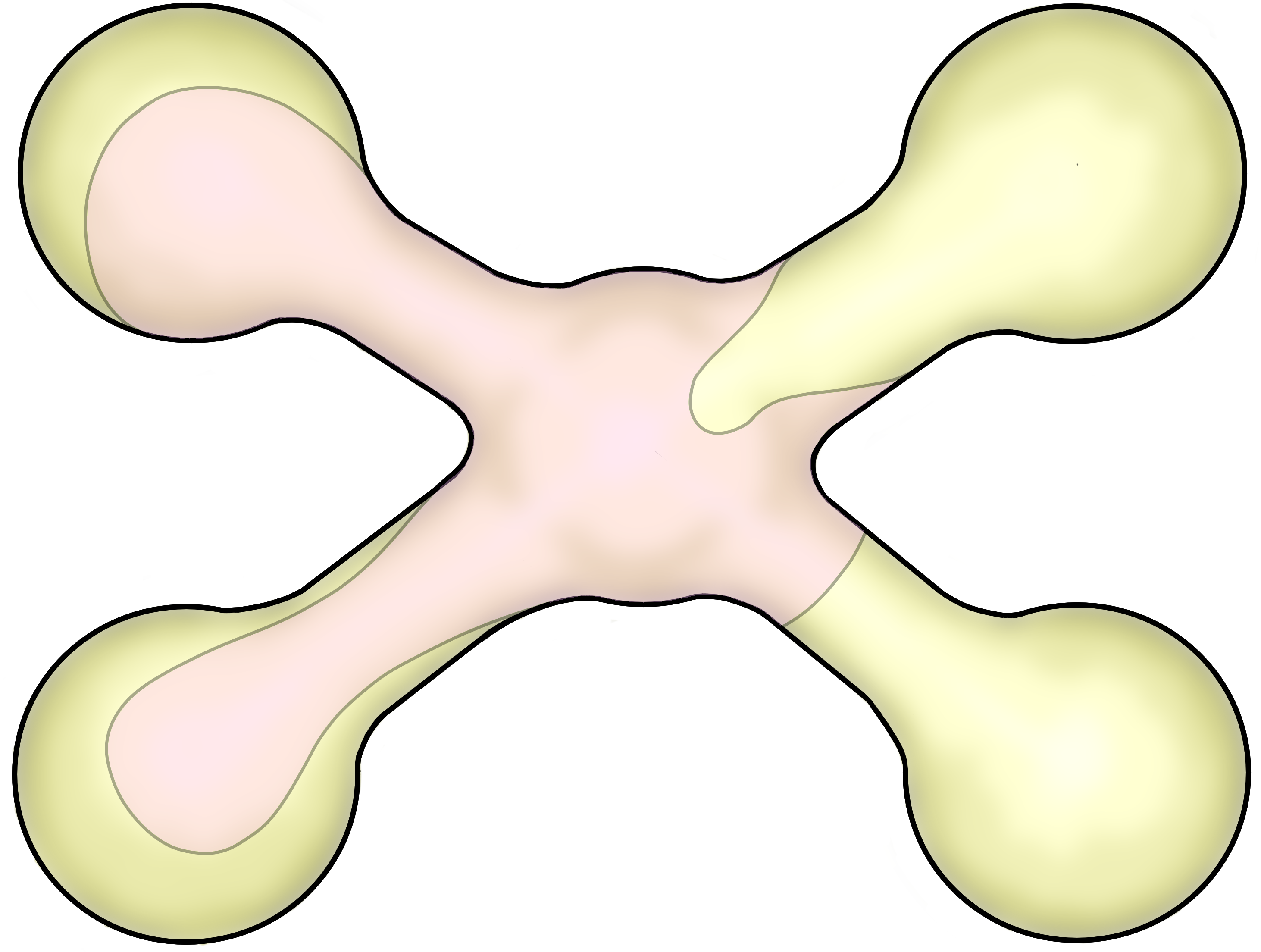}
\end{minipage}
\begin{minipage}{0.61\linewidth}
Defining coarse separation for a family of subgraphs and not only for a single subgraph is also motivated by the study of the coarse geometry of graphs of spaces. In a nutshell, if $\rho : X \to Y$ is a coarse embedding from some graph $X$ to some graph of spaces $Y$, either $\rho(X)$ is contained in a neighbourhood of some vertex-space or it is coarsely separated by $\{ \rho(X) \cap S \mid S \text{ edge-space}\}$. 
\end{minipage}

\medskip \noindent
Let us verify that coarse separation is preserved by quasi-isometries.

\begin{prop}
Let $X,Y$ be two graphs and $\varphi : X \to Y$ a quasi-isometry. If $X$ is coarsely separated by a family of subgraphs $\mathcal{Z}$, then $Y$ is coarsely separated by $\varphi(\mathcal{Z}):= \{ \varphi(Z) \mid Z \in \mathcal{Z}\}$. 
\end{prop}

\begin{proof}
Let $\bar{\varphi}$ be a quasi-inverse of $\varphi$ and let $A>0$, $B \geq 0$ be constants such that $\varphi, \bar{\varphi}$ are $(A,B)$-quasi-isometries and such that $\varphi \circ \bar{\varphi}, \bar{\varphi} \circ \varphi$ lie at distance $\leq B$ from identities (see Exercise~\ref{exo:ExQI}). Because $\mathcal{Z}$ coarsely separates $X$, there exists some $F \geq 0$ such that, for every $L \geq 0$, one can find a subgraph $Z \in \mathcal{Z}$ and two vertices $a,b \in V(X)$ at distance $\geq A^2(A+F+B) + A(L+3B)+F$ from $Z^{+F}$ such that $Z^{+F}$ separates $a$ and $b$. Notice that
$$\begin{array}{lcl} d(\varphi(a), \varphi(Z)^{+A(F+A+B)+2B}) & \geq & d(\varphi(a),\varphi(Z)) - A(F+A+B)-2B \\ \\ & \geq & \frac{1}{A} d(a,Z) - A(F+A+B) - 3B \\ \\ & \geq & \frac{1}{A} d(a,Z^{+F}) - \frac{F}{A} - A(F+A+B) -3B \geq L. \end{array}$$
Similarly, $d(\varphi(b), \varphi(Z)^{+A(F+A+B)+2B}) \geq L$. Now, we claim that $\varphi(Z)^{+A(F+A+B)+2B}$ separates $\varphi(a)$ and $\varphi(b)$. So let $\alpha$ be a path connecting $\varphi(a)$ and $\varphi(b)$. 
\begin{center}
\includegraphics[width=\linewidth]{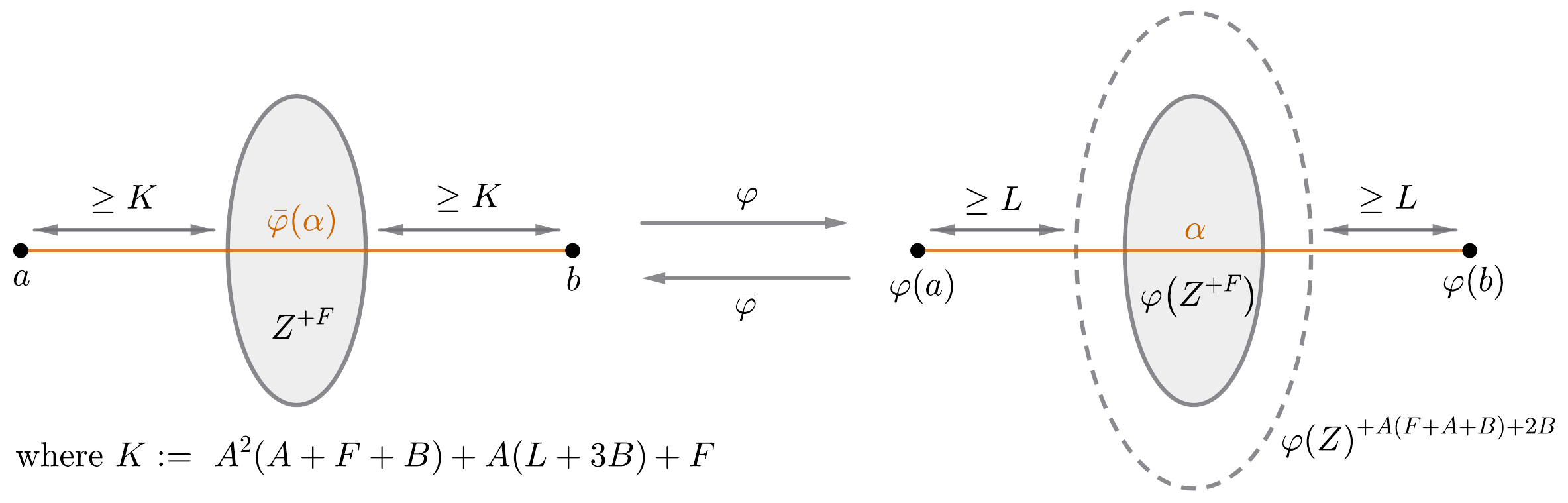}
\end{center}
Since $\bar{\varphi}(\alpha)$ is a sequence of vertices successively at distance $\leq A+B$ connecting $\bar{\varphi}(\varphi(a))$ and $\bar{\varphi}(\varphi(b))$, which lie at distance $\leq B$ from $a$ and $b$ respectively, necessarily $\bar{\varphi}(\alpha)$ intersects $Z^{+F+A+B}$ since $Z^{+F}$ separates $a$ and $b$. It follows that
$$\emptyset \neq \varphi(\bar{\varphi}(\alpha)) \cap \varphi(Z^{+F+A+B}) \subset   \alpha^{+B} \cap \varphi(Z)^{+A(F+A+B)+B},$$
from which we deduce that $\alpha$ intersects $\varphi(Z)^{+A(F+A+B)+2B}$.
\end{proof}

\noindent
We will be especially interested in coarse separation by (families of) vertices. This property is well-known in geometric group theory and it has its own name:

\begin{definition}
An unbounded graph $X$ is \emph{one-ended} if, for every $v \in V(X)$, $X$ is not coarsely separated by $\{v\}$; and it is \emph{uniformly one-ended} if it is not coarsely separated by $\{ \{v\} \mid v \in V(X)\}$. 
\end{definition}

\noindent
In other words, our graph $X$ is uniformly one-ended if, for every $F \geq 0$, there exists some $L \geq 0$ such that removing from $X$ a ball of radius $F$ yields a single unbounded connected component and bounded connected components all of diameters $\leq F$; and it is one-ended if removing a ball always yields a single unbounded component (with no restriction on the remaining bounded components). For instance, $\mathbb{N}$ is one-ended but not uniformly one-ended. For quasi-transitive graphs (such as Cayley graphs), the two notions coincide (see Exercise~\ref{exo:UnifOneEnded}).

\medskip \noindent
As shown by our next proposition, wreath products of graphs are usually uniformly one-ended.

\begin{prop}\label{prop:OneEnded}
Let $(X,o)$ be a pointed graph and $Y$ a graph of bounded. If $X$ contains at least two vertices and if $Y$ is unbounded, then $(X,o) \wr Y$ is uniformly one-ended.
\end{prop}

\begin{proof}
Fix a vertex $x \in V(X)$ distinct from $o$ and a constant $R \geq 0$. Our goal is to prove that, when we remove a ball of radius $R$ from $(X,o) \wr Y$, there is one unbounded connected components and a bunch of uniformly bounded components. For convenience, we we will assume that our ball is centred at $(c_o,o)$ where $c_o$ denotes the colouring identically equal to $o$. 

\medskip \noindent
Let $\kappa$ denote the colouring $V(Y) \to V(X)$ that is identically equal to $o$ outside $B(o,R)$ and identically equal to $x$ inside $B(o,R)$. We claim that, for every vertex $(c,p) \in (X,o) \wr Y$ satisfying $d((c,p),(c_o,o)) > (1+4R) |B(o,2R)|$, there exists a path connecting $(c,p)$ to $(\kappa,o)$ that stays at distance $>R$ from $(c_o,o)$. This will be sufficient to conclude the proof of the proposition. 

\medskip \noindent
First, notice that, along a geodesic connecting $(c_o,o)$ to $(c,p)$, we can find a vertex $(a,q)$ satisfying $d(q,o)>2R$. Otherwise, it would mean that $(c,p)$ can be obtained from $(c_o,o)$ by moving the arrow and modifying the lamps only in $B(o,2R)$, which would imply that $d((c,p),(c_o,o)) \leq (1+4R) |B(o,2R)|$, contradicting our assumption. 

\begin{figure}
\begin{center}
\begin{tabular}{|c|c|c|c|} \hline
\includegraphics[trim=0 19cm 24cm 0,clip,width=0.22\linewidth]{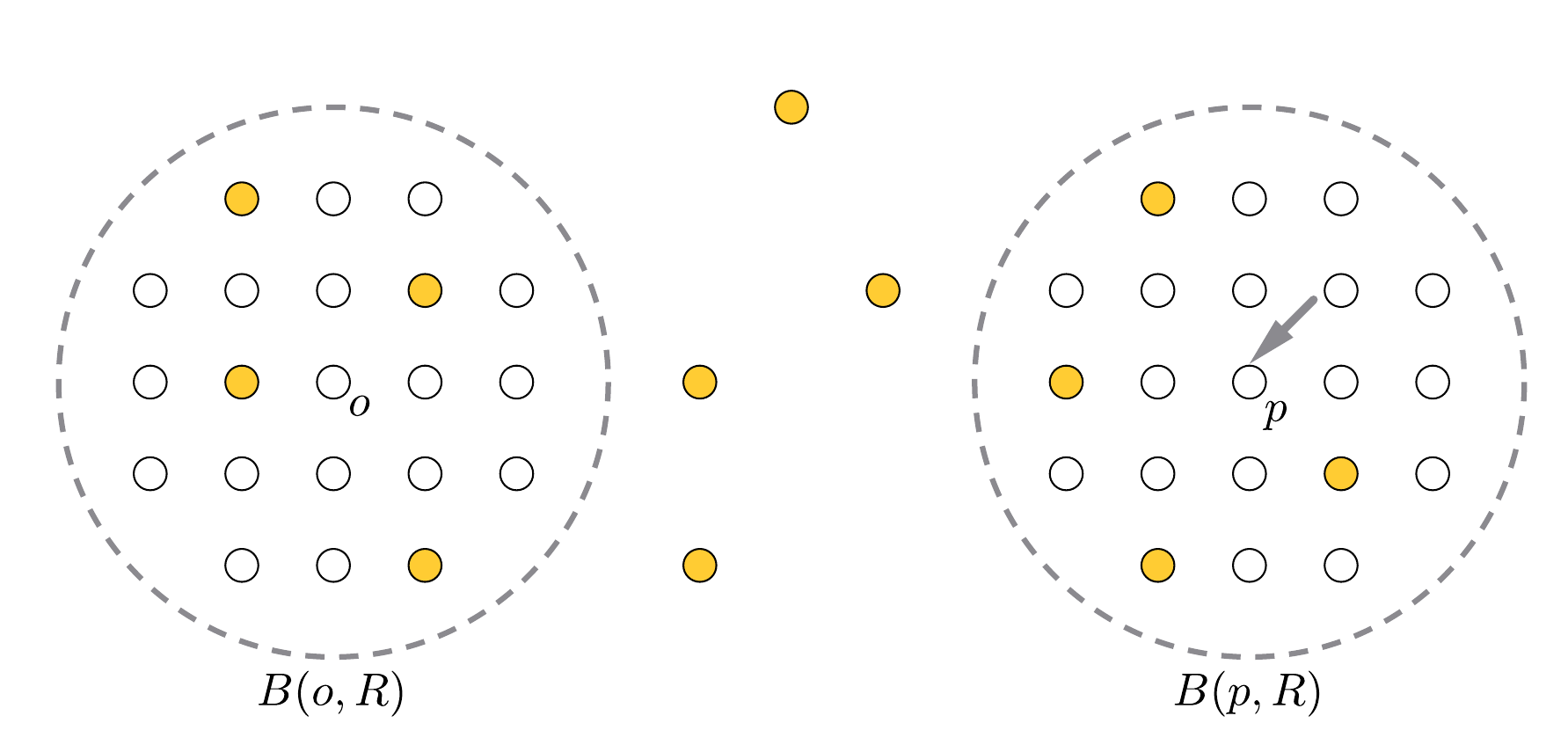} &
\includegraphics[trim=0 19cm 24cm 0,clip,width=0.22\linewidth]{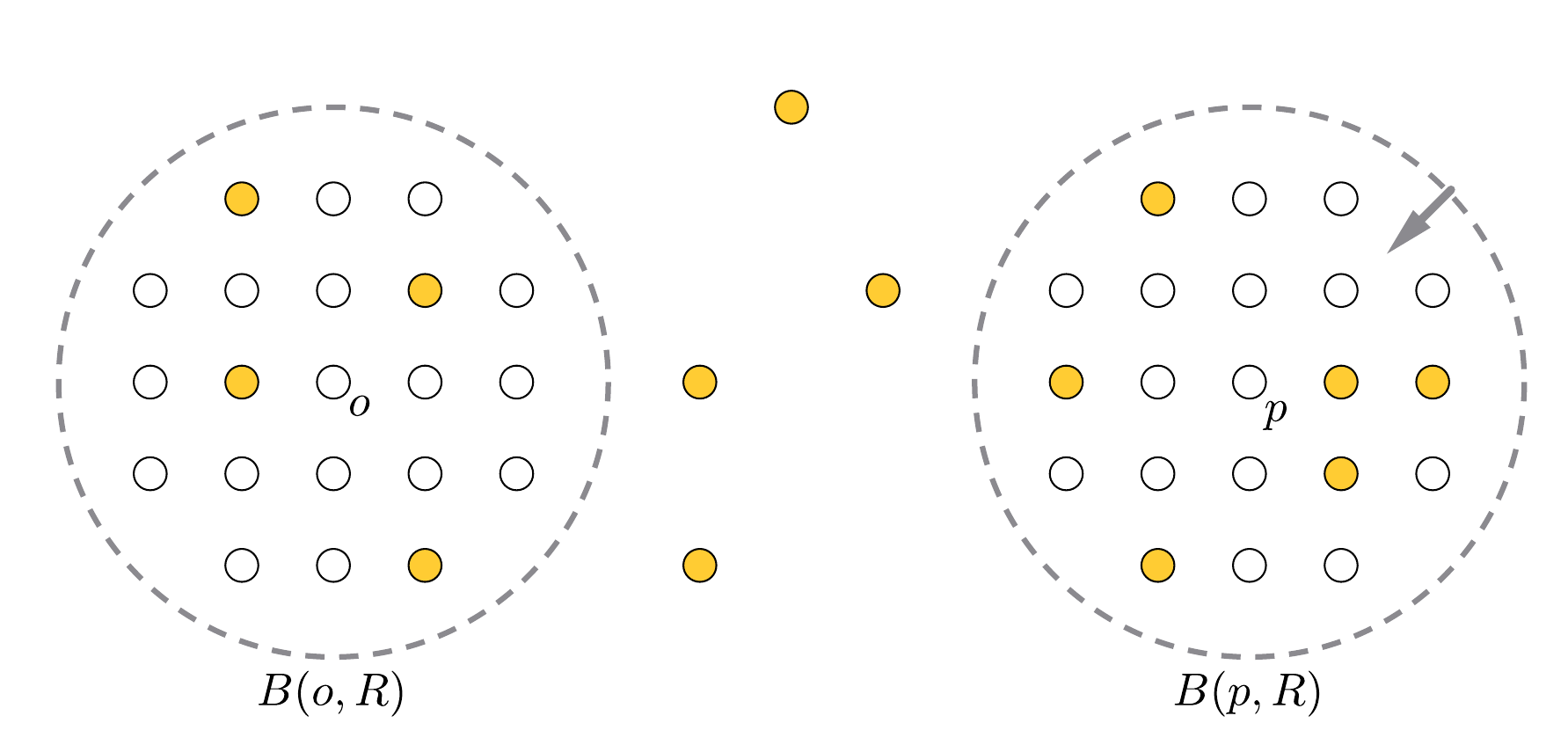} &
\includegraphics[trim=0 19cm 24cm 0,clip,width=0.22\linewidth]{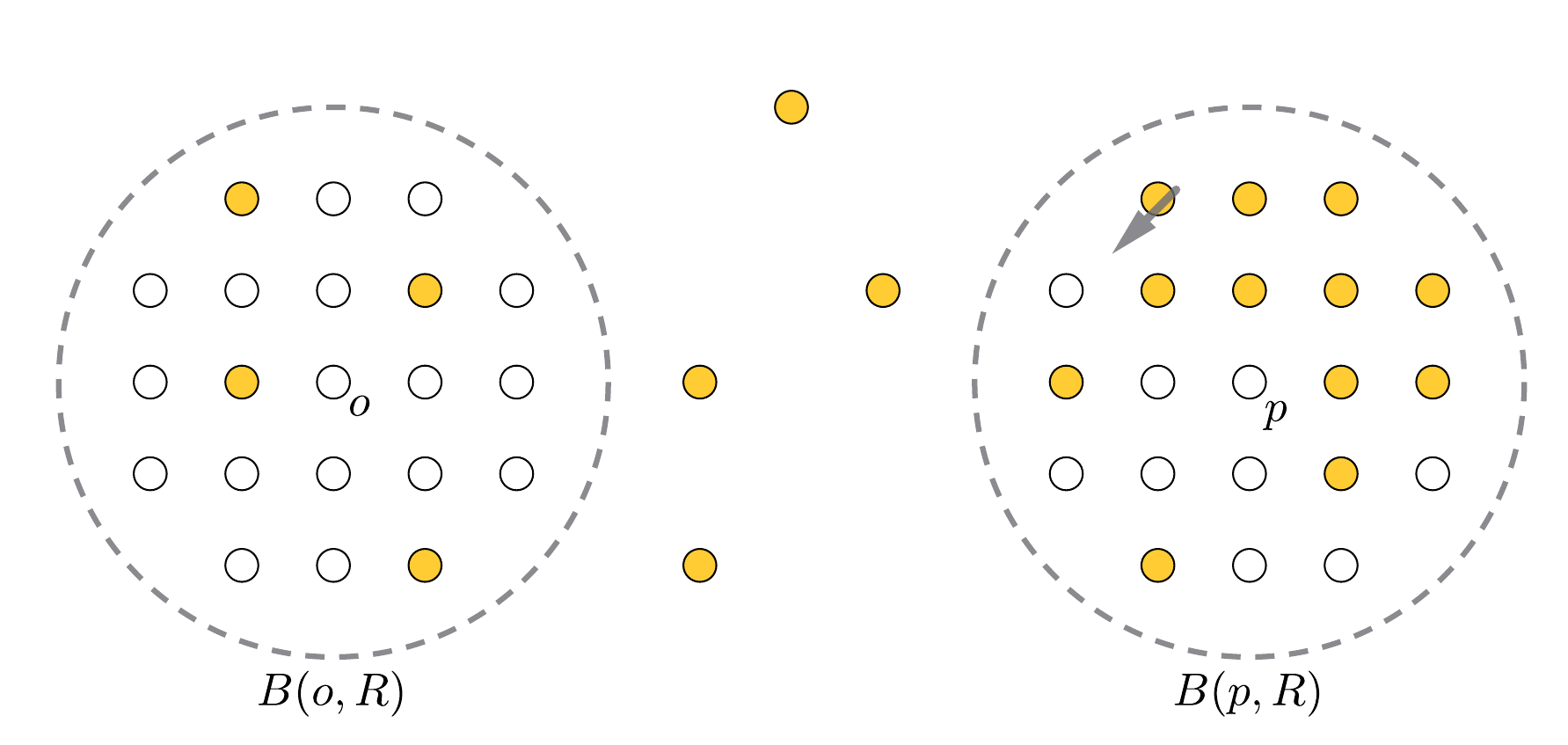} &
\includegraphics[trim=0 19cm 24cm 0,clip,width=0.22\linewidth]{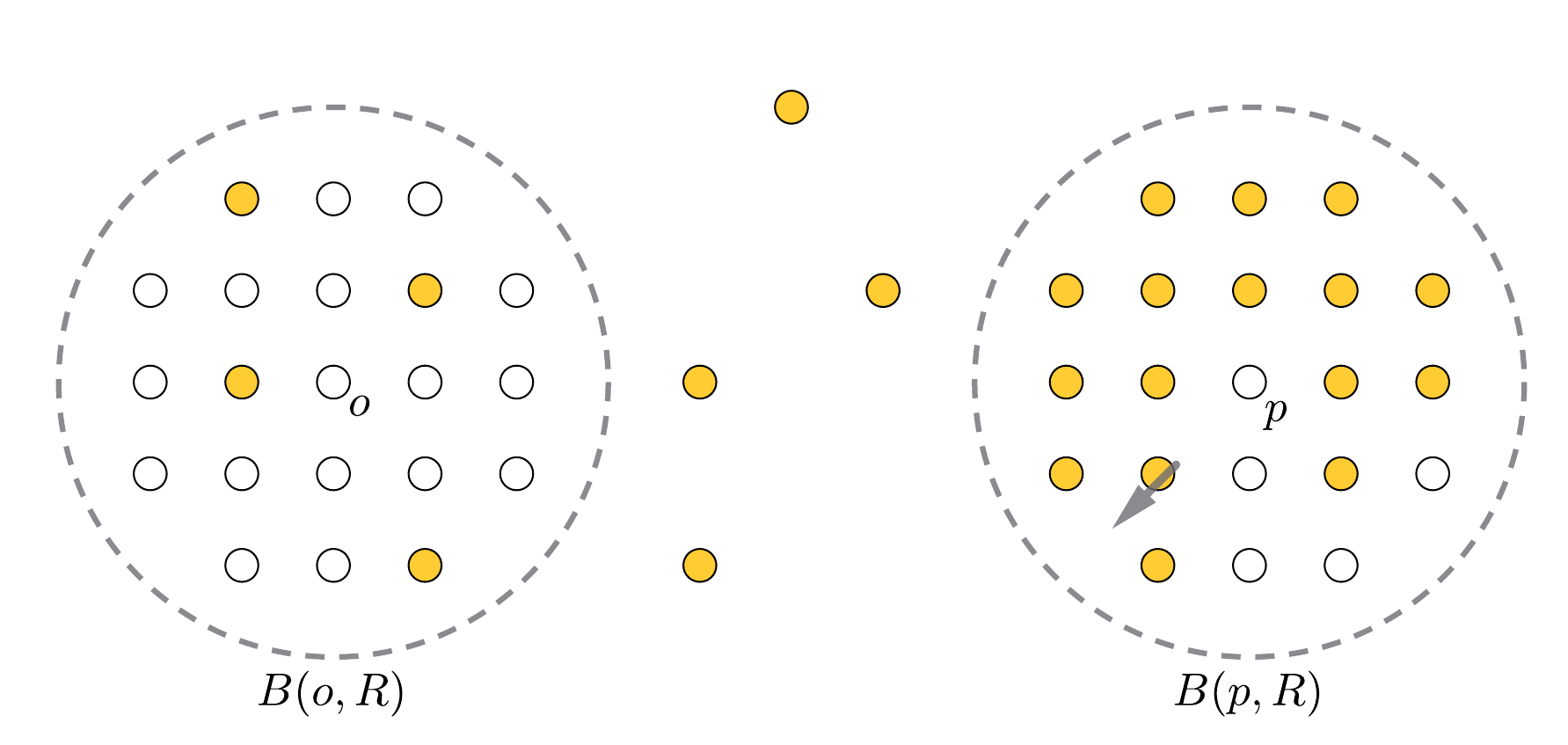} \\ \hline
\includegraphics[trim=0 19cm 24cm 0,clip,width=0.22\linewidth]{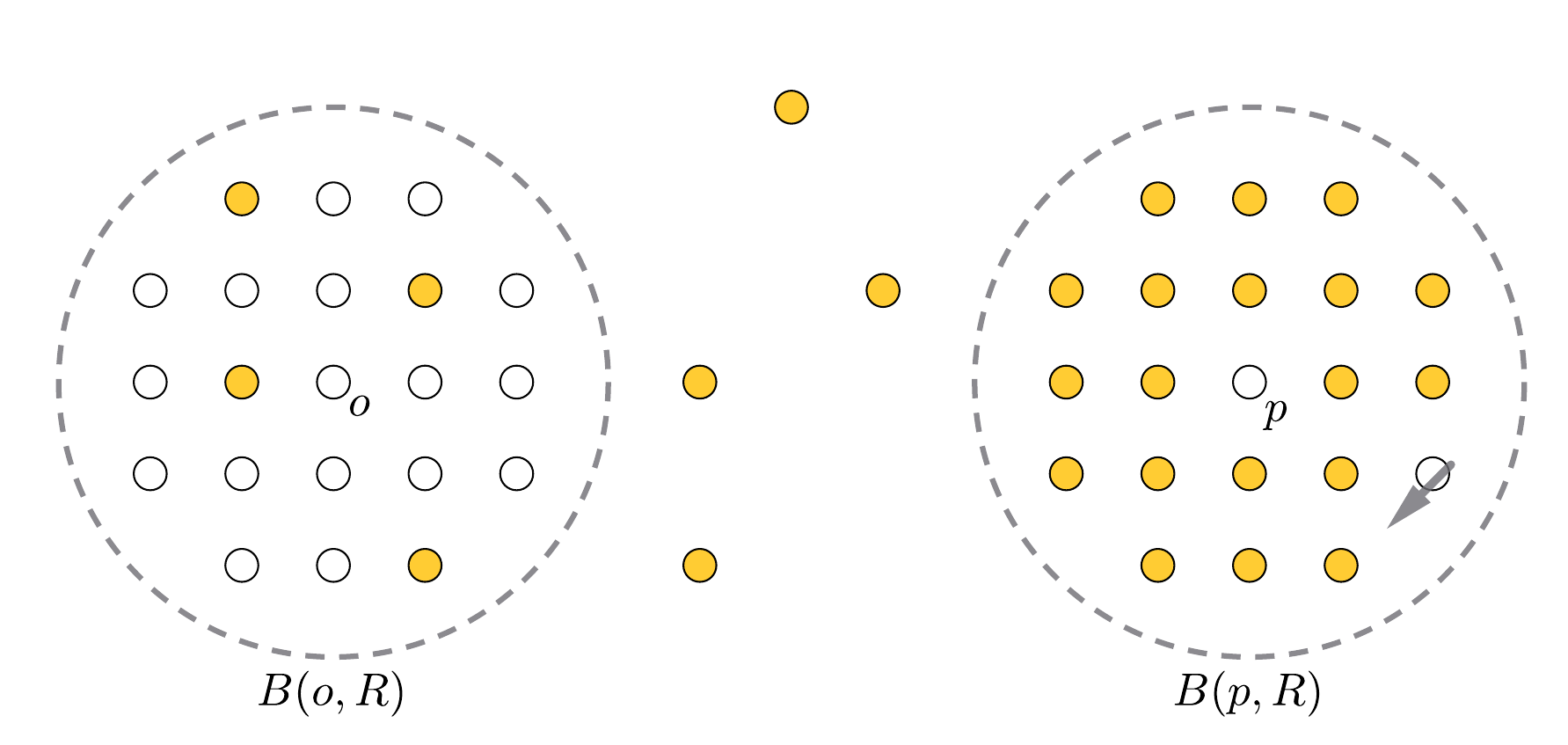} &
\includegraphics[trim=0 19cm 24cm 0,clip,width=0.22\linewidth]{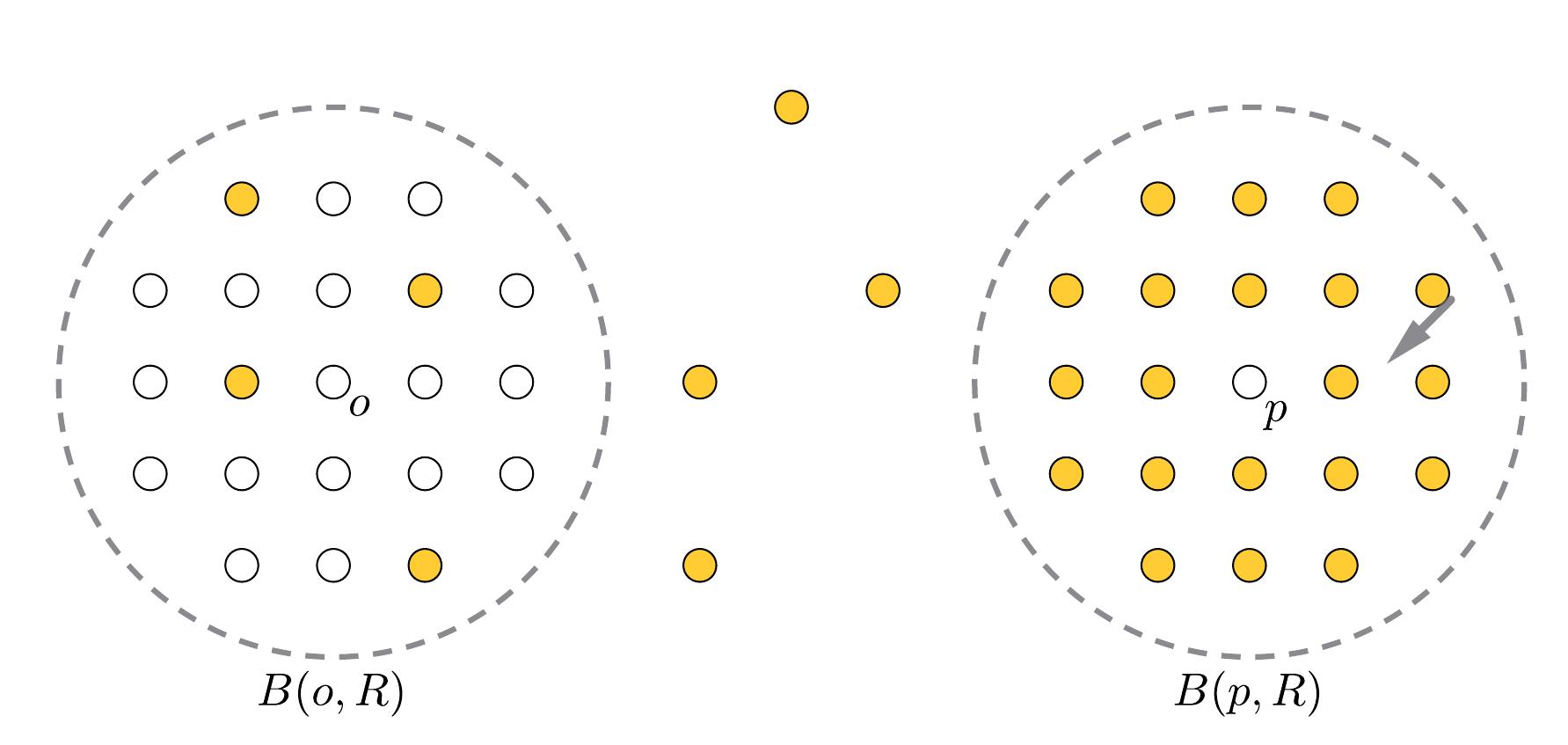} &
\includegraphics[trim=0 19cm 24cm 0,clip,width=0.22\linewidth]{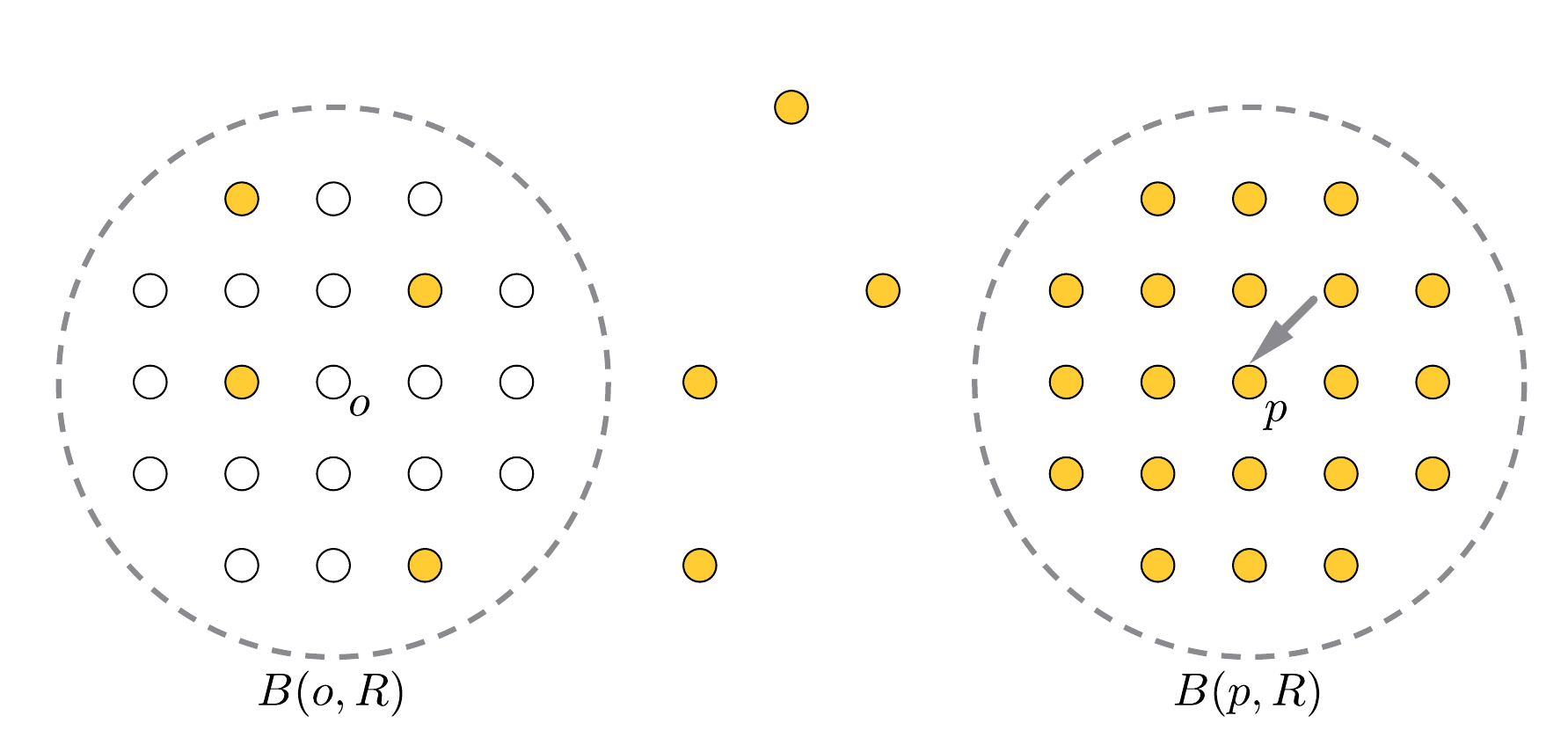} &
\includegraphics[trim=0 19cm 24cm 0,clip,width=0.22\linewidth]{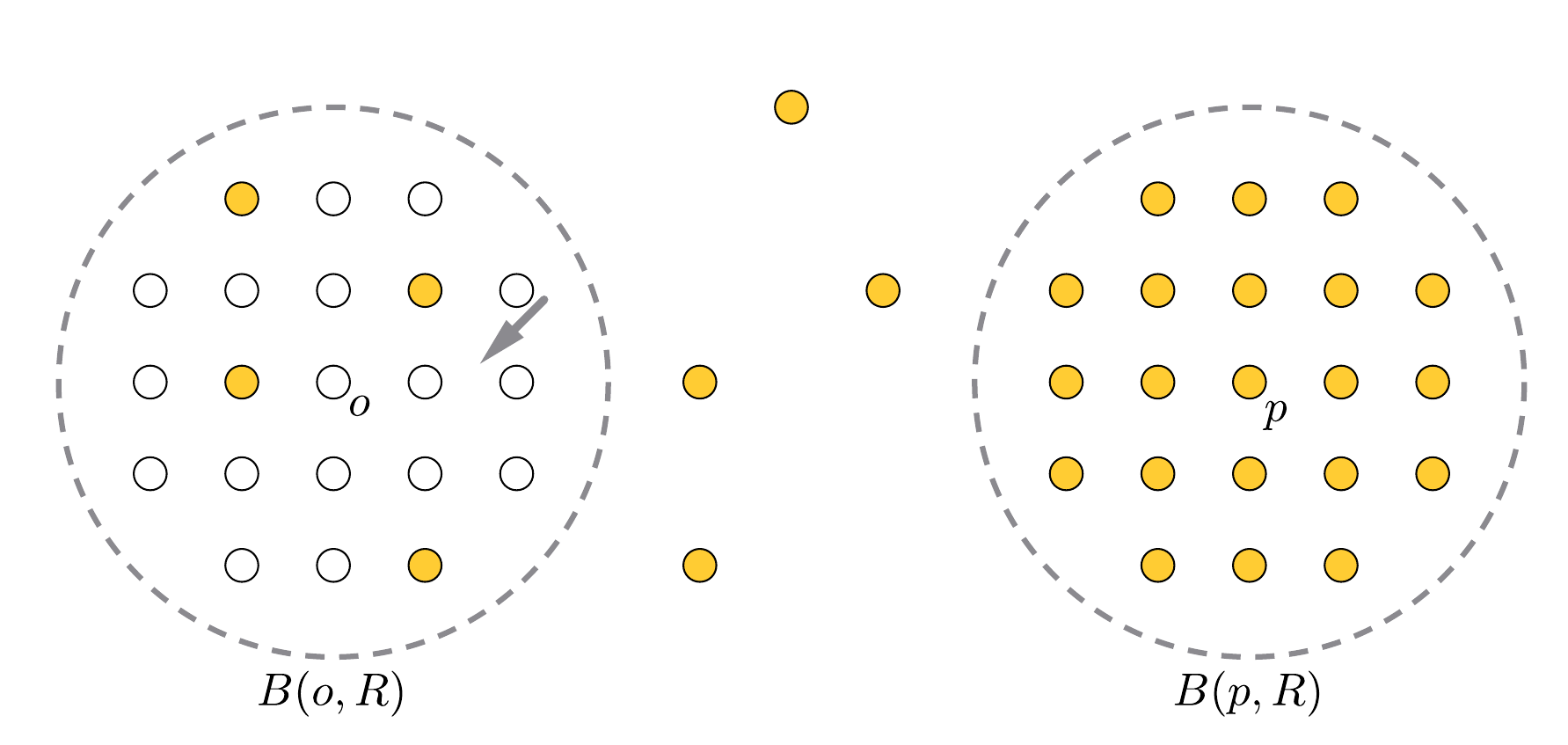} \\ \hline
\includegraphics[trim=0 19cm 24cm 0,clip,width=0.22\linewidth]{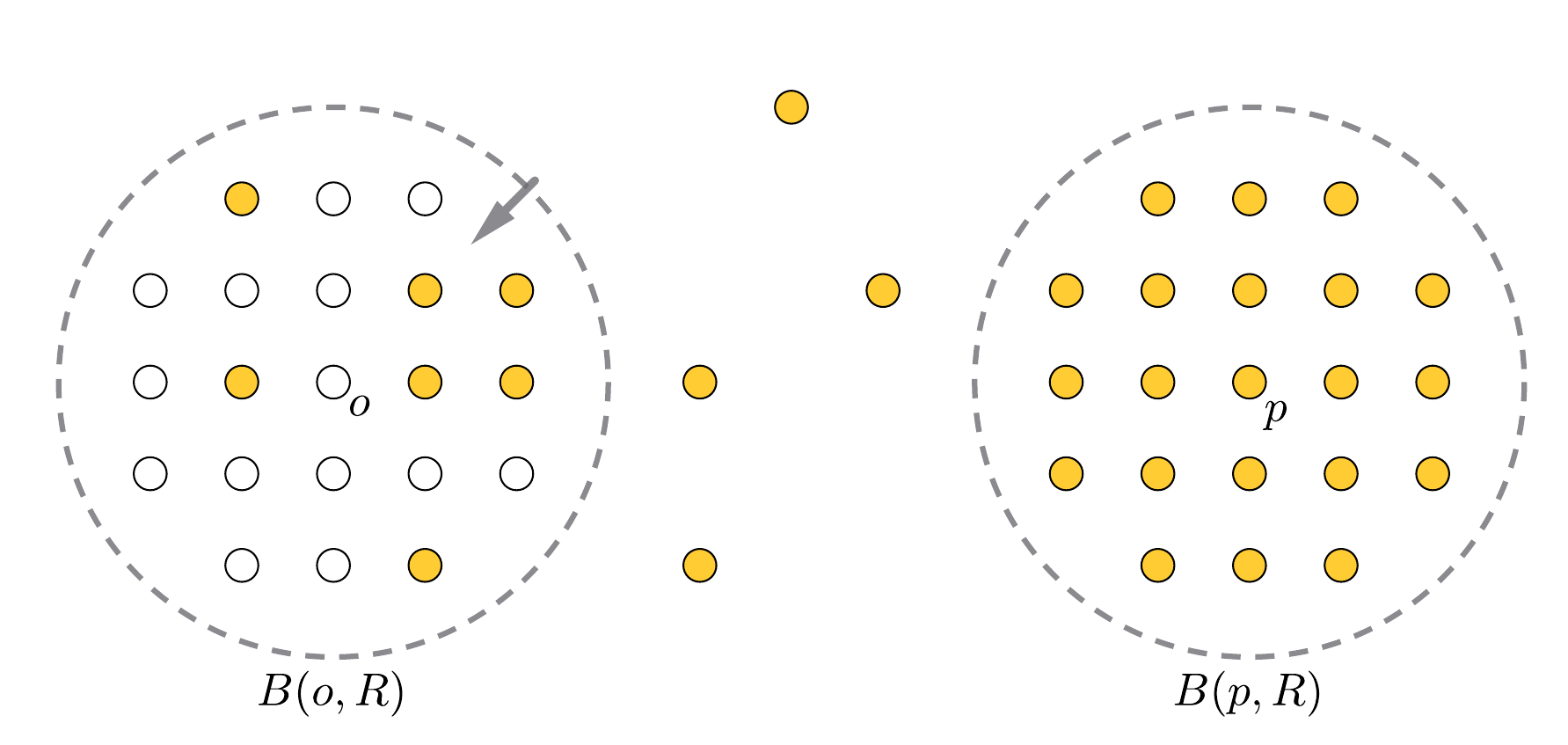} &
\includegraphics[trim=0 19cm 24cm 0,clip,width=0.22\linewidth]{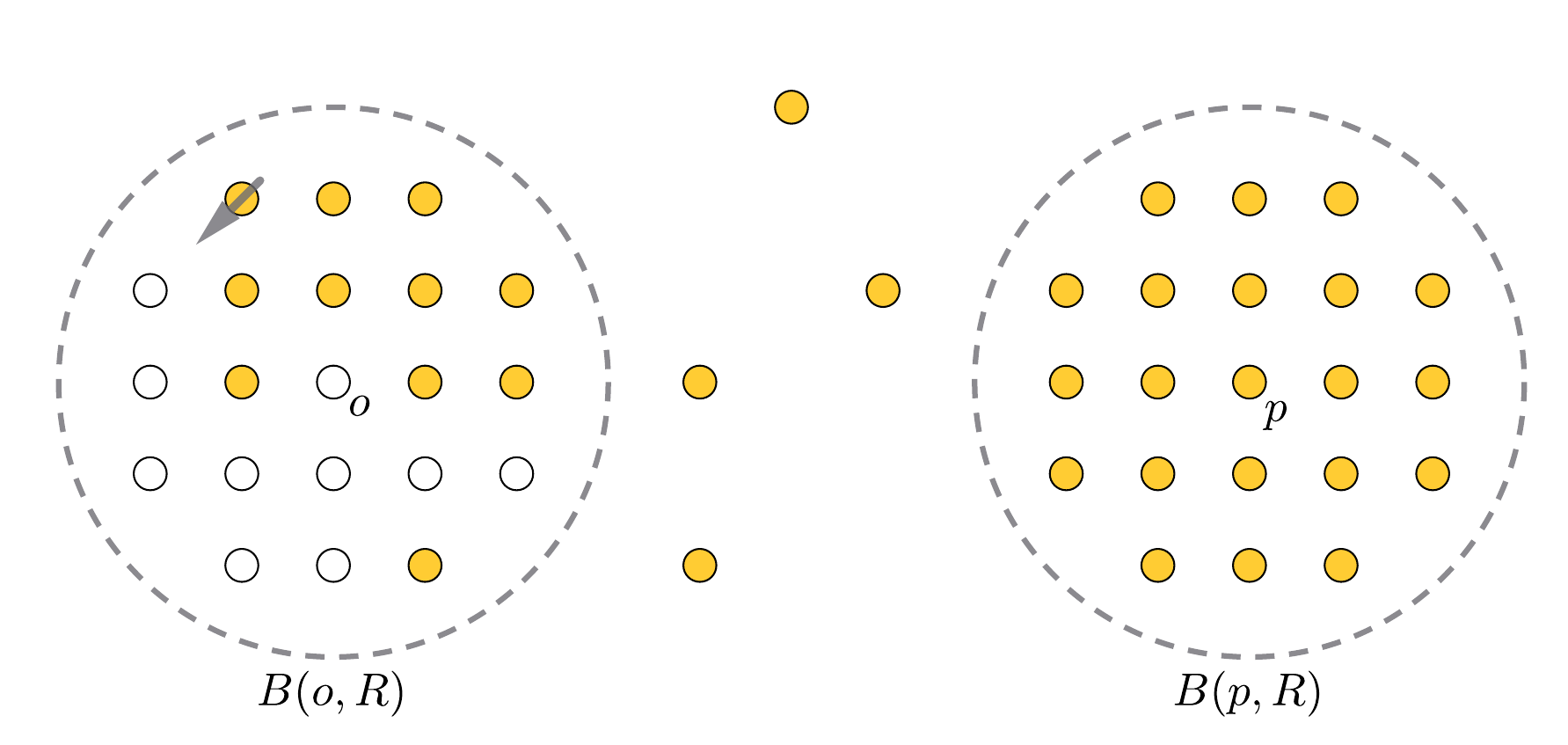} &
\includegraphics[trim=0 19cm 24cm 0,clip,width=0.22\linewidth]{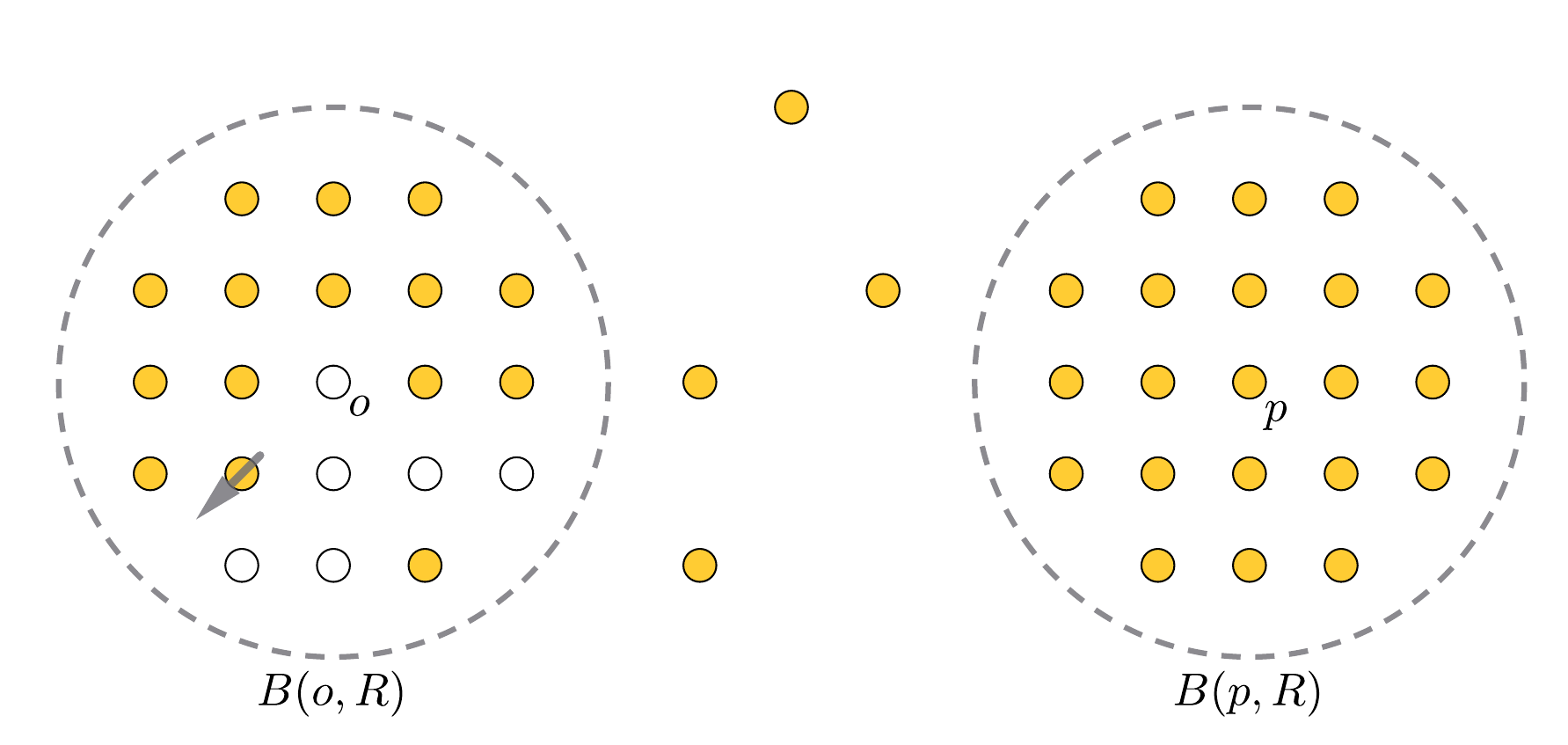} &
\includegraphics[trim=0 19cm 24cm 0,clip,width=0.22\linewidth]{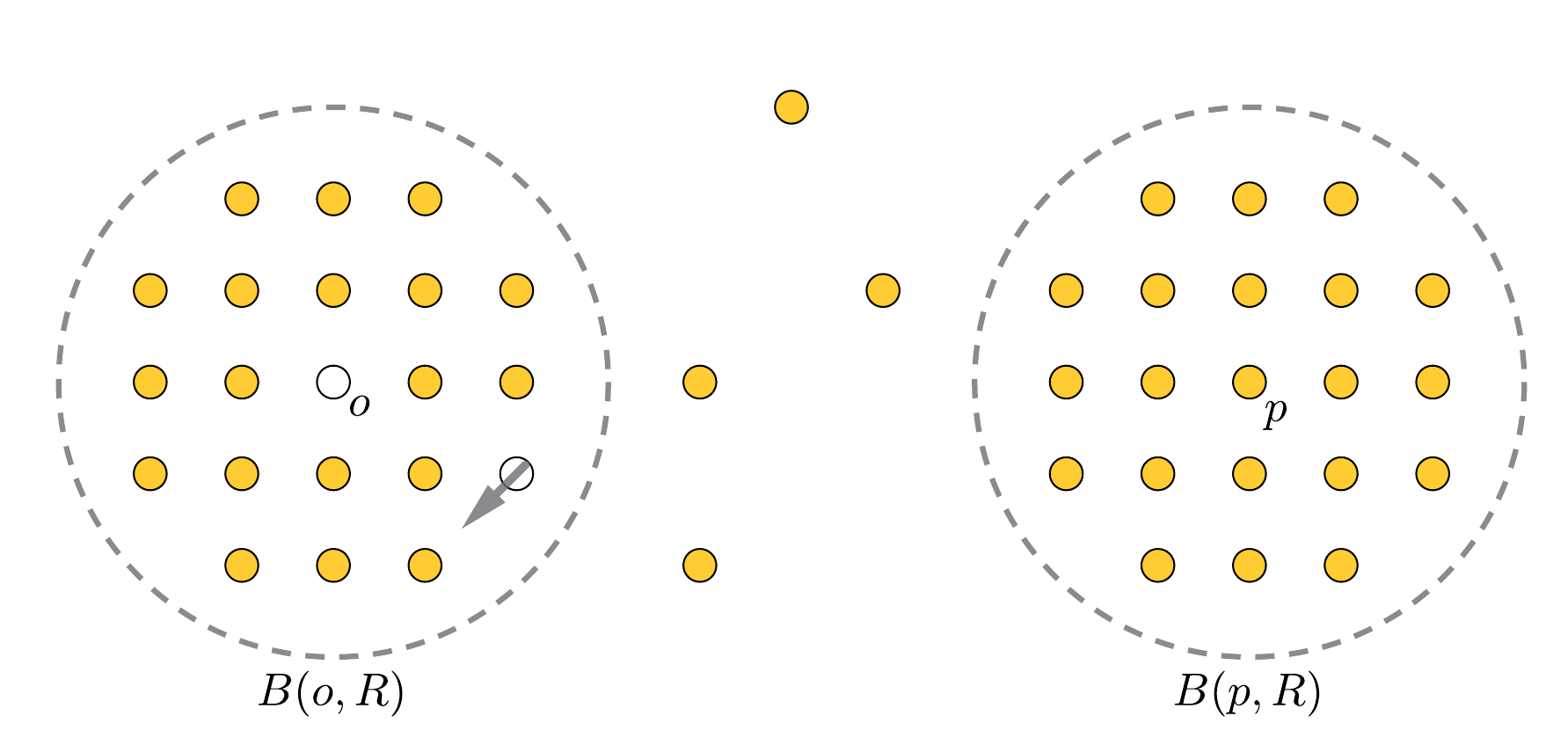} \\ \hline
\includegraphics[trim=0 19cm 24cm 0,clip,width=0.22\linewidth]{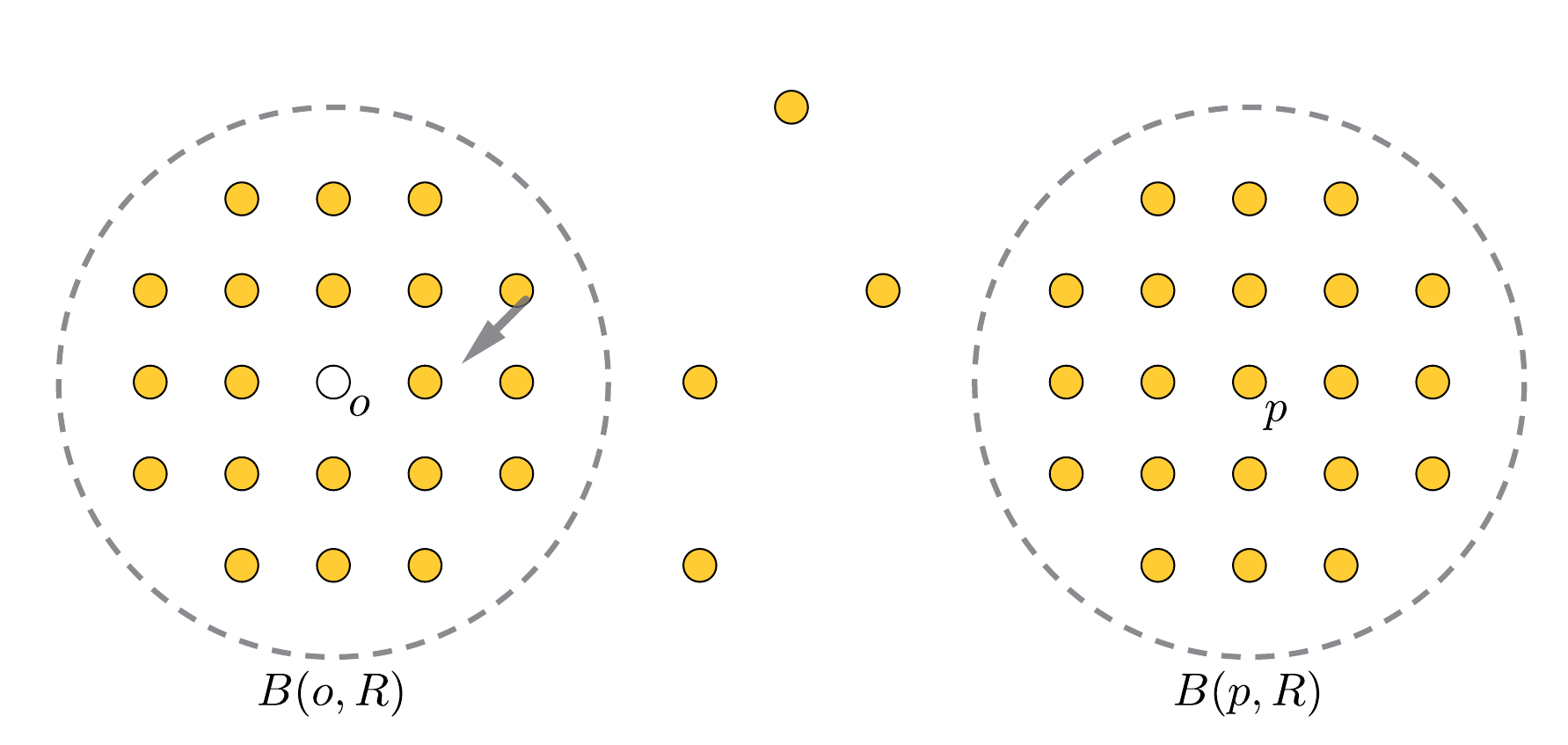} &
\includegraphics[trim=0 19cm 24cm 0,clip,width=0.22\linewidth]{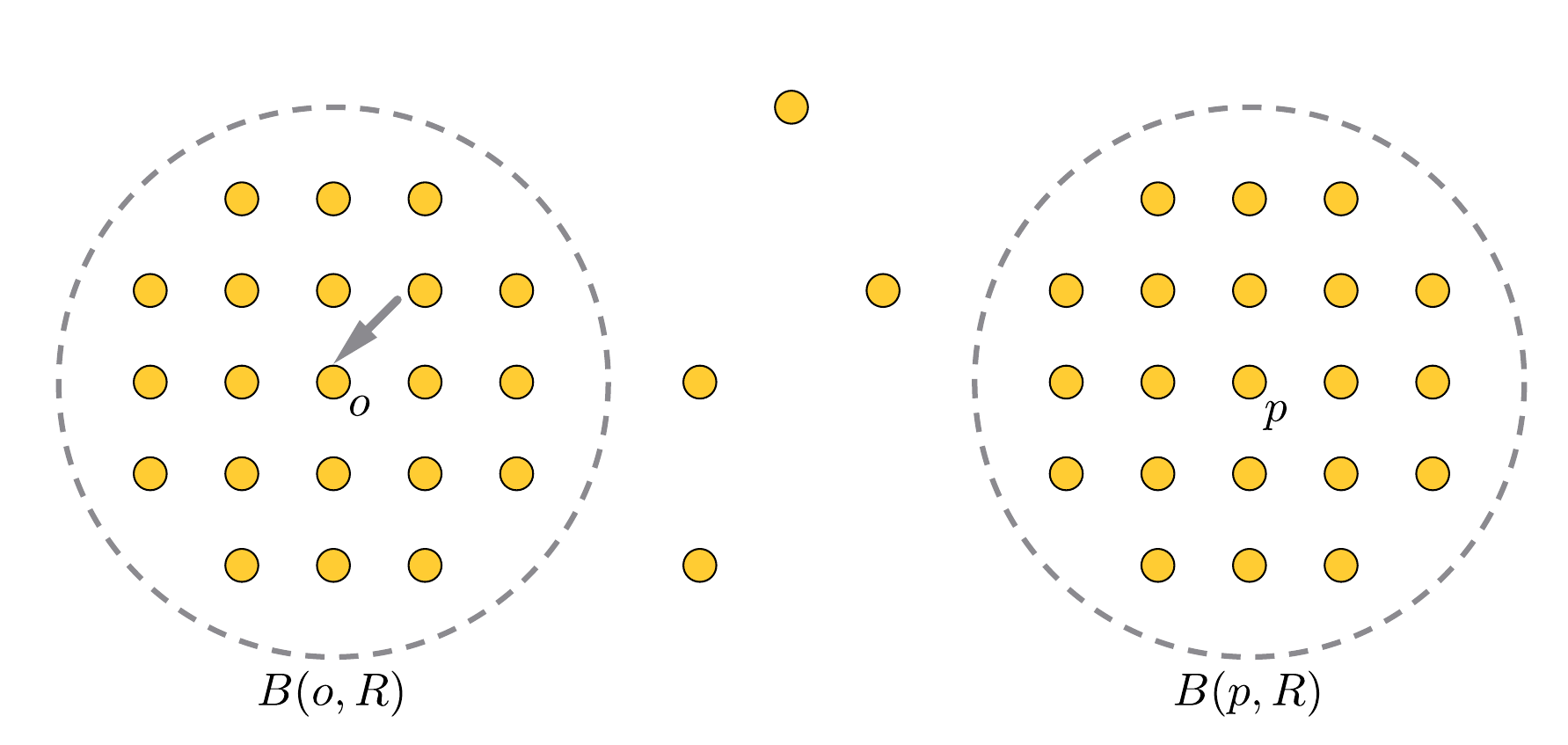} &
\includegraphics[trim=0 19cm 24cm 0,clip,width=0.22\linewidth]{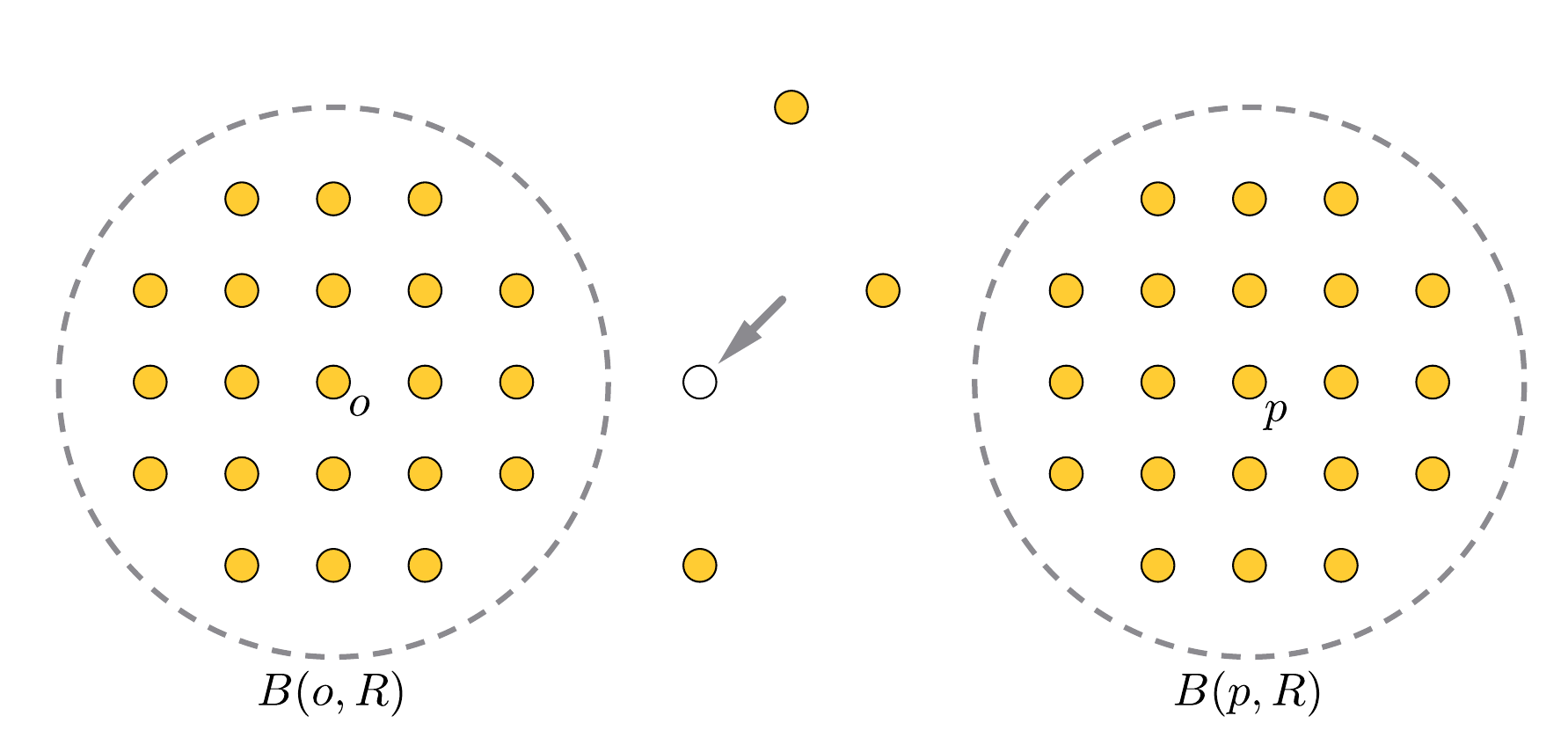} &
\includegraphics[trim=0 19cm 24cm 0,clip,width=0.22\linewidth]{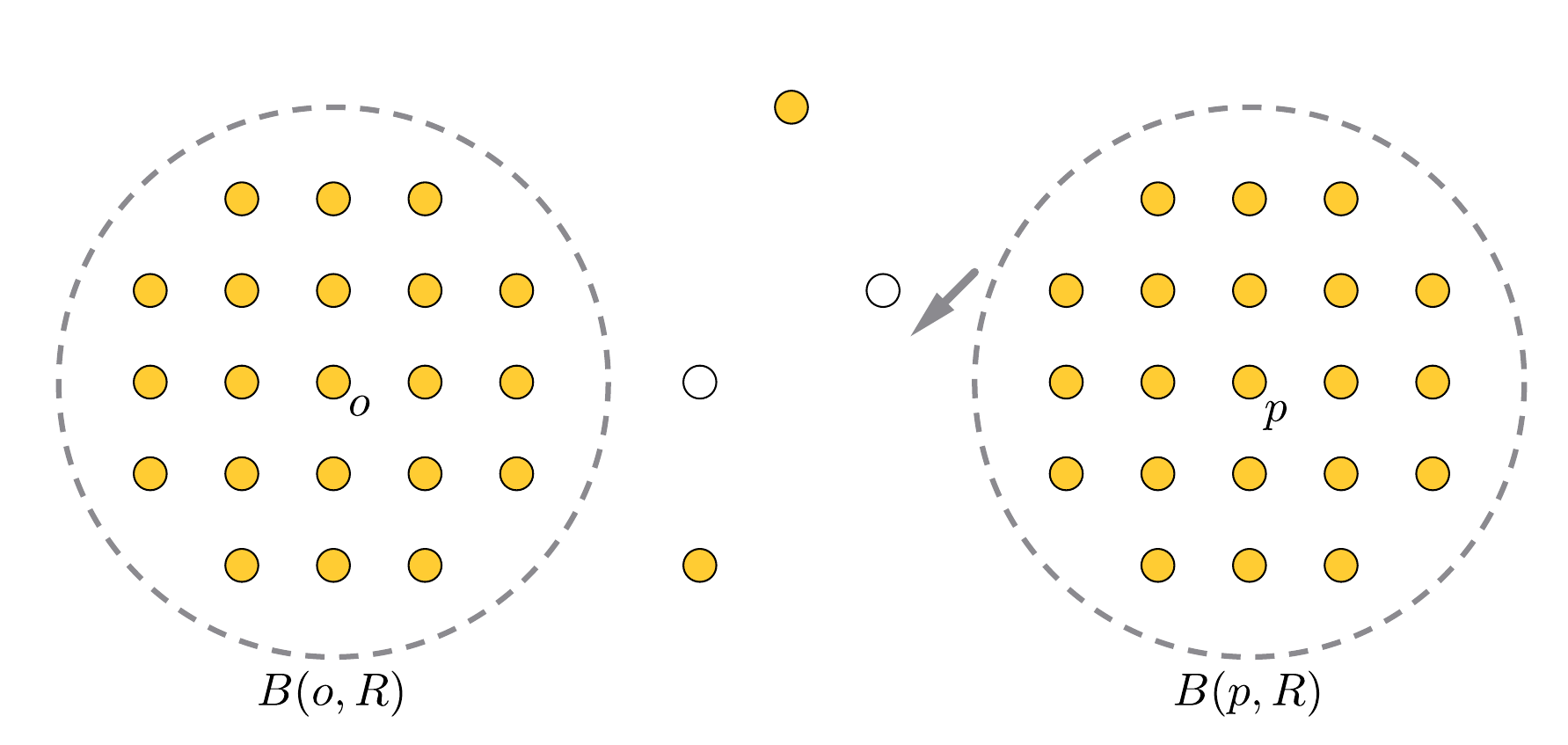} \\ \hline
\includegraphics[trim=0 19cm 24cm 0,clip,width=0.22\linewidth]{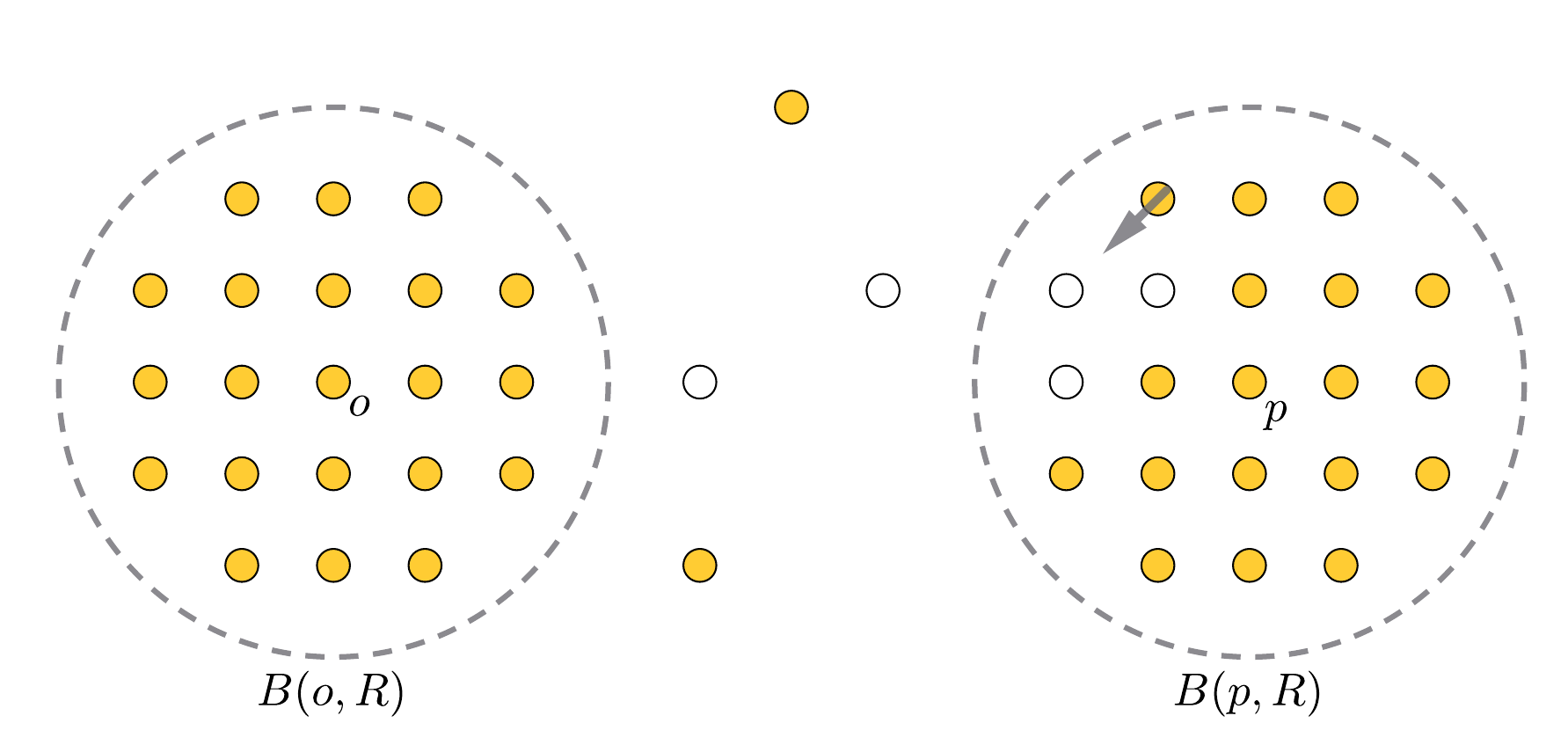} &
\includegraphics[trim=0 19cm 24cm 0,clip,width=0.22\linewidth]{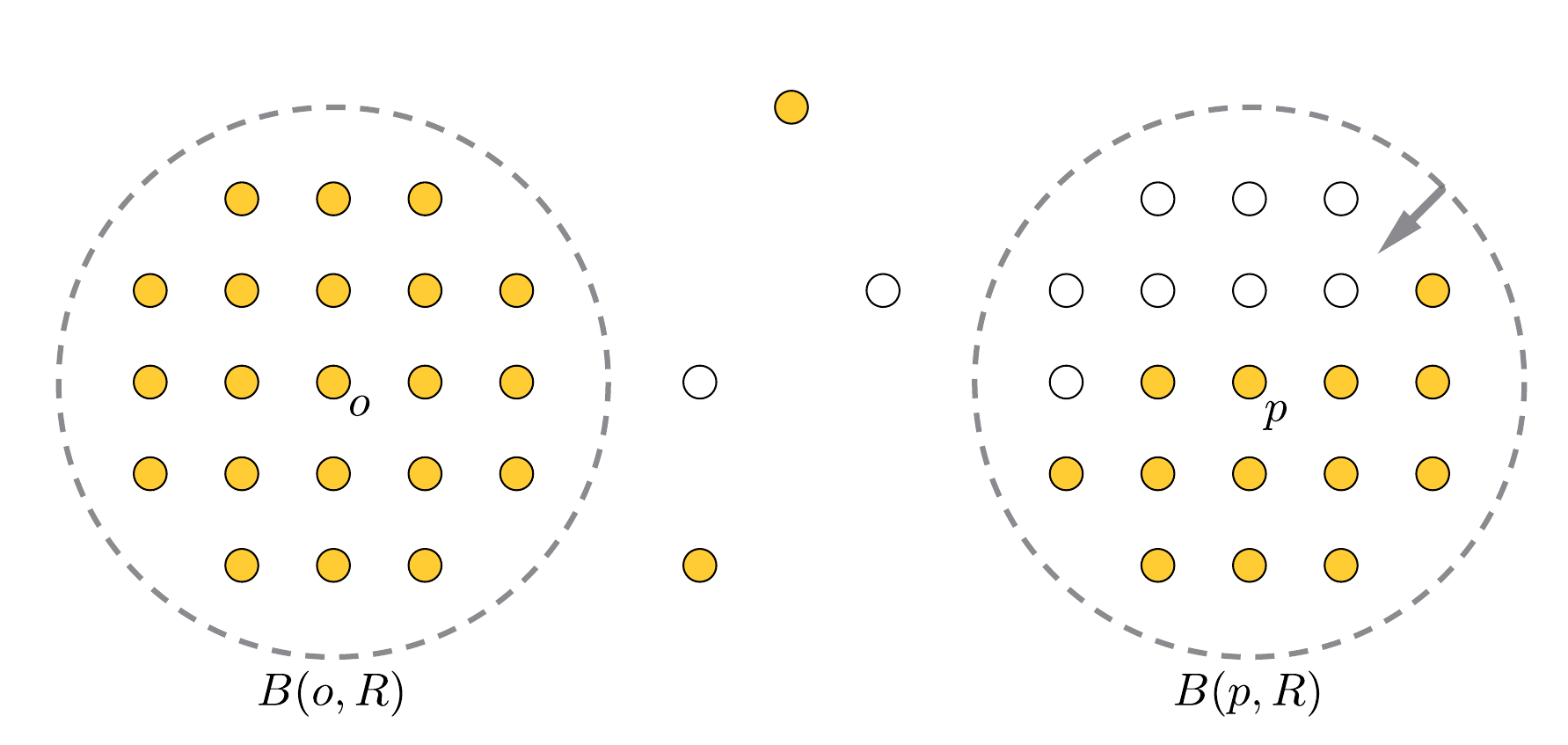} &
\includegraphics[trim=0 19cm 24cm 0,clip,width=0.22\linewidth]{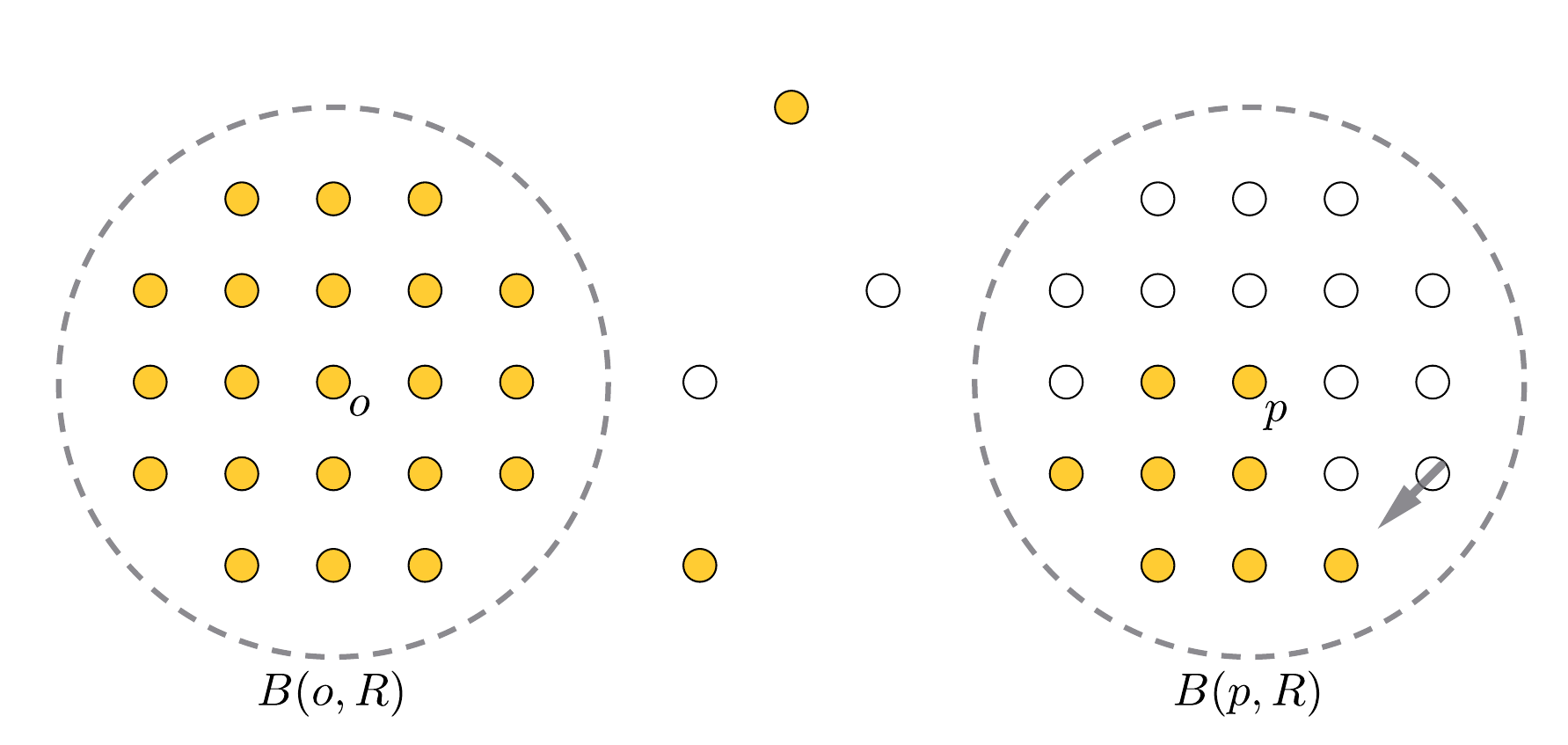} &
\includegraphics[trim=0 19cm 24cm 0,clip,width=0.22\linewidth]{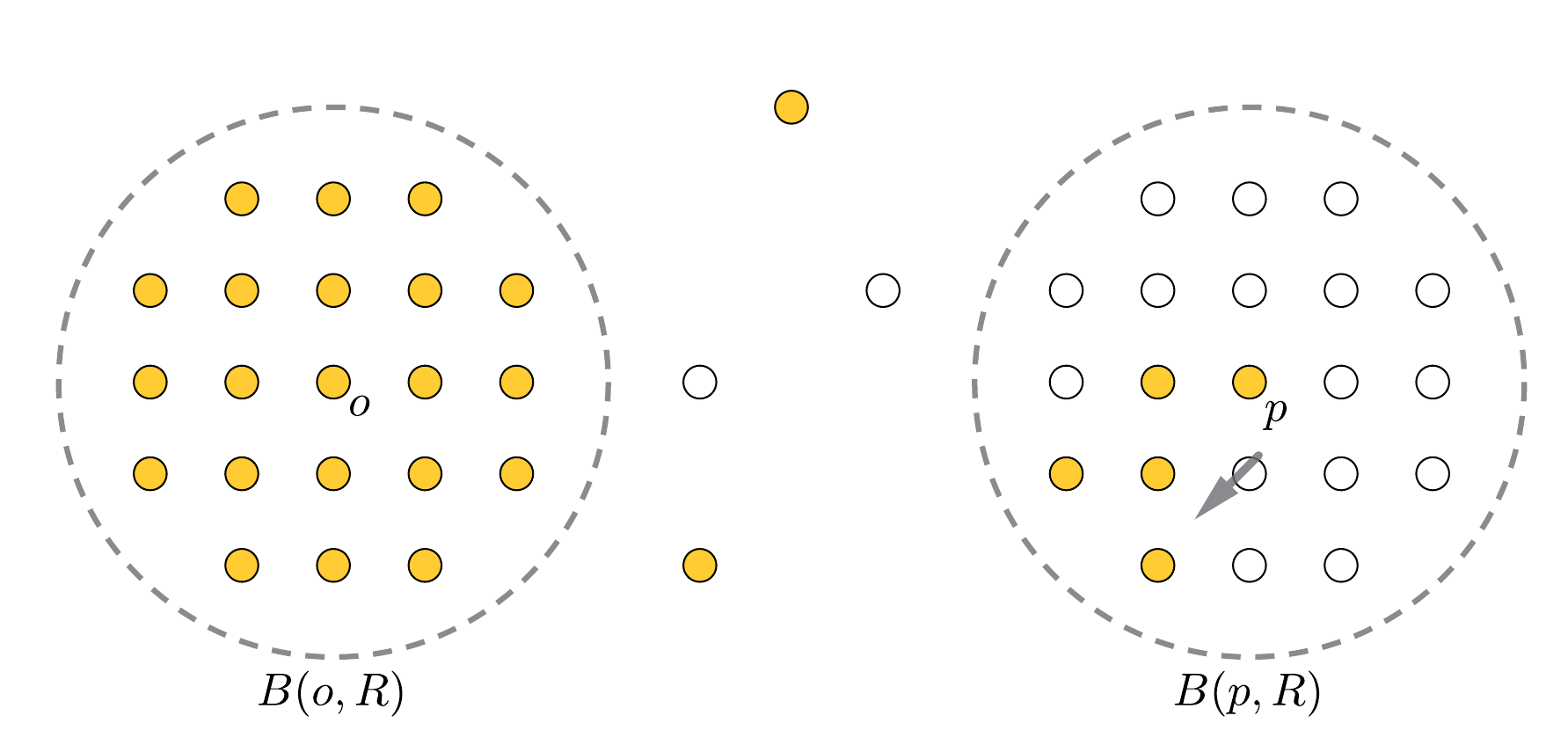} \\ \hline
\includegraphics[trim=0 19cm 24cm 0,clip,width=0.22\linewidth]{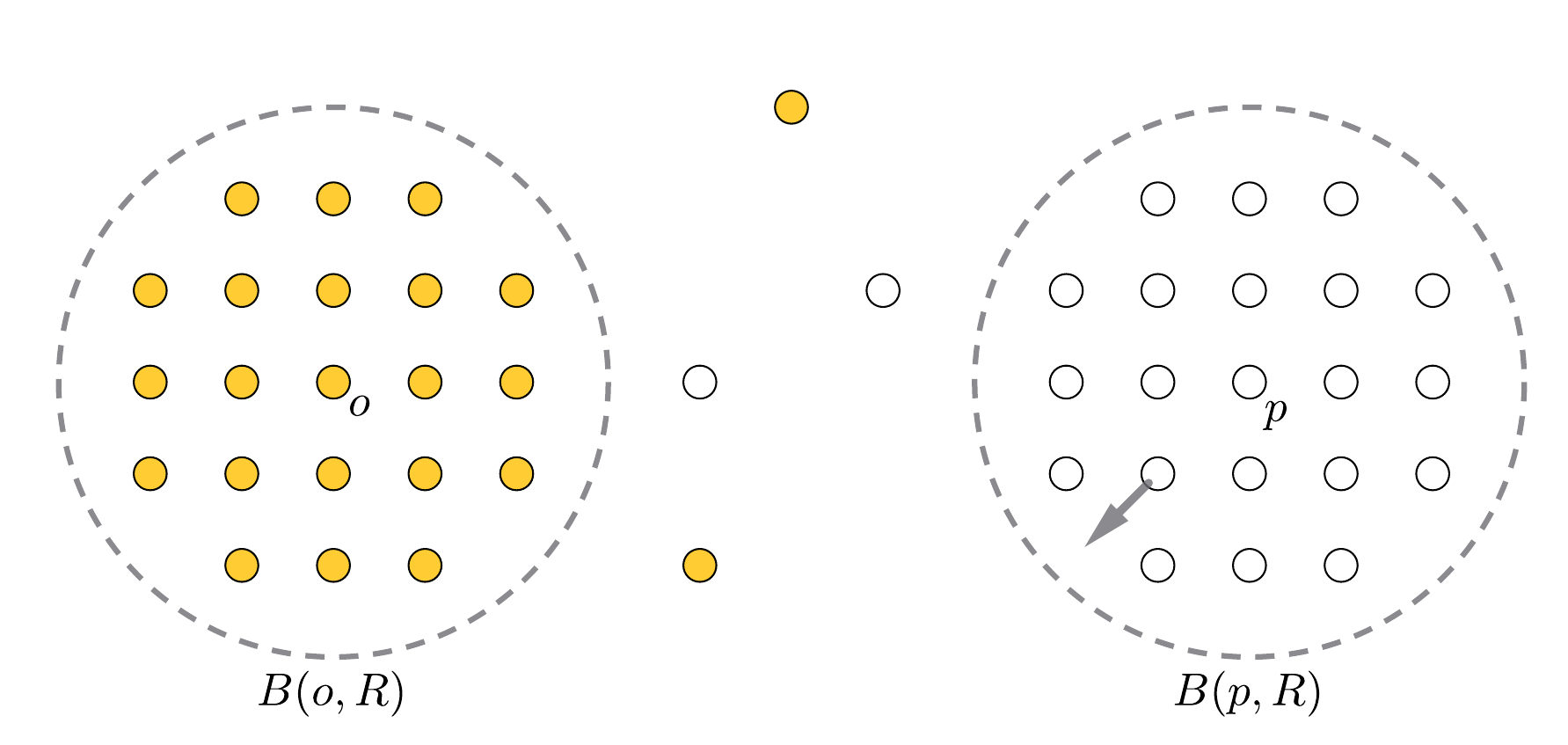} &
\includegraphics[trim=0 19cm 24cm 0,clip,width=0.22\linewidth]{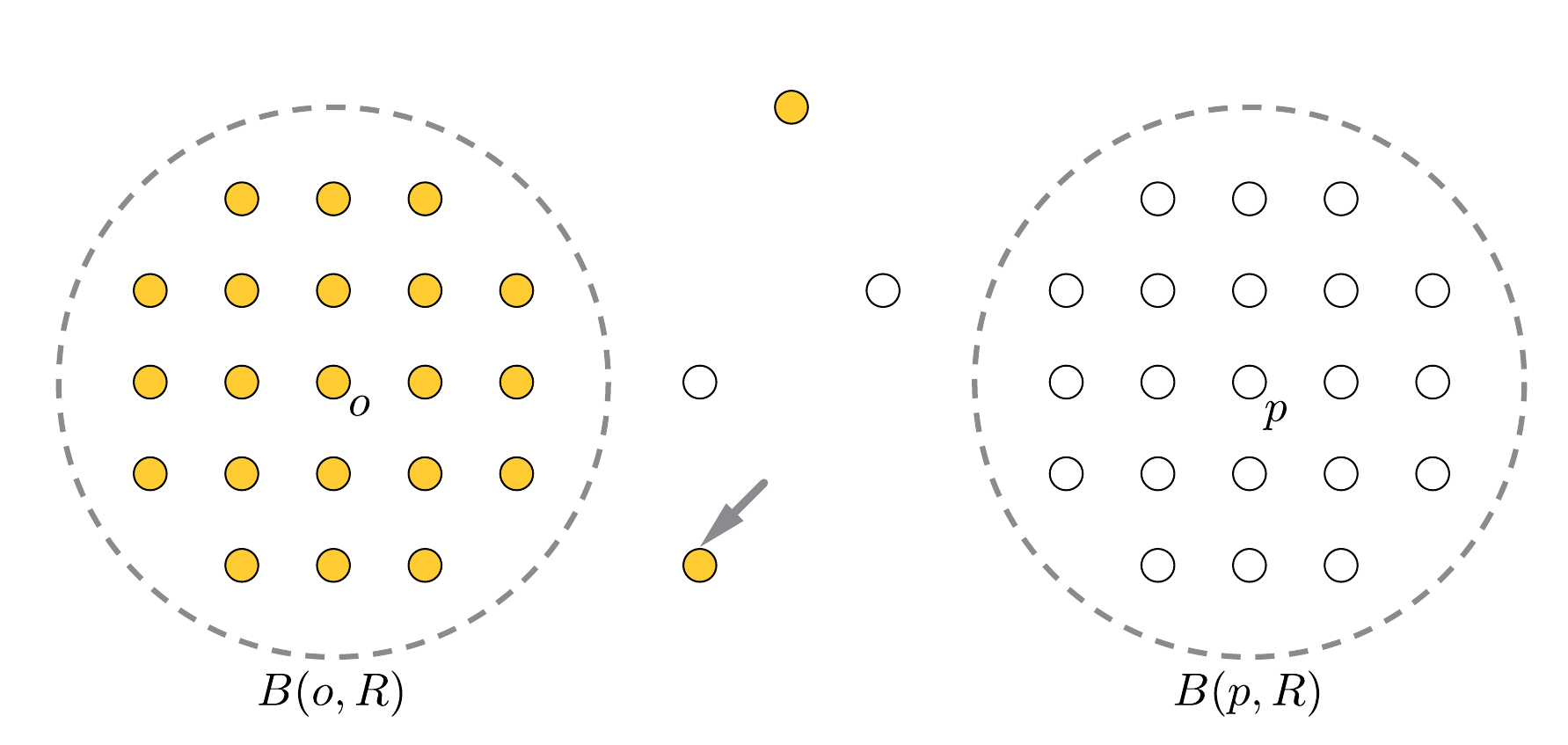} &
\includegraphics[trim=0 19cm 24cm 0,clip,width=0.22\linewidth]{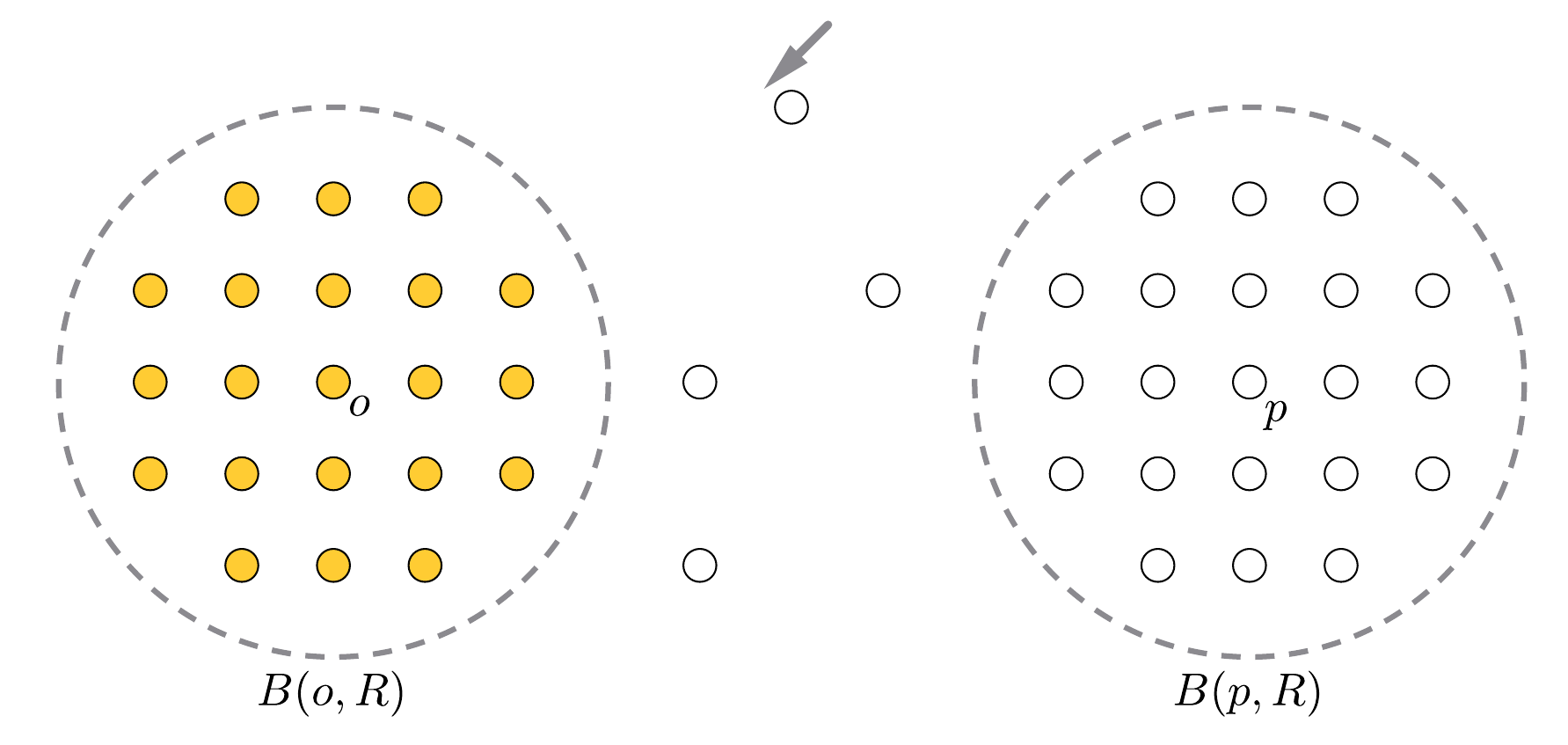} &
\includegraphics[trim=0 19cm 24cm 0,clip,width=0.22\linewidth]{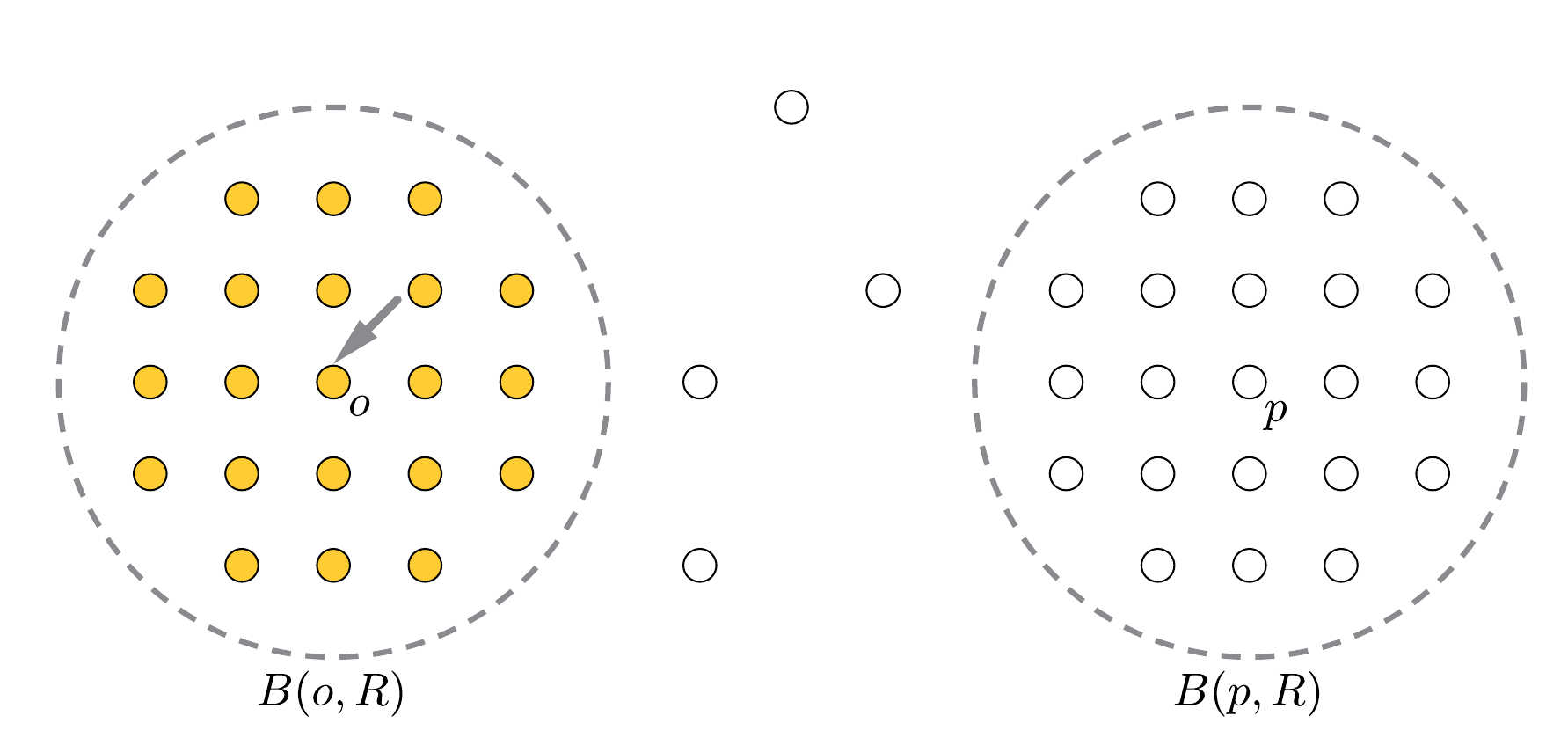} \\ \hline
\end{tabular}
\caption{Path constructed during the proof of Proposition~\ref{prop:OneEnded}.}
\label{Construction}
\end{center}
\end{figure}

\medskip \noindent
Now, we construct a path connecting $(c,p)$ to $(\kappa,o)$ as follows (Figure~\ref{Construction} illustrates our path from $(a,q)$ to $(c_o,o)$):
\begin{itemize}
	\item First, we connect $(c,p)$ to $(a,q)$ through a geodesic. Since $(a,q)$ belongs to a geodesic connecting $(c_o,o)$ to $(c,p)$, our first path stays at distance $d((a,q),(c_o,o)) \geq d(q,o)>R$ from $(c_o,o)$.
	\item Next, we connect $(a,q)$ to $(a',q)$ through a geodesic, where $a'$ denotes the colouring obtained from $a$ by forcing the colour to be $x$ in $B(q,R)$. Notice that, along such a geodesic, the arrow stays in $B(q,R)$, so our path stays at distance $\geq d(q,o)>R$ from $(c_o,o)$.
	\item Then, we connect $(a',q)$ to $(a'',o)$ through a geodesic, where $a''$ is the colouring obtained from $a'$ by forcing the colour to be $x$ in $B(o,R)$. Notice that, along such a geodesic, the colouring always take the colour $x$ in $B(q,R)$, so our path stays at distance $\geq |B(q,R)|>R$ from $(c_o,o)$.
	\item Finally, we connect $(a'',o)$ to $(\kappa,o)$. Notice that, along such a geodesic, the colouring always take the colour $x$ in $B(o,R)$, so our path stays at distance $\geq |B(o,R)|>R$ from $(c_o,o)$.
\end{itemize}
This concludes the proof of our proposition. 
\end{proof}

\subsection{Coarse local separation}\label{section:CoarseLocal}

\noindent
Following the previous section, we would like to coarsify the topological notion of local separation. Recall that, given a topological space $X$, a subspace $Z \subset X$ \emph{locally separates} $X$ whenever there exists some open subset $O \subset X$ that contains $Z$ and that is separated by $Z$. This might sound paradoxical, since coarse topology is supposed to ignore the local structure of our space. However, for nice topological spaces, it turns out to be possible to characterise local separation in a way that we will be able to coarsify. Indeed:

\medskip \noindent
\begin{minipage}{0.35\linewidth}
\begin{center}
\includegraphics[trim=0 19cm 26cm 0,clip,width=0.95\linewidth]{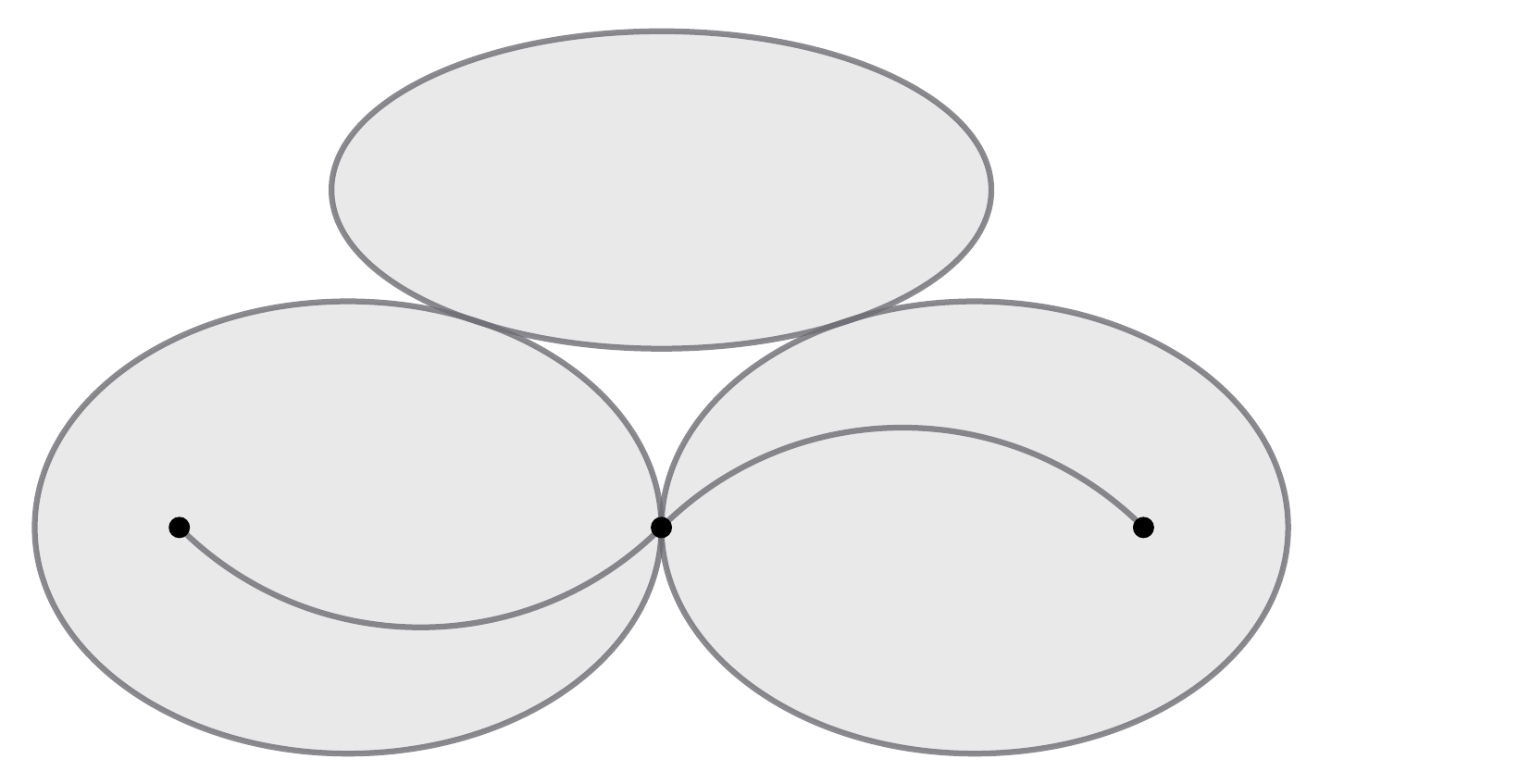}
\end{center}
\end{minipage}
\begin{minipage}{0.63\linewidth}
\begin{fact}\label{fact:LocaSepTopo}
Let $X$ be a CW-complex. A subcomplex $Z \subset X$ locally separates $X$ if and only if there exists two distinct points $x,y \in X$ and a path $\gamma$ connecting $x$ to $y$ such that any path homotopically equivalent to $\gamma$ intersects $Z$. 
\end{fact}
\end{minipage}

\medskip \noindent
We leave the proof of this assertion as an exercise to the interested reader. 

\medskip \noindent
First, since we saw in Section~\ref{section:CoarseSimplyConnected} how to coarsify the notion of homotopically equivalent paths, we can coarsify the notion of intersections persistent under homotopy equivalence. 

\begin{definition}
Let $X$ be a graph and $Z \leq X$ a subgraph. Given an $E \geq 0$, a path $\gamma$ intersects \emph{$E$-persistently} $Z$ if every path that is $E$-coarsely homotopy equivalent to $\gamma$ intersects $Z$. 
\end{definition}

\noindent
Then, inspired by our definition of coarse separation by families of subgraphs, we naturally obtain the following definition by coarsifying Fact~\ref{fact:LocaSepTopo}:

\begin{definition}
Let $X$ be a graph. A collection of subgraphs $\mathcal{Z}$ \emph{coarsely locally separates} $X$ if, for every $E \geq 0$, there exists $F \geq 0$ such that the following holds: for every $L \geq 0$, there exist a subgraph $Z \in \mathcal{Z}$ and a path $\gamma$ intersecting $Z^{+F}$ such that:
\begin{itemize}
	\item the endpoints of $\gamma$ lie at distance $\geq L$ from $Z$;
	\item the intersection between $\gamma$ and $Z^{+F}$ is $E$-persistent. 
\end{itemize}
\end{definition}

\noindent
Notice that, mimicking the fact that separation and local separation are equivalent in simply connected spaces, coarse separation and coarse local separation are equivalent in coarsely simply connected spaces.

\medskip \noindent
Again, we will be mainly interested in local coarse separation by (families of) vertices, so it makes sense to define:

\begin{definition}
A graph $X$ is \emph{locally one-ended} if, for every $v \in V(X)$, it is not coarsely locally separated by $\{v\}$; and it is \emph{uniformly locally one-ended} if it is not coarsely locally separated by $\{\{v\} \mid v \in V(X)\}$. 
\end{definition}

\noindent
We emphasize that, for quasi-transitive graphs (such as Cayley graphs), being locally one-ended and being uniformly locally one-ended is the same thing.

\medskip \noindent
\begin{minipage}{0.5\linewidth}
\includegraphics[width=0.95\linewidth]{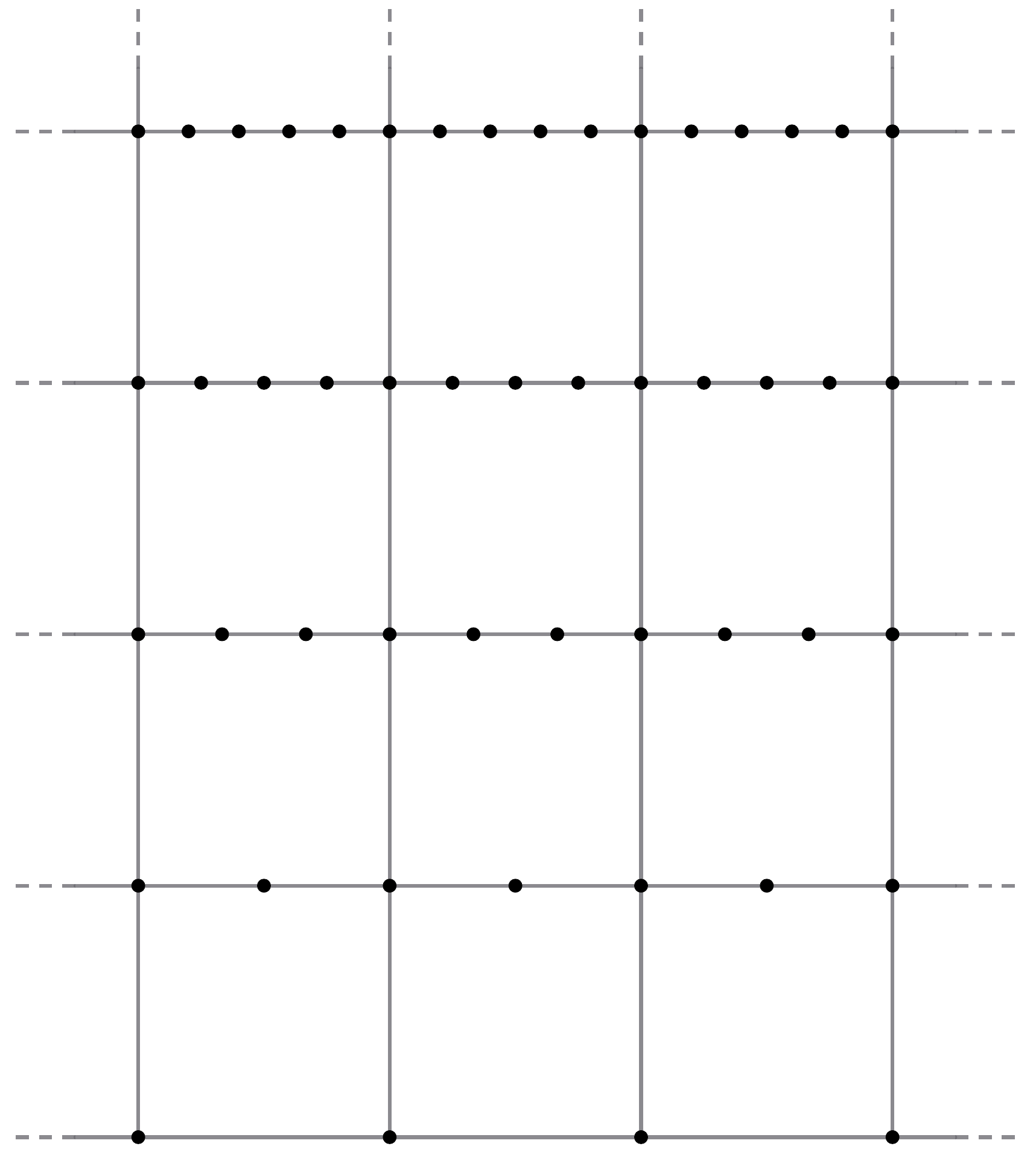}
\end{minipage}
\begin{minipage}{0.48\linewidth}
Let us given an example of a graph this is one-ended but not locally one-ended. Start with a square tiling of the upper halfplane, i.e.\ the graph $\mathbb{Z} \times \mathbb{N}$. Then, for every $k \geq 1$, subdivide $k$ times each horizontal edge in $\mathbb{Z} \times \{k\}$. The figure on the left illustrates the graph $X$ thus obtained. Clearly, $X$ is uniformly one-ended. However, it is coarsely locally separated by any of its vertices, so it is not locally one-ended. For instance, in order to deform a path $[-n,n] \times \{0\}$, with $n$ very large, and avoid a ball of large radius $R$ centred at $(0,0)$, one would need to bypass larger and larger cycles as $R$ grows, which cannot be done up to $E$-coarse homotopy equivalence with $E$ small compared to $R$. 
\end{minipage}

\medskip \noindent
We saw with Proposition~\ref{prop:OneEnded} that most wreath products are uniformly one-ended. But what about local one-endedness? We will see in Section~\ref{section:Embedding} that, actually, most lamplighter graphs are not locally one-ended, and that this is a key observation in order to understand their geometry. Theorem~\ref{thm:LampStringy} will prove much more than some non local one-endedness. For now, we only describe a general method that allows us to show that some infinitely presented groups are not locally one-ended and we apply it to lamplighters over~$\mathbb{Z}$.

\begin{thm}\label{thm:Truncating}
Let $G$ be an infinite group defined by an infinite presentation
$$\langle x_1, \ldots, x_n \mid r_1,r_2, \ldots \rangle.$$
For every $k \geq 1$, let $G_k$ denote the group defined by the truncated presentation
$$\langle x_1, \ldots, x_n \mid r_1, r_2, \ldots, r_k \rangle.$$
If no $G_k$ is one-ended, then $G$ is not locally one-ended.
\end{thm}

\noindent
In order to prove the theorem, the following notion will be necessary:

\begin{definition}
Let $G$ be a group, $S \subset G$ a generating set, and $\rho : G \to Q$ a morphism to some group $Q$. The \emph{injectivity radius of $\rho$ (relative to $S$)} is
$$\mathrm{inj}_S(\rho):= \min \{ \|g\|_S \mid g \in \mathrm{ker}(\rho) \text{ non-trivial}\}.$$
\end{definition}

\noindent
A key observation is that a surjective morphism always satisfies a weak lifting property for paths with respect to $E$-coarsely homotopy equivalence if $E$ is sufficiently small compared to the injectivity radius.  

\begin{lemma}\label{lem:LiftQuotient}
Let $G$ be a group endowed with a finite generating set $S \subset G$. Let $\rho : G \twoheadrightarrow H$ be a surjective morphism to some group $H$, which we endow with the finite generating set $\rho(S)$. Fix some $E< \mathrm{inj}_S(\rho)/2$. For every path $\alpha$ in $G$, every path $\beta$ in $H$ that is $E$-coarsely homotopy equivalent to $\rho(\alpha)$ admits a $\rho$-lift $\tilde{\beta}$ that is $E$-coarsely homotopy equivalent to $\alpha$. In particular, $\tilde{\beta}$ has the same endpoints as $\alpha$. 
\end{lemma}

\begin{proof}
The successive vertices of our path $\alpha$ can be written as
$$g, gs_1, gs_1s_2, \ldots, gs_1s_2 \cdots s_n$$
for some $g \in G$ and $s_1, \ldots, s_n \in S \cup S^{-1}$. Consequently, the path $\rho(\alpha)$ can be described as
$$\rho(g), \rho(g) \rho(s_1), \rho(g) \rho(s_1) \rho(s_2), \ldots, \rho(g) \rho(s_1) \rho(s_2) \cdots \rho(s_n).$$
For convenience, we set $w_k:= gs_1 \cdots s_k$ for every $0 \leq k \leq n$. Let $\beta$ be a path obtained from $\rho(\alpha)$ by replacing some subpath of length $\leq E$ with some subpath of length $\leq E$. In other words, there exist $\ell \leq E$, $0 \leq k \leq n- \ell$, and $r_1, \ldots, r_\ell \in S \cup S^{-1}$ such that $\beta$ is obtained from $\rho(\alpha)$ by replacing
$$\rho(w_k), \rho(w_k) \rho(s_{k+1}), \ldots, \rho(w_k) \rho(s_{k+1}) \cdots \rho(s_{k+ \ell})$$
with 
$$\rho(w_k), \rho(w_k) \rho(r_1) ,\ldots, \rho(w_k) \rho(r_1) \cdots \rho(r_\ell).$$
Since the terminal vertices of the previous two paths must coincide, the element $h:=s_{k+1} \cdots s_{k+ \ell} r_\ell^{-1} \cdots r_1^{-1}$ necessarily belongs to $\mathrm{ker}(\rho)$. But $\|h\|_S \leq 2 \ell \leq 2E< \mathrm{inj}_S(\rho)$, so $h=1$. Thus, we can define a new path $\tilde{\beta}$ from $\alpha$ by replacing
$$w_k, w_ks_{k+1}, \ldots, w_ks_{k+1} \cdots s_{k+\ell} \text{ with } w_k, w_kr_1, \ldots, w_kr_1 \cdots r_\ell$$
since now we know that the terminal vertices of these two paths coincide. Clearly, $\tilde{\beta}$ is $E$-homotopy equivalent to $\alpha$ and $\rho(\tilde{\beta})= \beta$. 
\end{proof}

\noindent
Before prove Theorem~\ref{thm:Truncating}, we need a last elementary observation:

\begin{lemma}\label{lem:DiameterBallQuotient}
Let $G,H$ be two finitely generated groups and $\rho : G \twoheadrightarrow H$ a surjective morphism. Assume that $H$ is infinite. For every $D \geq 0$, there exists some $K \geq 0$ such that the image under $\rho$ of every ball of radius $K$ has diameter $\geq D$. 
\end{lemma}

\begin{proof}
Because $H$ is infinite, we can find some element $h \in H$ satisfying $\|h\| \geq D$. Fix some $g \in G$ such that $\rho(g)=h$ and set $K:= \|g\|$. For every $x \in G$, we have
$$\mathrm{diam}( \rho(B(x,K))) \geq d(\rho(x), \rho(xg)) = \| \rho(g)\| = \|h\| \geq D,$$
which proves our lemma. 
\end{proof}

\begin{proof}[Proof of Theorem~\ref{thm:Truncating}.]
We endow $G$ and each $G_k$ with the generating sets given by their presentations. For every $k \geq 0$, we denote by $\rho_k$ the natural surjective morphism $G_k \twoheadrightarrow G$. 

\begin{claim}\label{claim:InjRadius}
For every $M \geq 0$, there exists $k \geq 1$ such that $\mathrm{inj}(\rho_k)>M$.
\end{claim}

\noindent
Fix some $k \geq 1$ sufficiently large such that, for every word $w$ of length $\leq M$ written over $\{x_1, \ldots, x_n\}$, if $w=1$ in $G$ then the equality can be deduced from the relations $r_1, \ldots, r_k$. Now, let $g \in G_k$ be an arbitrary element in $\mathrm{ker}(\rho_k)$ satisfying $\|g\| \leq M$. We can represent $g$ as a word $w$ of length $\leq M$ written over $\{x_1, \ldots, x_n\}$. Saying that $g \in \mathrm{ker}(\rho_k)$ amounts to saying that the equality $w=1$ holds in $G$, which implies that the equality $w=1$ also holds in $G_k$ by definition of $k$. Thus, $g$ is trivial in $G_k$. We conclude that $\mathrm{inj}(\rho_k) > M$, as desired, completing the proof of Claim~\ref{claim:InjRadius}. 

\medskip \noindent
Now, fix an arbitrary $E \geq 0$. According to Claim~\ref{claim:InjRadius}, there exists some $k \geq 1$ such that $\mathrm{inj}(\rho_k)>2E$. We know by assumption that $G_k$ is not one-ended. Since it is also infinite (as it surjects onto the infinite group $G$), there must exist some $F \geq 0$ such that the complement of the ball $B_k(F)$ of radius $F$ centred at $1$ contains at least two unbounded connected components. Consequently, fixing an arbitrary $L \geq 0$, we can find two elements $x,y \in G_k$ separated by $B_k(F)$ and satisfying $d(x,B_k(F)), d(y,B_k(F))>K$, where $K \geq 0$ is the constant given by Lemma~\ref{lem:DiameterBallQuotient} such that the image under $\rho_k$ of every ball of radius $K$ has diameter $\geq 2(F+L)$. As a consequence, $\rho_k$ cannot send the ball $B(x,K)$ inside the $L$-neighbourhood of the ball $B(F)$ of radius $F$ centred at $1$. In other words, there exists $a \in B(x,K)$ such that $d(\rho_k(a),B(F)) \geq L$. Similarly, there exists $b \in B(y,K)$ such that $d(\rho_k(b),B(F)) \geq L$.

\medskip \noindent
Thus, we have found two elements $a,b \in G_k$ such that $d(\rho_k(a),B(F)), d(\rho_k(b),B(F)) \geq L$ and such that $a$ and $b$ are separated by $B_k(F)$. (The latter assertion follows from the fact that $B_k(F)$ separates $x$ and $y$, which both lie at distance $>K$ from $B_k(F)$.) 

\medskip \noindent
Now, let $\alpha$ be a path connecting $a$ and $b$. Since $\rho_k$ is $1$-Lipschitz, necessarily $\rho_k(\alpha)$ intersects $B(F)$ (as $\alpha$ intersects $B_k(F)$). In order to conclude the proof of our theorem, we want to show that this intersection is $E$-persistent. But, if $\beta$ is a path $E$-coarsely homotopy equivalent to $\rho_k(\alpha)$, then we know from Lemma~\ref{lem:LiftQuotient} that $\beta$ admits a $\rho_k$-lift $\tilde{\beta}$ that has the same endpoints as $\alpha$. But $B_k(F)$ separates the endpoints of $\alpha$, so $\tilde{\beta}$ must intersect $B_k(F)$, and consequently $\beta$ must intersect $B(F)$. This concludes the proof of our theorem. 
\end{proof}

\noindent
Now, we are ready to show that some lamplighter groups are not locally one-ended. See Exercise~\ref{exo:ShufflerLocalOneEnded} for another example. 

\begin{cor}\label{cor:LampLocalOneEnded}
For every $n \geq 2$, the lamplighter group $\mathbb{Z}_n \wr \mathbb{Z}$ is one-ended but not locally one-ended.
\end{cor}

\begin{proof}
We already know from Proposition~\ref{prop:OneEnded} that $\mathbb{Z}_n \wr \mathbb{Z}$ is one-ended. Now, given an integer $k \geq 2$, consider the HNN extension of $\mathbb{Z}_n^{\{-k, \ldots, k\}}$ obtained by conjugating the two isomorphic subgroups $\mathbb{Z}_n^{\{-k+1, \ldots, k\}}$ and $\mathbb{Z}_n^{\{-k, \ldots, k-1\}}$. It admits as a presentation
$$\left\langle z_{-k}, \ldots, z_k, t \mid \begin{array}{l} z_i^2 = 1 \text{ for every } -k \leq i \leq k \\ \left[ z_i, z_j \right]= 1 \text{ for all } -k \leq i,j \leq k \\ tz_it^{-1}=z_{i+1} \text{ for every } -k \leq i \leq k-1 \end{array} \right\rangle.$$
Clearly, this is a truncation of the infinite presentation
$$\left\langle z_i \ (i \in \mathbb{Z}), t \mid \begin{array}{l} z_i^2 = 1 \text{ for every } i \in \mathbb{Z} \\ \left[ z_i, z_j \right]= 1 \text{ for all }  i,j \in \mathbb{Z} \\ tz_it^{-1}=z_{i+1} \text{ for every } i \in \mathbb{Z} \end{array} \right\rangle,$$
which can be simplified into the presentation
$$\langle z,t \mid z^2=1, [z, t^izt^{-i}]=1 \text{ for every } i \in \mathbb{Z} \rangle.$$
This is a presentation of $\mathbb{Z}_n \wr \mathbb{Z}$. Since HNN extensions of finite groups are virtually free, and in particular not one-ended, we can apply Theorem~\ref{thm:Truncating} and conclude that $\mathbb{Z}_n \wr \mathbb{Z}$ is not locally one-ended. 
\end{proof}

\noindent
\begin{minipage}{0.2\linewidth}
\includegraphics[width=0.98\linewidth]{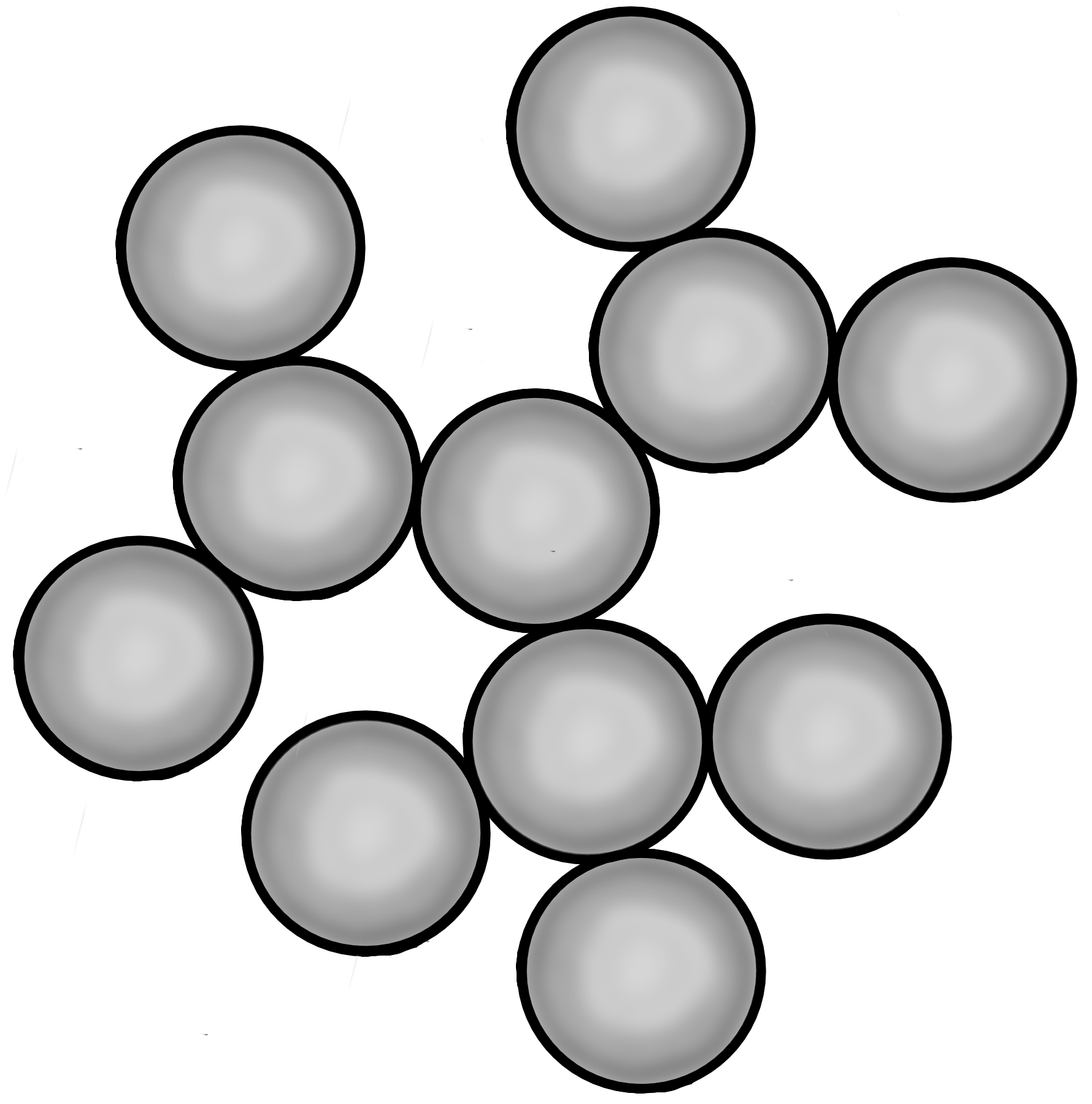} 
\end{minipage}
\begin{minipage}{0.78\linewidth}
Given a (sufficiently nice) topological space with cut points, we may expect to describe it as maximal pieces with no cut point glued together along isolated points in a tree-like way. Here, two points would belong to a common piece whenever they can be connected by two paths that meet only at their endpoints. (Loosely speaking, our two points are connected by a non-degenerate \emph{bigon}.)
\end{minipage}

\medskip \noindent
\begin{minipage}{0.78\linewidth}
For local cut points, something similar is possible, but pieces may have local cut points themselves, as suggested by the figure on the right. Two points would belong to a common piece whenever they can be connected by two paths that meet only at their endpoints and that are homotopy equivalent. (Loosely speaking, our two points are connected by a non-degenerate homotopically trivial bigon.) 
\end{minipage}
\begin{minipage}{0.2\linewidth}
\hspace{0.1\linewidth}\includegraphics[width=0.98\linewidth]{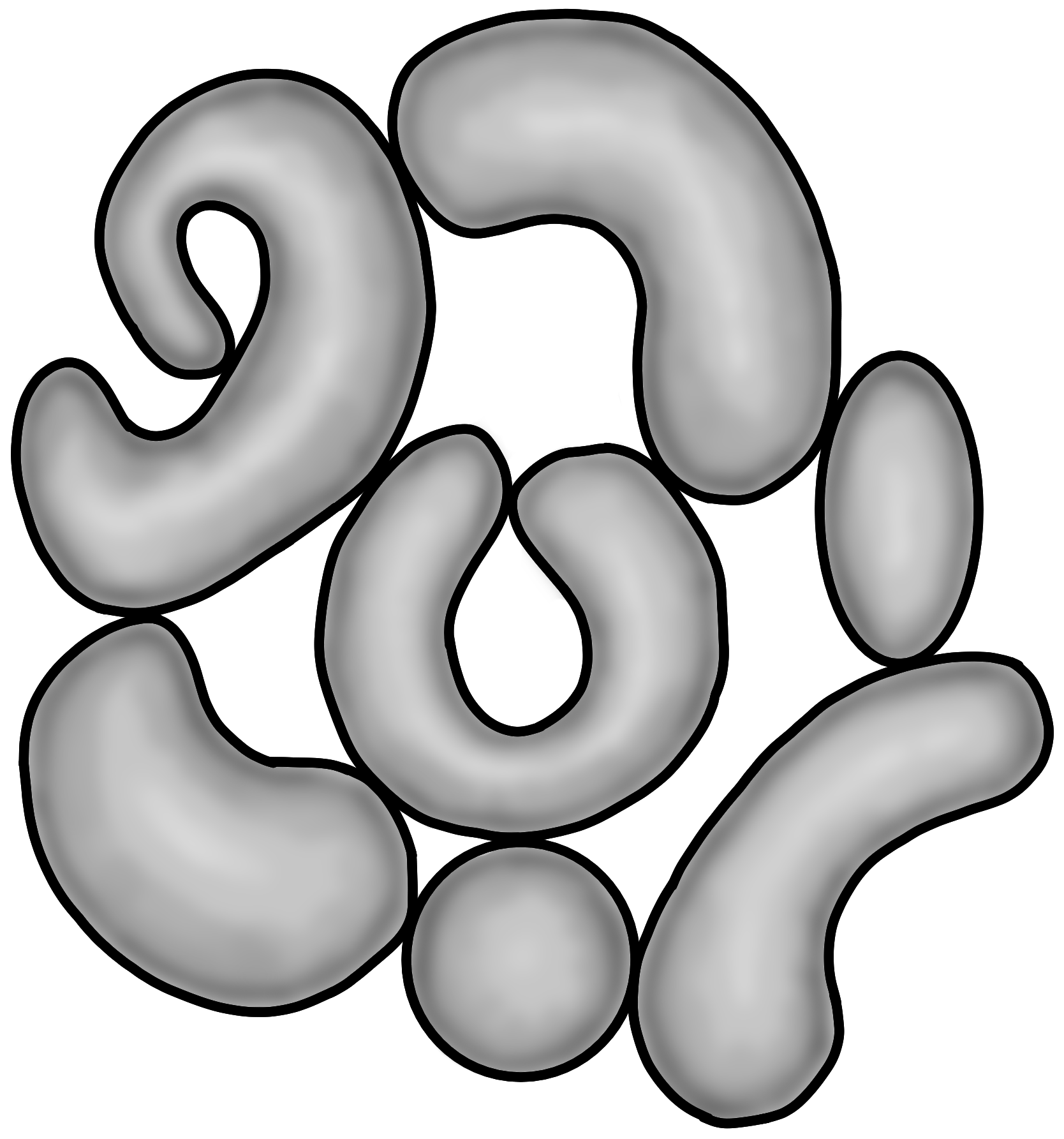} 
\end{minipage}

\medskip \noindent
The following definition capture this idea in the framework of coarse topology. This notion will allow us, in Section~\ref{section:BigEmbedding}, to identify the pieces when decomposing wreath products along their local cut points. 

\begin{definition}\label{def:Pancylindrical}
A graph $X$ is \emph{pancylindrical}\footnote{In \cite{MR4794592}, this property is referred to as the \emph{thick bigon property}.} if there exists $E \geq 0$ such that, for every $F \geq 0$, the following holds for some $L \geq 0$: any two vertices $x,y \in V(X)$ are connected by some path $\gamma$ such that, for every subgraph $B$ of diameter $\leq F$ intersecting $\gamma$, if $d(x,B),d(y,B) \geq L$ then the intersection between $\gamma$ and $B$ is not $E$-persistent. 
\end{definition}

\noindent
Notice that a uniformly locally one ended graph is automatically pancylindrical. Below is an example of a graph that is pancylindrical but not locally one-ended. 
\begin{center}
\includegraphics[width=\linewidth]{NotPan}
\end{center}
It can be proved by standard arguments that being pancylindrical is preserved by quasi-isometries. We leave the details as an exercise to the interested reader. As a consequence, it makes sense to refer to pancylindrical finitely generated groups. Examples of pancylindrical finitely generated groups include: wreath product of finitely generated groups $A \wr B$ with $A$ infinite and $B$ non-trivial \cite[Propositions~3.6 and~3.8]{MR4794592}; torsion-free finitely generated solvable groups that are not cyclic \cite[Corollary~3.13]{MR4794592}; and any finitely generated group containing an infinite finitely generated normal subgroup \cite[Proposition~3.16]{MR4794592}. Here, we restrict ourselves to the following source of examples:

\begin{lemma}
Let $A$ and $B$ be two unbounded graphs. The product $A \times B$ is uniformly locally one-ended, and a fortiori pancylindrical.
\end{lemma}

\begin{proof}
Fix an $F \geq 0$ and set $L:= 1+4F$. We want to prove that, given any two vertices $x_1=(a_1,b_1), x_2= (a_2,b_2) \in A\times B$, any path $\gamma$ connecting $x_1$ to $x_2$, and any ball $B$ of radius $F$ satisfying $d(x_1,B),d(x_2,B) \geq L$, $\gamma$ does not intersect $2$-persistently $B$. 

\medskip \noindent
\begin{minipage}{0.4\linewidth}
\includegraphics[width=0.98\linewidth]{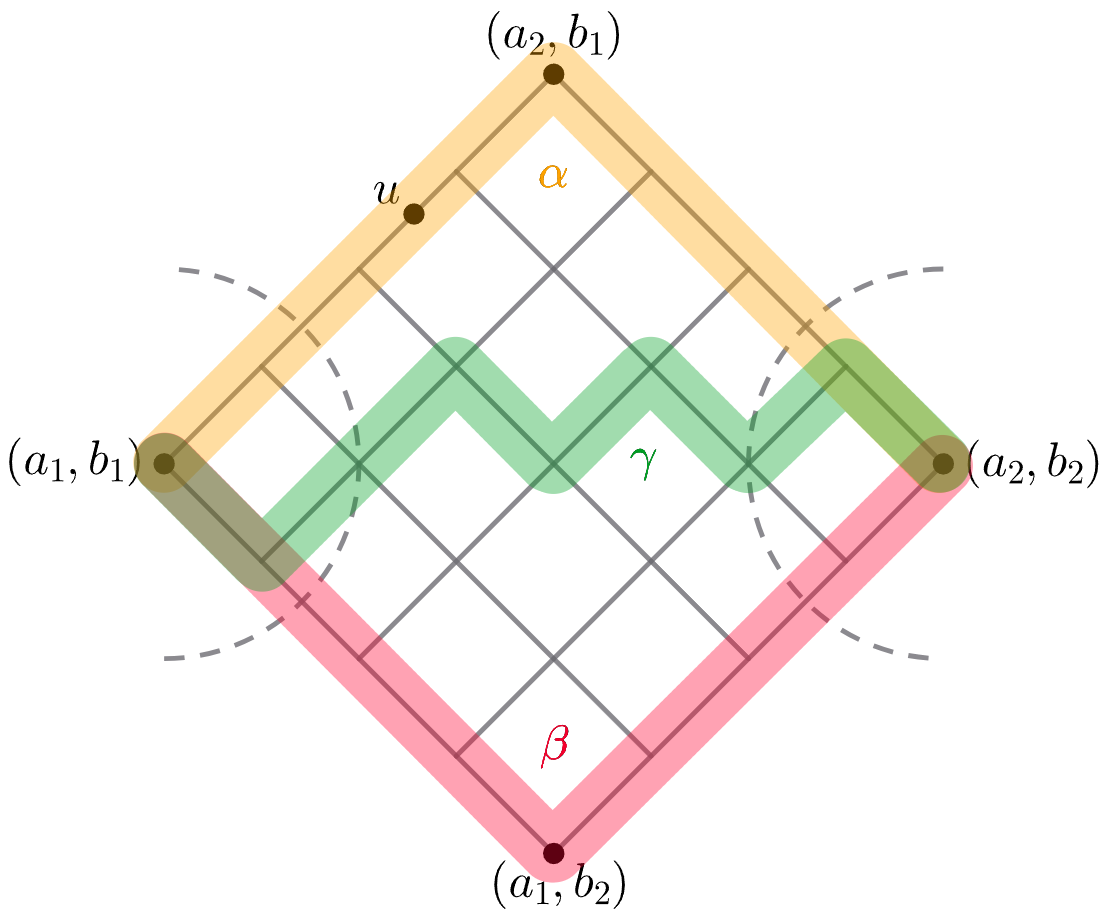}
\end{minipage}
\begin{minipage}{0.58\linewidth}
Projecting $\gamma$ to $A$ (resp.\ $B$), we get a path $[a_1,a_2]$ (resp.\ $[b_1,b_2]$) connecting $a_1$ to $a_2$ (resp.\ $b_1$ to $b_2$). Then, $\gamma$ is contained in the subproduct $[a_1,a_2] \times [b_1, b_2]$, which is a grid and is therefore $2$-coarsely simply connected. Consequently, $\gamma$ is $2$-coarsely homotopy equivalent to both the path $\alpha:= ([a_1,a_2] \times \{b_1\}) \cup( \{a_2\} \times [b_1,b_2])$ and the path $\beta:= (\{a_1\} \times [b_1,b_2]) \cup ([a_1,a_2] \times \{b_2\})$. 
\end{minipage}

\medskip \noindent
We claim that, if $d(a_1,a_2),d(b_1,b_2) \geq L/2$, then $\alpha \backslash (B(x_1,L) \cup B(x_2,L))$ and $\beta \backslash (B(x_1,L) \cup B(x_2,L))$ lie at distance $>2F$. So let $u$ be a vertex of $\alpha \backslash (B(x_1,L) \cup B(x_2,L))$. Two cases may happen: either $u \in [a_1,a_2] \times \{b_1\}$ or $u \in \{a_2\} \times [b_1,b_2]$. The two cases being symmetric, we assume that $u \in [a_1,a_2] \times \{b_1\}$. Also, let $v$ be a vertex of $\beta \backslash (B(x_1,L) \cup B(x_2,L))$. If $v \in [a_1,a_2] \times \{b_2\}$, then it is clear that $d(u,v) \geq d(b_1,b_2) \geq L/2> 2F$. Otherwise, if $v \in \{a_1\} \times [b_1,b_2]$, then we can write $v=(a_1,b)$ for some $b \in [b_1,b_2]$. Notice that, because $d(v,(a_1,b_1)) \geq L$, we must have $d(b,b_1) \geq L$. Consequently, $d(u,v) \geq d(b,b_1) \geq L> 2F$. Our claim is proved.

\medskip \noindent
Thus, if $d(a_1,a_2),d(b_1,b_2) \geq L/2$, then our ball $B$ must be disjoint from either $\alpha$ or $\beta$, proving that the intersection with $\gamma$ is not $2$-persistent. Now, assume that either $d(a_1,a_2)<L/2$ or $d(b_1,b_2)<L/2$. The two cases being symmetric, we assume that $d(a_1,a_2)<L/2$. Because $A$ is unbounded, we can find a vertex $a_3 \in V(A)$ such that $d(a_2,a_3) >2F$. Fix a geodesic $[a_2,a_3]$ connecting $a_2$ to $a_3$ in $A$. Let $\delta$ be the path
$$([a_1,a_2] \times \{b_1\}) \cup ([a_2,a_3] \times \{b_1\}) \cup (\{a_3\} \times [b_1,b_2] ) \cup ( [a_3,a_2] \times \{b_2\}).$$
Notice that, since $[a_2,a_3] \times [b_1,b_2]$ is $2$-coarsely simply connected, $\delta$ is $2$-coarsely homotopy equivalent to $\alpha$, and a fortiori to $\gamma$. We claim that $\alpha$ or $\delta$ is disjoint from $B$, which will prove, as desired, that the intersection between $\gamma$ and $B$ is not $2$-persistent.

\medskip \noindent
\begin{minipage}{0.4\linewidth}
\includegraphics[width=0.98\linewidth]{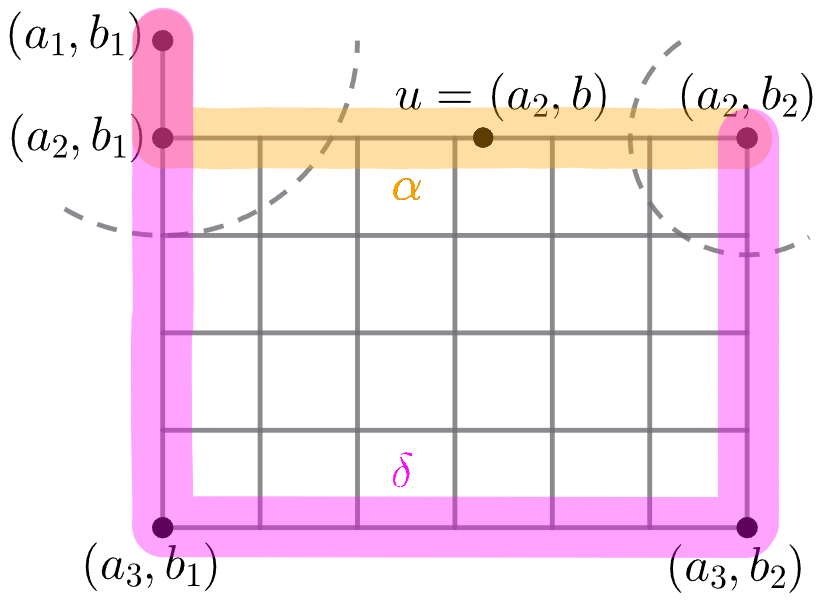}
\end{minipage}
\begin{minipage}{0.58\linewidth}
It suffices to show that $\alpha \backslash (B(x_1,L) \cup B(x_2,L))$ and $\delta \backslash (B(x_1,L) \cup B(x_2,L))$ lie at distance $>2F$ from each other. A vertex $u$ of $\alpha \backslash (B(x_1,L) \cup B(x_2,L))$ can be written as $(a_2,b)$ for some vertex $b \in V(B)$. First, notice that
$$\begin{array}{lcl} d(u, [a_2,a_3] \times \{b_2\}) & \geq & d(b,b_2)= d((a_2,b),(a_2,b_2)) \\ \\ & \geq & L>2F. \end{array}$$
\end{minipage}

\medskip \noindent
Then, notice that
$$d(u, \{a_3\} \times [b_1,b_2]) \geq d(a_2,a_3)> 2F.$$
Finally, notice that
$$\begin{array}{lcl} d(u,  [a_2,a_3] \times \{b_1\}) & \geq & d(b,b_1)= d((a_2,b),(a_2,b_1)) \\ \\ & \geq & d((a_2,b),(a_1,b_1)) - d((a_1,b_1), (a_2,b_1)) \\ \\ & \geq & L -L/2=L/2>2F.\end{array}$$
This concludes our proof.
\end{proof}

\subsection{Bonus: coarse embeddings in Hilbert spaces}\label{section:CoarseHilbert}

\noindent
As a natural way to compare the large-scale geometry of a given finitely generated group with the well-known Euclidean geometry, one can ask whether our group coarsely embeds into a Euclidean space. However, the family of groups admitting such coarse embeddings is rather too restrictive:

\begin{thm}
A finitely generated group coarsely embeds into some Euclidean space if and only if it is virtually nilpotent.
\end{thm}

\begin{proof}
A finitely generated group that coarsely embeds into some Euclidean space must have polynomial growth. Conversely, it is clear that a finitely generated group of polynomial growth is \emph{doubling} (i.e.\ there exists some $C \geq 0$ such that, for every $R$, the size of a ball of radius $2R$ is at most $C$ times the size of a ball of radius $R$), and consequently coarsely embed into some Euclidean space according to Assouad's embedding theorem (see for instance \cite[Chapter~12]{MR1800917}). Then, the desired conclusion follows from Gromov's characterisation of finitely generated groups of polynomial growth as virtually nilpotent groups. 
\end{proof}

\noindent
As a consequence, if $A,B$ are two finitely generated groups with $A$ non-trivial and $B$ infinite, then $A \wr B$ does not coarsely embed into a Euclidean space. In order to enlarge the family of accessible groups, one solution is to replace our Euclidean spaces with an infinite-dimensional model: when does a finitely generated groups coarsely embed into a Hilbert space? Then, one can show that wreath products preserve some Hilbertian geometry:

\begin{thm}[\cite{MR4449680}]\label{thm:CoarselyWreath}
Let $X,Y$ be two graphs of bounded degree and $o \in V(X)$ a basepoint. The wreath product $(X,o) \wr Y$ coarsely embeds into some Hilbert space if and only if so do $X$ and $Y$. 
\end{thm}

\noindent
However, most groups turn out to coarsely embed into Hilbert spaces, so this less informative than one may expect. In order to extract more information from such coarse embeddings, a possibility is to quantify to which extent one needs to deform the geometry of our group in order to embed it into a Hilbert space. This idea is captured as follows:

\begin{definition}
The \emph{compression} of a map $\varphi : X \to Y$ between two metric spaces is
$$\mathrm{comp}(\varphi):= \sup \{ \alpha>0 \mid \exists C >0 , \forall a,b \in X, C \cdot d(x,y)^\alpha \leq d(\varphi(x),\varphi(y)) \}.$$
The \emph{Hilbert space compression} of a metric space $X$ is
$$\alpha_2(X):= \sup \{ \mathrm{comp}(\varphi) \mid \text{for every Lipschitz } \varphi : X \to \text{Hilbert space}  \}.$$
\end{definition}

\noindent
Roughly speaking, the Hilbert space compression is a real number in $[0,1]$ that quantifies the compatibility between the large-scale geometry of our space (e.g.\ our finitely generated group) and the geometry of Hilbert spaces. The compatibility is greater as the compression approaches $1$. For instance, free abelian groups $\mathbb{Z}^n$ clearly have Hilbert space compression $1$, since they biLipschitz embed into (finite-dimensional) Hilbert spaces. Interestingly, $n$-regular trees for $n \geq 2$ also have Hilbert space compression $1$, despite the fact that they do not biLipschitz embed into Hilbert spaces \cite{MR880292}. 

\medskip \noindent
Computing Hilbert space compressions is often difficult, especially when the value lies strictly between $0$ and $1$. In full generality, Hilbert space compressions of wreath products are not well understood. Among the few cases when the compression is known is exactly, we can mention the following notable result:

\begin{thm}[\cite{MR2783928}]
Let $H$ be a group of polynomial growth. The Hilbert space compression of $\mathbb{Z} \wr H$ is $1/2$ if $H$ is not virtually $\mathbb{Z}$ and $2/3$ otherwise. 
\end{thm}

\noindent
Because its proof involves interesting constructions, we will spend some time to justify the following statement:

\begin{prop}\label{prop:CompressionLamp}
For every finite group $F$, the lamplighter group $F \wr \mathbb{Z}$ has Hilbert space compression $1$.
\end{prop}

\noindent
We focus on the case $F= \mathbb{Z}_2$ for simplicity. 

\medskip \noindent
\begin{minipage}{0.45\linewidth}
\includegraphics[width=0.98\linewidth]{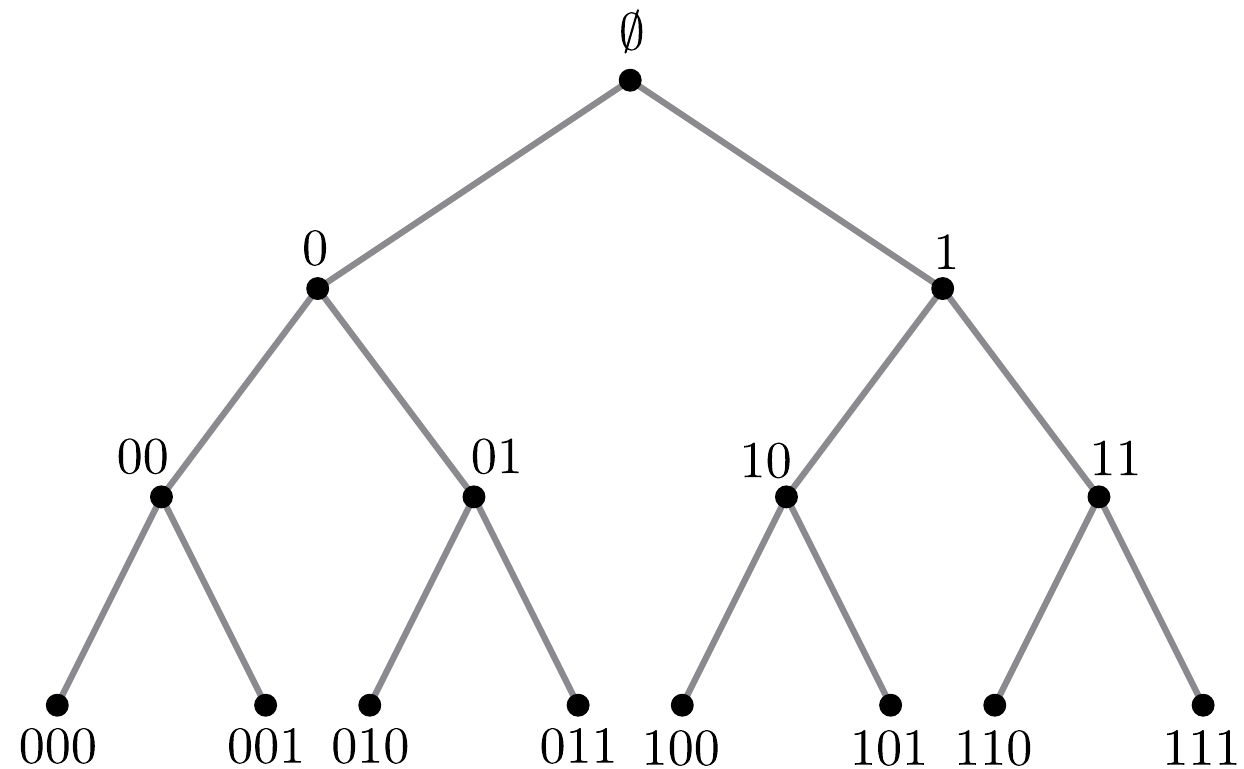}
\end{minipage}
\begin{minipage}{0.53\linewidth}
A well-known construction of a rooted $2$-regular tree is to start from an alphabet with two letters, say $\{0,1\}$ and to identify the vertices of our rooted tree with words of finitely many letters. Two words represent two adjacent vertices whenever one can be obtained from the other by adding a last letter. See the figure on the left. 
\end{minipage}

\medskip \noindent
Notice that the set of our words is just $\bigsqcup\limits_{n \geq 0} \mathbb{Z}_2^{[1,n]}.$

\medskip \noindent
By ``pushing the root to infinity'', one gets a similar construction of the (unrooted) $3$-regular tree $T$: its vertices are the left-infinite words in 
$$\mathscr{W}:= \bigsqcup\limits_{n \in \mathbb{Z}} \mathbb{Z}_2^{((- \infty,n])}$$
all but finitely many of whose letters are $0$, and two words represent two adjacent vertices whenever one can be obtained from the other by adding a last letter. The trick now is to notice that, given an element $(c,p) \in \mathbb{Z}_2 \wr \mathbb{Z}$, one has a left-infinite word by considering the colouring $c$ restricted to the left of the arrow $p$ and a right-infinite word by considering the colouring $c$ restricted to the right of the arrow $p$. Of course, the right-infinite word can be turned into a left-infinite word, so what we get is a pair of vertices in $T$. More formally,
$$\Psi : \left\{ \begin{array}{ccc} \mathbb{Z}_2 \wr \mathbb{Z} & \to & V(T \times T) \\ (c,p) & \mapsto & \left( c_{|(- \infty,p-1]}, c_{|[p,+ \infty)} (- \ast ) \right) \end{array} \right..$$
The map can be illustrated by Figure~\ref{Horo}.

\begin{figure}
\begin{center}
\begin{tabular}{|c|c|c|c|} \hline
\includegraphics[width=0.22\linewidth]{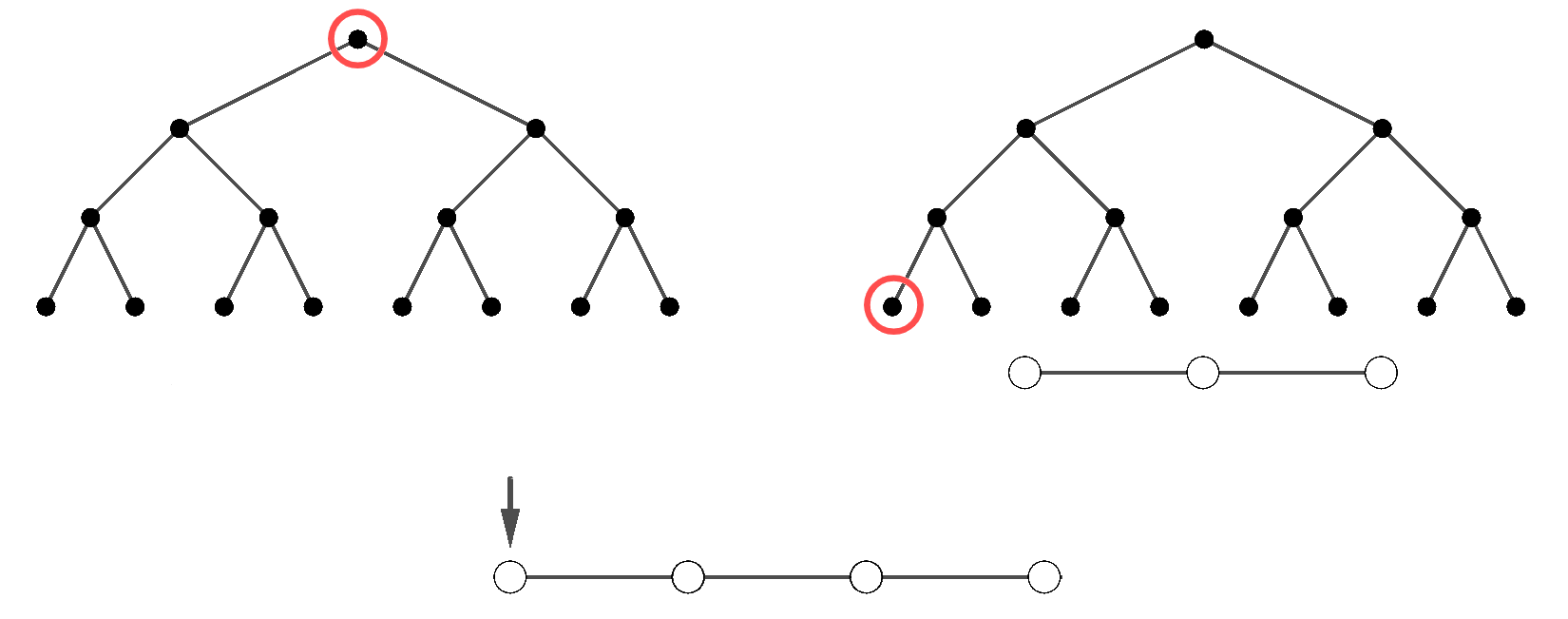} &
\includegraphics[width=0.22\linewidth]{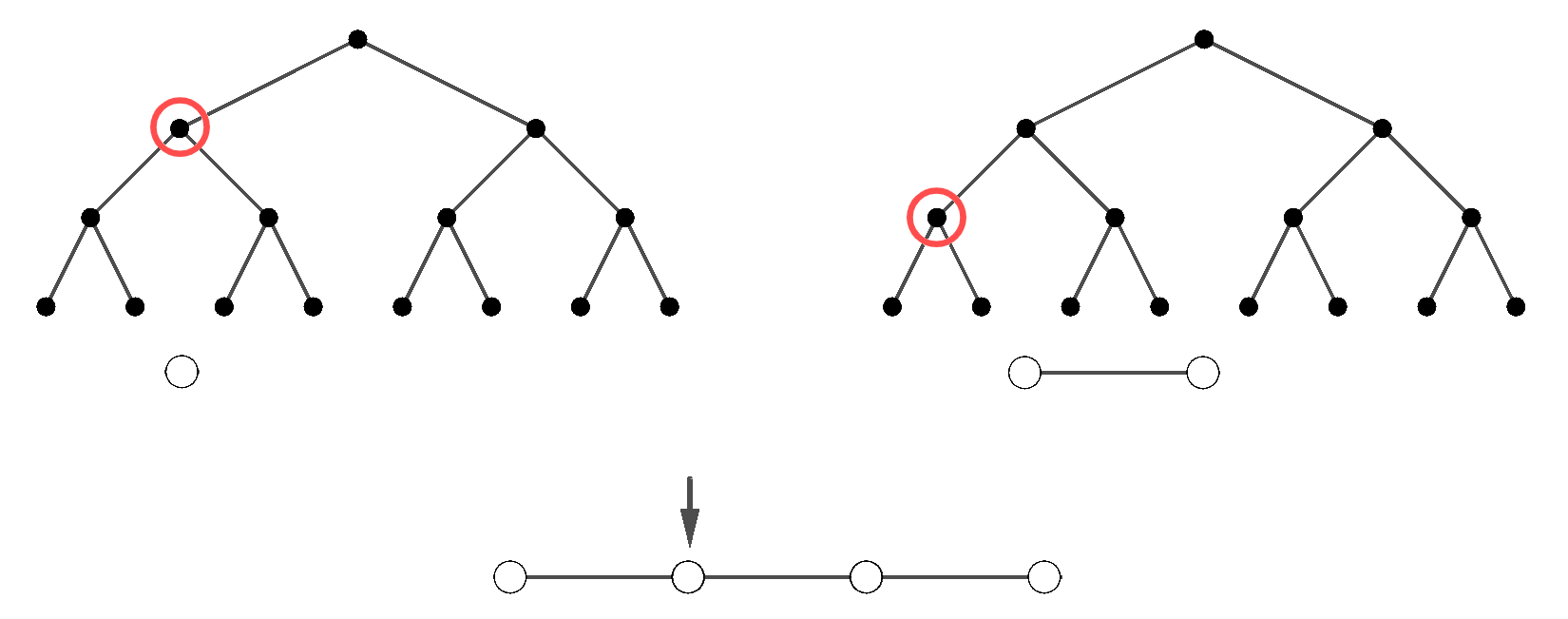} &
\includegraphics[width=0.22\linewidth]{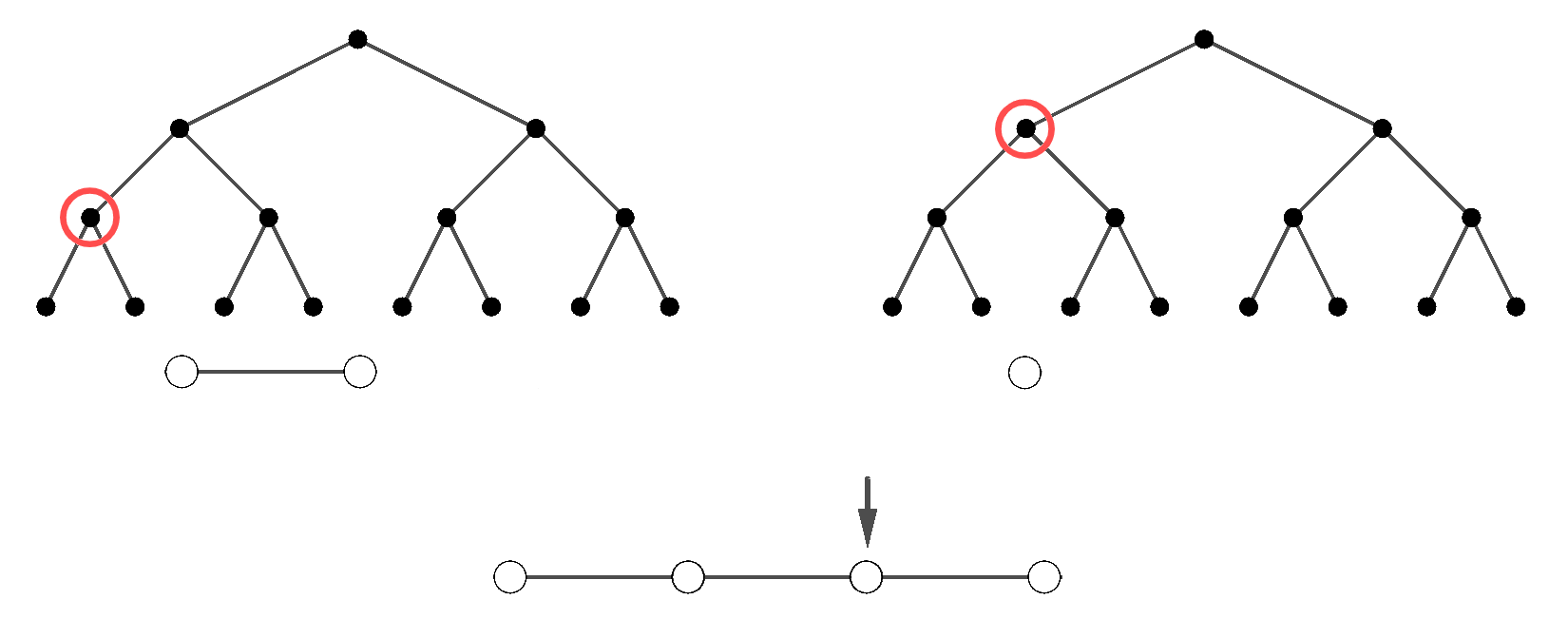} &
\includegraphics[width=0.22\linewidth]{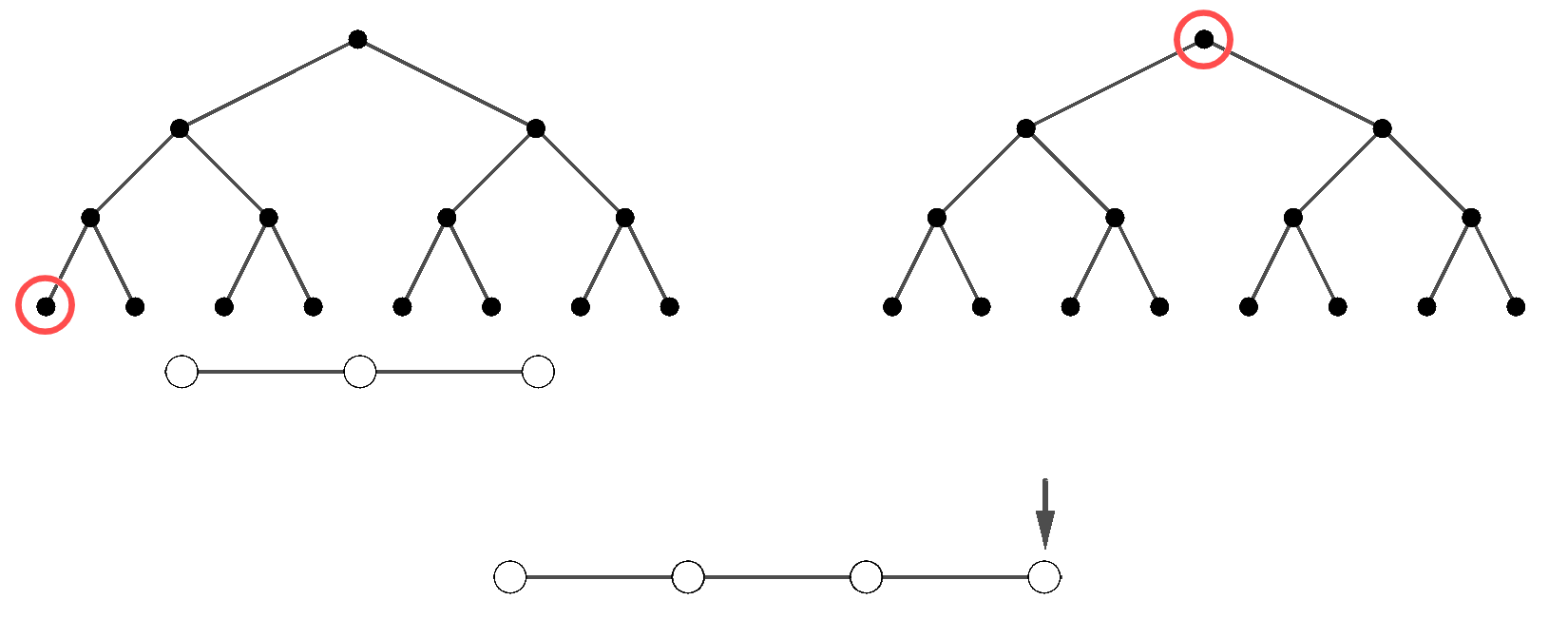} \\ \hline
\includegraphics[width=0.22\linewidth]{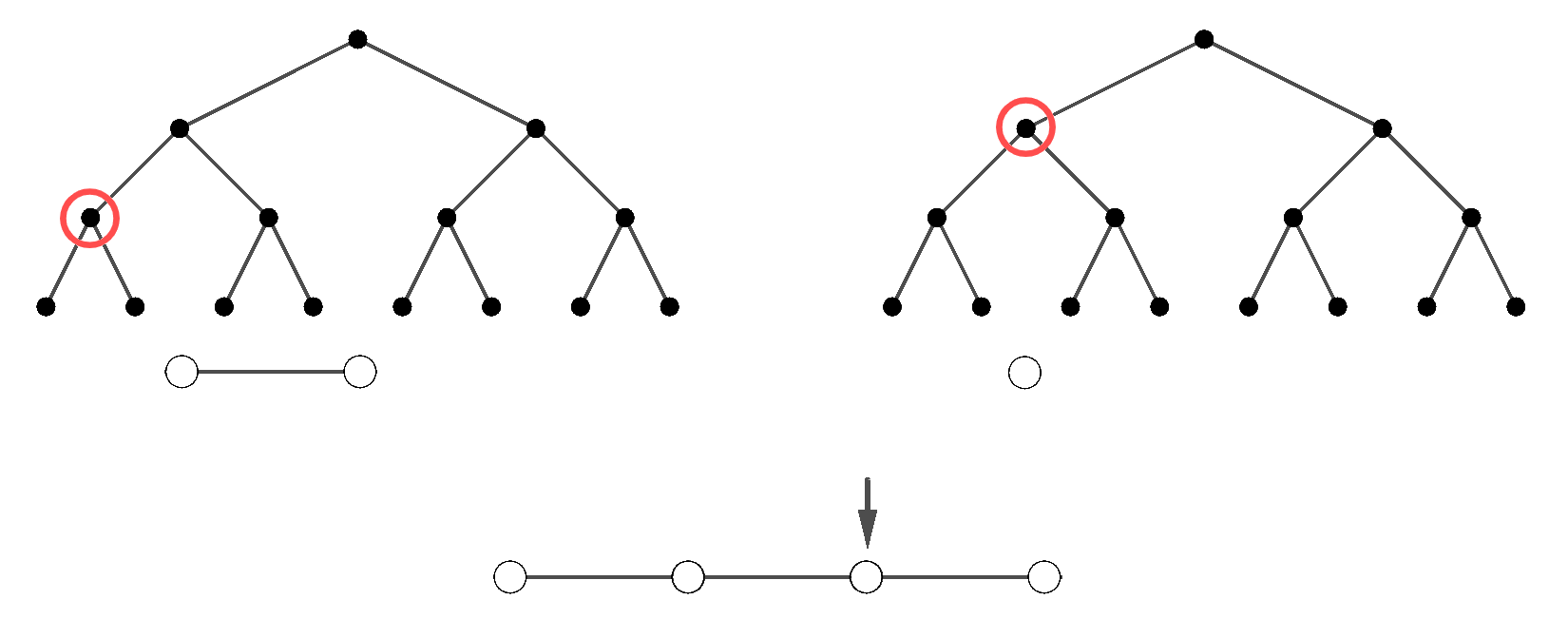} &
\includegraphics[width=0.22\linewidth]{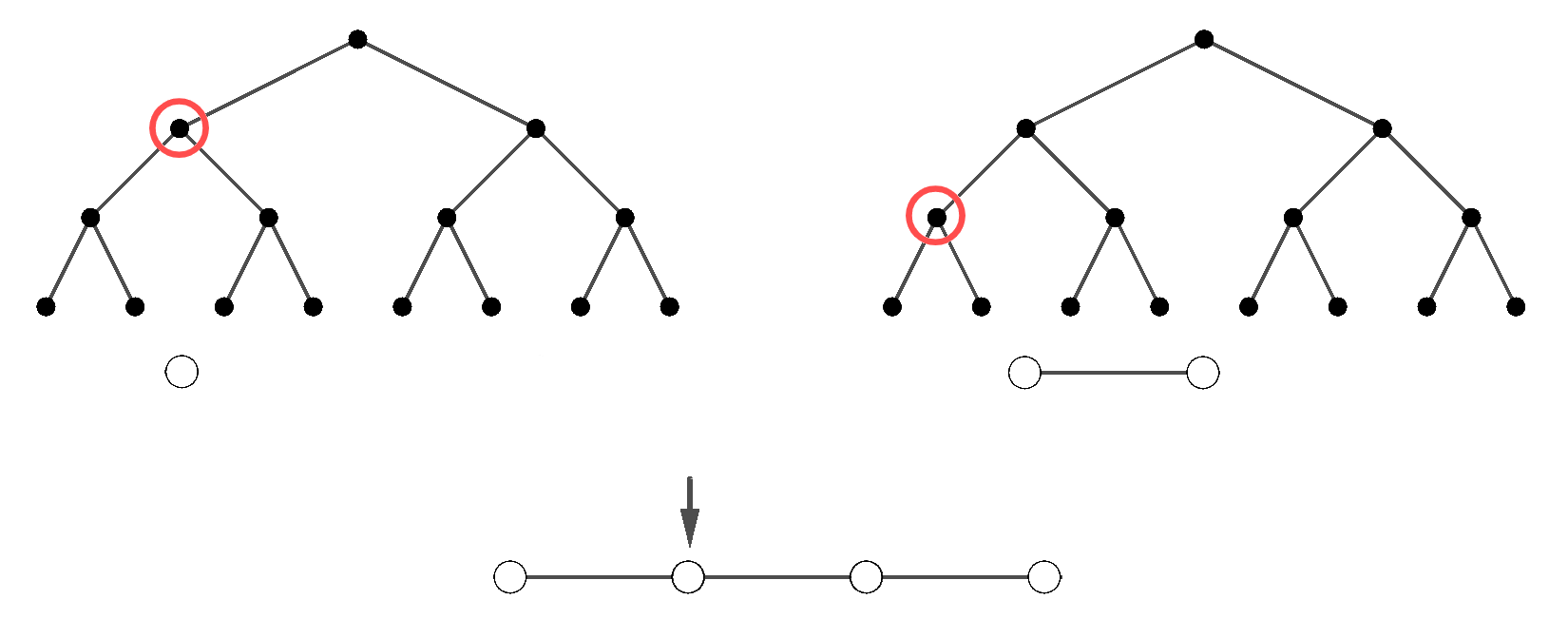} &
\includegraphics[width=0.22\linewidth]{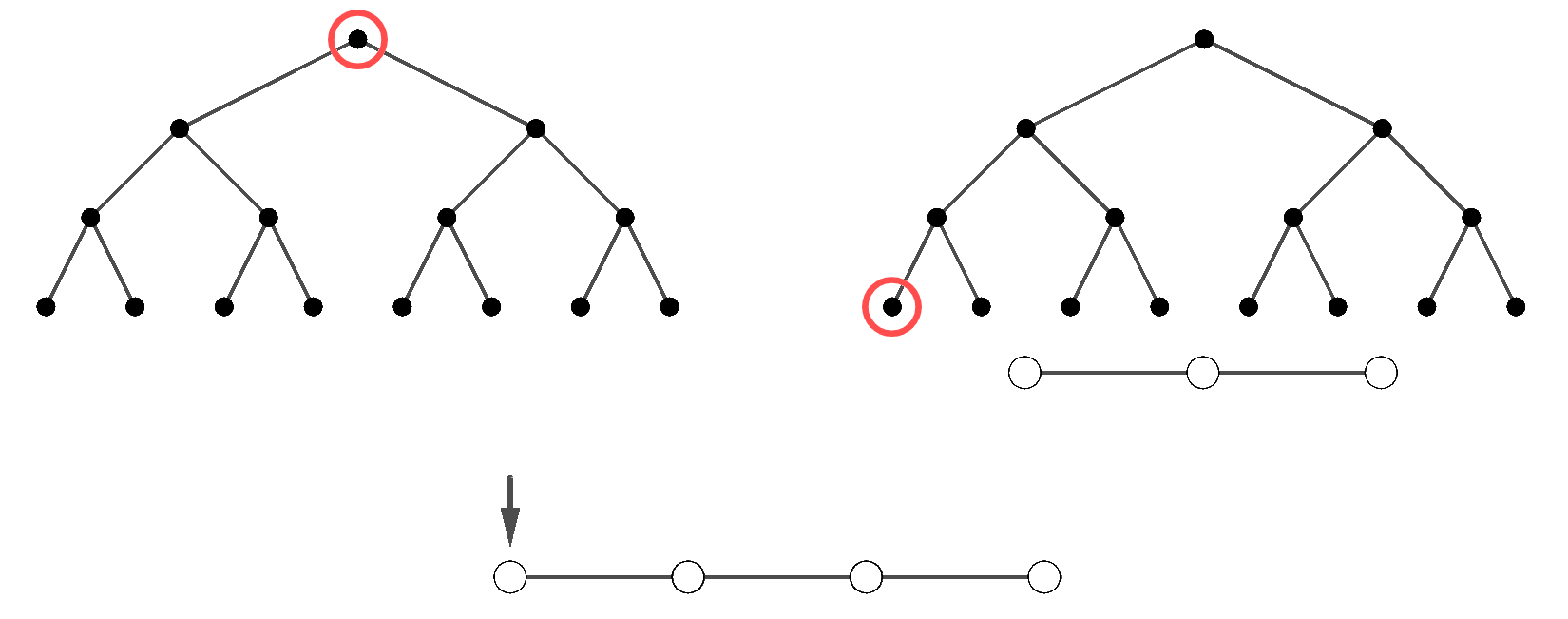} &
\includegraphics[width=0.22\linewidth]{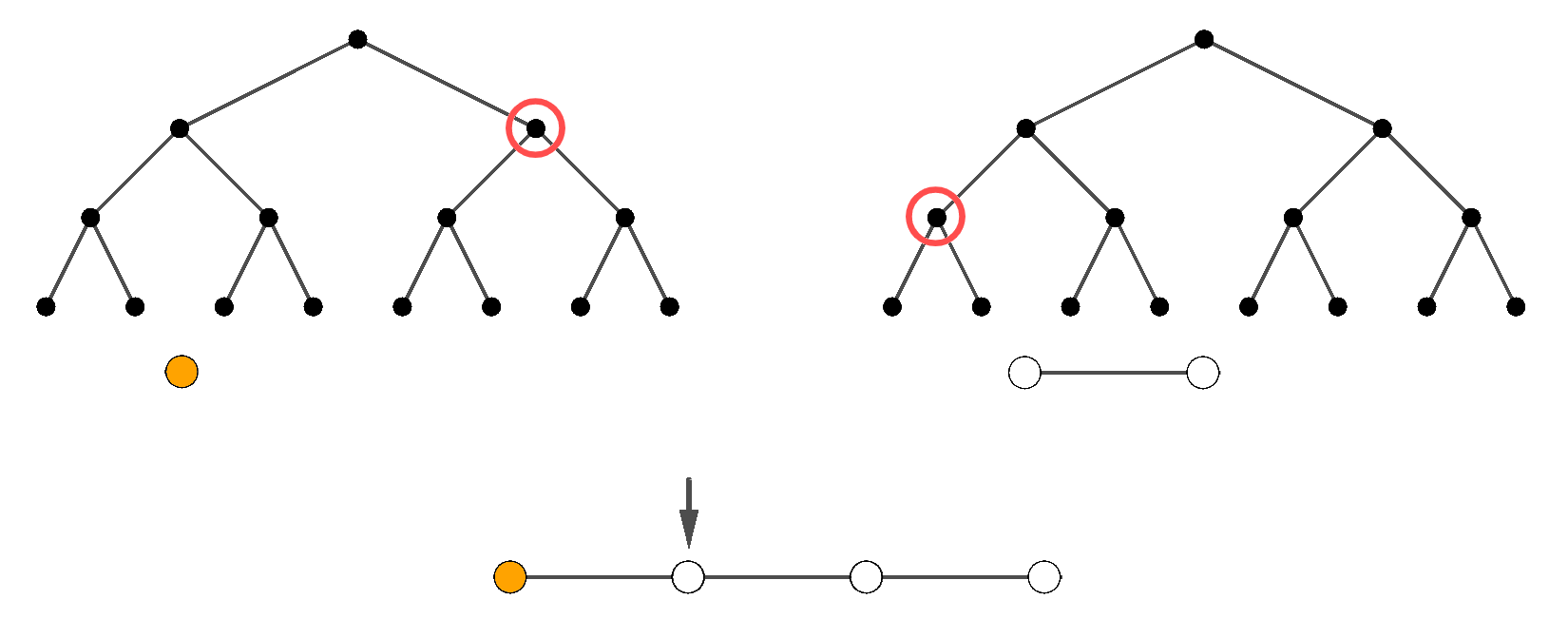} \\ \hline
\includegraphics[width=0.22\linewidth]{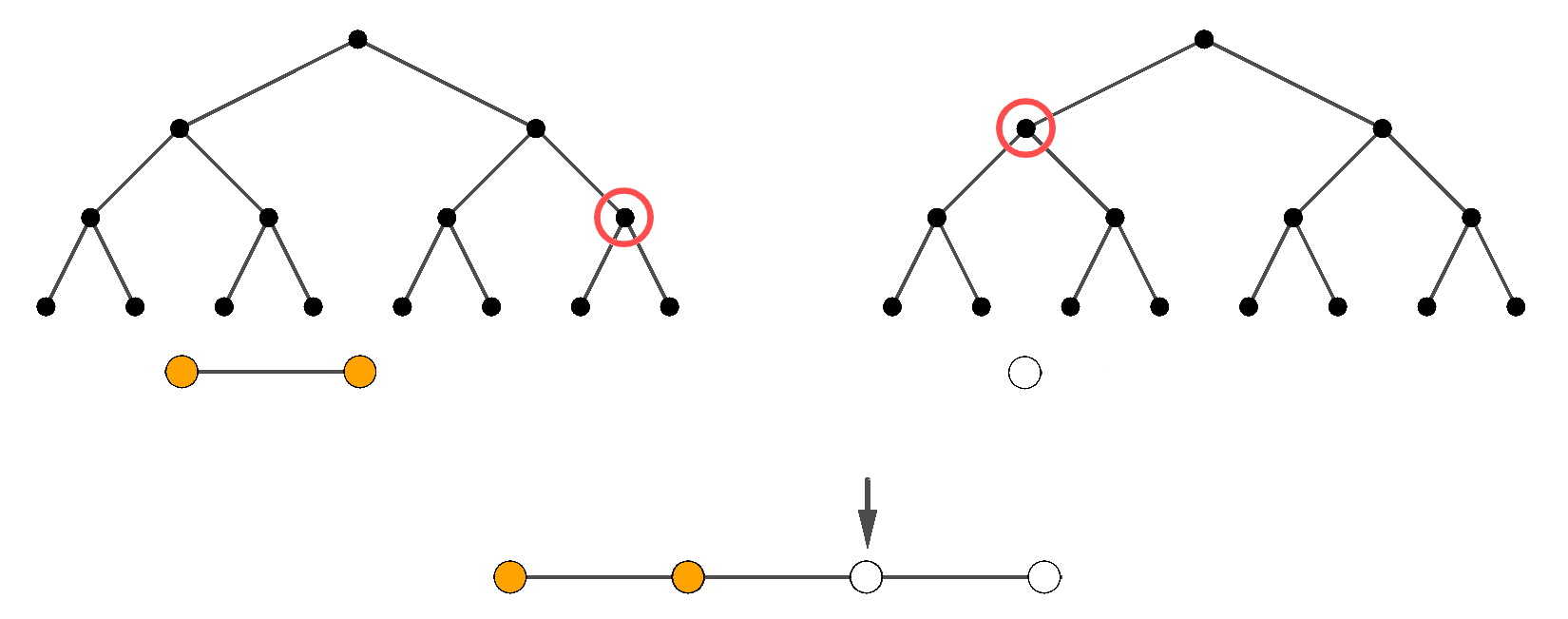} &
\includegraphics[width=0.22\linewidth]{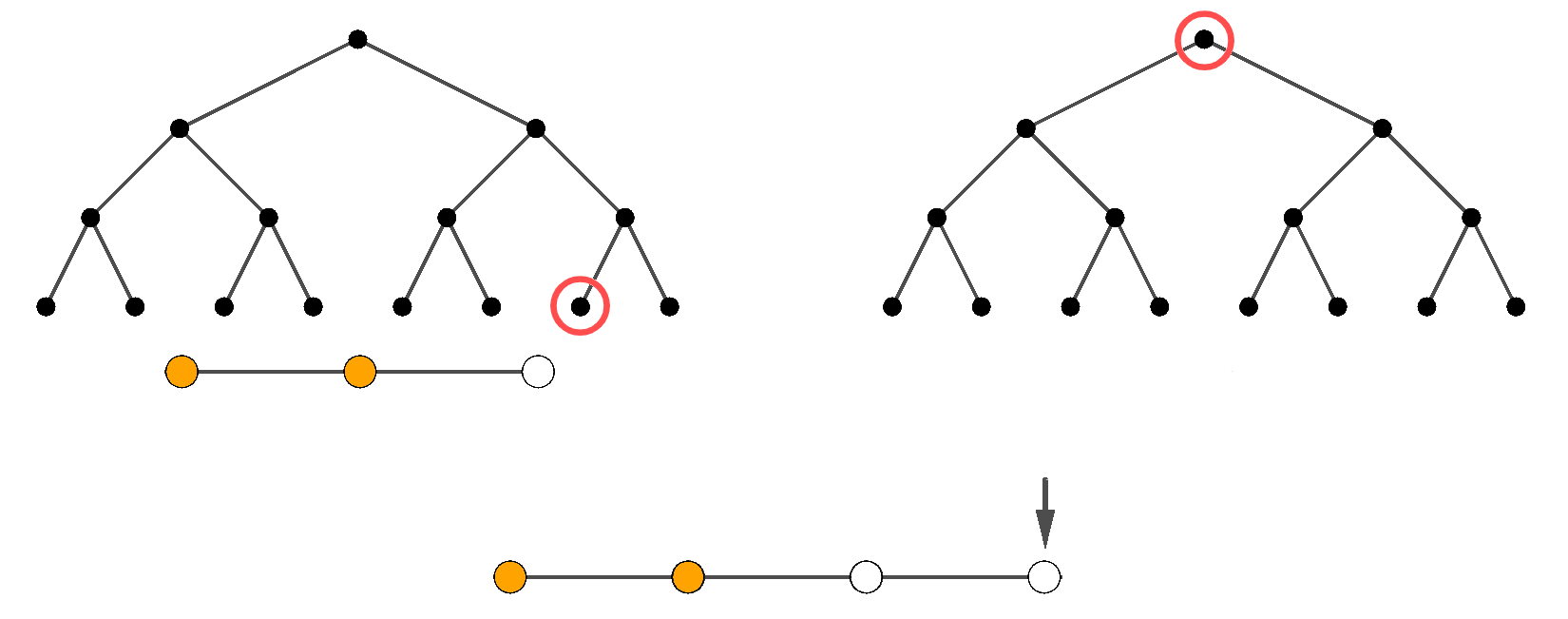} &
\includegraphics[width=0.22\linewidth]{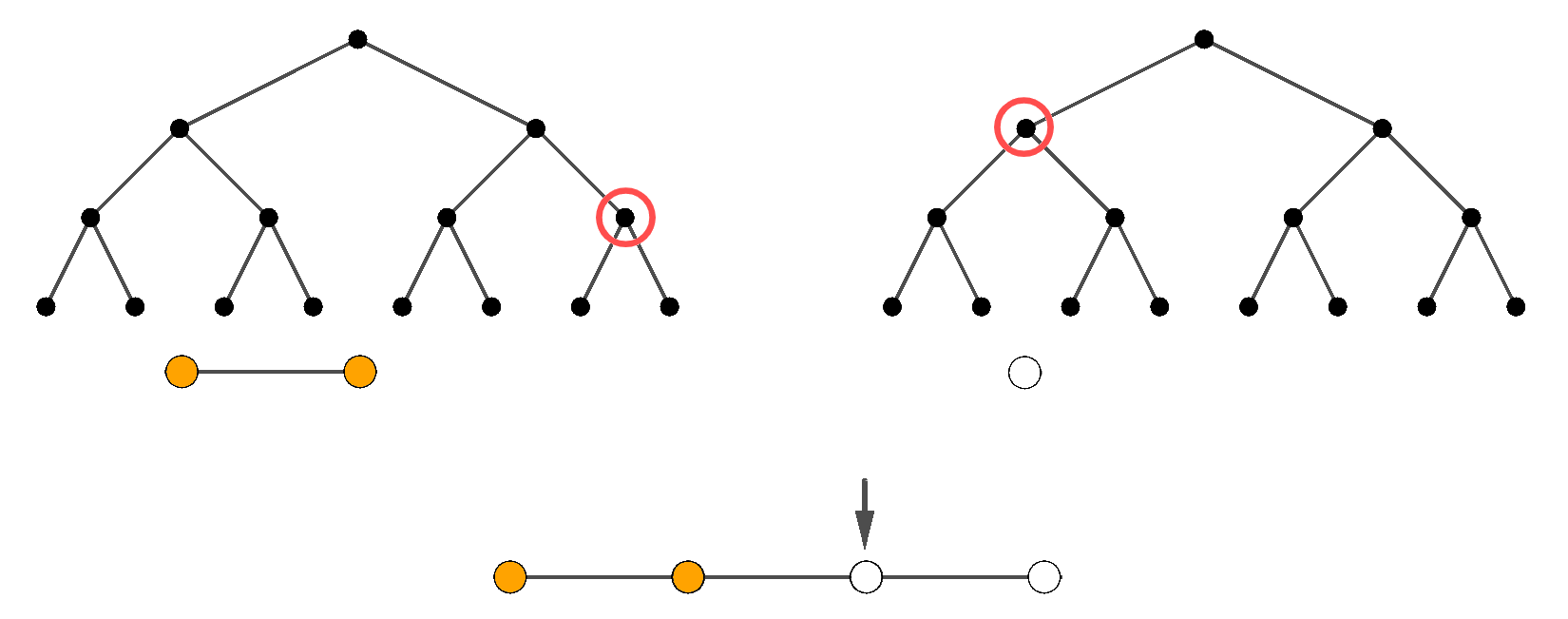} &
\includegraphics[width=0.22\linewidth]{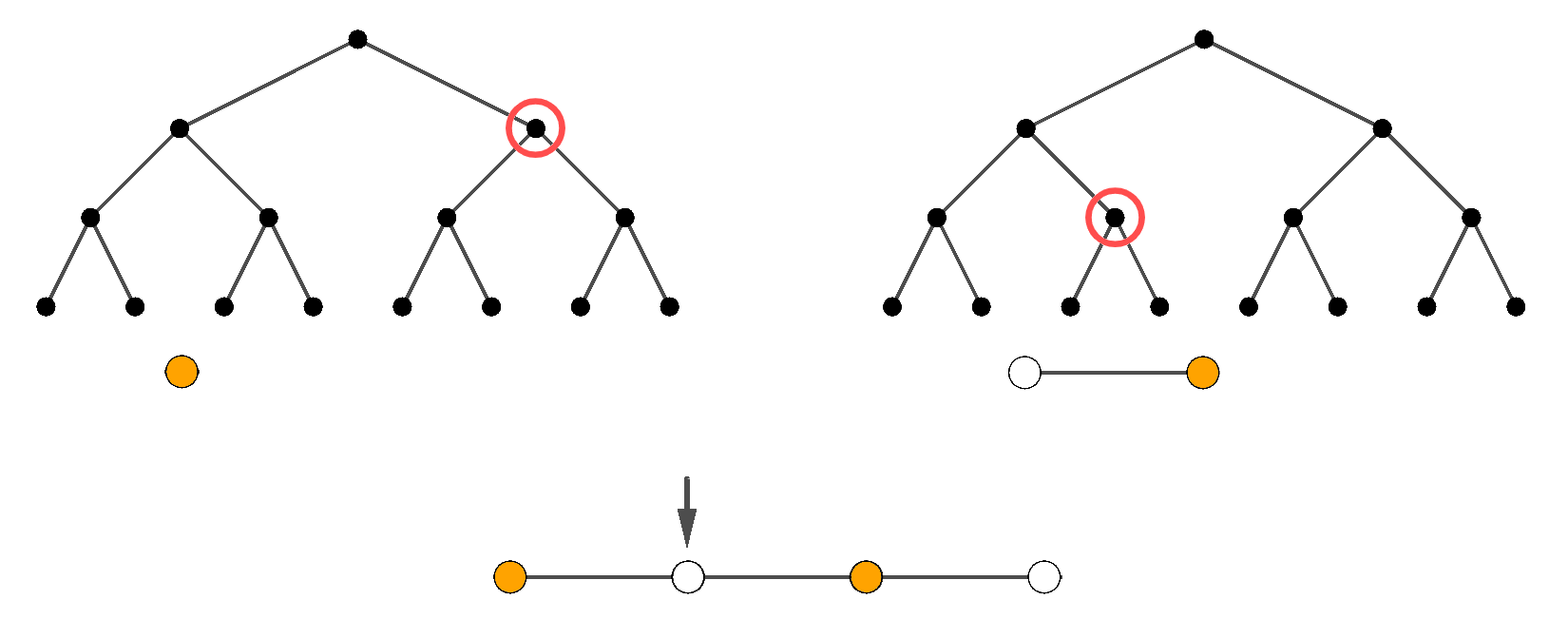} \\ \hline
\includegraphics[width=0.22\linewidth]{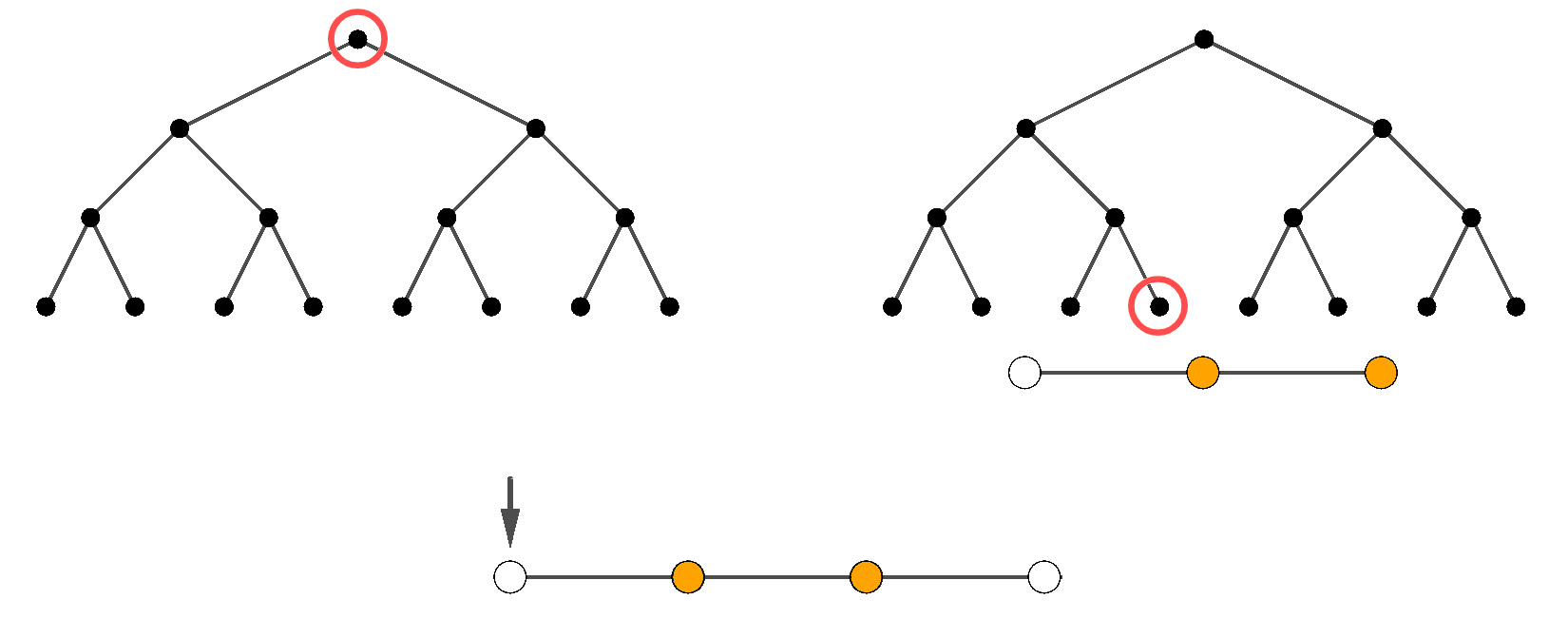} &
\includegraphics[width=0.22\linewidth]{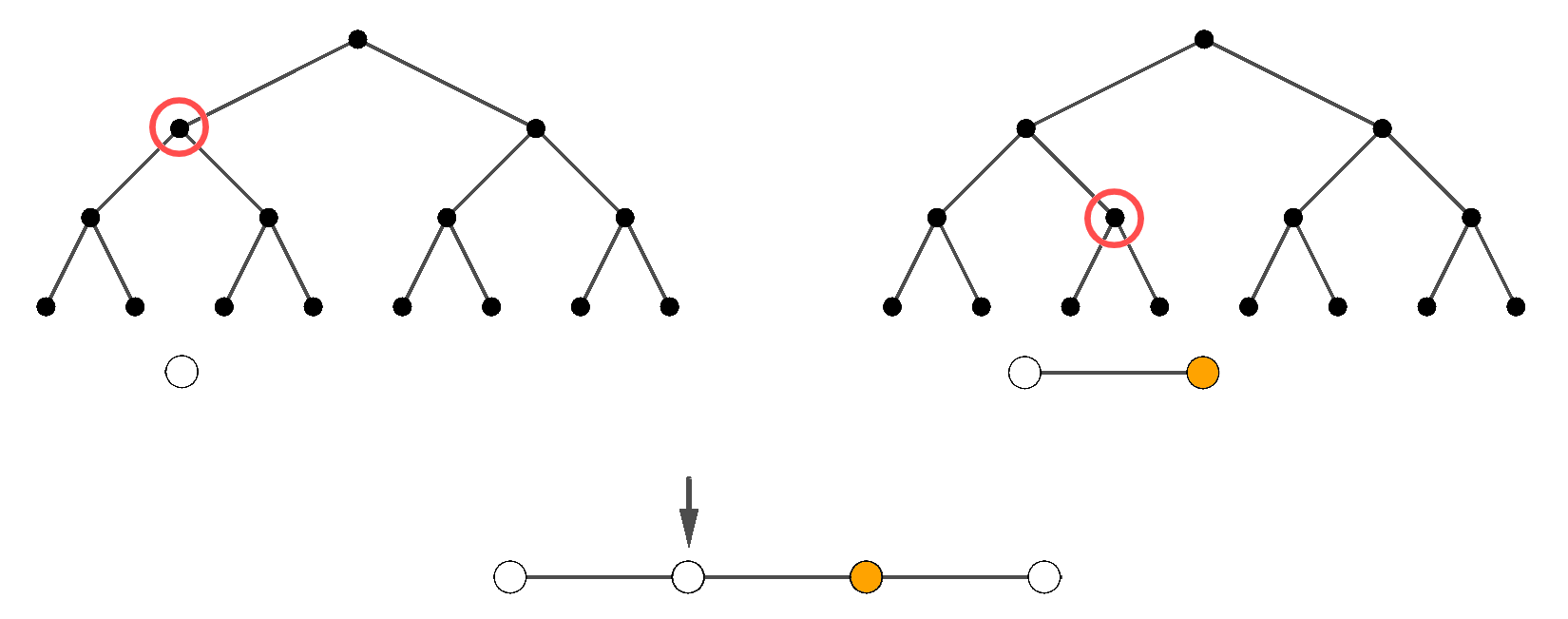} &
\includegraphics[width=0.22\linewidth]{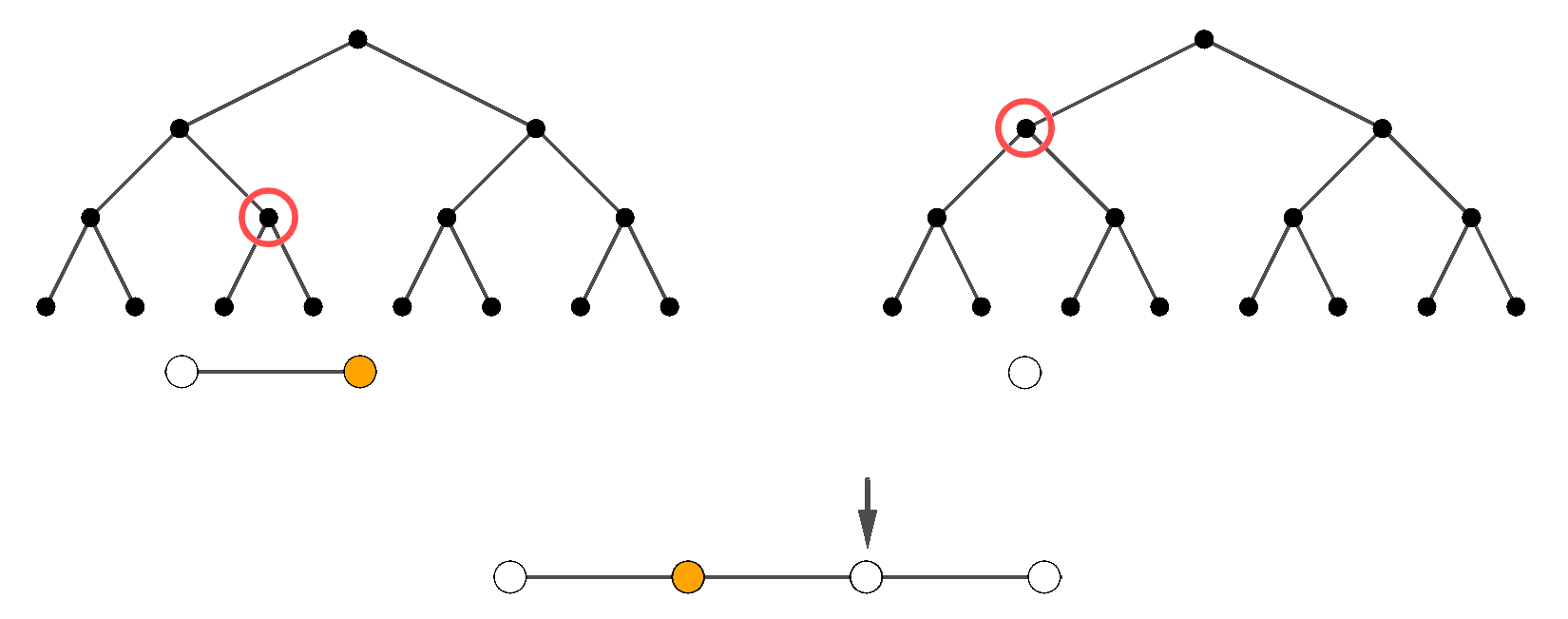} &
\includegraphics[width=0.22\linewidth]{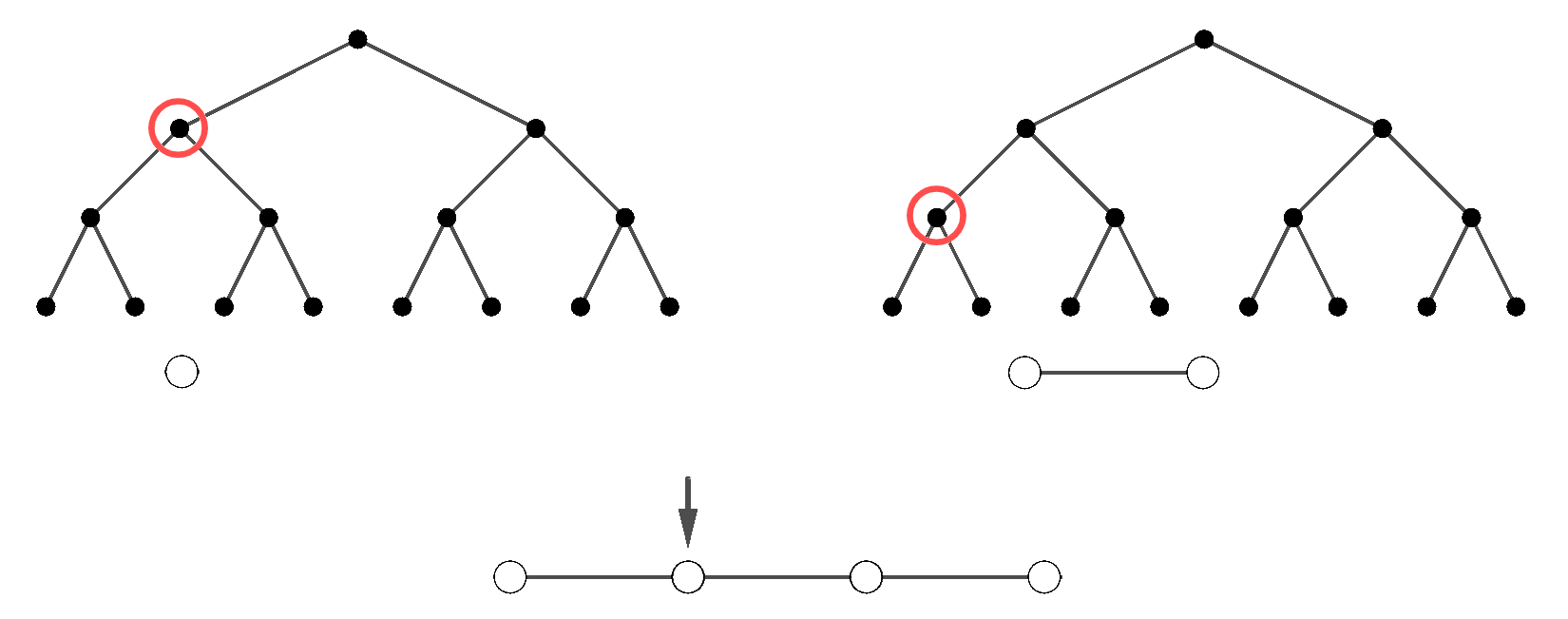} \\ \hline
\end{tabular}
\end{center}
\caption{The lamplighter as a horocyclic product.}
\label{Horo}
\end{figure}

\medskip \noindent
We can turn this set embedding into a graph embedding by choosing carefully a Cayley graph of $\mathbb{Z}_2 \wr \mathbb{Z}$. Let $a$ (resp.\ $t$) be a generator of the $\mathbb{Z}_2$-factor (resp.\ $\mathbb{Z}$-factor) of $\mathbb{Z}_2 \wr \mathbb{Z}$. Then, one can verify that $\Psi$ induces a graph embedding
$$\mathrm{Cayl}(\mathbb{Z}_2 \wr \mathbb{Z}, \{t, at \} ) \to T \boxtimes T,$$
where $\boxtimes$ denotes the strong product of graphs (i.e.\ the vertex set of $A \boxtimes B$ is $V(A) \times V(B)$ and two vertices $(a_1,b_1),(a_2,b_2)$ are connected by an edge if either $a_1=a_2$ and $b_1,b_2$ are adjacent, or $a_1,a_2$ are adjacent and $b_1=b_2$, or $a_1,a_2$ are adjacent and $b_1,b_2$ are adjacent). One can verify that this embedding is quasi-isometric, proving that $\mathbb{Z}_2 \wr \mathbb{Z}$ quasi-isometrically embeds into a product of two trees. Proposition~\ref{prop:CompressionLamp} follows as trees have Hilbert space compression $1$. 

\medskip \noindent
The Cayley graph $\mathrm{Cayl}(\mathbb{Z}_2 \wr \mathbb{Z}, \{t, at \})$, or equivalently its image in $T \boxtimes T$ under $\Psi$, is referred to as the \emph{Diestel-Leader graph} $\mathrm{DL}(n)$ or as the \emph{horocyclic product} $T \bowtie T$. It can be described as follows. First, notice that there is a natural map $\beta : V(T) \to \mathbb{Z}$ that sends, for every $n \in \mathbb{Z}$, every left-infinite word from $\mathbb{Z}_2^{((- \infty,n])}$ to $n$. (It is known as the \emph{Buseman function} of the point at infinity given by the limit of the $0$-words, and can be defined purely metrically.) Then, $\mathrm{DL}(n)$ is the subgraph of $T \boxtimes T$ induced by the set of vertices $\{ (u,v) \mid V(T) \times V(T) \mid \beta(u)+ \beta(v)=0 \}$. 

\medskip \noindent
This description of $\mathbb{Z}_2\wr \mathbb{Z}$ as a horocyclic product of two trees, which can be generalised to $\mathbb{Z}_n \wr \mathbb{Z}$ by replacing $T$ with en $(n+1)$-regular tree, is very specific to lamplighter groups over $\mathbb{Z}$ and is central in the proof of Theorem~\ref{thm:EFW}. 

\medskip \noindent
Of course, as a natural generalisation of our discussion, one could replace Hilbert spaces with other Banach spaces, such as $L^p$-spaces. We will only say a few words about $L^1$-spaces. First of all, it may be notice that:

\begin{thm}[\cite{MR2214573}]
A separable metric space coarsely embeds into a Hilbert space if and only if it coarsely embeds into an $L^1$-space.
\end{thm}

\noindent
Therefore, Theorem~\ref{thm:CoarselyWreath} shows that a wreath product of graphs that coarsely embed into $L^1$-spaces also coarsely embeds into an $L^1$-space. Actually, Theorem~\ref{thm:CoarselyWreath} is proved by constructing coarse embeddings into $L^1$-spaces, exploiting the strong connection that exists between $L^1$-spaces and median spaces.  

\begin{definition}
A metric space $(X,d)$ is \emph{median} if, for all $x_1,x_2,x_3 \in X$, there exists a unique point $m \in X$ satisfying
$$d(x_i,x_j)= d(x_i,m)+d(m,x_j) \text{ for all } i \neq j.$$
\end{definition}

\noindent
Basic examples of median spaces are $L^1$-spaces, but, conversely, every median space isometrically embeds into an $L^1$-space. Other typical examples of median spaces include $(\mathbb{R}^n, \| \cdot\|_1)$ and (simplicial or real) trees, as well as $\ell^1$-products of such spaces. Wreath products do not preserve median geometry. For instance, wreath products of trees are not median. In fact, median graphs are always coarsely simply connected, so we know from Theorem~\ref{thm:NotCoarselySimplyConnected} that most wreath products of graphs are not median. However, \cite{MR4449680} constructs \emph{diadem products} of median spaces, and shows that, given two graphs $X,Y$ endowed with coarse embeddings $X \hookrightarrow M$, $Y \hookrightarrow N$ into median spaces, a wreath product $(X,o) \wr Y$ coarsely embeds into a diadem product $(M,m) \circledast N$, which is median. 

\medskip \noindent
We refer to \cite{MR4449680} for a general definition of diadem products of median spaces, but we can at least give the definition of diadem products of median graphs:

\begin{definition}
Let $X,Y$ be two median graphs and $o \in V(X)$ a basepoint. A \emph{polytope} in $Y$ is the convex hull of finitely many vertices. Let $\mathscr{P}(Y)$ denote the poset of non-empty polytopes of $Y$ ordered by inclusion. The diadem product $(X,o) \circledast Y$ is the graph
\begin{itemize}
	\item whose vertices are the pairs $(c,P)$ where $P \in \mathscr{P}(Y)$ is a polytope and where $c : V(Y) \to V(X)$ is a colouring such that $c(y)=o$ for all but finitely many vertices $y \in V(Y)$;
	\item whose edges connect $(c_1,P_1)$ and $(c_2,P_2)$ if either
	\begin{itemize}
		\item $c_1,c_2$ differ only in $P_1=P_2$, where they take adjacent values at each vertex;
		\item $c_1=c_2$ and $P_2$ \emph{covers} $P_1$, i.e.\ $P_1 \subsetneq P_2$ but every polytope $P$ satisfying $P_1 \subset P \subset P_2$ must coincide with $P_1$ or $P_2$.
	\end{itemize}
\end{itemize}
\end{definition}

\noindent
\begin{minipage}{0.53\linewidth}
\begin{tabular}{|c|c|c|} \hline
\includegraphics[trim=0 22cm 43cm 0,clip,width=0.27\linewidth]{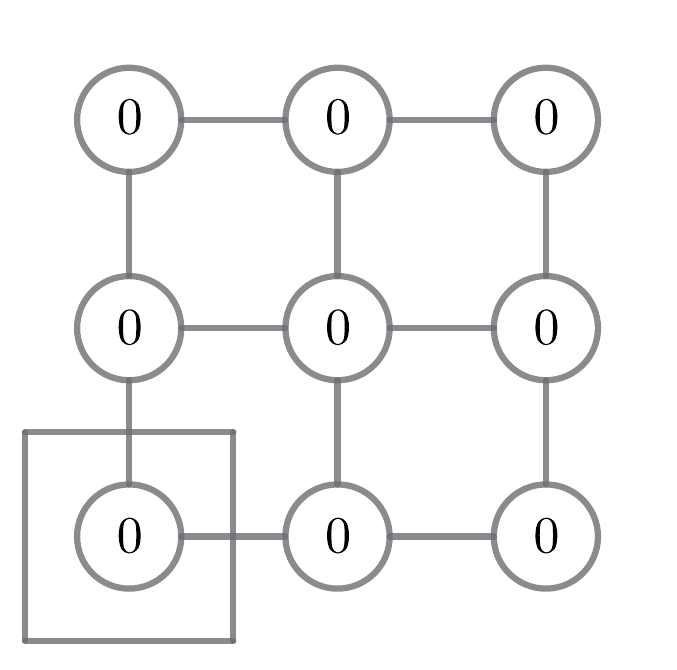} &
\includegraphics[trim=0 22cm 43cm 0,clip,width=0.27\linewidth]{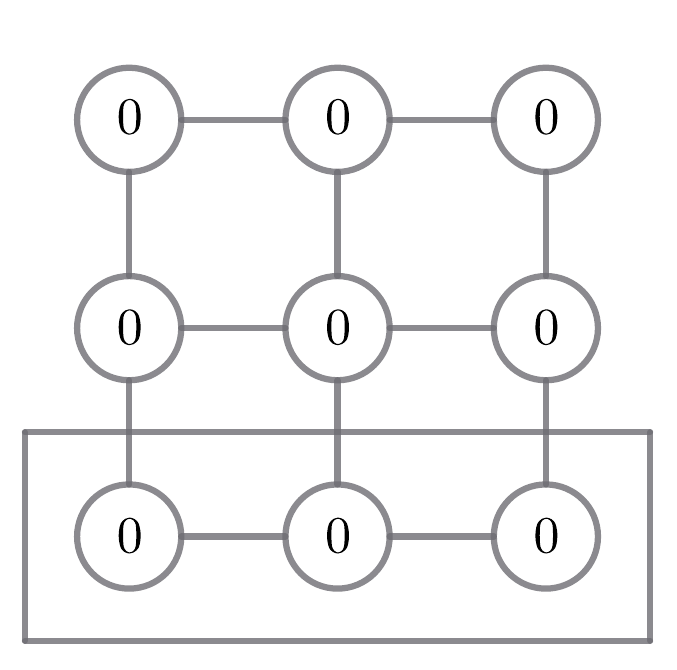} &
\includegraphics[trim=0 22cm 43cm 0,clip,width=0.27\linewidth]{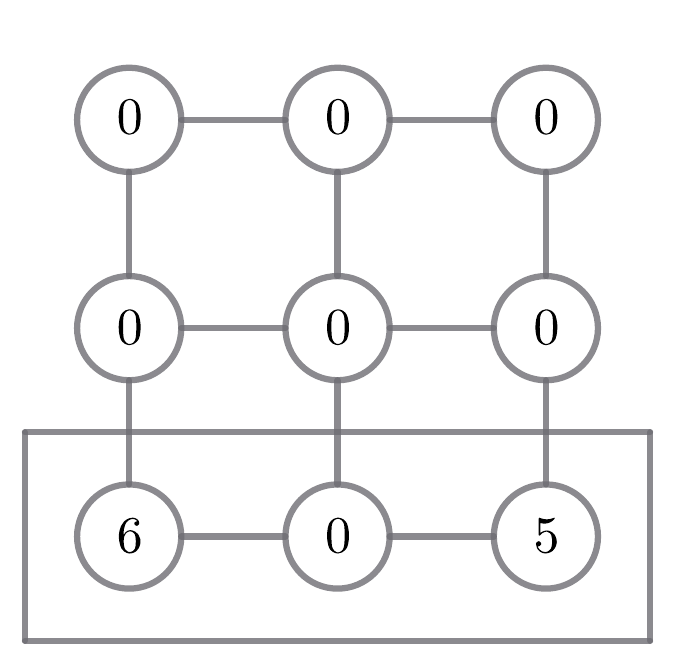} \\ \hline
\includegraphics[trim=0 22cm 43cm 0,clip,width=0.27\linewidth]{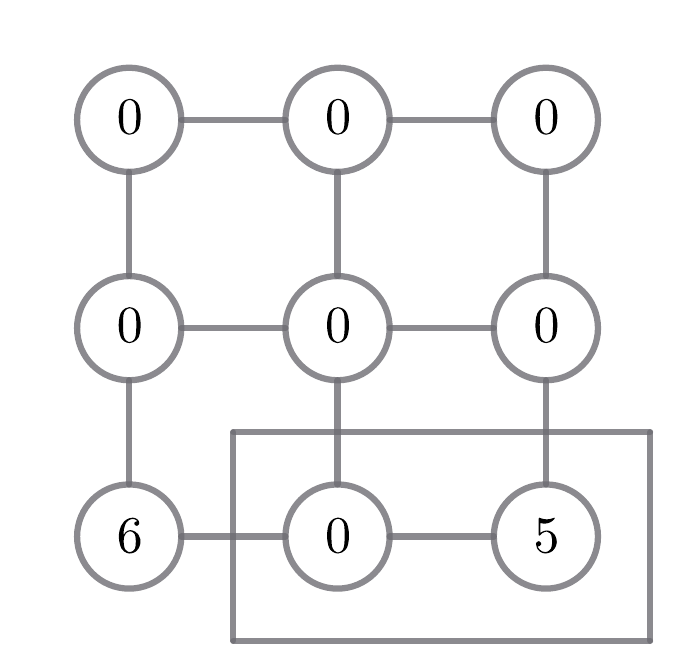} &
\includegraphics[trim=0 22cm 43cm 0,clip,width=0.27\linewidth]{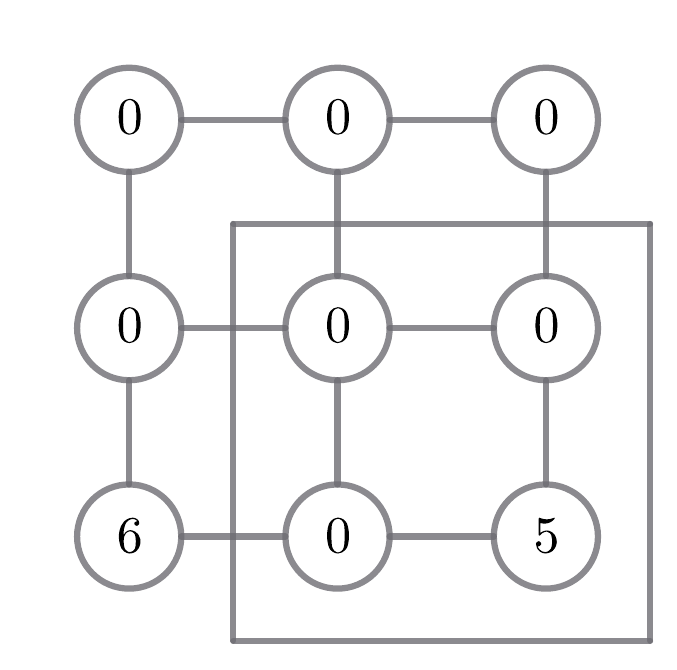} &
\includegraphics[trim=0 22cm 43cm 0,clip,width=0.27\linewidth]{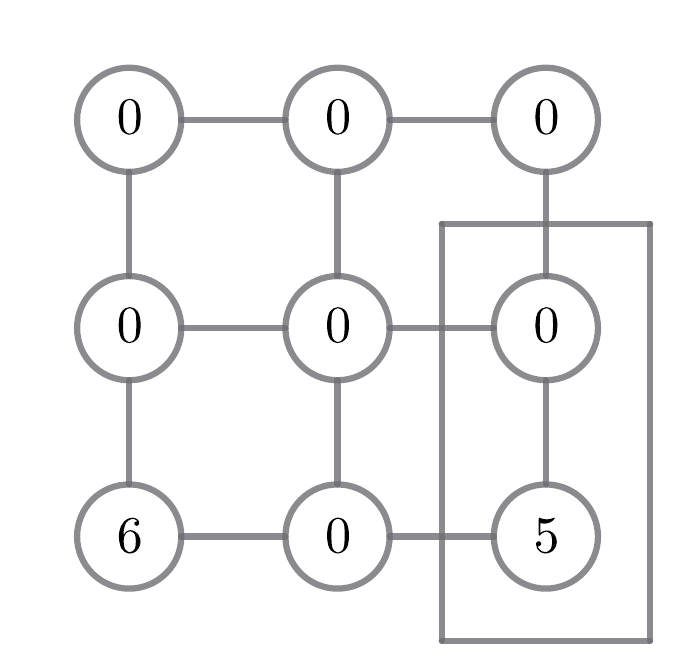} \\ \hline
\includegraphics[trim=0 22cm 43cm 0,clip,width=0.27\linewidth]{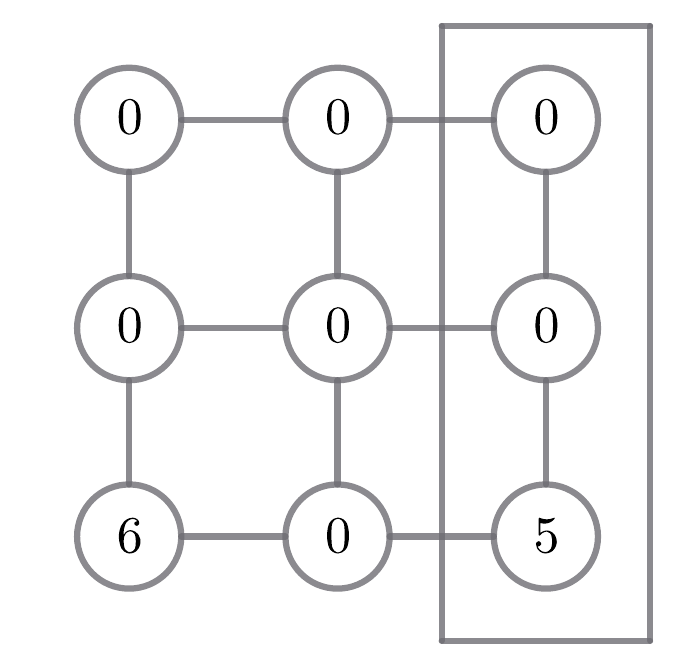} &
\includegraphics[trim=0 22cm 43cm 0,clip,width=0.27\linewidth]{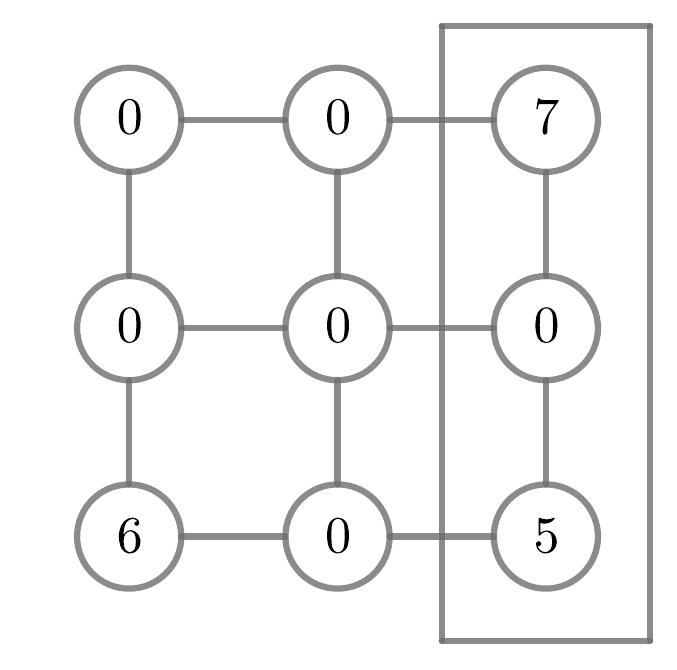} &
\includegraphics[trim=0 22cm 43cm 0,clip,width=0.27\linewidth]{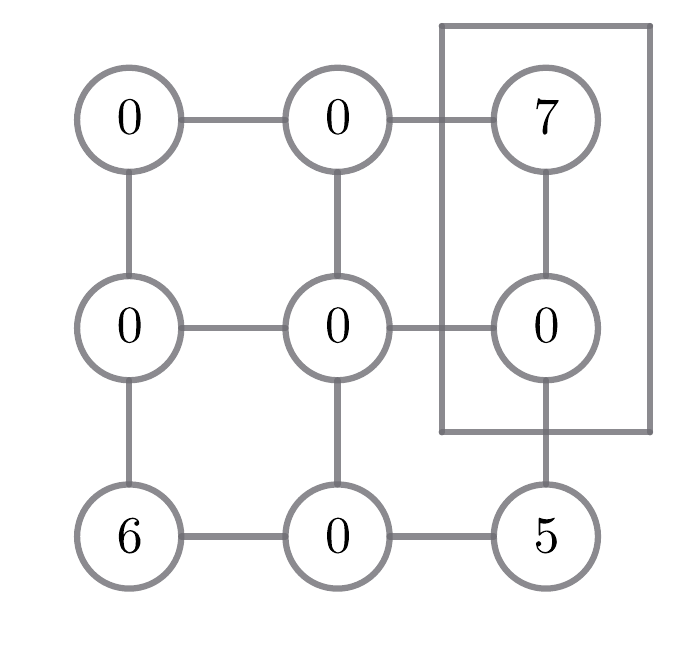} \\ \hline
\end{tabular}
\end{minipage}
\begin{minipage}{0.45\linewidth}
Loosely speaking, we start from the definition of wreath products of graphs and we replace the arrow, which points at a single vertex, with a polytope, which may encompass arbitrarily many vertices. Then, the polytope can modify the colour of every lamp it contains. 
\end{minipage}

\subsection{Bonus: coarse dimensions}\label{section:CoarseDim}

\noindent
In topology, the notion of dimension is delicate. Several definitions coexist and capture different features of topological spaces. From the perspective of coarse topology, topological dimension (also known as Lebesgue covering dimension) is worth mentioning as it naturally leads to a notion of coarse dimension. First, let us recall the definition of topological dimension:

\begin{definition}
The \emph{topological dimension} of a topological space is the smallest integer $d \geq 0$ for which every finite open cover has an open refinement of order $d+1$.
\end{definition}

\noindent
Recall that the \emph{order} of an open cover $\mathcal{O}$ is the smallest integer $n$ for which every point belongs to $\leq n$ open subsets of $\mathcal{O}$. Also, a \emph{refinement} of $\mathcal{O}$ is an open cover $\mathcal{O}'$ such that every $O' \in \mathcal{O}'$ is contained in some $O \in \mathcal{O}$. 

\medskip \noindent
Inspired by this definition of topological dimension:

\begin{definition}
A metric space $X$ has \emph{asymptotic dimension} $\leq d$ if, for every $R \geq 1$, there exists a uniformly bounded cover $\mathcal{O}$ of $X$ such that every closed $R$-balls in $X$ intersects at most $d+1$ subsets from $\mathcal{O}$.  A metric space has \emph{asymptotic dimension} $d$ if it has asymptotic dimension $\leq d$ but not asymptotic dimension $\leq d-1$. 
\end{definition}

\noindent
One can verify that the asymptotic dimension is preserved by quasi-isometries. (See Exercise~\ref{exo:AsDimCoarse}.) Clearly, bounded metric spaces have asymptotic dimension $0$; and, by covering the bi-infinite line with intervals of a given arbitrarily large width, one easily sees that $\mathbb{Z}$ has asymptotic dimension $1$. 

\medskip \noindent
\begin{minipage}{0.46\linewidth}
\includegraphics[width=\linewidth]{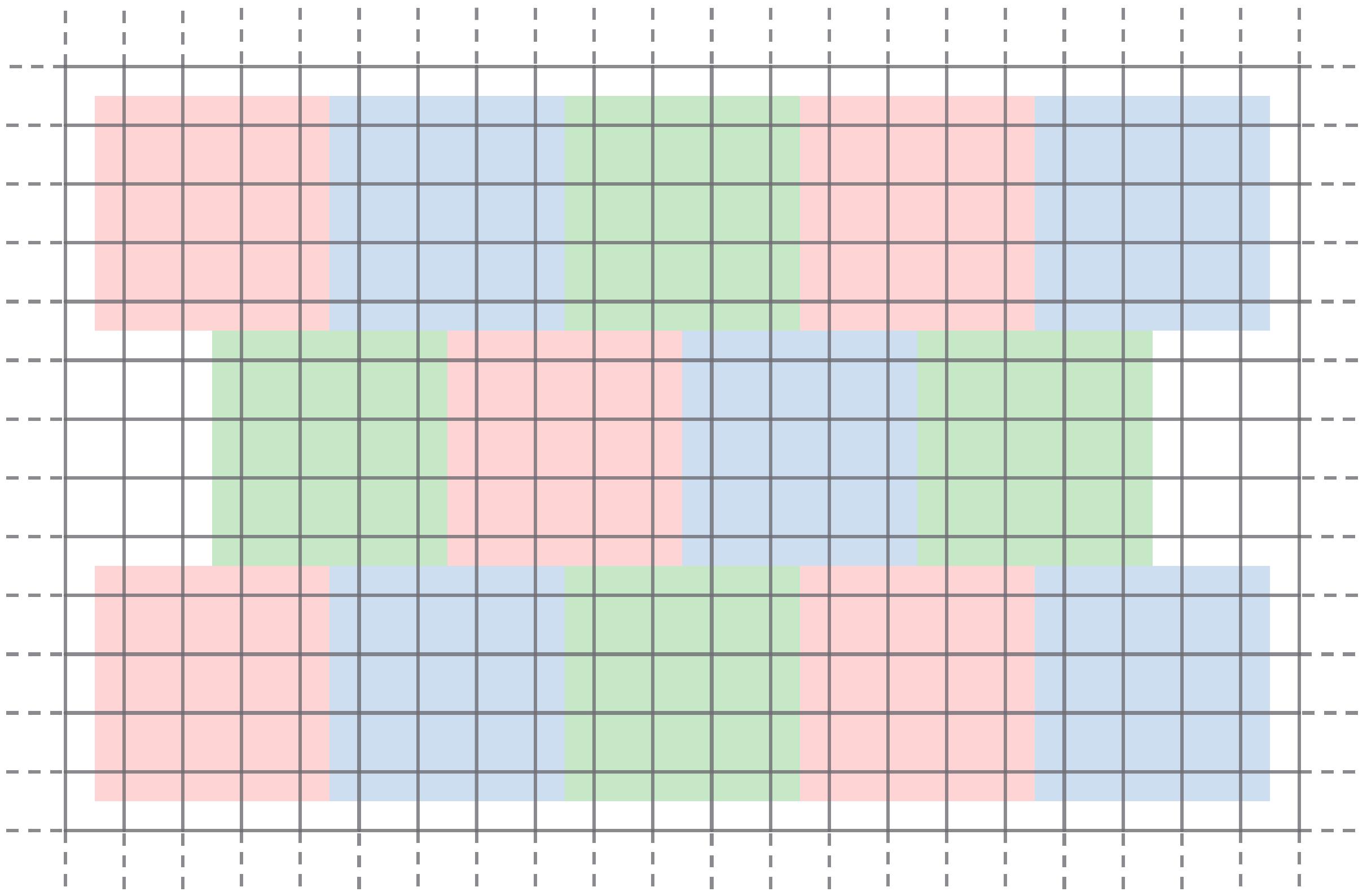}
\end{minipage}
\begin{minipage}{0.53\linewidth}
A more interesting example is $\mathbb{Z}^2$. The covering illustrated by the figure on the left shows that $\mathbb{Z}^2$ has asymptotic dimension $\leq 2$. One can show that $\mathbb{Z}^2$ has asymptotic dimension exactly $2$, but this requires a better understanding of asymptotic dimension. 
\end{minipage}

\medskip \noindent
We refer the reader to the survey \cite{MR2423966} and references therein for more information about asymptotic dimension. Here, we only record the following observation about wreath products:

\begin{prop}
Let $A$ and $B$ be two finitely generated groups, with $A$ non-trivial and $B$ infinite.
\begin{itemize}
	\item If $A$ is finite, then $\mathrm{asdim}(A \wr B) = \mathrm{asdim}(B)$.
	\item If $A$ is infinite, then $\mathrm{asdim}(A \wr B)= \infty$. 
\end{itemize}
\end{prop}

\begin{proof}
First, if $A$ is infinite, then there exists a quasi-isometric embedding $\mathbb{Z} \to A$. Since $B$ is infinite, $A \wr B$ contains a quasi-isometrically embedded copy of $A^d$, and a fortiori of $\mathbb{Z}^d$, for every $d \geq 1$. Since $\mathbb{Z}^d$ has asymptotic dimension $d$, necessarily $A \wr B$ has asymptotic dimension $\geq d$ (see Exercise~\ref{exo:AsDimCoarse}). Thus, $A \wr B$ has infinite asymptotic dimension. From now on, we assume that $B$ is infinite. Our wreath product $A \wr B$ satisfies a short exact sequence
$$1 \to \bigoplus\limits_B A \to A \wr B \to B \to 1,$$
hence $\mathrm{asdim}(A \wr B) \leq \mathrm{asdim}( \bigoplus_B A) + \mathrm{asdim}(B)$ according to \cite[Theorem~63]{MR2423966}. Here, $\bigoplus_BA$ is not finitely generated. However, as shown in Exercise~\ref{exo:AsDimCoarse}, every countable group can be endowed with a canonical metric up to coarse equivalence. Since asymptotic dimension is preserved by coarse equivalence, it makes sense to refer to the asymptotic dimension of a countable group. Since locally finite groups have asymptotic dimension $0$ (again, see Exercise~\ref{exo:AsDimCoarse}), it follows that $A \wr B$ has asymptotic dimension $\leq \mathrm{asdim}(B)$. The reverse equality follows from the fact that $B$ is a subgroup of $A \wr B$. 
\end{proof}

\noindent
Another interesting notion of dimension that is preserved by quasi-isometries is based on \emph{asymptotic cones}. Loosely speaking, given a metric space $(X,d)$, an asymptotic cone is the ``limit'' of the sequence of metric spaces $(X, d/n)$ as $n \to + \infty$, which would correspond to the picture one gets of our metric space when ``zooming out to infinity''. For instance, as illustrated by the figure below, one expects the asymptotic cone of the infinite grid $\mathbb{E}^2$ to be the Euclidean plane $\mathbb{R}^2$ (endowed with the $\ell^1$-metric). 

\begin{center}
\begin{tabular}{|c|c|c|c|} \hline
\includegraphics[trim=3cm 5cm 13cm 4cm,clip,width=0.22\linewidth]{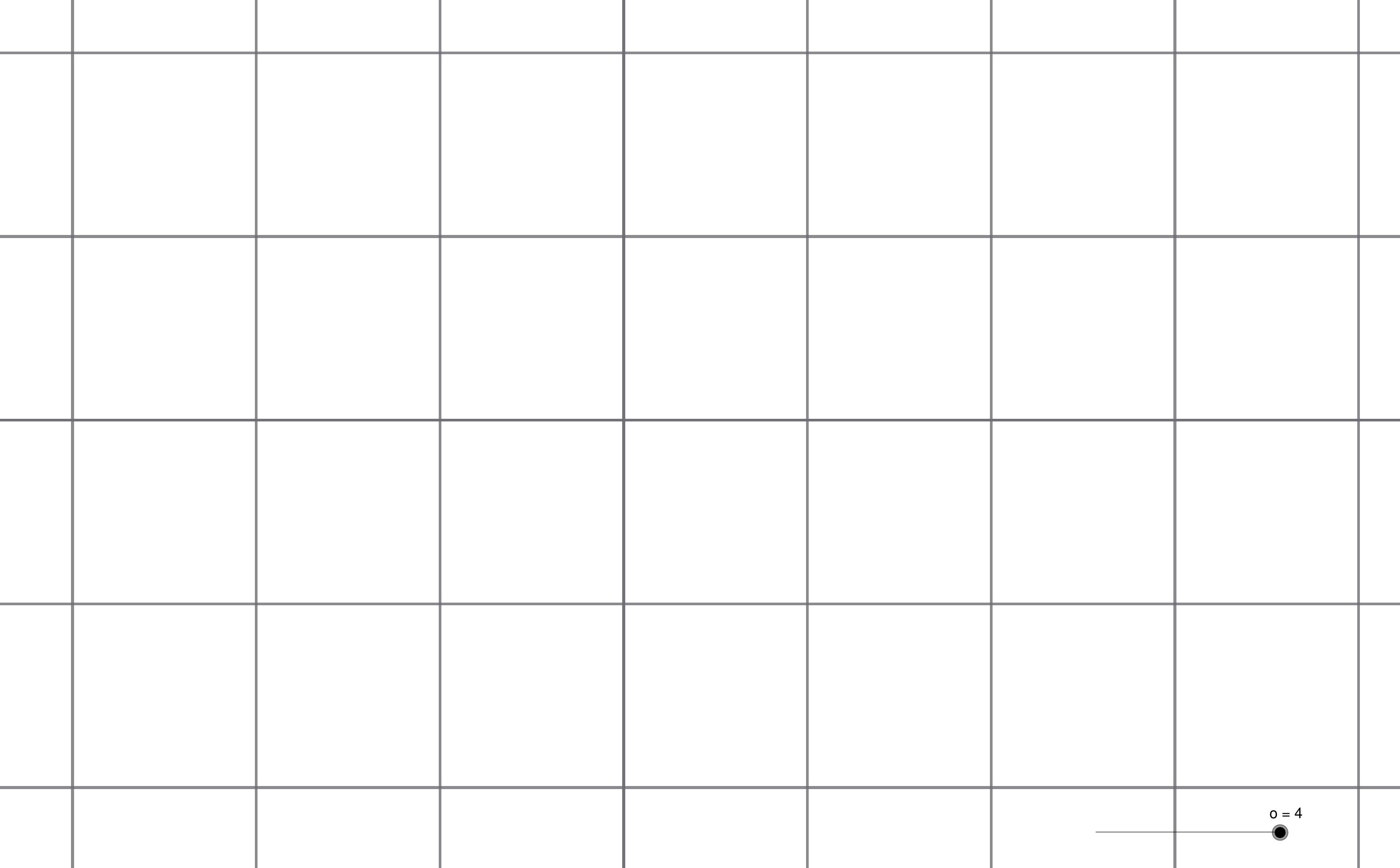} &
\includegraphics[trim=3cm 5cm 13cm 4cm,clip,width=0.22\linewidth]{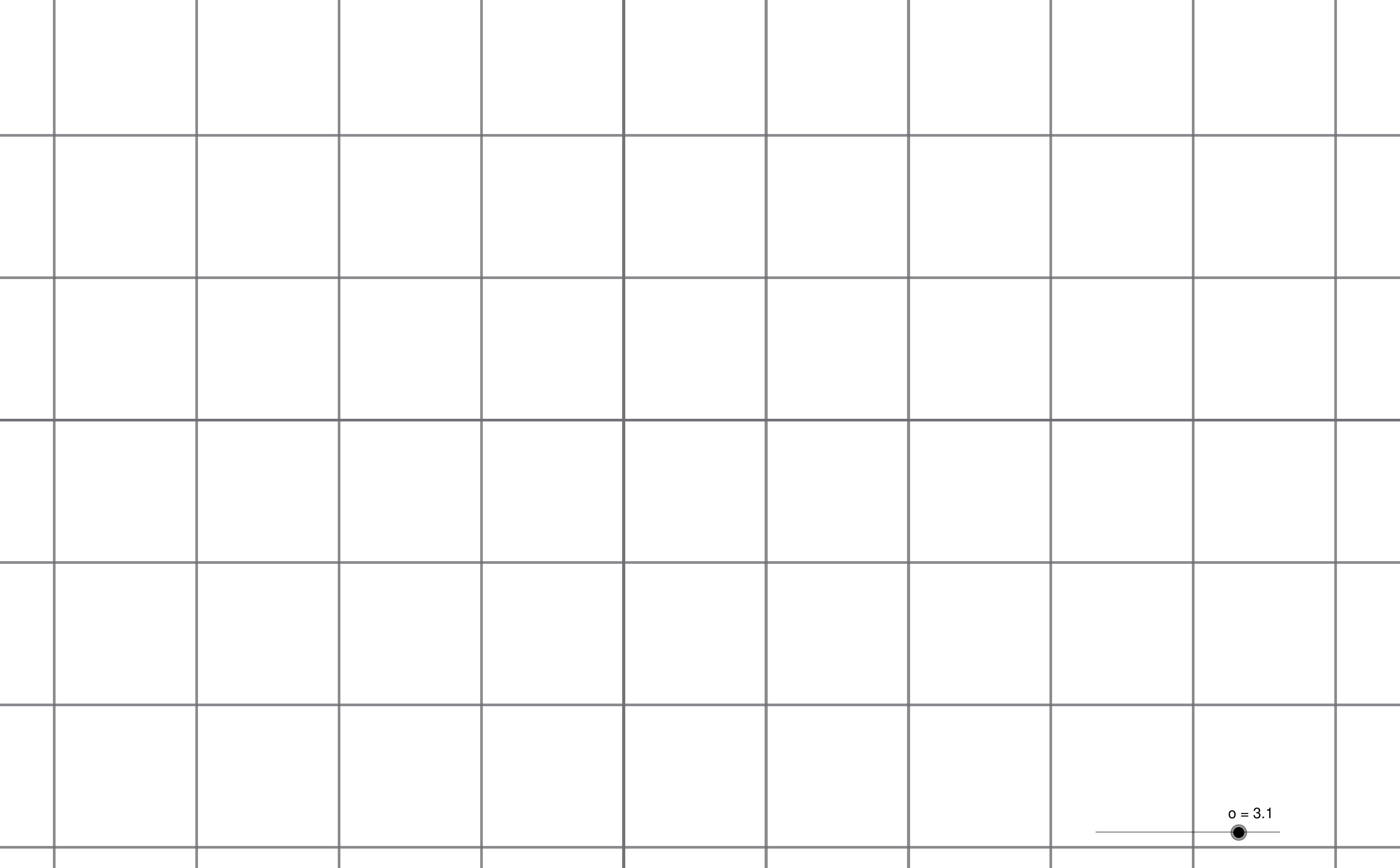} &
\includegraphics[trim=3cm 5cm 13cm 4cm,clip,width=0.22\linewidth]{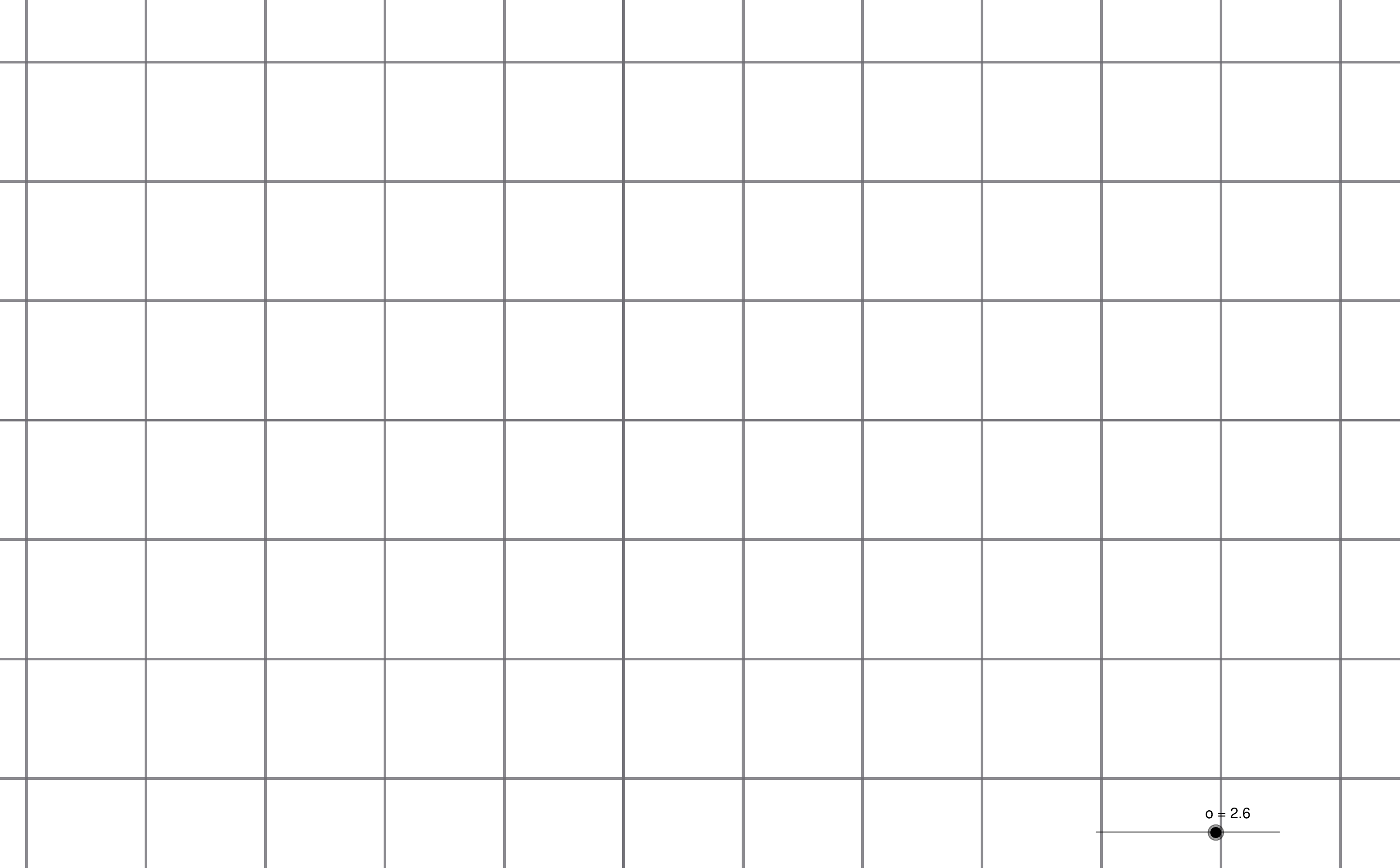} &
\includegraphics[trim=3cm 5cm 13cm 4cm,clip,width=0.22\linewidth]{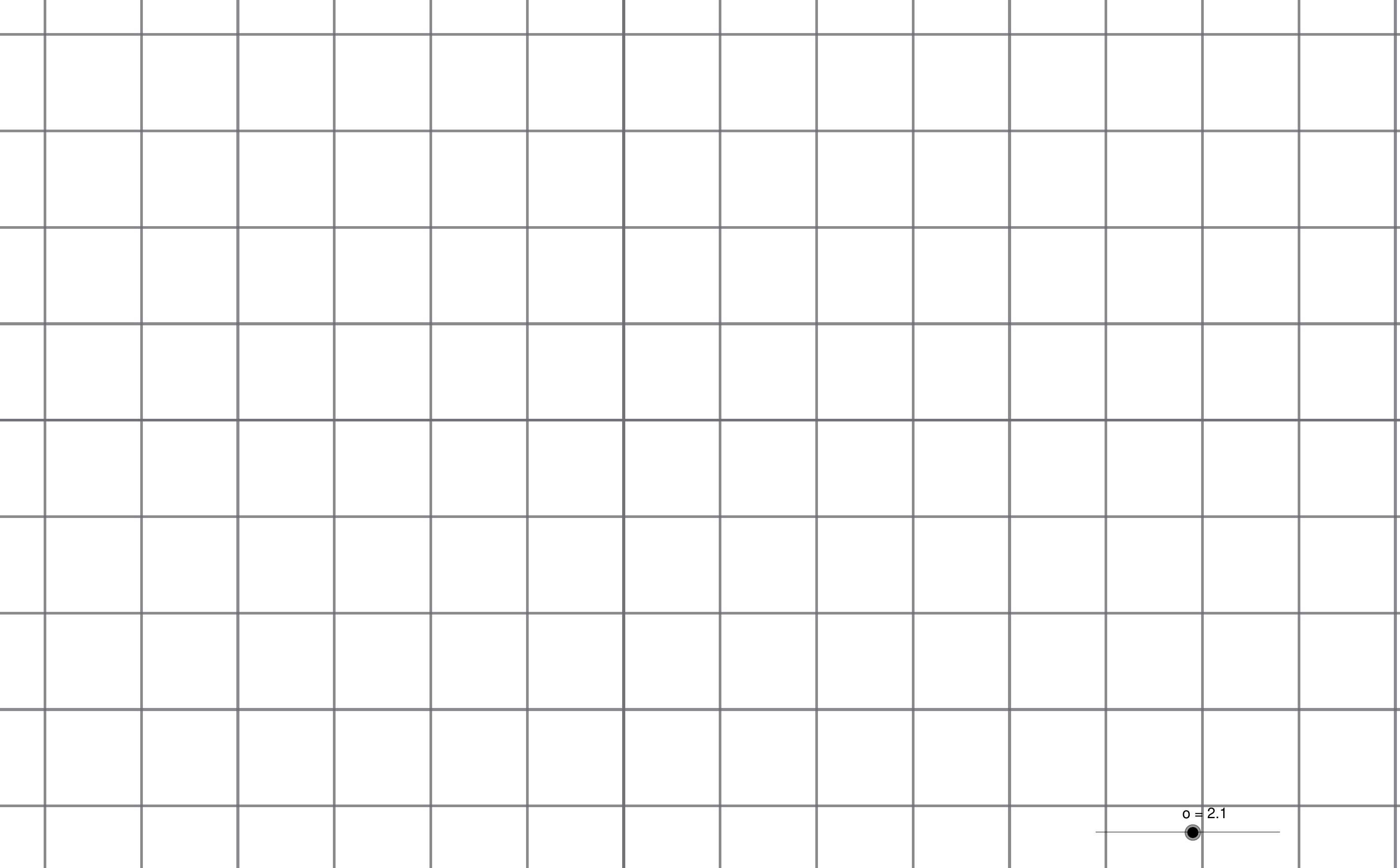} \\ \hline
\includegraphics[trim=3cm 5cm 13cm 4cm,clip,width=0.22\linewidth]{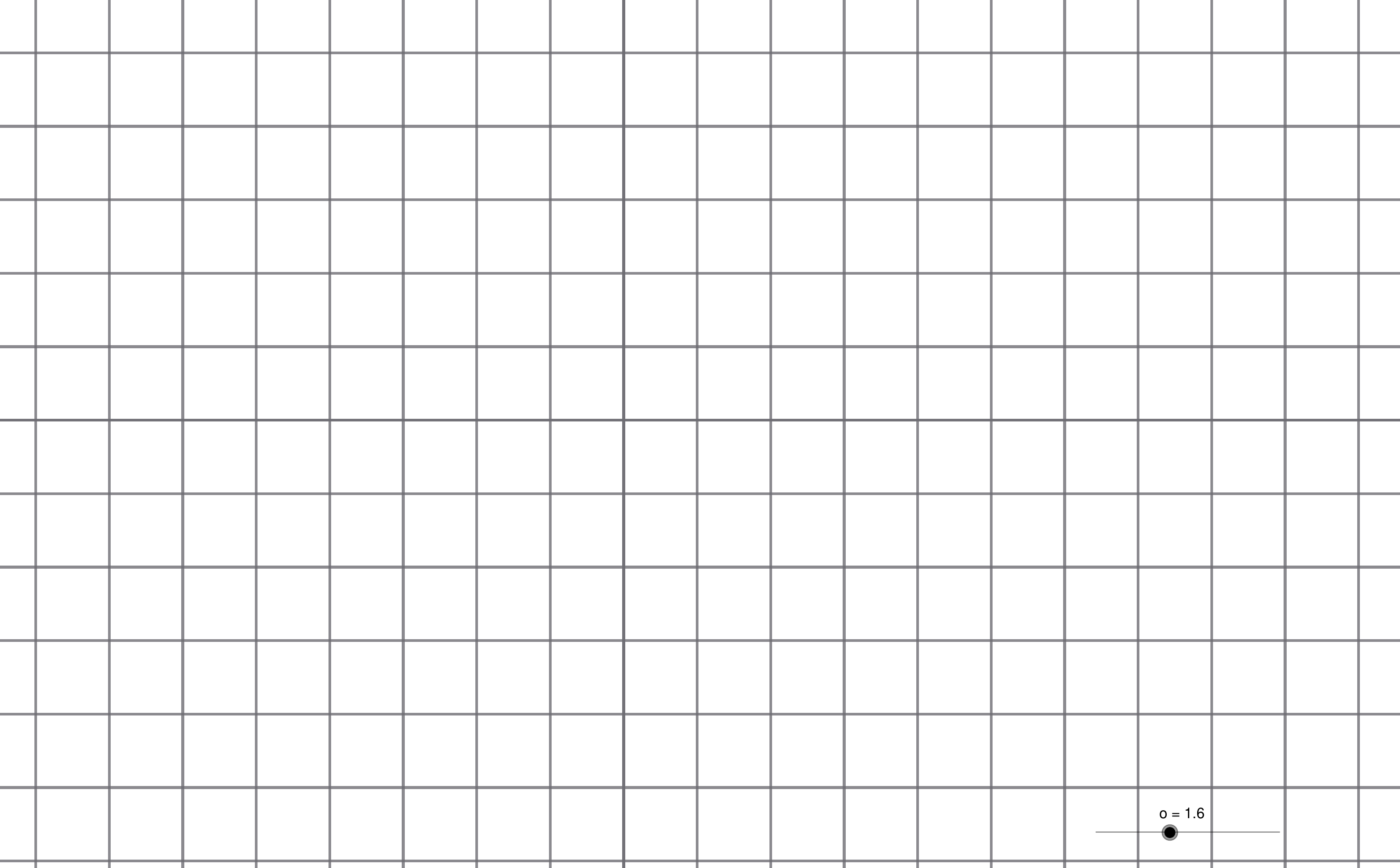} &
\includegraphics[trim=3cm 5cm 13cm 4cm,clip,width=0.22\linewidth]{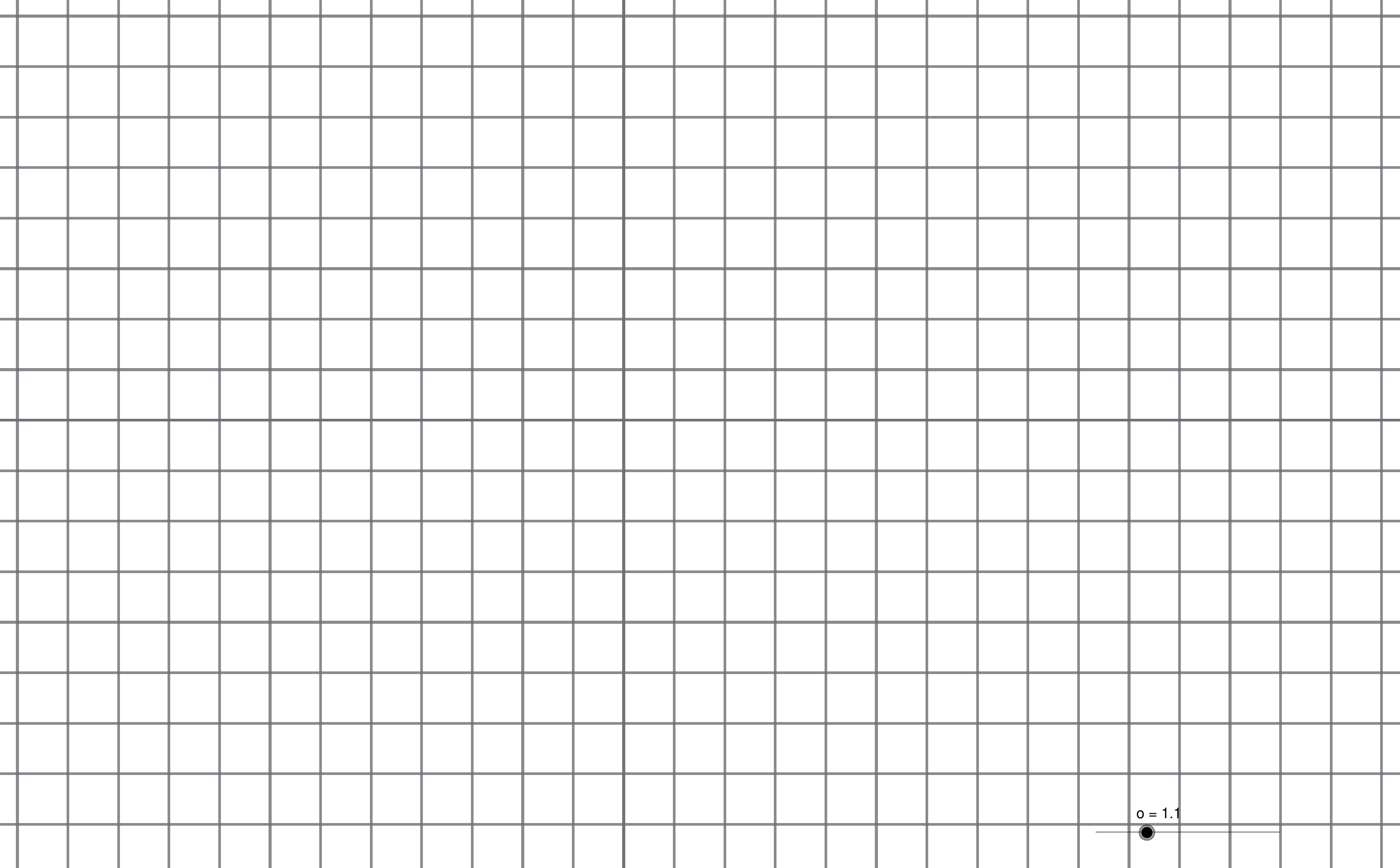} &
\includegraphics[trim=3cm 5cm 13cm 4cm,clip,width=0.22\linewidth]{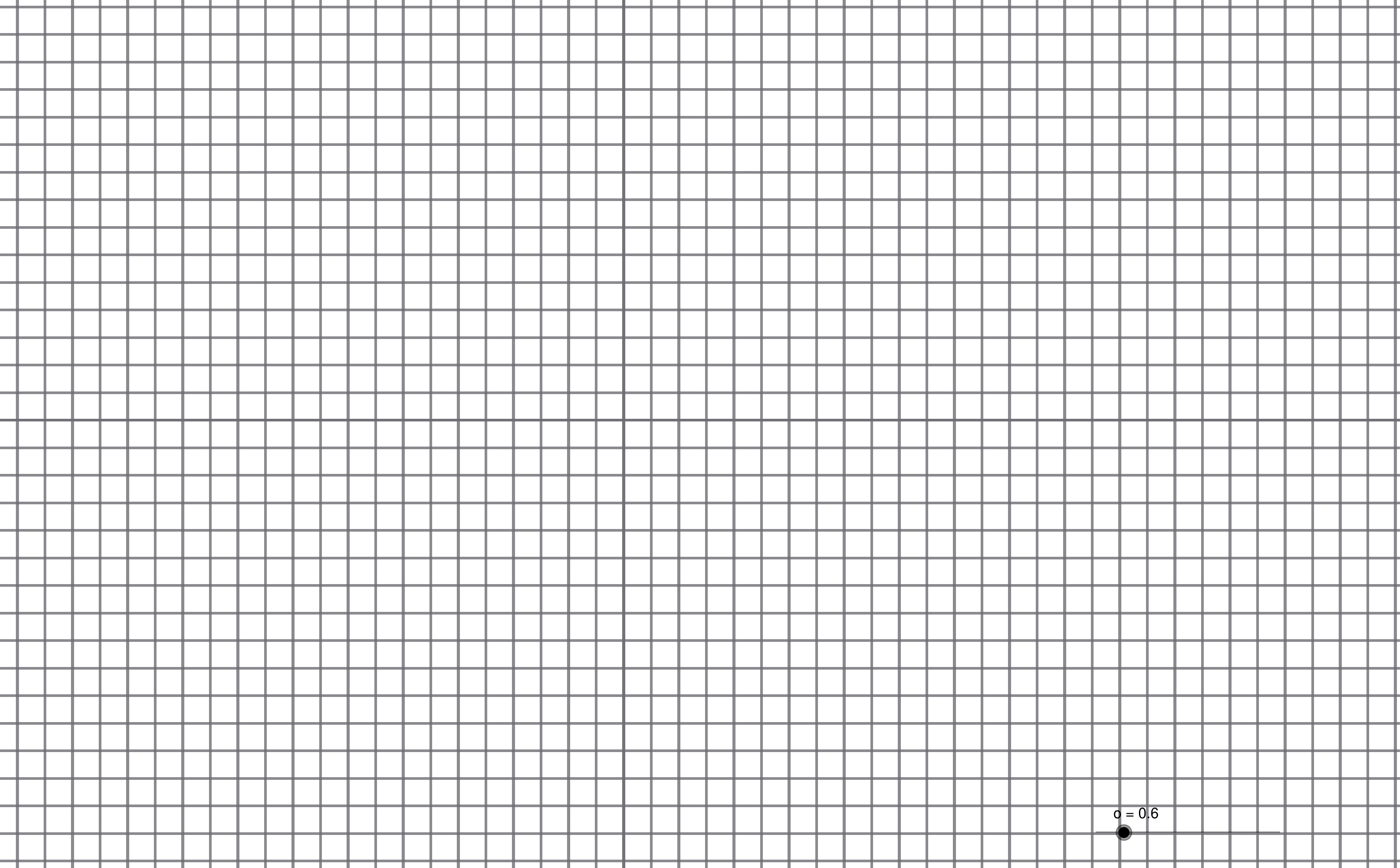} &
\includegraphics[trim=3cm 5cm 13cm 4cm,clip,width=0.22\linewidth]{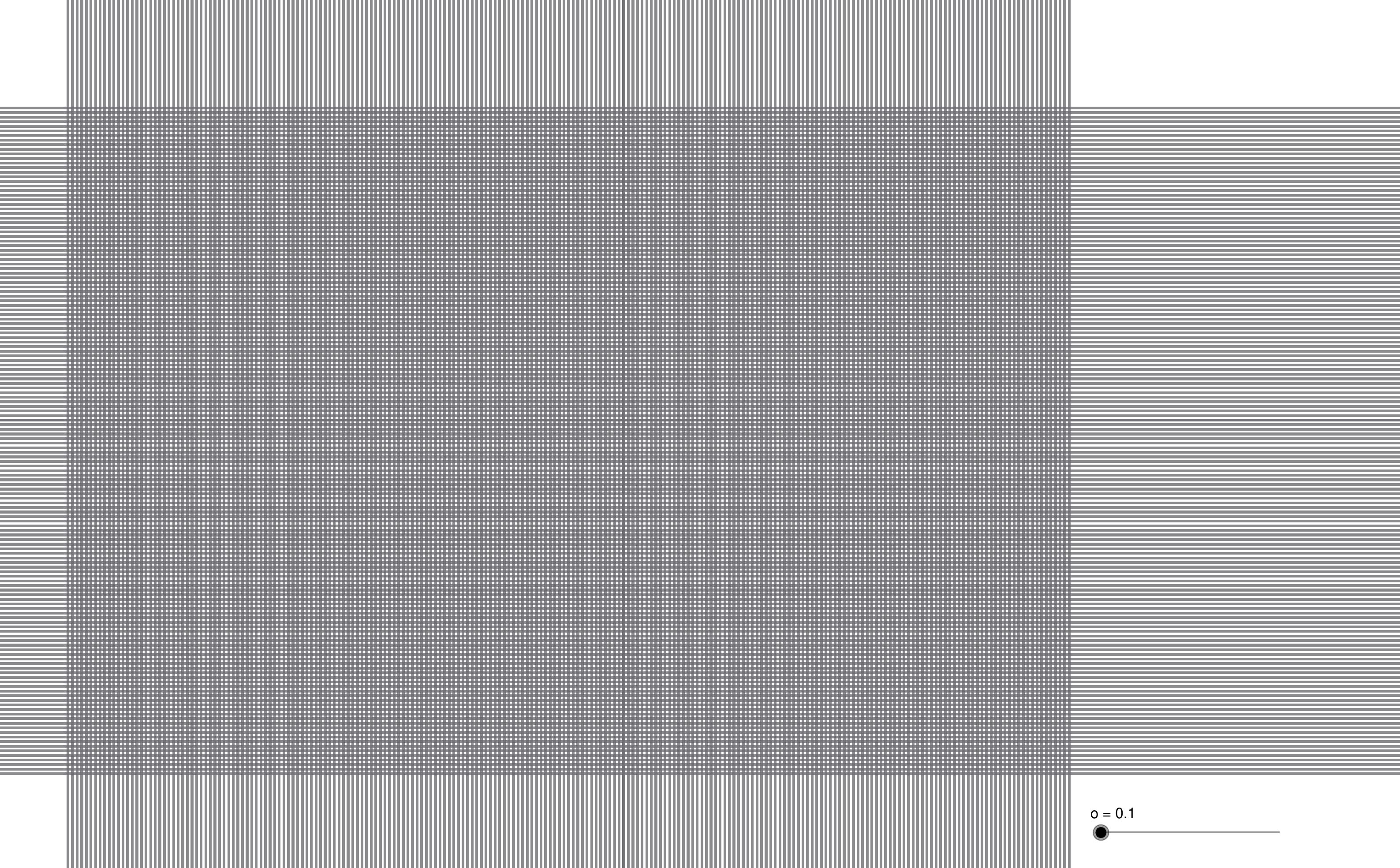} \\ \hline
\end{tabular}
\end{center}

\noindent
Unfortunately, formally the limit of the $(X,d/n)$ may not exist in a reasonable sense, and we need some formalism to define something meaningful. 

\medskip \noindent
An \emph{ultrafilter} over a set $S$ is a collection of subsets $\omega \subset \mathfrak{P}(S)$ satisfying the following conditions:
\begin{itemize}
	\item $\emptyset \notin \omega$ but $S \in \omega$;
	\item $A \cap B \in  \omega$ for all $A,B \in \omega$;
	\item for every $A \subset S$, either $A \in \omega$ or $S \backslash A \in \omega$.
\end{itemize}
The idea to keep in mind is that $\omega$ discriminates the subsets of $S$ as ``small'' (if not in $\omega$) and ``large'' (if in $\omega$). Thus, the axioms above amount to saying that: $\emptyset$ is small and $S$ is large; the intersection of two large subsets is again large; the complement of a small (resp.\ large) subset is large (resp.\ small). For instance, given an element $s \in S$, $\omega_s:= \{ R \subset S \mid s \in R\}$ is an ultrafilter, said \emph{principal}. Non-principal ultrafilters exist, which can be proved thanks to Zorn's lemma, but (a  weak form of) the axiom of choice is needed. In the sequel, we will assume implicitly that our ultrafilters are not principal.

\medskip \noindent
Given an ultrafilter $\omega$ over $\mathbb{N}$, we can naturally modify various definitions related to real sequences. The two variations that will interest us are the following: 
\begin{itemize}
	\item a sequence $(r_n)_n \in \mathbb{R}^\mathbb{N}$ is \emph{$\omega$-bounded} if there exists some $B \geq 0$ such that $$\{n \in \mathbb{N} \mid |r_n| \leq B \} \in \omega;$$
	\item a sequence $(r_n)_n \in \mathbb{R}^\mathbb{N}$ \emph{$\omega$-converges} to $r \in \mathbb{R}$ if, for every $\epsilon>0$, 
$$\{ n \in \mathbb{N} \mid |r_n-r| \leq \epsilon \} \in \omega.$$
\end{itemize}
The classical argument proving that a bounded sequence admits a converging subsequence shows similarly that every $\omega$-bounded sequence $\omega$-converges. We leave the details as an exercise. 

\medskip \noindent
Now, given an arbitrary sequence of pointed metric spaces $(X_n,d_n,o_n)_{n \geq 0}$, one can define its \emph{$\omega$-ultralimit}. First, we consider the set of sequences
$$B:= \{ (x_n)_n \in \prod\limits_{n \geq 0} X_n \mid (d(o_n,x_n))_n \text{ is $\omega$-bounded}\}.$$
In other words, we focus our attention to points at linear distance from the basepoints. Notice that, for any two $(x_n)_n,(y_n)_n \in B$, the sequence $(d(x_n,y_n))_n$ is $\omega$-bounded since $d(x_n,y_n) \leq d(x_n,o_n)+d(o_n,y_n)$ for every $n \geq 0$. Consequently, we can define a pseudo-metric on $B$ as
$$d ( (x_n)_n, (y_n)_n) := \lim\limits_{\omega} d(x_n,y_n) \text{ for all } (x_n)_n,(y_n)_n \in B.$$
Next, we obtain our ultralimit by quotienting $B$ in order to make this pseudo-metric a genuine metric, i.e.\ we define the $\omega$-ultralimit of $(X_n,d_n,o_n)_{n \geq 0}$ as 
$$B / \sim \text{ where } (x_n)_n \sim (y_n)_n \text{ if } d((x_n)_n,(y_n)_n)=0.$$
As a particular case, given a metric space $X$ and a sequence of points $o=(o_n)_{n \geq 0}$, we define the \emph{asymptotic cone} $\mathrm{Cone}_\omega(X,o)$ as the $\omega$-ultralimit of $(X, d/n, o_n)_{n \geq 0}$. 

\medskip \noindent
The main advantage of asymptotic cones is that they may turn questions about quasi-isometries into genuine questions of topology, as motivated by our next statement:

\begin{prop}\label{prop:AsConeLip}
Let $\varphi : X \to Y$ be a quasi-isometric embedding between two graphs. Given an ultrafilter $\omega$ over $\mathbb{N}$ and a sequence of basepoints $o=(o_n)_n$ in $X$, 
$$\bar{\varphi} : \left\{ \begin{array}{ccc} \mathrm{Cone}_\omega(X,o) & \to & \mathrm{Cone}_\omega(Y, \varphi(o)) \\ \left[ (x_n)_n \right] & \mapsto & \left[ (\varphi(x_n))_n \right] \end{array} \right.$$
is a continuous Lipschitz map. Consequently, if $\varphi$ is a quasi-isometry, then $\bar{\varphi}$ is a biLipschitz homeomorphism.
\end{prop}

\noindent
The proof is rather straightforward, and is left to the interested reader as an exercise. 

\medskip \noindent
As an illustration of the topological perspective offered by asymptotic cones, one can show that $\mathbb{Z}^n$ and $\mathbb{Z}^m$ are quasi-isometric if and only if $n=m$. Indeed, as suggested by the animation above, asymptotic cones of $\mathbb{Z}^n$ are homeomorphic to $\mathbb{R}^n$. So, if $\mathbb{Z}^n$ and $\mathbb{Z}^m$ are quasi-isometric, then their asymptotic cones $\mathbb{R}^n$ and $\mathbb{R}^m$ must be homeomorphic, which happens only when $n=m$. (Notice, however, that there exist more elementary proofs; see Exercise~\ref{exo:AbelianQI}.) 

\medskip \noindent
\hspace{2cm}
\begin{minipage}{0.83\linewidth}
\emph{[Cayley graphs] may appear boring and uneventful to a geometer's eye since it is discrete and the traditional local (e.g.\ topological and infinitesimal) machinery does not run in $\Gamma$. To regain the geometric perspective one has to change one's position and move the observation point far away from $\Gamma$. Then the metric in $\Gamma$ seen from the distance $d$ becomes the original distance divided by $d$ and  for $d \to \infty$ the points in $\Gamma$ coalesce into a connected continuous solid unity which occupies the visual horizon without any gaps or holes and fills our geometer's heart with joy.\\}
\cite[page~1]{MR1253544}
\end{minipage}

\medskip \noindent
This is an optimistic point of view on asymptotic cones. And, indeed, one meets useful and elegant asymptotic cones in the literature: asymptotic cones of nilpotent groups are Carnot groups \cite{MR741395}, asymptotic cones of hyperbolic groups are real trees, and more generally asymptotic cones of relatively hyperbolic groups are tree-graded spaces \cite{MR2153979}, and asymptotic cones of symmetric spaces are Euclidean buildings \cite{MR1608566}. Unfortunately, typically asymptotic cones are very difficult to describe. And, when they are, they may have some exotic topology. For instance, asymptotic cones of the lamplighter group $\mathbb{Z}_2 \wr \mathbb{Z}$ contain Hawaiian earrings (see Exercise~\ref{exo:Hawaiian}). As a consequence, they are not locally simply connected and the machinery of algebraic topology does not apply very well. Moreover, asymptotic cones loose a lot of information. As previously mentioned, hyperbolic groups have real trees as asymptotic cones. In fact, this is always the same: any two non-elementary hyperbolic groups have the same asymptotic cones (up to biLipschitz homeomorphism). Finally, it is wort mentioning that, due to the necessity of choosing an ultrafilter, some considerations from set theory may be necessary. As an illustration of this phenomenon, it is shown in \cite{MR2132762} that there exists a finitely presented group that has a unique asymptotic cone up to homeomorphism if the continuum hypothesis holds, and $2^{2^{\aleph_0}}$ many otherwise. 

\medskip \noindent
Anyway, asymptotic cones allow us to define a notion of dimension that is preserved by quasi-isometries: two quasi-isometric graphs have asymptotic cones with the same topological dimension. Notice that, in full generality, there is no connection between asymptotic dimension and topological dimension of asymptotic cones. Indeed, hyperbolic groups have one-dimensional asymptotic cones but may have arbitrarily large asymptotic dimension; and, as shown below, some lamplighter groups have finite asymptotic dimension but their asymptotic cones are infinite-dimensional. 

\begin{prop}[\cite{Halo}]\label{prop:AsConeDimLamp}
Let $F$ be a non-trivial finite group and $H$ an infinite finitely generated group. The wreath product $F \wr H$ has finite-dimensional asymptotic cones if and only if $H$ is virtually $\mathbb{Z}$. If so, its asymptotic cones are one-dimensional.
\end{prop}

\begin{proof}[Sketch of proof.] 
If $H$ is virtually $\mathbb{Z}$, then we can show that $H$ is biLipschitz equivalent to $\mathbb{Z}$, so Proposition~\ref{prop:BiLipWreathGroups} implies that $F \wr H$ is quasi-isometric to $\mathbb{Z}_n \wr \mathbb{Z}$ where $n:=|F|$. Consequently, it suffices to describe the asymptotic cones of $\mathbb{Z}_n \wr \mathbb{Z}$. We saw in Section~\ref{section:CoarseHilbert} how to describe $\mathbb{Z}_n \wr \mathbb{Z}$ as the horocyclic product of two $(n+1)$-regular trees. This decomposition turns out to be compatible with asymptotic cones, i.e.\ the asymptotic cones of $\mathbb{Z}_n \wr \mathbb{Z}$ are horocyclic products of real trees (in which ever point has degree $2^{\alpha_0}$). The proof we gave of the correspondence between $\mathbb{Z}_n \wr \mathbb{Z}$ and the horocyclic product of two $(n+1)$-regular trees extends to the continuous setting and shows that the horocyclic product of two $2^{\alpha_0}$-regular real trees can be identified with a lamplighter over $\mathbb{R}$ for which each lamp has $2^{\aleph_0}$ colours. This interpretation can be used in order to deduce that our asymptotic cone has topological dimension $1$. 

\medskip \noindent
Now, assume that $H$ is not virtually $\mathbb{Z}$. We want to prove that the asymptotic cones of $F \wr H$ are infinite-dimensional. For simplicity, we assume that $F = \mathbb{Z}_2$ and $H= \mathbb{Z}^2$. 

\medskip \noindent
\begin{minipage}{0.55\linewidth}
\includegraphics[width=0.99\linewidth]{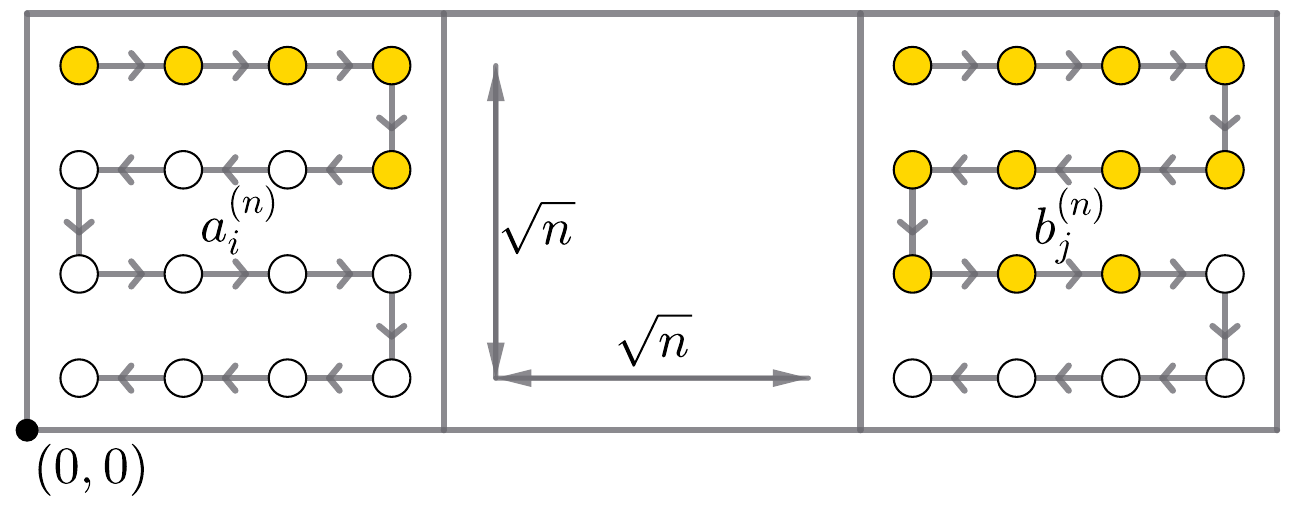}
\end{minipage}
\begin{minipage}{0.44\linewidth}
Given an $n \geq 2$, consider three adjacent squares of size $\sqrt{n}$ in $\mathbb{Z}^2$, with $(0,0)$ as the bottom left corner of the leftmost square. Order the vertices of the leftmost (resp.\ rightmost) square as shown by the figure on the left and define $a_i^{(n)}$ (resp.\ $b_i^{(n)}$) as the colouring that only colours the first $i$ lamps of this square. 
\end{minipage}
\medskip \noindent
Then, consider 
$$\varphi_n : \left\{ \begin{array}{ccc} [0,n]\times [0,n] & \to & \mathbb{Z}_2 \wr \mathbb{Z}^2 \\ (i,j) & \mapsto & (a_i^{(n)}+c_j^{(n)}, 0) \end{array} \right.$$
Our goal is to show that $\varphi_n$ converges as $n \to + \infty$ to a continuous embedding of $[0,1]^2$ into the asymptotic cones of $\mathbb{Z}_2 \wr \mathbb{Z}^2$. The key observation is that
$$2 \left( |i-r| + |j-s| \right) + \sqrt{n} \leq d(\varphi_n(i,j), \varphi_n(r,s)) \leq 2 \left( |i-r| + |j-s| \right) + 7 \sqrt{n}$$
for all $(i,j),(r,s) \in \mathbb{Z}^2$. Thus, $\varphi_n$ is not quite a quasi-isometric embedding, but almost: instead of a constant additive error, we have a sublinear additive error ($\sim \sqrt{n}$). But, when taking an asymptotic cone, we divide distance by $n$ and let $n$ grows to infinity. Therefore, the sublinear additive error disappears, and we find that $\varphi_n$ converges to a biLipschitz embedding $\varphi_\infty$ of $[0,1]^2$ in the asymptotic cones. This proves that the asymptotic cones of $\mathbb{Z}_2 \wr \mathbb{Z}^2$ are at least two-dimensional. But the same argument can be done with an arbitrary number of squares of size $\sqrt{n}$, producing biLispchitz embeddings of $[0,1]^d$ into the asymptotic cones for arbitrary dimensions $d$. 
\end{proof}

\noindent
It is worth mentioning that Proposition~\ref{prop:AsConeDimLamp} can be used in order to distinguish some groups up to quasi-isometry. For instance, a lamplighter over $\mathbb{Z}$ cannot be quasi-isometric to the lampshuffler $\mathrm{FSym}(\mathbb{Z}) \rtimes \mathbb{Z}$ because their asymptotic cones do not have the same dimension. See Exercise~\ref{exo:LighterShuffler}.

\subsection{Exercises}

\begin{exo}\label{exo:Rips}
Given a metric space $(X,d)$ and a scale $R \geq 0$, the \emph{Rips complex} $\mathrm{Rips}_R(X)$ is the simplicial complex whose vertex-set is $X$ and whose simplices are collections of points pairwise at distance $\leq R$.  Prove that a graph is coarsely simply connected if and only if it admits a simply connected Rips complex. 
\end{exo}

\begin{exo}\label{exo:WreathFP}
Let $A$ and $B$ be two groups.
\begin{enumerate}
	\item If $A \wr B$ is finitely generated, show that it admits a finite generating set of the form $\{a_1, \ldots, a_n, b_1, \ldots, b_m \}$ where $a_1, \ldots, a_n \in A^{(B)}$ and $b_1, \ldots, b _m \in B$.
	\item Fix an $A$-factor in $A^{(B)}$ and let $\pi : A^{(B)} \to A$ denote the corresponding projection. For every $a \in A^{(B)}$, verify that $\pi(bab^{-1}) = 1$ for all but finitely many $b \in B$.
	\item Prove that, if $A \wr B$ is finitely generated, then $A$ must be finitely generated as well.
	\item Deduce that $A \wr B$ is finitely generated if and only if $A$ and $B$ are both finitely generated.
	\item Prove that $A \wr B$ is finitely presented if and only if either $A$ is trivial and $B$ is finitely presented or $A$ is finitely presented and $B$ finite. 
\end{enumerate}
\end{exo}

\begin{exo}\label{exo:UnifOneEnded}
Let $X$ be a graph. Prove that, if $X$ is quasi-transitive, then $X$ is one-ended if and only if it is uniformly one-ended. 
\end{exo}

\begin{exo}
Prove Fact~\ref{fact:LocaSepTopo}, i.e.\ a subcomplex $Z$ in some CW-complex $X$ locally separates if and only if there exists a path $\gamma$ connecting two distinct points such that any path homotopy equivalent to $\gamma$ intersects $Z$. 
\end{exo}

\begin{exo}
Let $X$ and $Y$ be two graphs. Prove that $X$ is pancylindrical if and only if $Y$ is pancylindrical. 
\end{exo}

\begin{exo}\label{exo:ShufflerLocalOneEnded}
Given a group $H$, the \emph{lampshuffler over $H$} is the semidirect product
$$\mathrm{FSym}(H) \rtimes H$$
where $\mathrm{FSym}(H)$ denotes the group of finitely supported bijections $H \to H$. Following the proof of Corollary~\ref{cor:LampLocalOneEnded}, show that the lampshuffler over $\mathbb{Z}$ is not locally one-ended. 
\end{exo}

\begin{exo}\label{exo:LighterShuffler}
Inspired by the argument following Proposition~\ref{prop:AsConeDimLamp}, show that the asymptotic cones of the lampshuffler $\mathrm{FSym}(\mathbb{Z}) \rtimes \mathbb{Z}$ are infinite-dimensional. Conclude that lampshuffler and lamplighter over $\mathbb{Z}$ are not quasi-isometric. 
\end{exo}

\begin{exo}
Let $X$ and $Y$ be two graphs. If there exists a coarse embedding $X \hookrightarrow Y$, prove that $\mathrm{asdim}(X) \leq \mathrm{asdim}(Y)$. 
\end{exo}

\begin{exo}\label{exo:AsDimCoarse}
Let $G$ be a countable group. 
\begin{enumerate}
	\item Let $S \subset G$ be a generating set satisfying $S=S^{-1}$. Fix a \emph{weight function} $w : S \to \mathbb{N}$ such that $w(s)=w(s^{-1})$ for every $s \in S$ and such that $w^{-1}(n)$ is finite for every $n \in \mathbb{N}$. Define $$\|g\|_w := \min \left\{ \sum\limits_{i=1}^r w(s_i) \mid \exists s_1, \ldots s_r \in S, \ g=s_1 \cdots s_r \right\}$$ for every $g \in G$. Prove that $(g,h) \mapsto \|g^{-1}h\|_w$ defines a metric on $A$ that is left-invariant and proper (i.e.\ balls are finite).
	\item Let $d_1,d_2$ be two left-invariant proper metrics on $G$. Set $$\mu(t):= \min \{ d_1(1,g) \mid g \in G \backslash B_{d_2}(1,t)\}$$ and $$\nu(t):= \max \{ d_1(1,g) \mid g \in B_{d_2}(1,t) \}.$$ Prove that $\mu(t),\nu(t) \to + \infty$ as $t \to + \infty$. Conclude that $(A,d_1)$ and $(A,d_2)$ are coarsely equivalent.
	\item It follows that every countable group admits a well-defined geometry up to coarse equivalence. Justify that the coarse equivalence between countable groups coincides with the quasi-isometry between finitely generated groups when restricted to finitely generated groups. 
	\item Justify that it makes sense to refer to the asymptotic dimension of a countable group. Show that $\mathrm{asdim}(H) \leq \mathrm{asdim}(G)$ for every subgroup $H \leq G$. 
	\item Prove that $\mathrm{asdim}(G)= \mathrm{sup} \{ \mathrm{asdim}(H) \mid H \leq G \text{ finitely generated}\}$. Deduce that $G$ has asymptotic dimension $0$ if and only if it is locally finite. 
\end{enumerate}
\end{exo}

\begin{exo}
Using the definition of asymptotic dimension, 
\begin{enumerate}
	\item show that, for every $n \geq 1$, $\mathbb{Z}^n$ has asymptotic dimension $\leq n$;
	\item and that $\mathbb{Z}_2 \wr \mathbb{Z}$ has asymptotic dimension $1$. 
\end{enumerate}
\end{exo}

\begin{exo}
Let $\omega$ be an ultrafilter over $\mathbb{N}$.
\begin{enumerate}
	\item Prove that an $\omega$-bounded sequence of real numbers $\omega$-converges. (\emph{Hint: If our sequence lies in $[a,b]$ $\omega$-almost surely, then it lies either in $[a, (a+b)/2]$ or in $[(a+b)/2,b]$ $\omega$-almost surely.})
	\item If $\omega$ is a principal ultrafilter, when does a real sequence $\omega$-converge?
\end{enumerate}
\end{exo}

\begin{exo}
Prove Proposition~\ref{prop:AsConeLip}.
\end{exo}

\begin{exo}\label{exo:AbelianQI}
Given a graph $X$ of bounded degree, its \emph{growth function} is
$$\gamma_X : R \mapsto \max \left\{ \# B(x,R) \mid x \in V(X) \right\}.$$
We compare growth functions of graphs up to the following order relation: given two maps $f,g : \mathbb{N} \to \mathbb{N}$, write $f \prec g$ if there exists $C>0$ such that $f(n) \leq C g(Cn)$ for every $n \geq 0$; and write $f \sim g$ if $f \prec g$ and $g \prec f$. 
\begin{enumerate}
	\item Verify that, for all $a,b>0$, the maps $n \mapsto a^n$ and $n \mapsto b^n$ are $\sim$-equivalent.
	\item Given two integers $r,s \geq 1$, verify that the maps $n \mapsto n^r$ and $n \mapsto n^s$ are $\sim$-equivalent if and only if $r=s$. 
	\item Show that, if there exists a coarse embedding $X \to Y$ between two graphs of bounded degree, then $\gamma_X \prec \gamma_Y$. 
	\item Prove that non-abelian free groups have exponential growth.
	\item For every $d \geq 1$, show that $\mathbb{Z}^d$ has polynomial growth of degree $d$.
	\item Deduce that, given two integers $n,m \geq 1$, there exists a coarse embedding $\mathbb{Z}^n \hookrightarrow \mathbb{Z}^m$ if and only if $n \leq m$. 
\end{enumerate}
\end{exo}

\begin{exo}\label{exo:AsConeDim}
Let $X$ be a graph of bounded degree and $n \geq 2$ an integer. Show that, if $X$ has superlinear growth, then the asymptotic cones of $\mathcal{L}_n(X)$ are infinite-dimensional.
\end{exo}

\begin{exo}\label{exo:Hawaiian}
Our goal is to construct a copy of the Hawaiian earrings in every asymptotic cone of a lamplighter graph $\mathcal{L}_n(X)$ where $X$ is an unbounded graph of bounded degree. 
\begin{enumerate}
	\item Given an $\ell \geq 1$, construct a sequence of colourings $a_n^\ell \in \mathbb{Z}_2^{(0,\sqrt{n}]}$ such that, for all $\ell_1 \neq \ell_2$, one can find some $N \geq 0$ such that $a_n^{\ell_1} \neq a_n^{\ell_2}$ for  every $n \geq N$. Also, construct a sequence of integers $p_n^\ell$ such that $p_n^\ell \to 1/4\ell$ as $n \to + \infty$. 
	\item Let $C_n(\ell)$ denote the loop in $\mathcal{L}_2(\mathbb{N})$ defined as follows. Start from $(a_n^\ell,0)$, switch on the lamp at $0$, move arrow to $p_n^\ell$, switch on the lamp at $p_n^\ell$, move the arrow back at $0$, switch off the lamp at $0$, move the arrow back to $p_n^\ell$, switch off the lamp at $p_n^\ell$, and finally move the arrow back again at $0$. Prove that, in an asymptotic cone of $\mathcal{L}_2(\mathbb{N})$, $C_n(\ell)$ converges to a homotopically non-trivial loop $C(\ell)$ of length $1/\ell$. 
	\item Verify that, for all $\ell_1 \neq \ell_2$, the intersection between $C(\ell_1)$ and $C(\ell_2)$ is reduced to a single point. 
	\item If $X$ is an unbounded graph of bounded degree, show that, for every $n \geq 2$, the lamplighter graph $\mathcal{L}_n(X)$ contains a quasi-isometrically embedded copy of $\mathcal{L}_2(\mathbb{N})$.
	\item Conclude that $\mathcal{L}_n(X)$ contains a copy of the Hawaiian earrings. 
\end{enumerate}
\end{exo}

\begin{exo}
Let $G$ be a group admitting $\langle s_1, \ldots, s_n \mid r_1,r_2 , \ldots \rangle$ as an infinite presentation. For every $k \geq 1$, let $G_k$ denote the group given by the presentation $\langle s_1, \ldots, s_n \mid r_1, \ldots, r_k \rangle$. Show that, given a finitely presented group $H$, every morphism $H \to G$ factors through $G_k$ for every $k$ large enough. Deduce that every finitely presented group that surjects onto $\mathbb{Z}_n \wr \mathbb{Z}$ for some $n \geq 2$ is large (i.e.\ it contains a finite-index subgroup that surjects onto a non-abelian free group). \\
(For more information on this result, see \cite{MorphWreath} and references therein.) 
\end{exo}

\section{The Embedding Theorem and applications}\label{section:BigEmbedding}

\noindent
As announced in Section~\ref{section:CoarseLocal}, we will prove that there is some coarse local separation in lamplighter graphs $\mathcal{L}_n(X)$. In fact, even more precisely, we will characterise the maximal pancylindrical pieces in such graphs. They turn out to be contained in the natural copies of $X$ in $\mathcal{L}_n(X)$, namely:

\begin{definition}
Let $X$ be a graph and $n \geq 2$ an integer. An \emph{$X$-leaf} in the lamplighter graph $\mathcal{L}_n(X)$ is the subgraph induced by
$$\{ (c,p) \in \mathcal{L}_n(X) \mid p \in V(X) \}$$
for some colouring $c \in \mathbb{Z}_n^{(X)}$. 
\end{definition}

\noindent
Most of this section is dedicated to the proof of the following statement, which essentially motivates the idea that a lamplighter graph $\mathcal{L}_n(X)$ is obtained by gluing copies of $X$ along coarse cut points. 

\begin{thm}\label{thm:Embedding}
Let $X$ be a graph of bounded degree and $n \geq 2$ an integer. For every unbounded pancylindrical graph $Y$ and every coarse embedding $\eta : Y \to X$, the image of $\eta$ is contained in a neighbourhood of some $X$-leaf.
\end{thm}

\noindent
We emphasize that such a phenomenon is specific to lamplighter graphs, since wreath products with infinite lamp-groups are usually pancylindrical \cite[Propositions~3.6 and~3.8]{MR4794592}. (However, in this case, there is a similar statement but with a condition of coarse local separation that depends on the lamp-groups; see \cite{CoarseSep}.)

\subsection{Relatively stringy graphs}

\noindent
In this section, we prove an embedding result in the spirit of Theorem~\ref{thm:Embedding} but in the abstract framework of \emph{relatively stringy} graphs, which we now define:

\begin{definition}\label{def:Stringy}
A graph $X$ is \emph{stringy relative to} $\mathcal{P}$, a collection of connected subgraphs, if the following conditions are satisfied:
\begin{itemize}
	\item for every $K \geq 0$, there exists $M \geq 0$ such that $\mathrm{diam}(P^{+K} \cap Q^{+K}) \leq M$ for all distinct $P,Q \in \mathcal{P}$;
	\item for every $A_1 \geq 0$, there exists $F \geq 0$ such that the following holds for every $A_2 \geq 0$: there exists $K \geq 0$ such that, if $x,y \in V(X)$ are two vertices not contained in the $K$-neighbourhood of some subgraph of $\mathcal{P}$, then every path connecting $x$ and $y$ intersects $A_1$-persistently some subgraph $B$ of diameter $\leq F$ satisfying $d(x,B),d(y,B) \geq A_2$.
\end{itemize}
\end{definition}

\noindent
We will see in the next section that lamplighter graphs $\mathcal{L}_n(X)$ are stringy relative to their $X$-leaves. For now, we prove the following result:

\begin{prop}\label{prop:EmbeddingStringy}
Let $X$ be a graph of bounded degree that is stringy relative to some collection $\mathcal{P}$. For every unbounded pancylindrical graph $Y$ and every coarse embedding $\eta : Y \to X$, there exists $P \in \mathcal{P}$ such that the image of $Y$ under $\eta$ is contained in a neighbourhood of $P$. 
\end{prop}

\begin{proof}
For convenience, we will assume that $\eta$ is $1$-Lipschitz. Notice that there is no loss of generality, since it is always possible to subdivide $Y$ sufficiently and to extend $\eta$ by sending every subdivided edge to a path connecting the images of its endpoints. In particular, this implies that $\eta$ sends paths to paths.

\medskip \noindent
We start our proof of the proposition by showing the following observation:

\begin{claim}\label{claim:BoundedDegree}
For every $K \geq 0$, there exists $N \geq 0$ such that at most $N$ subgraphs of $\mathcal{P}$ intersects a given ball of radius $K$. 
\end{claim}

\noindent
Fix a ball $B(K)$ of radius $K$. Because $X$ is stringy relative to $\mathcal{P}$, we know that there exists some $M \geq 0$ such that $\mathrm{diam}(P \cap Q) \leq M$ for all distinct $P,Q \in \mathcal{P}$. Set $K':= K+M+1$. Let $B(K')$ denote the ball of radius $K'$ with the same centre as $B(K)$. 

\medskip \noindent
Let $N_1$ denote the number of $P \in \mathcal{P}$ such that $P \cap B(K) \neq \emptyset$ and $P \subset B(K')$. Clearly,
$$N_1 \leq N_0:= 2^{\mathrm{deg}(X)^{K'}}.$$
Let $N_2$ denote the number of $P \in \mathcal{P}$ such that $P \cap B(K) \neq \emptyset$ but $P \nsubseteq B(K')$. If $N_2>N_0$, then we can find two distinct subgraphs $P_1, P_2 \in \mathcal{P}$ such that $P_1 \cap B(K), P_2 \cap B(K) \neq \emptyset$, $P_1,P_2 \nsubseteq B(K')$, and $P_1 \cap B(K')= P_2 \cap B(K')$. Notice that, since $P_1$ is contained and that it contains a vertex in $B(K)$ and a vertex outside $B(K')$, necessarily $P_1$ intersects the sphere $S(K')$ of radius $K'$ with the same centre as $B(K)$. Thus, we can find two vertices $a \in B(K) \cap P_1 \cap P_2$ and $b \in S(K') \cap P_1 \cap P_2$. It follows that
$$\mathrm{diam}(P_1 \cap P_2) \geq d(a,b) \geq K'-K >M,$$
hence $P_1=P_2$ by definition of $M$, a contradiction. This concludes the proof of Claim~\ref{claim:BoundedDegree}. 

\medskip \noindent
Let $E$ denote the constant given by Definition~\ref{def:Pancylindrical} when applied to $Y$. Then, let $F$ denote the constant given by Definition~\ref{def:Stringy} when applied to $X$, $\mathcal{P}$, and $A_1:=E$. We denote by $F' \geq 0$ the smallest constant such that the pre-image under $\eta$ of a subgraph of diameter $\leq F$ has diameter $\leq F'$. By definition of $E$, we know that there exists $L \geq 0$ such that any two vertices of $Y$ can be connected by some path that does not intersect $E$-persistently any subgraph of diameter $\leq F'$ lying at distance $\geq L$ from its endpoints. And, by definition of $F$, we know that there exists $K \geq 0$ such that, if $x,y \in V(X)$ are two vertices not contained in the $K$-neighbourhood of some subgraph of $\mathcal{P}$, then every path connecting $x$ and $y$ intersects $E$-persistently some subgraph $B$ of diameter $\leq F$ satisfying $d(x,B),d(y,B) \geq L$. 

\begin{claim}\label{claim:PairBoundedDist}
For all $y_1,y_2 \in V(Y)$, there exists $P \in \mathcal{P}$ such that $\eta(y_1),\eta(y_2) \in P^{+K}$.
\end{claim}

\noindent
Suppose to the contrary that there exist $y_1,y_2 \in V(Y)$ satisfying $\eta(y_1),\eta(y_2) \notin P^{+K}$ for every $P \in \mathcal{P}$. We know that there exists some path $\alpha$ connecting $y_1$ and $y_2$ that does not intersect $E$-persistently any subgraph of diameter $\leq F'$ lying at distance $\geq L$ from its endpoints. We also know that there exists some subgraph $B$ of diameter $\leq F$ such that $\eta(\alpha)$ intersects $E$-persistently $B$ and such that $d(x,B),d(y,B) \geq L$. Because $\eta$ is $1$-Lipschitz, $\alpha$ intersects $E$-persistently $\eta^{-1}(B)$, which has diameter $\leq F'$, and
$$d(y_i, \eta^{-1}(B)) \geq d(\eta(y_i) , \eta(\eta^{-1}(B))) \geq d(\eta(y_i), B) \geq L$$
for $i=1,2$. We get a contradiction, proving Claim~\ref{claim:PairBoundedDist}.

\medskip \noindent
Because $X$ is stringy relative to $\mathcal{P}$, we know that there exists some $M \geq 0$ such that $\mathrm{diam}(P^{+K} \cap Q^{+K}) \leq M$ for all distinct $P,Q \in \mathcal{P}$. Setting the constant
$$C:= \max \left( K, N(1+ \mathrm{deg}(X)^{M+1} ) \right),$$
let us verify that:

\begin{claim}\label{claim:TripleBoundedDist}
For all vertices $y_1,y_2,y_2 \in V(Y)$, there exists a subgraph $P \in \mathcal{P}$ such that $\eta(y_1),\eta(y_2),\eta(y_3) \in P^{+C}$. 
\end{claim}

\noindent
Suppose to the contrary that there exist $y_1,y_2,y_3 \in V(Y)$ such that $\eta(y_1),\eta(y_2),\eta(y_3) \notin P^{+C}$ for every $P \in \mathcal{P}$. We know from Claim~\ref{claim:PairBoundedDist} that there exist $P_1,P_2,P_3 \in \mathcal{P}$ such that $\eta(y_1),\eta(y_2) \in P_1^{+K}$, $\eta(y_2),\eta(y_3) \in P_2^{+K}$, and $\eta(y_3),\eta(y_1) \in P_3^{+K}$. 
\begin{center}
\includegraphics[width=0.6\linewidth]{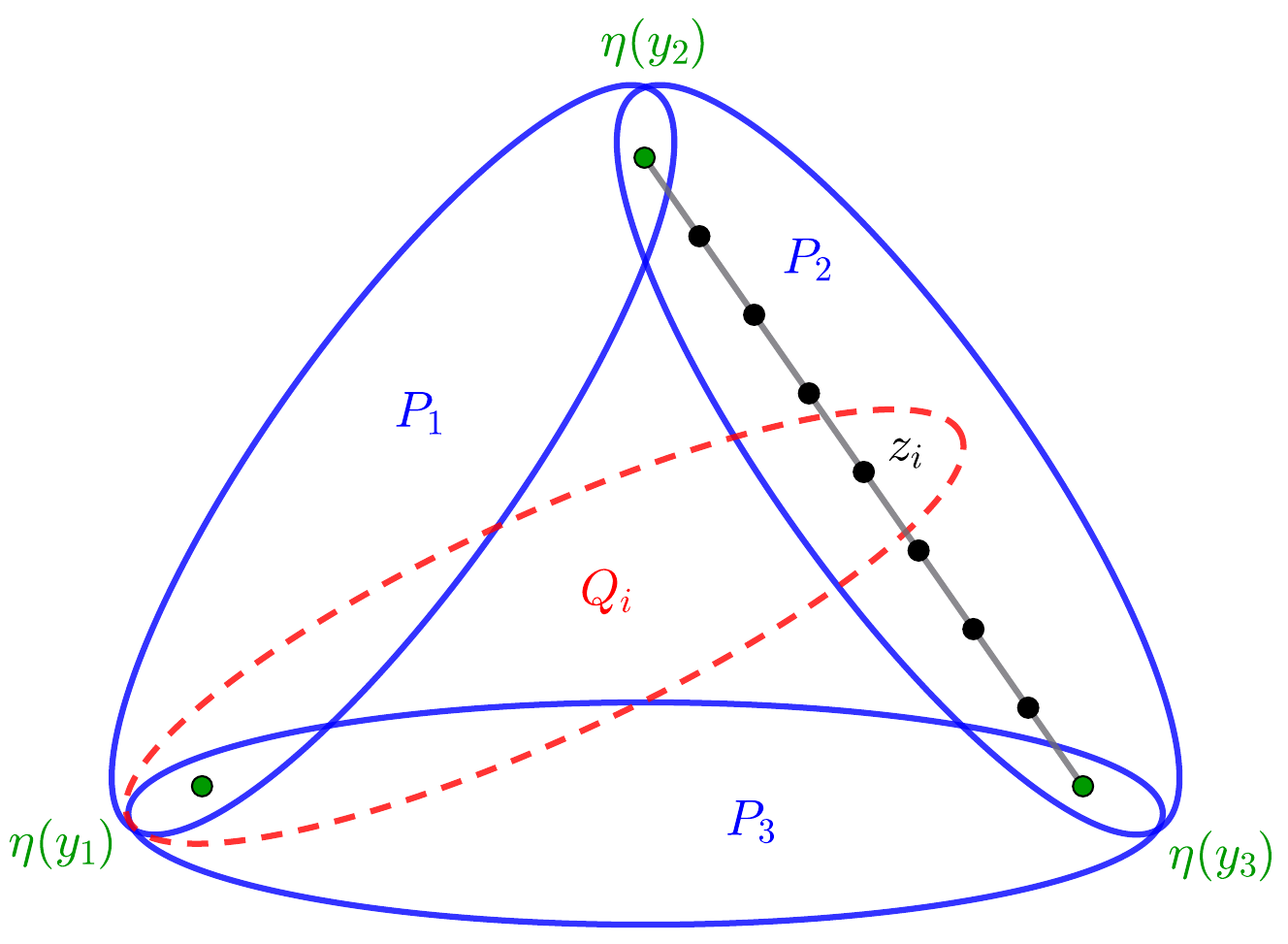}
\end{center}
Because $P_2$ is connected, we can fix a vertex-path $z_1, \ldots, z_n$ in $P_2$ connecting $\eta(y_2)$ and $\eta(y_3)$. Notice that, because $\eta(y_2) \notin P_3^{+C}$ since $P_3^{+C}$ already contains $\eta(y_1)$ and $\eta(y_3)$, we must have $d(\eta(y_2), \eta(y_3)) \geq C$, hence $n \geq C$. Again according to Claim~\ref{claim:PairBoundedDist}, we know that, for every $1 \leq i \leq n$, there exists $Q_i \in \mathcal{P}$ such that $\eta(y_1),z_i \in Q_i^{+K}$. Since the $Q_i$ all contain $\eta(y_1)$ in their $K$-neighbourhood, it follows from Claim~\ref{claim:BoundedDegree} that at least $\lfloor n/N \rfloor \geq \lfloor C/N \rfloor$ of the $Q_i$ must be identical, say to some $Q \in \mathcal{P}$. Since $Q$ contains at least $\lfloor C/N \rfloor$ of the $z_i$ in its $K$-neighbourhood, we have
$$\mathrm{diam}(Q^{+K} \cap P_2^{+K}) \geq \log_{\mathrm{deg}(X)} |Q^{+K} \cap P_2^{+K}|  \geq \log_{\mathrm{deg}(X)} \lfloor C/N \rfloor > M,$$
where the first inequality is justified by the fact that a ball in $X$ of radius $r$ has cardinality $\leq \mathrm{deg}(X)^r$. We deduce from the definition of $M$ that $Q$ and $P_2$ must coincide, hence
$$\eta(y_1) \in Q^{+K} = P_2^{+K} \subset P_2^{+C},$$
a contradiction since $P_2^{+C}$ already contains $\eta(y_2)$ and $\eta(y_3)$. This proves Claim~\ref{claim:TripleBoundedDist}. 

\medskip \noindent
We are now ready to conclude the proof of our proposition. Because $X$ is stringy relative to $\mathcal{P}$, we know that there exists some $M' \geq 0$ such that $\mathrm{diam}(P^{+C} \cap Q^{+C}) \leq M'$ for all distinct $P,Q \in \mathcal{P}$. Since $Y$ is unbounded, we can fix two vertices $a,b \in V(Y)$ such that $d(\eta(a),\eta(b)) > M'$. According to Claim~\ref{claim:TripleBoundedDist}, for every vertex $y \in V(Y)$, there exists $P(y) \in \mathcal{P}$ such that $\eta(a),\eta(b),\eta(y) \in P(y)^{+C}$. Notice that, for all $y_1,y_2 \in V(Y)$, 
$$\mathrm{diam}(P(y_1)^{+C} \cap P(y_2)^{+C}) \geq d(\eta(a),\eta(b)) > M',$$
hence $P(y_1)=P(y_2)$ by the definition of $M'$. Thus, if $P \in \mathcal{P}$ denotes this common subgraph, we have proved that $\eta(y) \in P^{+C}$ for every $y \in V(Y)$, which concludes the proof of our proposition. 
\end{proof}

\begin{remark}\label{remark:UniformConstant}
It follows from the proof above that the size of the neighbourhood from Proposition~\ref{prop:EmbeddingStringy} only depends on the constants given by the fact that $X$ is stringy relative to $\mathcal{P}$, on the constants given by the fact that $Y$ is pancylindrical, on the degree of $X$, and on the parameters of the coarse embedding $\eta$. 
\end{remark}

\subsection{Proof of the Embedding Theorem}\label{section:Embedding}

\noindent
As a consequence of Proposition~\ref{prop:EmbeddingStringy}, our embedding theorem for lamplighter graphs (Theorem~\ref{thm:Embedding}) will be a consequence of the following observation:

\begin{thm}\label{thm:LampStringy}
Let $X$ be a graph of bounded degree and $n \geq 2$ an integer. The lamplighter graph $\mathcal{L}_n(X)$ is stringy relative to its $X$-leaves. 
\end{thm}

\noindent
First of all, we need a tool in order to show that some intersections between paths and subgraphs are persistent. This is the purpose of our next proposition. The picture to keep in mind is that, given a path $\gamma$ in a graph, there exists a representative $\gamma'$ of $\gamma$ up to homotopy such that every path homotopy equivalent to $\gamma$ covers $\gamma'$ entirely. 

\begin{prop}\label{prop:CreatingPersist}
Let $X$ be a graph and $\mathcal{O}$ a covering of its vertices. Assume that:
\begin{itemize}
	\item there exists $E \geq 1$ such that every subset $S \subset V(X)$ of diameter $\leq E$ is contained in some $O \in \mathcal{O}$;
	\item the nerve complex $\mathcal{N}(\mathcal{O})$ of $\mathcal{O}$ is a graph.
\end{itemize}
Then, for any two distinct $A,B \in \mathcal{O}$ and for every path $\gamma$ connecting $A$ to $B$, there exists a path
$$C_1=A, \ C_2, \  \ldots, \ C_{n-1}, \ C_n=B \text{ in } \mathcal{N}(\mathcal{O})$$
such that $\gamma$ intersects $E$-persistently each $C_i$. 
\end{prop}

\begin{proof}
Fix two distinct $A,B \in \mathcal{O}$. For every path $\alpha$ connecting $A$ to $B$, we want to define some path $\Pi(\alpha)$ in $\mathcal{N}(\mathcal{O})$. For this, we need the following observation:

\begin{claim}\label{claim:UniqueO}
Let $O \in \mathcal{O}$ and let $\{a,b\} \in E(X)$ be an edge with $a \in O$ but $b \notin O$. There exists a unique $O' \in \mathcal{O}$ such that $\{a,b\} \subset O'$. 
\end{claim}

\noindent
Because $\{a,b\}$ has diameter $\leq 1 \leq E$, we know that there exists at least one $O' \in \mathcal{O}$ such that $\{a,b\} \subset O'$. Let $O'' \in \mathcal{O}$ be another such element. Notice that $O,O',O''$ pairwise intersect. (Indeed, they all contain $a$.) But $\mathcal{N}(\mathcal{O})$ is a graph, so $\mathcal{O}$ cannot contain three distinct pairwise intersecting elements. Since $O \notin \{O',O''\}$ as $b \notin O$, necessarily $O'=O''$. This concludes the proof of Claim~\ref{claim:UniqueO}. 

\medskip \noindent
Now, we are ready to define our path $\Pi(\alpha)$. We decompose $\alpha$ as a concatenation of successive subsegments $\alpha_0 \cdot \alpha_1 \cdots \alpha_n$ such that:
\begin{itemize}
	\item $\alpha_0$ is the maximal initial segment of $\alpha$ contained in $A$;
	\item for every $i \geq 1$, $\alpha_i$ is the maximal subsegment of $\alpha$ starting from the terminal point of $\alpha_{i-1}$ that is contained in the unique $O_i \in \mathcal{O}$ to which belongs the edge following $\alpha_{i-1}$.
\end{itemize}
Then, we define $\Pi(\alpha)$ as the path
$$A, \ O_1, \ \ldots, \ O_n, \ B$$
in $\mathcal{N}(\mathcal{O})$. The key observation is that, when we modify $\alpha$ up to $E$-coarse homotopy equivalence, then $\Pi(\alpha)$ is modified up to homotopy equivalence. This is a consequence of our next claim.

\begin{claim}\label{claim:UpToBacktrack}
Decompose $\alpha$ as $\mu \cdot \zeta \cdot \nu$ for three successive subsegments $\mu,\zeta,\nu \subset \alpha$, and let $\xi$ be a path with the same endpoints as $\zeta$ such that $\zeta \cup \xi$ has diameter $\leq E$. Then $\Pi(\alpha)= \Pi(\mu \zeta \nu)$ and $\Pi(\mu \xi \nu)$ differ by at most one edge-backtrack. 
\end{claim}

\noindent
We keep the previous decomposition $\alpha = \alpha_0 \cdot \alpha_1 \cdots \alpha_n$ and the description of $\Pi(\alpha)$ as $A,O_1, \ldots, O_n,B$. In order to describe $\Pi(\mu \xi \nu)$, we distinguish different cases, depending on how the two decompositions $\mu \zeta \nu$ and $\alpha_0 \alpha_1 \cdots \alpha_n$ interact. We denote by $p$ and $q$ the common endpoints of $\zeta$ and $\xi$, and we fix some $O \in \mathcal{O}$ containing $\zeta \cup \xi$. We also denote by $i$ the index such that $i=0$ if $p$ is the initial point of $\alpha$ and such that $\alpha_i$ contains the edge preceding $p$ otherwise. 

\medskip \noindent
Case 1: Assume that $p \in \alpha_i \cap \alpha_{i+1}$. Because $O_{i+1}$ is the unique element of $\mathcal{O}$ containing the edge following $\alpha_i$, we know that $O=O_{i+1}$. Moreover, since $\alpha_{i+1}$ is the maximal subsegment of $\alpha$ starting from the terminal point of $\alpha_i$ that is contained in $O_{i+1}$, we deduce from $\zeta \subset O$ that $q \in \alpha_{i+1}$. Thus, the configuration is the following:
\begin{center}
\includegraphics[width=0.5\linewidth]{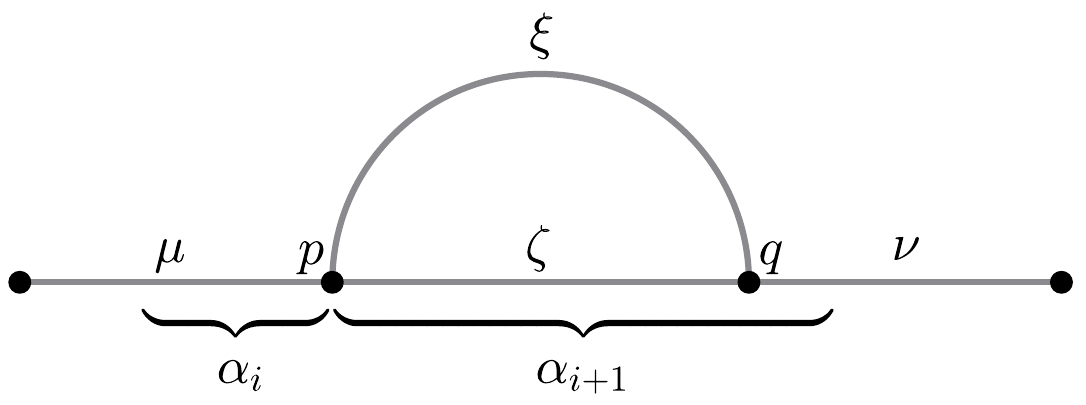}
\end{center}
It follows that $\Pi(\mu \xi \nu)$ coincides with $\Pi(\mu \zeta \nu) = \Pi(\alpha)$. 

\medskip \noindent
Case 2: Assume that $p \notin \alpha_{i+1}$, $q \notin \alpha_i$, and $O_i$ contains $\xi$, $\alpha_{i+1} \cap \nu$, and the first edge of $\alpha_{i+2}$. The configuration is the following:
\begin{center}
\includegraphics[width=0.5\linewidth]{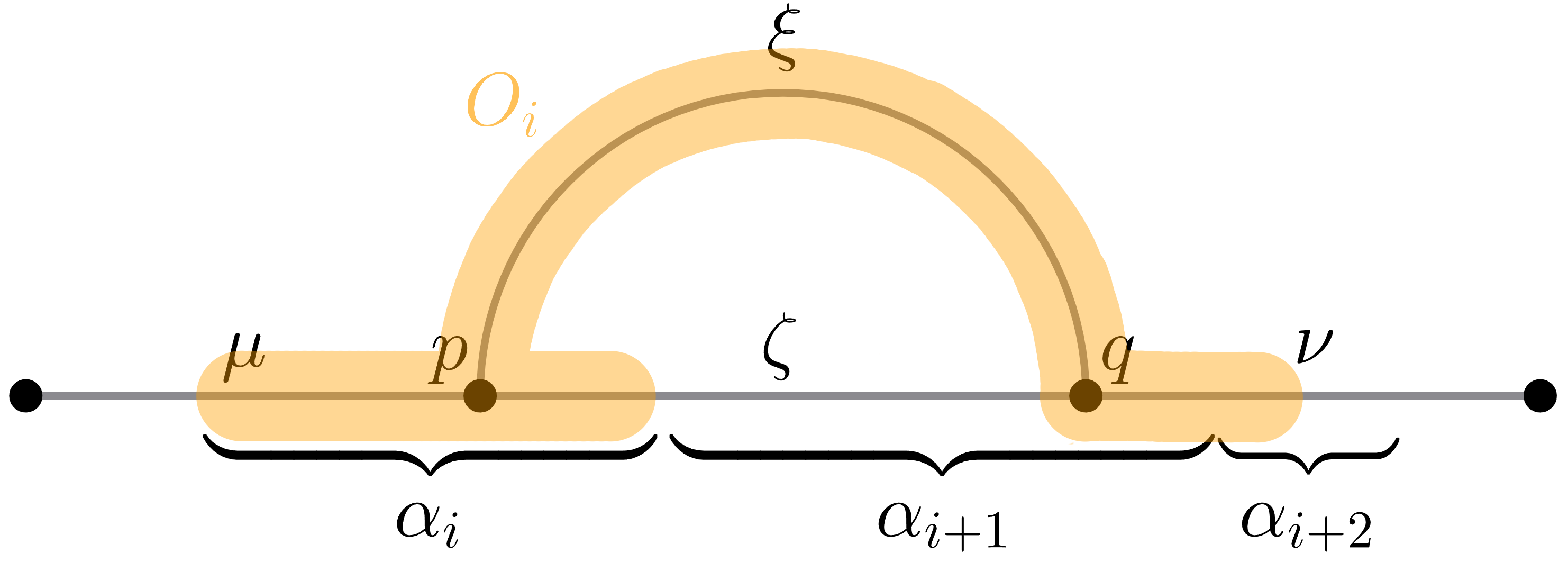}
\end{center}
By just applying the definition of $\Pi(\cdot)$, we find that $\Pi(\mu \xi \nu)$ is 
$$A, O_1, \ldots, O_{i-1}, O_i, O_{i+2}, O_{i+3}, \ldots, O_n, B.$$
Notice that, because $O_{i+2}$ is the unique element of $\mathcal{O}$ that contains the first edge of $\alpha_{i+2}$, necessarily $O_{i+2}=O_i$. Thus, $\Pi(\mu \xi \nu)$ can be obtained from $\Pi(\mu \zeta \nu)= \Pi(\alpha)$ by removing an edge-backtrack.

\medskip \noindent
Case 3: Assume that $p \notin \alpha_{i+1}$, $q \notin \alpha_i$, and $O_i$ does not contain the union of $\xi$, $\alpha_{i+1} \cap \nu$, and the first edge of $\alpha_{i+2}$. Our configuration is illustrated by one of the two following pictures:
\begin{center}
\includegraphics[width=0.45\linewidth]{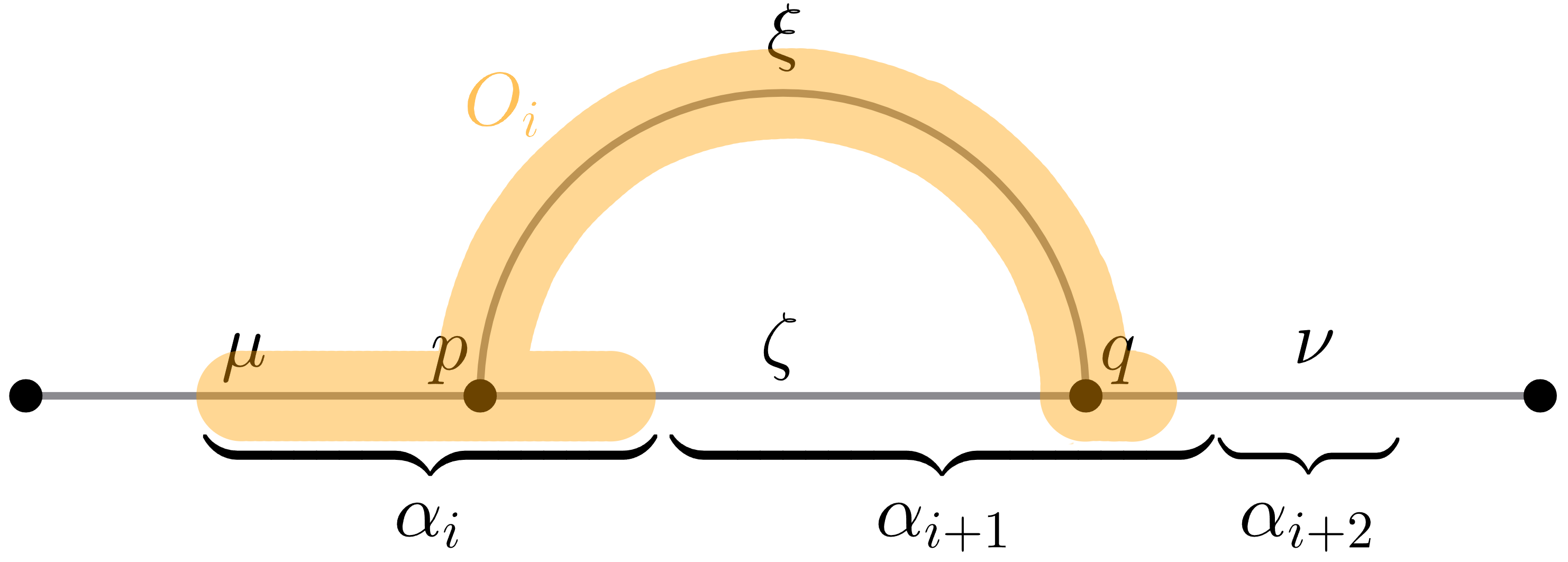}
\includegraphics[width=0.45\linewidth]{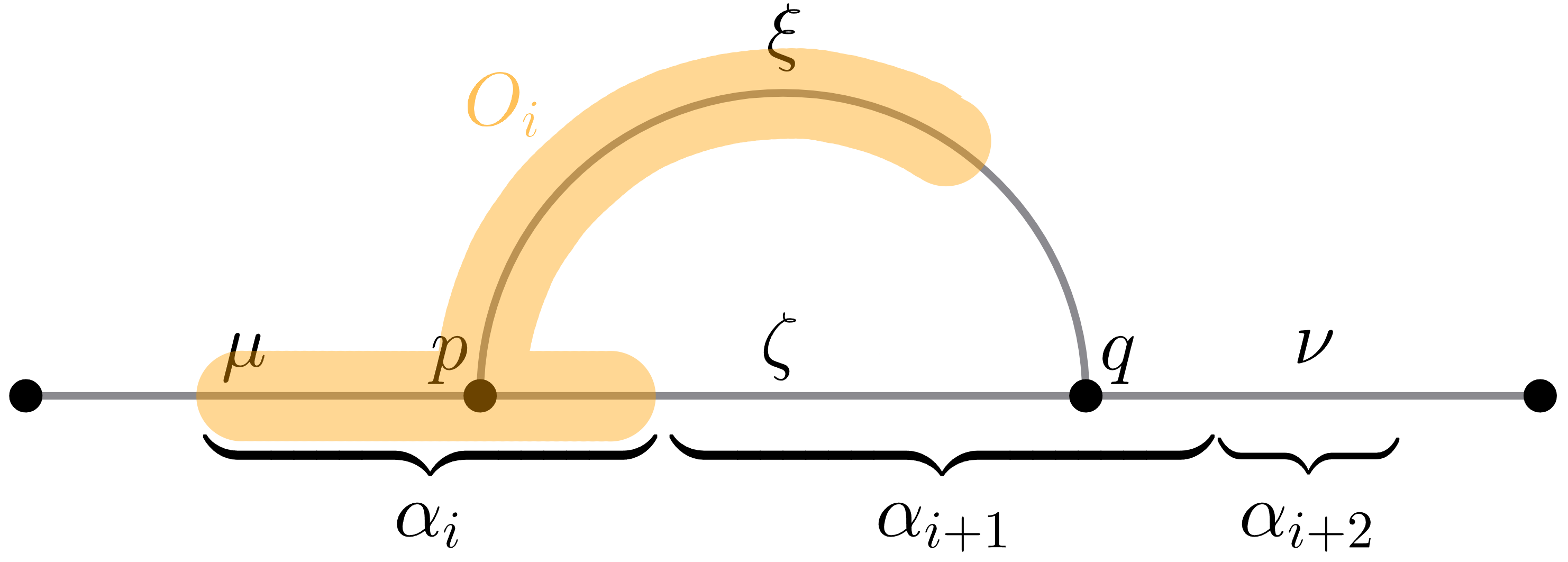}
\end{center}
Notice that, because $O_{i+1}$ is the unique element of $\mathcal{O}$ that contains the first edge of $\alpha_{i+1}$, necessarily $O_{i+1}=O$. By just applying the definition of $\Pi(\cdot)$, we find that $\Pi(\mu \xi \nu)$ is
$$A, O_1, \ldots, O_{i-1}, O_i, O=O_{i+1}, O_{i+2} , \ldots, O_n, B.$$
Thus, $\Pi(\mu \xi \nu)$ coincides with $\Pi(\mu \zeta \nu)= \Pi(\alpha)$.

\medskip \noindent
Case 4: Assume that $q \in \alpha_i$ and $\xi \subset O_i$. The configuration is the following:
\begin{center}
\includegraphics[width=0.5\linewidth]{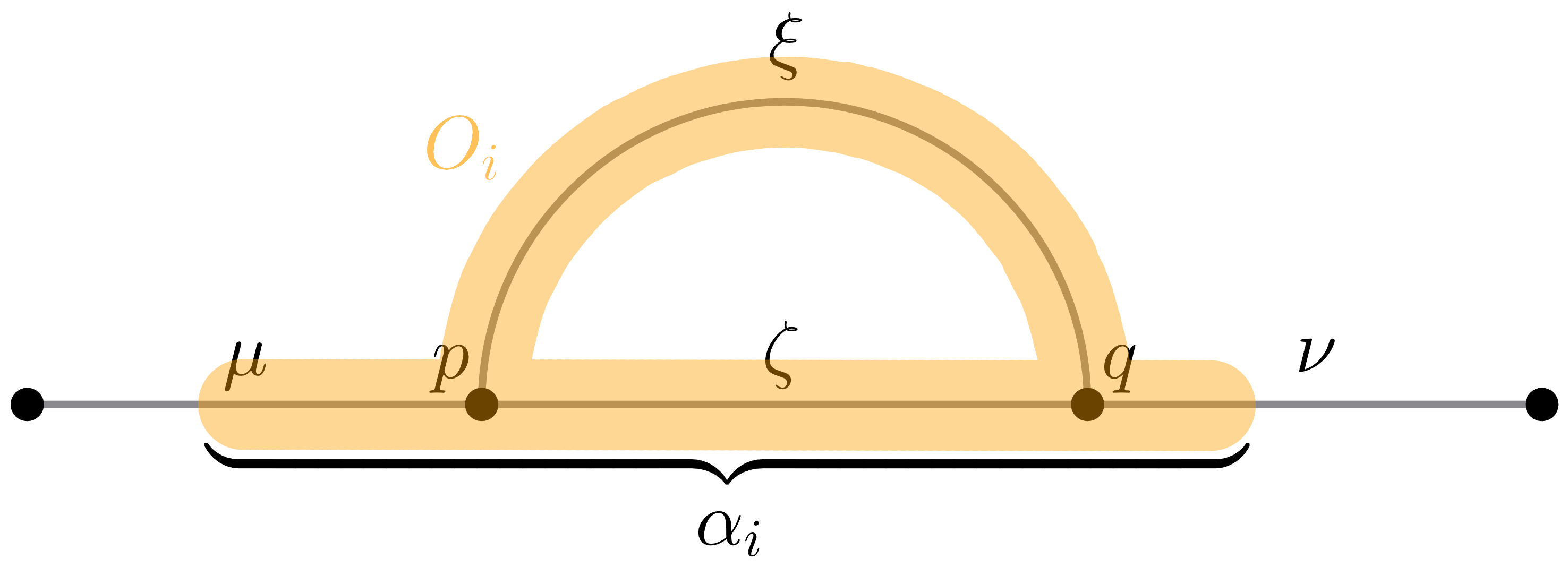}
\end{center}
Just applying the definition of $\Pi(\cdot)$, we find that $\Pi(\mu \xi \nu)$ coincides with $\Pi(\mu \zeta \nu)= \Pi(\alpha)$. 

\medskip \noindent
Case 5: Assume that $q \in \alpha_i$, $\xi \nsubseteq O_i$, and the first edge of $\alpha_{i+1}$ is contained in $O$. The configuration is the following:
\begin{center}
\includegraphics[width=0.5\linewidth]{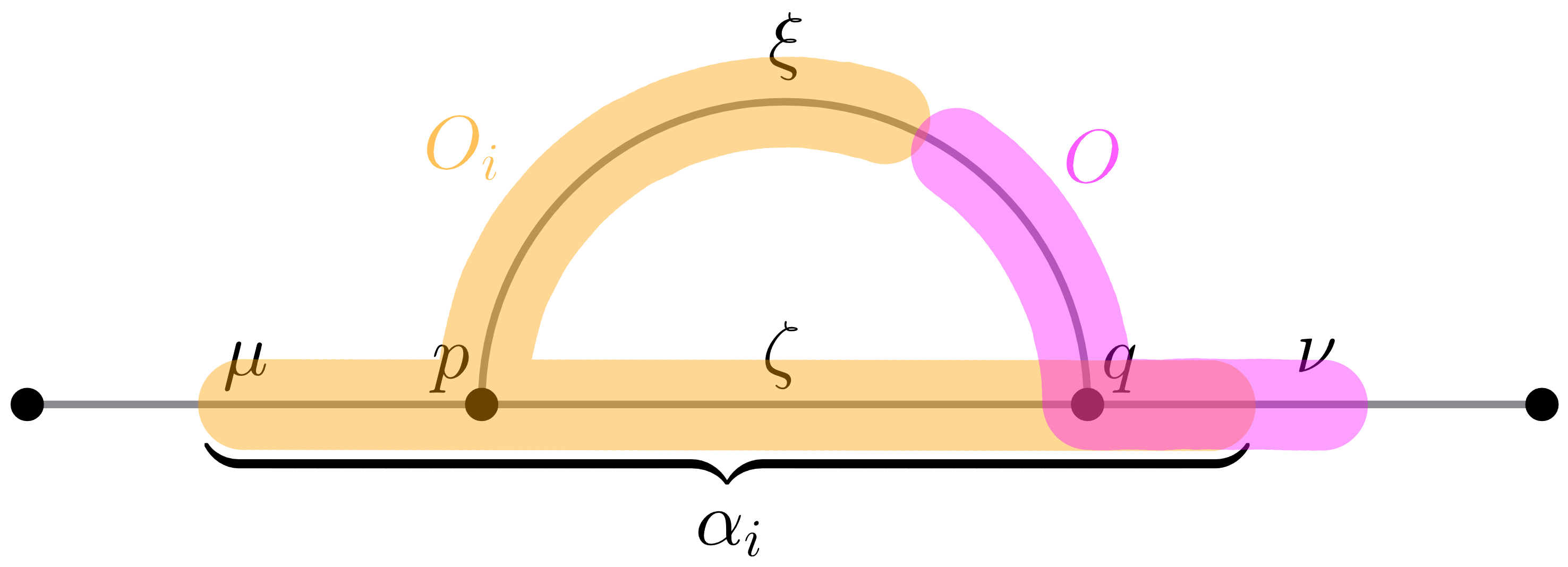}
\end{center}
Notice that, since $O_{i+1}$ is the unique element of $\mathcal{O}$ that contains the first edge of $\alpha_{i+1}$, necessarily $O_{i+1}=O$. By applying the definition of $\Pi(\cdot)$, we find that $\Pi(\mu \xi \nu)$ is
$$A, O_1, \ldots, O_{i-1}, O_i, O=O_{i+1}, O_{i+2}, \ldots, O_n, B.$$
Thus, $\Pi(\mu \xi \nu)$ coincides with $\Pi(\mu \zeta \nu)= \Pi(\alpha)$.

\medskip \noindent
Case 6: Assume that $q \in \alpha_i$, $\xi \nsubseteq O_i$, and the first edge of $\alpha_{i+1}$ is not contained in $O$. Our configuration is illustrated by one of the two following pictures, depending on whether the last edge of $\alpha_i$ belongs to $O$ or not.
\begin{center}
\includegraphics[width=0.45\linewidth]{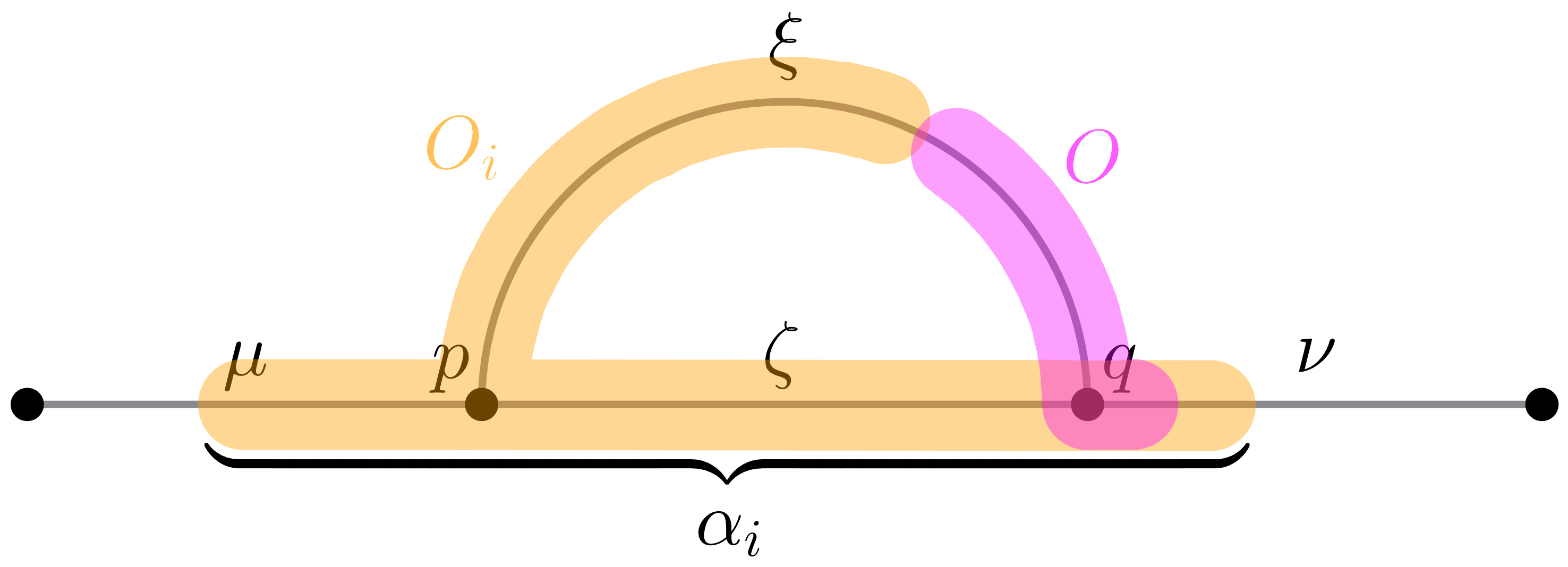}
\includegraphics[width=0.45\linewidth]{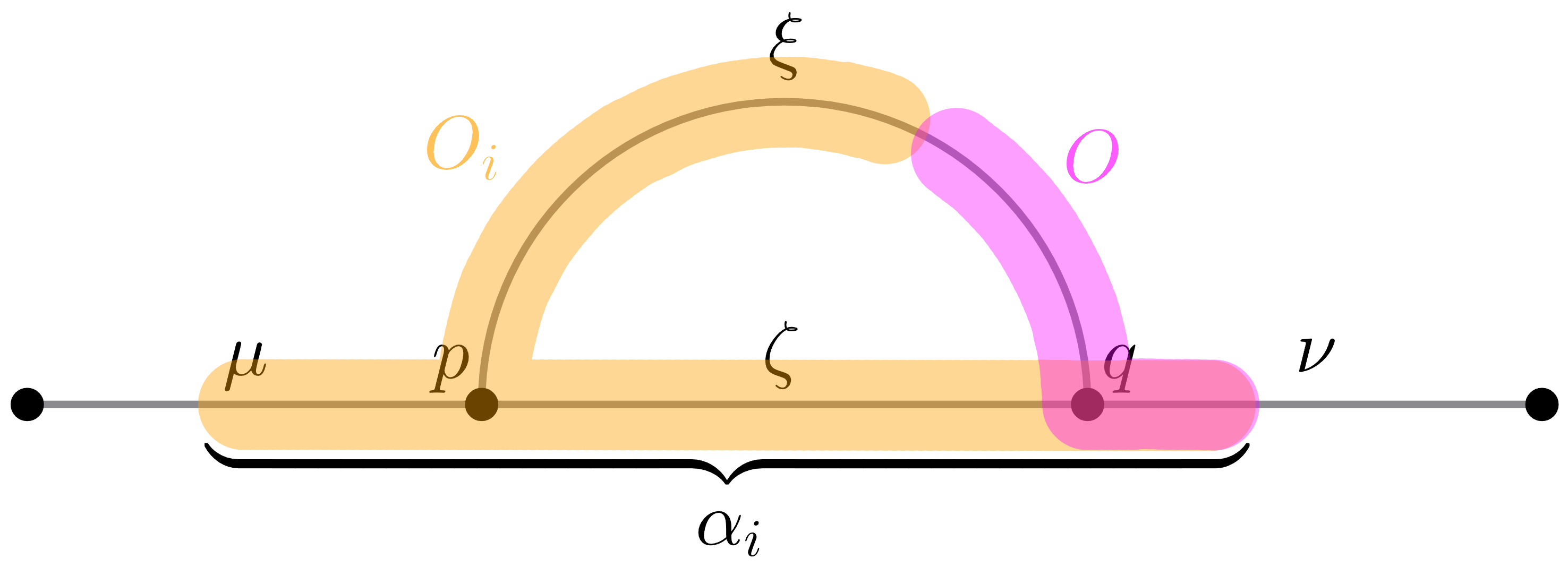}
\end{center}
In the first case, we find that $\Pi(\mu \xi \nu)$ is
$$A, O_1, \ldots, O_{i-1}, O_i, O, O_i, O_{i+1}, \ldots, O_n, B.$$
Thus, $\Pi(\mu \xi \nu)$ is obtained from $\Pi(\mu \zeta \nu)= \Pi(\alpha)$ by adding an edge-backtrack. And, in the second case, we find that $\Pi(\mu \xi \nu)$ is
$$A, O_1, \ldots, O_{i-1}, O_i, O, O_{i+1}, \ldots, O_n, B.$$
Notice that $O$ intersects both $O_i$ and $O_{i+1}$, hence $O \in \{O_i, O_{i+1}\}$ since $\mathcal{N}(\mathcal{O})$ is triangle-free. Thus, $\Pi(\mu \xi \nu)$ coincides with $\Pi(\mu \zeta \nu) = \Pi(\alpha)$. The proof of Claim~\ref{claim:UpToBacktrack} is complete.

\medskip \noindent
Now, let $\gamma$ be an arbitrary path in $X$ connecting $A$ to $B$. In the graph $\mathcal{N}(\mathcal{O})$, the path $\Pi(\gamma)$ is homotopy equivalent to a unique path
$$C_1=A, \ C_2, \ \ldots, \ C_{n-1}, \ C_n=B$$
of minimal length. Moreover, every path in $\mathcal{N}(\mathcal{O})$ homotopy equivalent to $\Pi(\gamma)$ has to pass through each $C_i$. If $\gamma'$ is a path in $X$ that is $E$-coarsely homotopy equivalent to $\gamma$, then it follows from Claim~\ref{claim:UpToBacktrack} that $\Pi(\gamma')$ is homotopy equivalent to $\Pi(\gamma)$ in $\mathcal{N}(\mathcal{O})$. Consequently, $\gamma'$ has to intersect each $C_i$. 
\end{proof}

\noindent
Before turning to the proof of Theorem~\ref{thm:LampStringy}, we need a last preliminary lemma:

\begin{lemma}\label{lem:NotNeighLeaf}
Let $X$ be a graph of bounded degree and $n \geq 2$ an integer. For every $D \geq 0$, there exists some $K \geq 0$ such that the following holds. If two vertices $(c_1,p_1), (c_2,p_2) \in \mathcal{L}_n(X)$ are not contained in the $K$-neighbourhood of some $X$-leaf, then there exists $p \in c_1 \triangle c_2$ such that $d(p,p_1),d(p,p_2)>D$. 
\end{lemma}

\begin{proof}
Let $(c_1,p_1),(c_2,p_2) \in \mathcal{L}_n(X)$ be two vertices and $D \geq 0$ a constant. Assume that $c_1 \triangle c_2 \subset B(p_1,D) \cup B(p_2,D)$. Let $c$ denote the colouring that agrees with both $c_1$ and $c_2$ outside $B(p_1,D) \cup B(p_2,D)$, that agrees with $c_1$ on $B(p_2, D)$, and that agrees with $c_2$ on $B(p_1,D)$. By construction, for $i=1,2$, we have $c_i \triangle c \subset B(p_i,D)$. Therefore, if $X(c)$ denotes the $X$-leaf associated to $c$, we have
$$d( (c_i,p_i), X(c)) \leq |B(p_i,D)| \mathrm{diam}(B(p_i,D)) + |B(p_i,D)| \leq (1+ 2D) \mathrm{deg}(X)^{D}=:K.$$
Thus, we have found the constant $K$ we were looking for. 
\end{proof}

\begin{proof}[Proof of Theorem~\ref{thm:LampStringy}.]
We start by verifying that coarse intersections of $X$-leaves are bounded.

\begin{claim}\label{claim:InterLeaves}
For every $K \geq 0$, there exist $M \geq 0$ such that $\mathrm{diam}(P^{+K} \cap Q^{+K}) \leq M$ for all distinct $X$-leaves $P$ and $Q$.
\end{claim}

\noindent
Let $c_1$ and $c_2$ be two distinct colourings, and let $X(c_1),X(c_2)$ denote the corresponding $X$-leaves. If $X(c_1)^{+K} \cap X(c_2)^{+K}$ is empty, then there is nothing to prove. So, from now on, we assume that the intersection is non-empty. Let $(a_1,p_1), (a_2,p_2) \in X(c_1)^{+K} \cap X(c_2)^{+K}$ be two vertices. For $i,j=1,2$, we must have $a_i \triangle c_j \subset B(p_i,K)$. First, this implies that
$$c_1 \triangle c_2 = (c_1 \triangle a_i) \triangle (a_i \triangle c_2) \subset B(p_i,K).$$
Consequently, $c_1 \triangle c_2$ has diameter $\leq 2K$ and $p_i \in (c_1 \triangle c_2)^{+K}$. Second, this implies that
$$a_1 \triangle a_2 = (a_1 \triangle c_1) \triangle (c_1 \triangle a_2) \subset B(p_1,K) \subset (c_1 \triangle c_2)^{+2K},$$
where the inclusion is justified by the fact $p_1 \in (c_1 \triangle c_2)^{+K}$, as previously noticed. Since $p_1,p_2 \in (c_1 \triangle c_2)^{+K}$ and $a_1 \triangle a_2 \subset (c_1 \triangle c_2)^{+2K}$, we must have
$$\begin{array}{lcl} d((a_1,p_1),(a_2,p_2)) & \leq & \displaystyle \mathrm{diam}\left( (c_1 \triangle c_2)^{+2K} \right) \left| (c_1 \triangle c_2)^{+2K} \right| + \left| (c_1 \triangle c_2)^{+2K} \right| \\ \\ & \leq & \displaystyle (1+4K) \mathrm{deg}(X)^{6K}, \end{array}$$
where the last inequality is justified by the fact that $c_1 \triangle c_2$ has diameter $\leq 2K$, as previously noticed. This concludes the proof of Claim~\ref{claim:InterLeaves}. 

\medskip \noindent
Now, we want to verify the second item from Definition~\ref{def:Stringy}. So we fix some $A_1\geq 0$ and we define 
$$F:= (1+4A_1) \mathrm{deg}(X)^{2A_1}. $$
 Given some $A_2 \geq 0$, let $K \geq 0$ denote the constant given by Lemma~\ref{lem:NotNeighLeaf} for $D:=2A_1+A_2$.  

\medskip \noindent
Fix two vertices $(c_1,p_1),(c_2,p_2) \in \mathcal{L}_n(X)$ not contained in the $K$-neighbourhood of an $X$-leaf, and fix a path $\alpha$ connecting them. Our goal is to apply Proposition~\ref{prop:CreatingPersist} in order to show that $\alpha$ intersect persistently some bounded subgraph. We know from the definition of $K$ that there exists $p \in c_1 \triangle c_2$ such that $d(p,p_1),d(p,p_2)\geq 2A_1 +A_2$. Define
$$\mathcal{I}(d):= \{ (c,q) \in \mathcal{L}_n(X) \mid q \in B(p,2A_1), c_{|X \backslash B(p,3A_1)} = d\} \text{ for every } d \in \mathbb{Z}_n^{(X \backslash B(p,3A_1))}$$
and
$$\mathcal{O}(d):= \{ (c,q) \in \mathcal{L}_n(X) \mid q \notin B(p,2A_1), \ c_{|B(p,A_1)} = d\} \text{ for every } d \in \mathbb{Z}_n^{(B(p,A_1))}.$$
Notice that every set of vertices $S$ of diameter $\leq A_1$ is contained in some $\mathcal{I}(d)$ or some $\mathcal{O}(d)$. Indeed, if there exists some $(c,q) \in S$ such that $q \in B(p,2A_1)$, then, for every $(c',q') \in S$, the colourings $c$ and $c'$ must coincide outside $B(p,3A_1)$ since the arrow cannot leave $B(p,3A_1)$ along a geodesic connecting $(c,q)$ to $(c',q')$. Thus, $S \subset \mathcal{I}(c_{|X\backslash B(p,3A_1)})$. Otherwise, if $q \notin B(p,2A_1)$ for every $(c,q) \in S$, then, for all $(c,q),(c',q') \in S$, the colourings $c$ and $c'$ must coincide on $B(p,A_1)$ since the arrow must stay outside $B(p,A_1)$ along a geodesic connecting $(c,q)$ and $(c',q')$. Thus, $S \subset \mathcal{O}(c_{|B(p,A_1)})$ for any $(c,q) \in S$. 

\medskip \noindent
Notice that the nerve complex $\mathcal{N}$ of the covering given by the $\mathcal{I}(d)$ and the $\mathcal{O}(d)$ is clearly a graph since no two distinct $\mathcal{I}(d)$ (resp.\ $\mathcal{O}(d)$) intersect. Thus, $\mathcal{N}$ is a bipartite graph with vertices of type $\mathcal{I}$ and vertices of type $\mathcal{O}$.

\medskip \noindent
Our vertices $(c_1,p_1)$ and $(c_2,p_2)$ respectively belong to $\mathcal{O}(c_{1 | B(p,A_1)})$ and $\mathcal{O}(c_{2 | B(p,A_1)})$, which represent two distinct vertices of $\mathcal{N}$ since $c_1(p) \neq c_2(p)$. Thus, it follows from Proposition~\ref{prop:CreatingPersist} that there exists some $\mathcal{I}(d)$ that our path $\alpha$ intersects $A_1$-persistently. Notice that 
$$d((c_i,p_i), \mathcal{I}(d)) \geq d(p_i,p) - 2A_1 \geq A_2$$
for every $i=1,2$; and, finally, that $\mathcal{I}(d)$ has diameter $\leq F$, since
$$\mathrm{diam}(\mathcal{I}(d)) \leq \mathrm{diam}(B(p,A_1)) \cdot |B(p,A_1)| + |B(p,A_1)|.$$
This concludes the proof of our theorem. 
\end{proof}

\begin{proof}[Proof of Theorem~\ref{thm:Embedding}.]
The theorem follows directly from Theorem~\ref{thm:LampStringy} and Proposition~\ref{prop:EmbeddingStringy}. 
\end{proof}

\subsection{Applications}

\noindent
Let us mention a few easy but interesting applications of Theorem~\ref{thm:Embedding}. We refer to Exercises~\ref{exo:PancyclicSubgroups}, ~\ref{exo:QIbase}, and~\ref{exo:IteratedLamp} for some information. 

\begin{cor}\label{cor:PancyclicSubgroups}
Let $F$ be a finite group and $H$ a finitely generated group. Every pancylindrical subgroup $K \leq F \wr H$ is conjugate into $H$. 
\end{cor}

\begin{cor}\label{cor:QIbase}
Let $F_1,F_2$ be two finite groups and $H_1,H_2$ two pancylindrical finitely generated groups. If $F_1 \wr H_1$ and $F_2 \wr H_2$ are quasi-isometric, then $H_1$ and $H_2$ must be quasi-isometric.
\end{cor}

\begin{cor}\label{cor:IteratedLamp}
Let $F$ be a non-trivial finite group, $I$ an infinite finitely generated group, and $H$ a non-trivial finitely generated group. The wreath products $F\wr (I \wr H)$ and $I \wr (F \wr H)$ are not quasi-isometric.
\end{cor}

\noindent
Interestingly, Theorem~\ref{thm:Embedding} also provides some non-trivial information about arbitrary finitely generated groups quasi-isometric to lamplighter groups. Most notably:

\begin{thm}[{\cite[Theorem~4.23]{Halo}}]
Let $F$ be a finite group and $H$ a pancylindrical group. Every finitely generated group $G$ quasi-isometric to $F \wr H$ contains a finite collection of subgroups $\mathcal{H}$ such that:
\begin{itemize}
	\item all the groups of $\mathcal{H}$ are quasi-isometric to $H$;
	\item $\mathcal{H}$ is almost malnormal, i.e.\ for all $g \in G$ and $H_1,H_2 \in \mathcal{H}$, if $H_1 \cap gH_2g^{-1}$ is infinite then $H_1=H_2$ and $g \in H_1$;
	\item every pancylindrical subgroup of $G$ is conjugate into some subgroup in $\mathcal{H}$.
\end{itemize}
\end{thm}

\noindent
With some additional information about the groups involved, it is possible to be more precise:

\begin{cor}[{\cite[Theorem~4.27]{Halo}}]
Let $F$ be a finite group and $H$ a finitely generated nilpotent group. If an elementary amenable group $G$ is quasi-isometric to $F \wr H$, then $G$ virtually splits as $L \rtimes \bar{H}$ where $\bar{H}$ is quasi-isometric to $H$, where $L$ is locally finite, and where $\bar{H}$ acts by conjugation on $L \backslash \{1\}$ with finite stabilisers. 
\end{cor}

\noindent
However, there is no known description of all the finitely generated groups quasi-isometric to a given lamplighter over a one-ended finitely presented group, even for the basic example $\mathbb{Z}_2 \wr \mathbb{Z}^2$.

\subsection{Exercises}

\begin{exo}
Let $(X,x),(Y,y)$ be two pointed graphs. Show that the pointed sum $(X,x) \vee (Y,y)$, i.e.\ the graph obtained from two disjoint copies of $X$ and $Y$ by identifying the vertices $x$ and $y$, is stringy relative to $X$ and $Y$. 
\end{exo}

\begin{exo}
A graph $X$ is \emph{stringy} if, for every $A_1 \geq 0$, there exists $F \geq 0$ such that the following holds for every $A_2 \geq 0$: there exists $K \geq 0$ such that, if $x,y \in V(X)$ are two vertices satisfying $d(x,y) \geq K$, then every path connecting $x$ and $y$ intersects $A_1$-persistently some ball $B$ of radius $\leq F$ satisfying $d(x,B),d(y,B) \geq A_2$. Prove that lamplighter graphs over trees are stringy. 
\end{exo}

\begin{exo}\label{exo:PancyclicSubgroups}
Let $A$ and $B$ be two groups.
\begin{enumerate}
	\item Show that $B$ is an almost malnormal subgroup of $A \wr B$, i.e.\ if $B \cap gBg^{-1}$ is infinite for some $g \in A \wr B$, then necessarily $g \in B$. 
	\item Prove Corollary~\ref{cor:PancyclicSubgroups}. 
\end{enumerate}
\end{exo}

\begin{exo}\label{exo:QIbase}
Prove Corollary~\ref{cor:QIbase}. (\emph{Hint: Apply the Embedding Theorem twice, first for a quasi-isometry $F_1 \wr H_1 \to F_2 \wr H_2$ and then for a quasi-inverse.})
\end{exo}

\begin{exo}\label{exo:IteratedLamp}
Prove Corollary~\ref{cor:IteratedLamp}. (\emph{Hint: Remember that a wreath product $A \wr B$ with $A$ infinite and $B$ non-trivial is pancylindrical.})
\end{exo}

\section{Aptolic quasi-isometries}\label{section:Aptolic}

\noindent
As a consequence of our Embedding Theorem (Theorem~\ref{thm:Embedding}), it can be shown that, given two pancylindrical graphs $X,Y$ and two integers $n,m \geq 2$, a quasi-isometry $\varphi : \mathcal{L}_n(X) \to \mathcal{L}_m(Y)$ can always be modified up to finite distance so that it sends an $X$-leaf quasi-isometrically onto a $Y$-leaf, inducing a bijection from the set of $X$-leaves to the set of $Y$-leaves. In other words, there exists a bijection $\alpha : \mathbb{Z}_n^{(X)} \to \mathbb{Z}_m^{(Y)}$ such that $\varphi$ induces a quasi-isometry $X(c) \to Y(\alpha(c))$ for every $c \in \mathbb{Z}_n^{(X)}$. Consequently, we can describe $\varphi$ as a map $(c,p) \mapsto (\alpha(c), \beta_c(p))$ for some collection of quasi-isometries $\{\beta_c : X \to Y \mid c \in \mathbb{Z}_n^{(X)}\}$. Our first goal in this section is to show that all these quasi-isometries actually agree up to finite distance, providing a strong restriction on the possible quasi-isometries between lamplighters over pancylindrical graphs. 

\noindent
More formally, define:

\begin{definition}
A quasi-isometry $\varphi : \mathcal{L}_n(X) \to \mathcal{L}_m(Y)$ between two lamplighter graphs is \emph{aptolic}\footnote{This adjective comes from the contraction of the two Greek words $\alpha\pi\tau\omega$ (to light) and $\lambda\upsilon\chi\nu o\varsigma$ (lamp). It refers to a map that preserves the lamplighter structure.} if there exists a bijection $\alpha : \mathbb{Z}_n^{(X)} \to \mathbb{Z}_m^{(Y)}$ and a quasi-isometry $\beta : X \to Y$ such that $\varphi : (c,p) \mapsto (\alpha(c),\beta(p))$. 
\end{definition}

\noindent
Then, as described above, our goal is to prove that:

\begin{thm}\label{thm:AptolicQI}
Let $X,Y$ be two unbounded pancylindrical graphs of bounded degree and $n,m \geq 2$ two integers. Every quasi-isometry $\mathcal{L}_n(X) \to \mathcal{L}_m(Y)$ lies at finite distance from an aptolic quasi-isometry.
\end{thm}

\noindent
Actually, we will prove an even more general statement, restricting the possible coarse embeddings between lamplighters over pancylindrical graphs. See Theorem~\ref{thm:AptolicCoarse}. 

\medskip \noindent
The rest of the section is mainly dedicated to the deduction from Theorem~\ref{thm:AptolicQI} of the classification provided by Theorem~\ref{thm:QIclassification}, which we generalise for lamplighters over pancylindrical graphs:

\begin{thm}\label{thm:AptolicQIstrong}
Let $X,Y$ be two unbounded pancylindrical graphs of bounded degree and $n,m \geq 2$ two integers. The lamplighter graphs $\mathcal{L}_n(X)$ and $\mathcal{L}_m(Y)$ are quasi-isometric if and only if one of the following two conditions occurs:
\begin{itemize}
	\item $X,Y$ are non-amenable and quasi-isometric, and $n,m$ have the same prime divisors;
	\item $X,Y$ are amenable, $n,m$ are powers of a common number, say $n=q^a$ and $m=q^b$, and there exists a quasi-$\frac{b}{a}$-to-one quasi-isometry $X \to Y$. 
\end{itemize}
\end{thm}

\noindent
Our proof requires the concept of \emph{quasi-$\kappa$-to-one quasi-isometries}, which we introduce in Section~\ref{section:MeasureScaling}, and which will be fundamental in the understanding of the coarse geometry of lamplighters over amenable graphs.

\subsection{Ladders of leaves}

\noindent
A vertex $(c,p)$ of a lamplighter graph $\mathcal{L}_n(X)$ is the data of a colouring $c \in \mathbb{Z}_n^{(X)}$ and an arrow $p \in V(X)$. The colouring $c$ determines a leaf $X(c)$, and we know from the Embedding Theorem that leaves are preserved by quasi-isometries. The challenge now is to recognise the position of the arrow $p$ geometrically. It will be possible to do it thanks to specific configurations of leaves, which we now define.

\begin{definition}
Let $X$ be a graph, $n \geq 2$ an integer, and $\epsilon, L \geq 0$ two constants. An \emph{$(\epsilon,L)$-square of leaves} in $\mathcal{L}_n(X)$ is the data of four distinct $X$-leaves $\{P_i \mid i \in \mathbb{Z}_4 \}$ such that
\begin{itemize}
	\item $d(P_i,P_{i+1}) \leq \epsilon$ for every $i \in \mathbb{Z}_4$;
	\item $d(P_i,P_{i-1}^{+\epsilon}), d(P_i,P_{i+1}^{+\epsilon}) \geq L$ for every $i \in \mathbb{Z}_4$.
\end{itemize}
An \emph{$(\epsilon, L)$-ladder of leaves} is the data of two sequences of $X$-leaves $P_1, Q_1, \ldots, P_k,Q_k$ such that $P_i,P_{i+1}, Q_{i+1}, Q_i$ defines an $(\epsilon,L)$-square for every $1 \leq i \leq k-1$.
\end{definition}

\noindent
\begin{minipage}{0.35\linewidth}
\includegraphics[width=\linewidth]{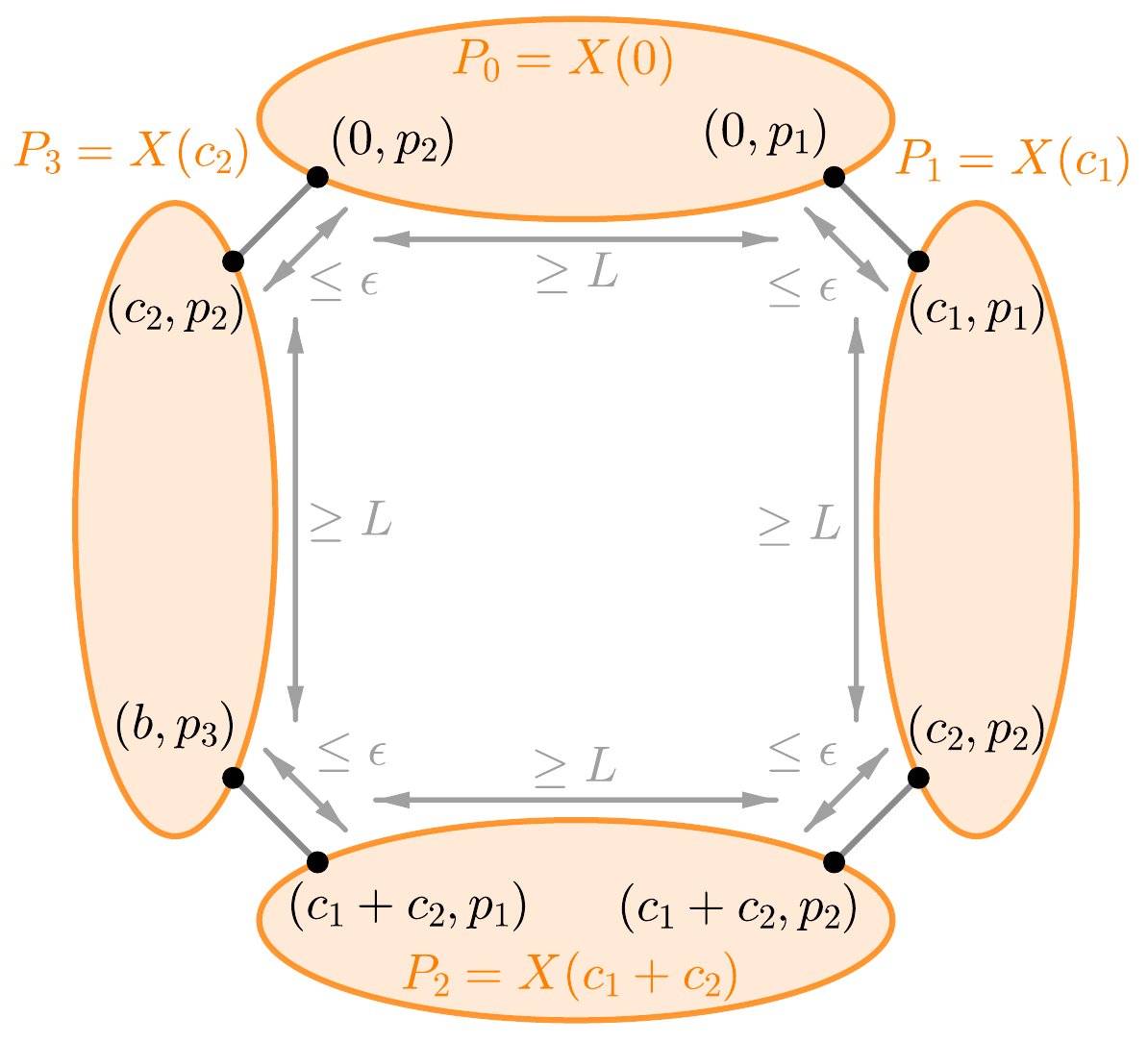}
\end{minipage}
\begin{minipage}{0.63\linewidth}
The typical example of a square of leaves to keep in mind is the following. Fix two small balls $B(p_1,\epsilon_1)$ and $B(p_2,\epsilon_2)$ in $X$ at large distance $\geq L$ apart. Let $c_1$ (resp.\ $c_2$) be a colouring supported in $B_1$ (resp.\ $B_2$). Then, the leaves $P_0:= X(0)$, $P_1:= X(c_1)$, $P_2:=X(c_1+c_2)$, $P_3:=X(c_2)$ define an $(\epsilon,L)$-square of leaves for some small $\epsilon$. 
\end{minipage}

\medskip \noindent
Our next lemma shows that such squares of leaves turn out to be, under some reasonable assumptions, the only squares of leaves.

\begin{lemma}\label{lem:SquareLeaf}
Let $X$ be a graph and $n \geq 2$ an integer. If $\{P_i \mid i \in \mathbb{Z}_4\}$ is an $(\epsilon,L)$-square of leaves for some $\epsilon \geq 0$ and $L>3\epsilon$, then 
$$P_0= X(c), \ P_1=X(c+a), \ P_4=X(c+b), \text{ and } P_2= X(c+a+b)$$
for some colouring $c \in \mathbb{Z}_n^{(X)}$ and two colourings $a,b \in \mathbb{Z}_n^{(X)}$ both supported in balls of radii $\leq \epsilon$ at distance $\geq L- 2\epsilon$ from each other. 
\end{lemma}

\noindent
The intuition is roughly the following. Since $P_0$ and $P_1$ (resp.\ $P_3$) are close, we can obtain the colouring of $P_1$ (resp.\ $P_3$) from the colouring of $P_0$ by a modification in some small ball $B(p_0, \eta)$ (resp.\ $B(q_0,\eta)$). Say $P_0=X(c)$ and $P_1=X(c+a)$ (resp.\ $P_3= X(c+b)$) with $a$ (resp.\ $b$) supported in $B(p_0,\eta)$ (resp.\ $B(q_0,\eta)$). Notice that the balls $B(p_0,\eta)$ and $B(q_0,\eta)$ must be at distance $\gtrsim L$ apart. Similarly, the colouring $d$ associated to $P_3$ can be obtained from $c+a$ by a modification in some small ball $B(q_1,\eta)$ far away from $p_0$, but also from $c+b$ by a modification in some small ball $B(p_3,\eta)$ far away from $q_0$. The only possibility, then, is that $d=c+a+b$. 

\begin{proof}[Proof of Lemma~\ref{lem:SquareLeaf}.]
Let $c\in \mathbb{Z}_n^{(X)}$ be such that $P_0=X(c)$. In order to shorten the notation, we will assume for convenience that $c \equiv 0$. 

\medskip \noindent
\begin{minipage}{0.35\linewidth}
\includegraphics[width=\linewidth]{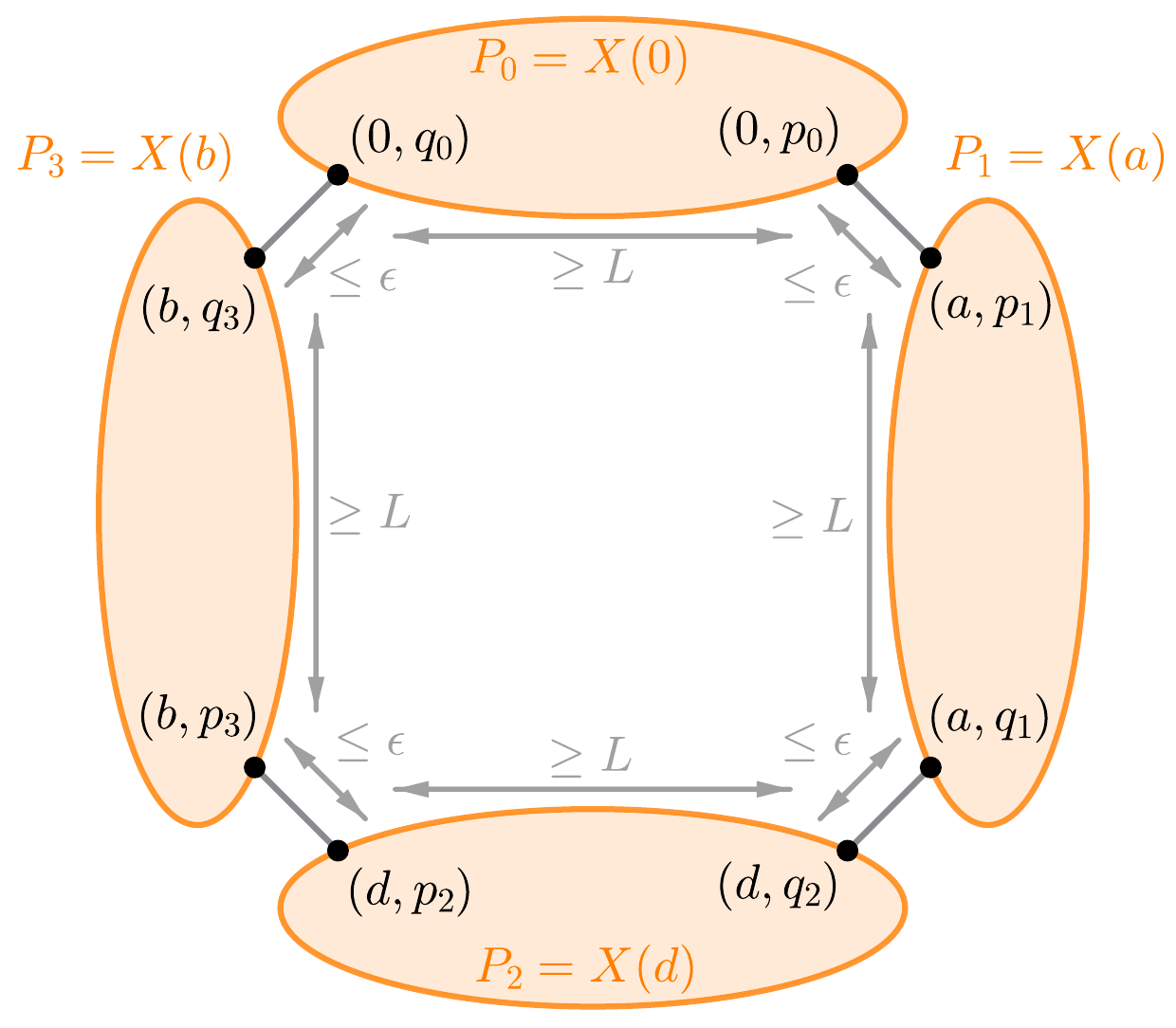}
\end{minipage}
\begin{minipage}{0.63\linewidth}
Let $a,b \in \mathbb{Z}_n^{(X)}$ be such that $P_1=X(a)$ and $P_3=X(b)$. Fix two vertices $(0,p_0) \in P_0$ and $(a,p_1) \in P_1$ at distance $\leq \epsilon$. Necessarily, $a$ has support in $B(p_0,\epsilon)$. Symmetrically, fix two vertices $(0,q_0) \in P_0$ and $(b,q_3) \in P_3$ at distance $\leq \epsilon$. Then, $b$ has support in $B(q_0,\epsilon)$. Notice that
$$\begin{array}{lcl} d(B(p_0,\epsilon), B(q_0,\epsilon)) & \geq & d(p_0,q_0)- 2 \epsilon  \\ \\ & \geq & d((0,p_0),(0,q_0)) -2 \epsilon \\ \\ & \geq & L- 2 \epsilon. \end{array}$$
\end{minipage}

\medskip \noindent
Thus, it only remains to show that $P_2= X(a+b)$. Let $d \in \mathbb{Z}_n^{(X)}$ be such that $P_2= X(d)$.

\medskip \noindent
 Fix two vertices $(a,q_1) \in P_1$ and $(d,q_2) \in P_2$ at distance $\leq \epsilon$. Necessarily, $d=a+d'$ for some colouring $d'$ with support in $B(q_1, \epsilon)$. Notice that
$$d(p_0,q_1) \geq d(p_0,q_0) - d(q_0,q_1) \geq d((0,p_0),(0,q_0)) - d((0,q_0),(a,q_1)) \geq L- \epsilon>2 \epsilon,$$
so the support of $a$, which is contained in $B(p_0,\epsilon)$, must be disjoint from the support of $d'$, which is contained in $B(q_1,\epsilon)$. Symmetrically, we can fix two vertices $(b,p_3) \in P_3$ and $(d,p_2) \in P_2$ at distances $\leq \epsilon$, and we can decompose $d$ as $b+d''$ for some colouring $d''$ with support in $B(p_3,\epsilon)$ that is disjoint from the support of $b$. 

\medskip \noindent
From the equality $a+d'=b+d''$ and from the restrictions on the supports of $a,b,d',d''$ we have, it follows that we must have $d'=b$ and $d''=a$. Thus, $d=a+b$ as desired. 
\end{proof}

\noindent
Given a vertex of a lamplighter graph, the recognition of the position of the arrow can now be done thanks to ladders of leaves as just defined, thanks to our next proposition:

\begin{prop}\label{prop:LadderLeaf}
Let $X$ be a graph and $n \geq 2$ an integer. Let $P_1,Q_1, \ldots, P_k,Q_k$ be an $(\epsilon,L)$-ladder of leaves for some $\epsilon\geq 0$ and $L> 3 \epsilon$. If $(a,p)$ and $(b,q)$ are two vertices satisfying $(a,p) \in P_1^{+\eta} \cap Q_1^{+\eta}$ and $(b,q) \in P_k^{+\eta} \cap Q_k^{+\eta}$ for some $\eta \geq 0$, then $d(p,q) \leq 6\eta$. 
\end{prop}

\begin{proof}
Let $a_1,b_1, \ldots, a_k,b_k \in \mathbb{Z}_n^{(X)}$ be colourings such that $P_i=X(a_i)$ and $Q_i=X(b_i)$ for every $1 \leq i \leq k$. By applying Lemma~\ref{lem:SquareLeaf} inductively to the $(\epsilon,L)$-squares of leaves $P_i,P_{i+1},Q_{i+1},Q_i$ for $1 \leq i \leq k-1$, we find that $b_i-a_i = b_j-a_j$ for all $1 \leq i ,j \leq k$. 

\medskip \noindent
Fix two vertices $(a_1,p_1) \in P_1$ and $(b_1,q_1) \in Q_1$ at distance $\leq \eta$ from $(a,p)$. Since $d((a_1,p_1),(b_1,q_1)) \leq 2 \eta$, necessarily the support of $b_1-a_1$ must be contained in the ball $B(p_1,2 \eta)$. On the other hand, fixing two vertices $(a_k,p_k) \in P_k$ and $(b_k,q_k) \in Q_k$ at distance $\leq \eta$ from $(b,q)$, we find similarly that $b_k-a_k$ has its support contained in $B(q_k,2\eta)$. Thus, the ball $B(p_1,2\eta)$ and $B(q_k,2\eta)$ intersect, since they both contain the support of $b_1-a_1=b_k-a_k$. Consequently, $d(p_1,q_k) \leq 4 \eta$. We conclude that
$$d(p,q) \leq d(p,p_1) + d(p_1,q_k)+d(q_k,q) \leq d((a,p),(a_1,p_1)) + 4 \eta + d((b_k,q_k),(b,q)) \leq 6\eta,$$
as desired.
\end{proof}

\noindent
Now, we have all the tools required to prove Theorem~\ref{thm:AptolicQI}. In fact, we are going to prove a stronger statement, extracting information not only about quasi-isometries but also about coarse embeddings. 

\begin{definition}
A coarse embedding $\varphi : \mathcal{L}_n(X) \to \mathcal{L}_m(Y)$ between two lamplighter graphs is \emph{aptolic} if there exists an injective map $\alpha : \mathbb{Z}_n^{(X)} \to \mathbb{Z}_m^{(Y)}$ and a coarse embedding $\beta : X \to Y$ such that $\varphi : (c,p) \mapsto (\alpha(c),\beta(p))$. 
\end{definition}

\noindent
Our main result in this section is then:

\begin{thm}\label{thm:AptolicCoarse}
Let $X,Y$ be two graphs and $n,m \geq 2$ two integers. If $X$ is an unbounded pancylindrical graph of bounded degree, then every coarse embedding $\mathcal{L}_n(X) \to \mathcal{L}_m(Y)$ either has its image contained in a neighbourhood of some leaf or it lies at finite distance from an aptolic coarse embedding.
\end{thm}

\noindent
Before turning the proof of the theorem, let us record for future use the following elementary observation of graph theory:

\begin{lemma}\label{lem:Partition}
Let $X$ be a graph of bounded degree. For every $L \geq 0$, there exists a partition $V(X)=V_1 \sqcup \cdots \sqcup V_N$ such that $d(a,b) \geq L$ for all $1 \leq i< j \leq N$, $a \in V_i$, and $b \in V_j$. 
\end{lemma}

\begin{proof}
Let $\Xi$ be the graph whose vertex-set is $V(X)$ and such that two elements in $V(X)$ are linked by an edge if and only if they are at distance $\leq L$ in $X$. Because $X$ has bounded degree, so does $\Xi$, i.e. there exists some $N \geq 1$ such that every vertex of $\Xi$ has degree $\leq N -1$. Then, one can colour the vertices of $\Xi$ with $N$ colours so that any two adjacent vertices have different colours. The partition of $V(X)$ induced by this colouring is the partition we are looking for. 
\end{proof}

\begin{proof}[Proof of Theorem~\ref{thm:AptolicCoarse}.]
Let $\varphi : \mathcal{L}_n(X) \to \mathcal{L}_m(Y)$ be a coarse embedding. Fix two non-decreasing maps $\rho_-,\rho_+ : [0, + \infty) \to [0,+ \infty)$ satisfying $\rho_{\pm}(t) \to + \infty$ as $t \to + \infty$ such that
$$\rho_-(d(x,y)) \leq d( \varphi(x), \varphi(y)) \leq \rho_+( d(x,y)) \text{ for all } x,y \in \mathcal{L}_n(X).$$
Because $X$ is pancylindrical, it follows from Theorem~\ref{thm:Embedding} (and Remark~\ref{remark:UniformConstant}) that there exists a constant $K \geq 0$ such that, for every colouring $c \in \mathcal{Z}_n^{(X)}$, $\varphi$ sends the $X$-leaf $X(c)$ in the $K$-neighbourhood of the $Y$-leaf $Y(\alpha(c))$. We are interested in the injectivity of $\alpha$. More precisely, we want to prove that either $\alpha$ is injective or it is constant. We start by proving the following observation:

\begin{claim}\label{claim:ForAlphaInjective}
For every $\epsilon>0$, there exists $L_0 \geq 0$ such that the following holds for every $L \geq L_0$. Given an $(\epsilon,L)$-square of leaves $\{ X(c_i) \mid i \in \mathbb{Z}_4 \}$ in $\mathcal{L}_n(X)$, the $\alpha(c_i)$ are either pairwise distinct or all identical.
\end{claim}

\noindent
According to Claim~\ref{claim:InterLeaves}, there exists a constant $D \geq 0$ such that $P^{+\rho_+(\epsilon)+2K} \cap Q^{+\rho_+(\epsilon)+2K}$ has diameter $\leq D$ for any two distinct $Y$-leaves $P$ and $Q$. Let $L_0 \geq 0$ be sufficiently large so that $\rho_-(L_0) > \max \left( D+2K, 4 \rho_+(\epsilon) + 10K\right)$. Now, set some $L \geq L_0$. If the $\alpha(c_i)$ are pairwise distinct, we are done, so we assume that at least two of these colourings coincide. 

\medskip \noindent
First, assume that two opposite leaves of our square are sent in a neighbourhood of some common $Y$-leaf, say $\alpha(c_0)= \alpha(c_2)$. Because $X(c_1)$ contains two points in $X(c_0)^{+\epsilon}$ and $X(c_2)^{+\epsilon}$ that lie at distance $\geq L$ from each other, it follows that $Y(\alpha(c_1))$ contains two points in $Y(\alpha(c_0))^{+\rho_+(\epsilon)+2K} = Y(\alpha(c_2))^{+\rho_+(\epsilon)+2K}$ that lie at distance $\geq \rho_-(L) -2K >D$. We deduce from the definition of $D$ that $\alpha(c_0)= \alpha(c_1)$. Similarly, we obtain that $\alpha(c_0)= \alpha(c_3)$. Thus, we have proved that our $X$-leaves are all sent in a neighbourhood of a common $Y$-leaf as desired.

\medskip \noindent
Next, assume that two consecutive leaves of our square are sent in a neighbourhood of some common $Y$-leaf, say $\alpha(c_0)= \alpha(c_1)$. If $\alpha(c_2)$ (resp.\ $\alpha(c_3)$) coincide with $\alpha(c_0)$ (resp.\ $\alpha(c_1)$), then the previous case applies. So we can assume that $\alpha(c_2), \alpha(c_3) \neq \alpha(c_0)= \alpha(c_1)$. If $\alpha(c_2)= \alpha(c_3)$, then $Y(\alpha(c_2))= Y(\alpha(c_3))$ and $Y(\alpha(c_0)) = Y(\alpha(c_1))$ are two leaves containing two points in their $(\rho_+(\epsilon)+2K)$-neighbourhoods that lie at distance $\geq \rho_-(L)-2K$ from each other, and we conclude as before that the $\alpha(c_i)$ all agree. So we can assume that $\alpha(c_2) \neq \alpha(c_3)$. 

\medskip \noindent
\begin{minipage}{0.46\linewidth}
\includegraphics[width=\linewidth]{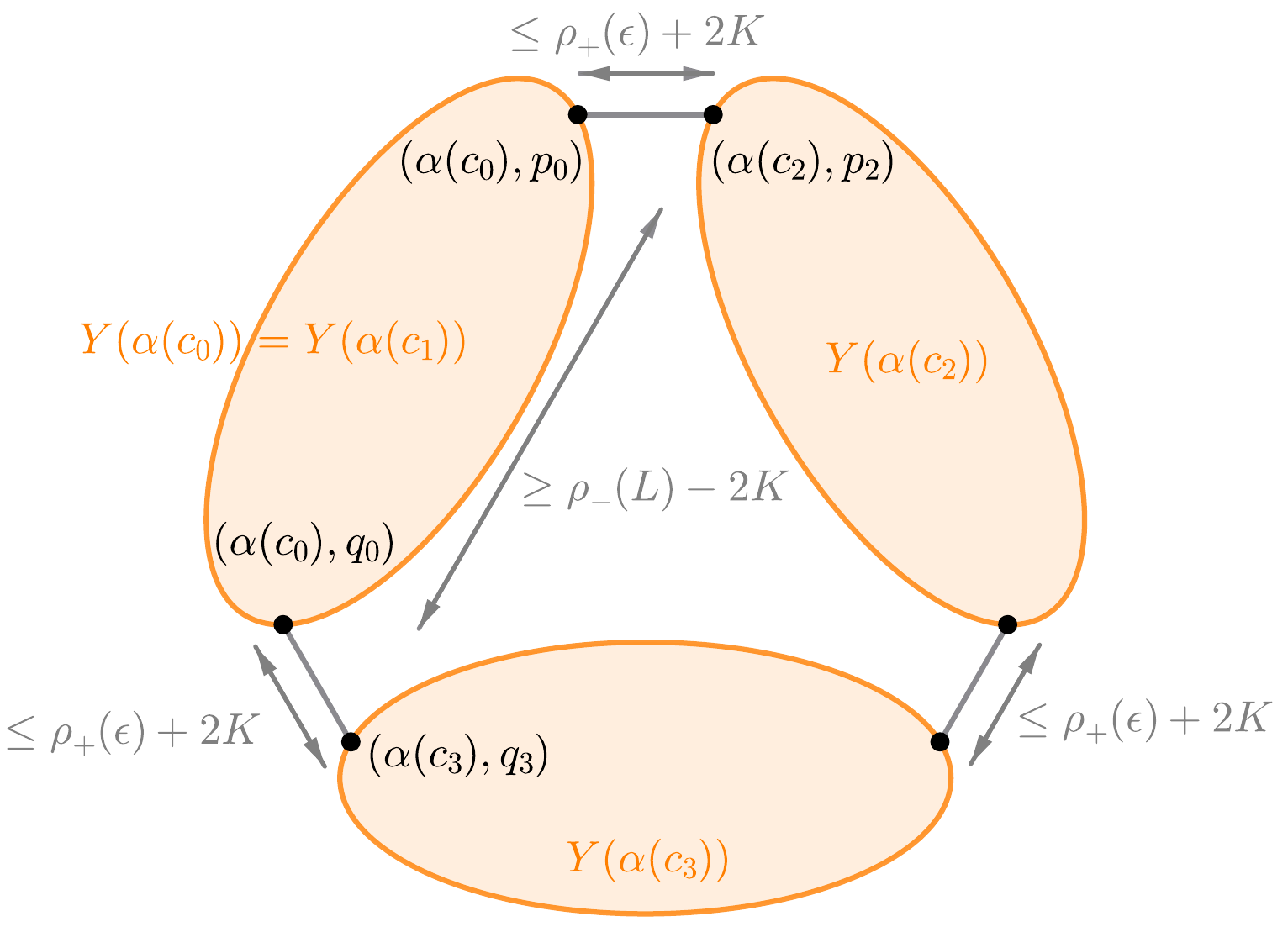}
\end{minipage}
\begin{minipage}{0.53\linewidth}
In other words, we obtain a triangle of three $Y$-leaves $Y(\alpha(c_0))= Y(\alpha(c_1))$, $Y(\alpha(c_2))$, and $Y(\alpha(c_3))$. See the figure on the left. We deduce from our configuration that
$$\alpha(c_2) \in \alpha(c_0) + \mathbb{Z}_m^{B(p_0, \rho_+(\epsilon)+2K)}$$
and 
$$\alpha(c_3) \in \alpha(c_0) + \mathbb{Z}_m^{B(q_0, \rho_+(\epsilon)+2K)}$$
where $d(p_0,q_0) \geq \rho_-(L)-2K$. 
\end{minipage}

\medskip \noindent
Consequently, since
$$\alpha(c_2) \triangle \alpha(c_3) = \left( \alpha(c_2) \triangle \alpha(c_0) \right) \triangle \left( \alpha(c_0) \triangle \alpha(c_3) \right)$$
intersects both $B(p_0, \rho_+(\epsilon)+2K)$ and $B(q_0, \rho_+(\epsilon)+2K)$, it follows that $\alpha(c_2) \triangle \alpha(c_3)$ has diameter $\geq \rho_-(L)-6K -2 \rho_+(\epsilon)$. But, on the other hand, since $Y(\alpha(c_2))$ and $Y(\alpha(c_3))$ contain points at distance $\leq \rho_+(\epsilon)+2K$, $\alpha(c_2) \triangle \alpha(c_3)$ must be contained in a ball of radius $\rho_+(\epsilon)+2K$, and consequently must have diameter $\leq 2 \rho_+(\epsilon)+4K$. Since we chose $L$ sufficiently large, we get a contradiction, concluding the proof of Claim~\ref{claim:ForAlphaInjective}. 

\medskip \noindent
Now, we are ready to show that, if $\alpha$ is not injective, then it must be constant, which amounts to saying that our coarse embeddings $\varphi$ sends $\mathcal{L}_n(X)$ inside the a neighbourhood of a single $Y$-leaf of $\mathcal{L}_m(Y)$. So let $a,b \in \mathbb{Z}_n^{(X)}$ be two distinct colourings satisfying $\alpha(a)= \alpha(b)$. Given an arbitrary colouring $c \in \mathbb{Z}_n^{(X)}$, we want to prove that $\alpha(c)= \alpha(a)= \alpha(b)$. 

\medskip \noindent
\begin{minipage}{0.42\linewidth}
\includegraphics[width=\linewidth]{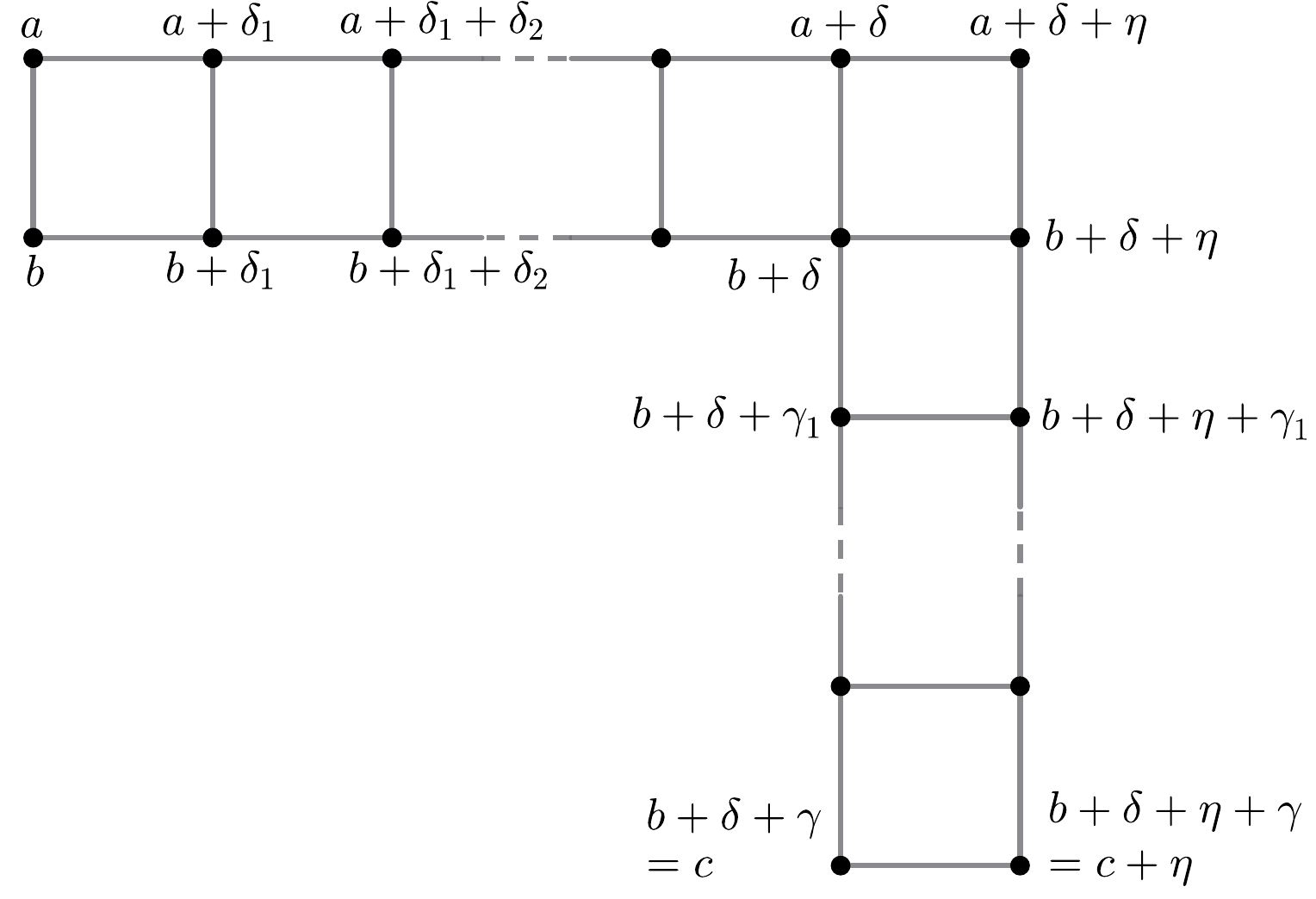}
\end{minipage}
\begin{minipage}{0.56\linewidth}
Fix an $\epsilon>0$ such that the distance between the leaves $X(a)$ and $X(b)$ is $ \leq \epsilon$, and let $L \geq 0$ be a sufficiently large constant as given by Claim~\ref{claim:ForAlphaInjective}. Because $a$ and $b$ agree outside some ball $B(o,\epsilon)$, there exists a colouring $\delta$ supported outside $B(o,\epsilon+ L)$ such that $a+ \delta$ and $b+ \delta$ both agree with $c$ outside $B(\epsilon+L)$. We can decompose $\delta$ as a sum $\delta_1+ \cdots+ \delta_r$ of colourings supported on single vertices. We can use such a data to construct an $(\epsilon,L)$-ladder of leaves as shown by the figure on the left. 
\end{minipage}

\medskip \noindent
Next, fix an arbitrary vertex $q$ not in the ball $B(o,\epsilon+2L)$ and a non-trivial colouring $\eta$ whose support is $\{q\}$. Because $b+ \delta$ coincides with $c$ outside $B(o,\epsilon+L)$, we can find a colouring $\gamma$ supported in $B(o,\epsilon+L)$ such that $b+\delta+ \gamma=c$. We can decompose $\delta$ as a sum $\gamma_1+ \cdots + \gamma_s$ of colourings supported on single vertices. Then, we can use such a data to construct an $(\epsilon,L$-ladder of leaves as shown by the figure above. 

\medskip \noindent
By applying Claim~\ref{claim:ForAlphaInjective} iteratively to the squares of our two ladders, we deduce that all the corresponding colourings have the same image under $\alpha$. In particular, $\alpha(c)= \alpha(a)= \alpha(b)$, as desired. Thus, we have proved that, if $\alpha$ is not injective, then our coarse embedding $\varphi : \mathcal{L}_n(X) \to \mathcal{L}_m(Y)$ has its image entirely contained in a neighbourhood of a single leaf. 

\medskip \noindent
Now, our goal is to prove that, under the assumption that $\alpha$ is injective, then $\varphi$ must be aptolic up to finite distance. We know that $\varphi(X(c)) \subset Y(\alpha(c))^{+K}$ for every $c \in \mathbb{Z}_n^{(X)}$. Up to modifying $\varphi$ up to finite distance, we can assume that $\varphi(X(c)) \subset Y(\alpha(c))$ for every $c \in \mathbb{Z}_n^{(X)}$. In other words, we have a collection of coarse embeddings $\{ \beta_c : X \to Y \mid c \in \mathbb{Z}_n^{(X)} \}$ such that
$$\varphi(c,p) = (\alpha(c), \beta_c(p)) \text{ for every } (c,p) \in \mathcal{L}_n(X).$$
In order to conclude the proof of our theorem, we want to show that each $\beta_c$ lies at uniform finite distance from $\beta_0$, which will imply that our coarse embedding $\varphi$ lies at finite distance from the aptolic coarse embedding $(c,p) \mapsto (\alpha(c), \beta_0(p))$. More precisely, we fix a vertex $p \in V(X)$ and a constant $L \geq 0$ such that $\rho_-(L) > 3\rho_+(1)$, and our goal is to prove that $d(\beta_c(p),\beta_0(p)) \leq 6 \rho_+(1)$. 

\medskip \noindent
Fix a partition $V(X)= V_1 \sqcup \cdots \sqcup V_N$ as given by Lemma~\ref{lem:Partition}. We decompose our colouring $c$ as a sum $c_0 +c_1+ \cdots +c_N$ such that $c_0$ is the restriction of $c$ to the ball $B(p,L)$ and each $c_i$ for $1 \leq i \leq N$ is supported on $V_i \backslash B(p,L)$. For convenience, we set $a_i:= c_1+ \cdots + c_i$ for every $0 \leq i \leq N$. (By convention, $a_0:=0$.) We have
\begin{equation}\label{equation:First}
d(\beta_0(p), \beta_c(p)) \leq \sum\limits_{i=0}^{N-1} d( \beta_{a_i}(p), \beta_{a_{i+1}}(p))  + d(\beta_{a_N}(p), \beta_c(p)).
\end{equation}
First, notice that
$$\begin{array}{lcl} d(\beta_{a_N}(p), \beta_c(p)) & \leq & d( (\alpha(a_N),\beta_{a_N}(p)), (\alpha(c),\beta_c(p))) \\ \\ & \leq & d( \varphi(a_N,p), \varphi(c,p)) \leq \rho_+( d( (a_N,p), (c,p))). \end{array}$$
But $c$ and $a_N$ may only disagree on $B(p,L)$, so
$$d( (a_N,p), (c,p)) \leq \mathrm{diam}(B(p,L)) |B(p,L)| + |B(p,L)| \leq (1+ 2L) \mathrm{deg}(X)^L.$$
Hence
\begin{equation}\label{equation:Second}
d(\beta_{a_N}(p), \beta_c(p)) \leq \rho_+\left( (1+2L) \mathrm{deg}(X)^L \right).
\end{equation}

\medskip \noindent
\begin{minipage}{0.45\linewidth}
\begin{center}
\includegraphics[width=\linewidth]{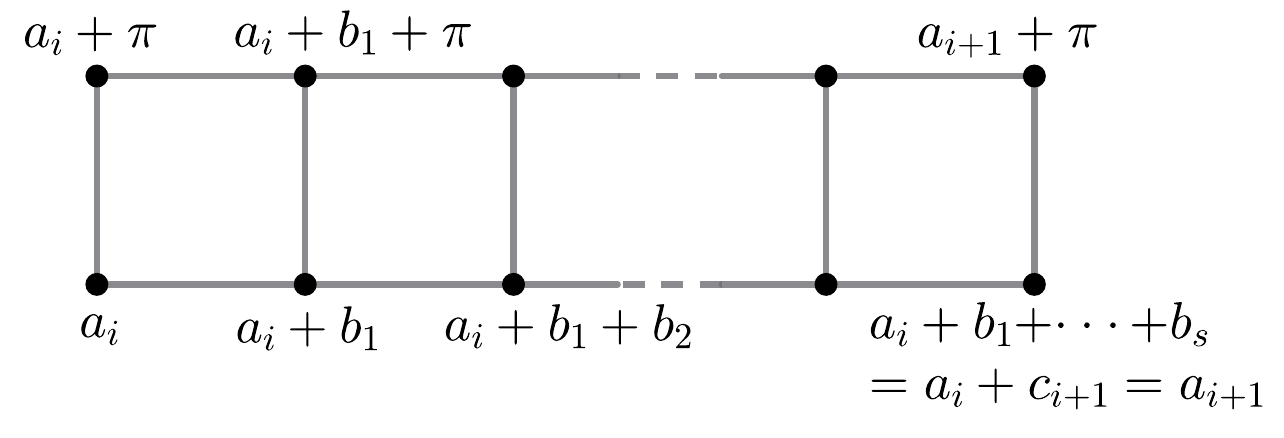}
\end{center}
\end{minipage}
\begin{minipage}{0.54\linewidth}
Next, given an index $0 \leq i \leq N-1$, decompose $c_{i+1}$ as a sum $b_1+ \cdots + b_s$ of colourings supported on single vertices (which belong to the support of $c_{i+1}$, and a fortiori to $V_{i+1} \backslash B(p,L)$). Then, we have a $(1, L)$-ladder of leaves as illustrated by the figure on the left, where $\pi$ is a non-trivial colouring whose support is $\{p\}$.
\end{minipage}

\medskip \noindent
Since $\alpha$ is injective, $\varphi$ sends this ladder to a $(\rho_+(1), \rho_-(L))$-ladder of leaves in $\mathcal{L}_m(Y)$. Notice that $\rho_+(1) >3 \rho_-(L)$ by definition of $L$. Thus, since we know that $(a_i,p) \in X(a_i) \cap X(a_i+ \pi)^{+1}$ and $(a_{i+1},p) \in X(a_{i+1}) \cap X(a_{i+1}+\pi)^{+1}$, which implies that $(\alpha(a_i),\beta_{a_i}(p)) \in Y(\alpha(a_i)) \cap Y(\alpha(a_i+ \pi))^{+1}$ and $(\alpha(a_{i+1}), \beta_{a_{i+1}}(p)) \in Y( \alpha(a_{i+1})) \cap Y(\alpha(a_{i+1}+ \pi))^{+1}$, we can apply Proposition~\ref{prop:LadderLeaf} and deduce that $d(\beta_{a_i}(p), \beta_{a_{i+1}}(p)) \leq 6 \rho_+(1)$. Hence
\begin{equation}\label{equation:Third}
\sum\limits_{i=0}^{N-1} d \left( \beta_{a_i}(p), \beta_{a_{i+1}} (p) \right) \leq 6N \rho_+(1).
\end{equation}
By combining Equations~\ref{equation:Second} and~\ref{equation:Third} with Equation~\ref{equation:First}, we conclude that
$$d(\beta_0(p), \beta_c(p)) \leq 6N \rho_+(1) + \rho_+\left( (1+2L) \mathrm{deg}(X)^L \right),$$
which completes the proof of our theorem.
\end{proof}

\begin{proof}[Proof of Theorem~\ref{thm:AptolicQI}.]
Let $\varphi : \mathcal{L}_n(X) \to \mathcal{L}_m(Y)$ be a quasi-isometry and $\bar{\varphi} : \mathcal{L}_m(Y) \to \mathcal{L}_n(X)$ a quasi-inverse. According to Theorem~\ref{thm:AptolicCoarse}, up to replacing $\varphi$ and $\bar{\varphi}$ with maps at finite distance, we can assume that $\varphi$ and $\bar{\varphi}$ are aptolic coarse embeddings. Consequently, there exist two injective maps $\alpha : \mathbb{Z}_n^{(X)} \to \mathbb{Z}_m^{(Y)}$ and $\bar{\alpha} : \mathbb{Z}_m^{(Y)} \to \mathbb{Z}_n^{(X)}$, as well as two coarse embeddings $\beta : X \to Y$ and $\bar{\beta} : Y \to X$, such that $\varphi : (c,p) \mapsto (\alpha(c), \beta(p))$ and $\bar{\varphi} : (c,p) \mapsto (\bar{\alpha}(c), \bar{\beta}(p))$. Since $\varphi \circ \bar{\varphi}$ lies at finite distance, say $K$, from the identity, we know that
$$d((\alpha \circ \bar{\alpha}(c), \beta \circ \bar{\beta}(p)), (c,p)) \leq K \text{ for every } (c,p) \in V(\mathcal{L}_n(X)).$$
This amounts to saying that 
$$d(\beta \circ \bar{\beta}(p),p) \leq K \text{ for every } p \in V(X),$$
and
$$d( (\alpha \circ \bar{\alpha}(c),p), (c,p)) \leq 2K \text{ for every } (c,p) \in V(\mathcal{L}_n(X)).$$
Notice that, given a colouring $c$, we must have $\alpha \circ \bar{\alpha}(c)=c$ since the distance between $(\alpha \circ \bar{\alpha}(c),p)$ and $(c,p)$ can be bounded below by $d(p, c \triangle \alpha \circ \bar{\alpha}(c))$ if $c$ and $\alpha \circ \bar{\alpha}(c)$ disagree. Thus, $\beta \circ \bar{\beta}$ lies at distance $\leq K$ from the identity and $\alpha \circ \bar{\alpha}$ is the identity. Symmetrically, we obtain that $\bar{\beta} \circ \beta$ lies at finite distance from the identity and that $\bar{\alpha} \circ \alpha$ is the identity. Therefore, $\alpha$ is bijective and $\beta$ is a quasi-isometry, proving that $\varphi$ is aptolic. 
\end{proof}

\subsection{A few words about amenability}

\noindent
As shown by Theorem~\ref{thm:AptolicQIstrong}, the coarse geometry of lamplighter graphs depends on whether our original graph is \emph{amenable}, a property that we know define.

\begin{definition}
A locally finite graph $X$ is \emph{amenable} if it admits a \emph{F{\o}lner sequence}, i.e.\ a sequence of finite subsets of vertices $(F_n)_{n \geq 0}$ satisfying
$$\frac{ |\partial F_n| }{|F_n|} \to 0 \text{ as } n \to + \infty.$$
\end{definition}

\noindent
Here, the \emph{boundary} $\partial S$ of a set of vertices $S$ refers to the collection of all the vertices that do not belong to $S$ but that are adjacent to some vertices of $S$. Thus, a graph is amenable if one can find sets of vertices for which the boundary is arbitrarily small compared to the size of the set itself. In the opposite direction, a graph is non-amenable whenever it satisfies a linear isoperimetric inequality: there exists some $\epsilon>0$ such that $|\partial S| \geq \epsilon |S|$ for every finite set $S$ of vertices. 

\medskip \noindent
It is worth noticing that amenability is preserved by quasi-isometries (see Exercise~\ref{exo:AmenableQI}), so it makes sense to say that a given finitely generated group is or is not amenable. Let us consider a few basic examples of amenable and non-amenable graphs and groups. We first show the amenability of (Cayley graphs of) $\mathbb{Z}^n$. 

\begin{lemma}
For every $d \geq 1$, the graph $\mathbb{E}^d$ is amenable.
\end{lemma}

\begin{proof}
For every $n \geq 1$, set $F_n:= [n]^d$ where $[n]:= \{ 1, \ldots, n\}$. Of course, $|F_n|= n^d$. Then, since
$$\partial F_n = \bigsqcup\limits_{i=1}^d \left\{ (x_i)_{1 \leq i \leq d} \in [n]^d \mid x_i \in \{0,n+1\} \text{ and } x_j \in [n] \text{ for every } j \neq i\right\},$$
we have
$$|\partial F_n| = \sum\limits_{i=1}^d 2n^{d-1}= 2dn^{d-1}.$$
Therefore,
$$\frac{|\partial F_n|}{|F_n|} = \frac{2dn^{d-1}}{n^d} = \frac{2d}{n} \to 0 \text{ as } n\to + \infty.$$
Thus, $(F_n)_{n \geq 1}$ defines a F{\o}lner sequence in $\mathbb{E}^d$, proving that $\mathbb{E}^d$ is amenable. 
\end{proof}

\noindent
In the opposite direction, free groups of finite rank $\geq 2$, or more generally regular trees of degree $\geq2$, are not amenable:

\begin{lemma}\label{lem:TreeNotAmenable}
For every $d \geq 3$, the $d$-regular tree $T_d$ is not amenable.
\end{lemma}

\noindent
Before proving our lemma, we show a couple of elementary observations regarding F{\o}lner sets. 

\begin{fact}\label{fact:FolnerNeigh}
Let $X$ be a graph of bounded degree. If $(F_n)_{n \geq 0}$ is a F{\o}lner sequence, then, for every $R \geq 0$, $(F_n^{+R})_{n \geq 0}$ is a F{\o}lner sequence as well.
\end{fact}

\begin{proof}
It suffices to notice that
$$|\partial F_n^{+R}| \leq \sum\limits_{x \in \partial F_n} |B(x,R)| \leq \mathrm{deg}(X)^R |\partial F_n|$$
for every $n \geq 0$, and to deduce that
$$\frac{|\partial F_n^{+R}|}{|F_n^{+R}|} \leq \mathrm{deg}(X)^R \frac{|\partial F_n|}{|F_n|} \underset{n \to + \infty}{\longrightarrow} 0.$$
\end{proof}

\begin{fact}\label{fact:FolnerConnected}
Let $X$ be a graph of bounded degree. If $X$ is amenable, then there exists a F{\o}lner sequence $(F_n)_{n \geq 0}$ such that $F_n$ induces a connected subgraph for every $n \geq 0$.
\end{fact}

\begin{proof}
Given some $R \geq 1$, we say that a set of vertices $S$ is $R$-coarsely connected if, for any two vertices $a,b \in S$, there exists a sequence
$$x_1=a, x_2, \ldots, x_{n-1},x_n=b$$
such that $d(x_i,x_{i+1}) \leq R$ for every $1 \leq i \leq n-1$. Our goal is to show that there exists a F{\o}lner sequence $(F_n)_{n \geq 0}$ of $2$-coarsely connected sets of vertices. Then, thanks to Fact~\ref{fact:FolnerNeigh}, we can conclude that $(F_n^{+1})_{n \geq 0}$ yields a F{\o}lner sequence where each $F_n^{+1}$ is $1$-coarsely connected, which amounts to saying that each $F_n^{+1}$ induces a connected subgraph. 

\medskip \noindent
Let us start with an arbitrary F{\o}lner sequence $(F_n)_{n \geq 0}$. For every $n \geq 0$, let
$$F_n= F_n^{(1)} \sqcup F_n^{(2)} \sqcup \cdots \sqcup F_n^{k_n}$$
denote the decomposition of $F_n$ into $2$-coarsely connected components. Notice that
$$\partial F_n= \partial F_n^{(1)} \sqcup \partial F_n^{(2)} \sqcup \cdots \sqcup \partial F_n^{k_n}.$$
We claim that, for every $\epsilon>0$, there exist $n \geq 0$ and $1 \leq i \leq k_n$ such that $|\partial F_n^{(i)}| \leq \epsilon |F_n^{(i)}|$. Otherwise, there exists $\epsilon>0$ such that $|\partial F_n^{(i)}| > \epsilon |F_n^{(i)}|$ for every $n \geq 0$ and $1 \leq i \leq k_n$. Then, 
$$|\partial F_n| = \sum\limits_{i=1}^{k_n} |\partial F_n^{(i)}| > \epsilon \sum\limits_{i=1}^{k_n} |F_n^{(i)}| = \epsilon |F_n|$$
for every $n \geq 0$, which contradicts the fact that $(F_n)_{n \geq 0}$ is a F{\o}lner sequence. Therefore, for every $j \geq 1$, we can find $n(j) \geq 0$ and $1 \leq i(j) \leq k_{n(j)}$ such that $|\partial F_{n(j)}^{i(j)}| \leq |F_{n(j)}^{i(j)}| / j$. Thus, $(F_{n(j)}^{i(j)} )_{j \geq 1}$ yields a F{\o}lner sequence of $2$-coarsely connected sets of vertices. 
\end{proof}

\begin{proof}[Proof of Lemma~\ref{lem:TreeNotAmenable}.]
If $T_d$ is amenable, then we know from Fact~\ref{fact:FolnerConnected} that it admits a F{\o}lner sequence $(F_n)_{n \geq 0}$ such that each $F_n$ induces a connected subgraph, i.e.\ a subtree. Then, thanks to Fact~\ref{fact:FolnerNeigh}, we get a F{\o}lner sequence $(F_n^{+1})_{n \geq 0}$ such that each $F_n^{+1}$ is the vertex-set of a finite $d$-regular tree. Thus, in order to show that $T_d$ is not amenable, it suffices to notice that:

\begin{claim}
For every finite $d$-regular tree $T$,  
$$|\partial V(T)| \geq (d-2) |V(T)| + 2.$$
\end{claim}

\noindent
We argue by induction over the number of vertices of $T$. If $T$ is reduced to a single vertex, then $|\partial V(T) | = d$.  

\medskip \noindent
\begin{minipage}{0.2\linewidth}
\includegraphics[width=0.98\linewidth]{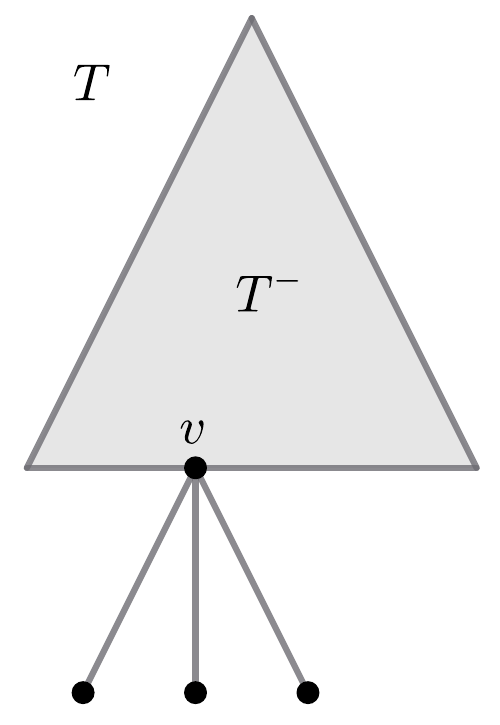}
\end{minipage}
\begin{minipage}{0.78\linewidth}
From now on, we assume that $T$ has at least two vertices. Fix a vertex $v \in V(T)$ adjacent to some leaf and let $T^-$ denote the subtree obtained by removing the leaves adjacent to $v$. We have 
$$|\partial V(T)| = |\partial V(T^-)| - (d-1) + (d-1)^2,$$
hence
$$|\partial V(T)| \geq (d-2) |V(T^-)|  +2 - (d-1) + (d-1)^2$$
\end{minipage}

\medskip \noindent
as $T^-$ is a $d$-regular tree smaller than $T$. Since $|V(T)| = |V(T^-)| + (d-1)$, we conclude that
$$|\partial V(T)| \geq (d-2) |V(T)|   +2 - (d-2)(d-1)  -(d-1) + (d-1)^2 = (d-2) |V(T)| +2, $$
which concludes the proof of our claim. 
\end{proof}

\noindent
Proving that some graphs or groups are not amenable may be tricky, because one has to consider boundaries of all finite subsets. For finitely generated groups, a convenient strategy is to find a simpler subgroup for which we already know that it is not amenable. Indeed:

\begin{lemma}\label{lem:SubAmenable}
Let $G$ be a finitely generated group. If $G$ is amenable, then all its finitely generated subgroups are amenable.
\end{lemma}

\noindent
We emphasize that the result is really algebraic: the pointed sum of an amenable graph with a non-amenable graph is amenable, despite the fact that it contains a(n isometrically embedded) non-amenable subgraph. Even worth, many amenable groups (e.g.\ $\mathbb{Z}_2 \wr \mathbb{Z}$) contains a (quasi-isometrically embedded) free subsemigroups, despite the fact that they cannot contain non-abelian free subgroups. (See Exercise~\ref{exo:Subsemigroups}.) 

\begin{proof}[Proof of Lemma~\ref{lem:SubAmenable}.]
Let $H \leq G$ be a finitely generated subgroup and let $(F_n)_{n \geq 0}$ be a F{\o}lner sequence in (a Cayley graph of) $G$. For every $n \geq 0$, we have the decomposition
$$F_n = F_n^{(1)} \sqcup F_n^{(2)} \sqcup \cdots \sqcup F_n^{k_n}$$
where the $F_n^{(i)}$ denote the non-empty intersections of $F_n$ with the cosets of $H$. We claim that, for every $\epsilon>0$, there exist $n \geq 0$ and $1 \leq i \leq k_n$ such that $|\partial_H F_n^{(i)}| / |F_n^{(i)}| \leq \epsilon$, where $\partial_H$ denotes the boundary in the corresponding $H$-coset. Otherwise, there exists $\epsilon>0$ such that $|\partial_H F_n^{(i)}| / |F_n^{(i)}| > \epsilon$ for all $n \geq 0$ and $1 \leq i \leq k_n$. If so,
$$\frac{|\partial F_n|}{|F_n|} \geq \frac{1}{|F_n|} \sum\limits_{i=1}^{k_n} |\partial_H F_n^{(i)}| > \epsilon,$$
contradicting the fact that $|\partial F_n| / |F_n| \to 0$ as $n \to + \infty$. This proves our claim. As a consequence, we can find two sequences of indices $(r(n))_{n \geq 0}$ and $(i(n))_{n \geq 0}$ such that 
$$\frac{|\partial_H F_{r(n)}^{i(n)} | }{|F_{r(n)}^{i(n)}|} \to 0 \text{ as } n \to + \infty.$$
Thus, thinking of each $F_{r(n)}^{i(n)}$ as a subset of $H$ (formally, this is a subset of an $H$-coset), we conclude that $(F_{r(n)}^{(i(n))})_{n \geq 0}$ defines a F{\o}lner sequence in $H$, proving that $H$ is indeed amenable. 
\end{proof}

\noindent
As an immediate consequence of our discussion:

\begin{cor}\label{cor:NotAmenable}
If a finitely generated group contains a non-abelian free subgroup, then it is not amenable.
\end{cor}

\noindent
This elementary obstruction to amenability turns out to be very useful. In fact, it was thought for a while that every non-amenable finitely generated group must contain a non-abelian free subgroup. Known as the von Neumann conjecture, this was finally disproved. See for instance \cite{MR586204, MR682486, MR1985031, MR3047655} for counterexamples. 

\begin{proof}[Proof of Corollary~\ref{cor:NotAmenable}.]
Combine Lemmas~\ref{lem:SubAmenable} and~\ref{lem:TreeNotAmenable}. 
\end{proof}

\noindent
We conclude this section by showing that amenability is preserved by wreath products. 

\begin{lemma}
Let $(X,o)$ be a pointed graph and $Y$ a graph. If $X$ and $Y$ are both amenable, then $(X,o) \wr Y$ is amenable as well.
\end{lemma}

\begin{proof}
Let $(A_n)_{n \geq 0}$ (resp.\ $(B_n)_{n \geq 0}$) be a F{\o}lner sequence in $X$ (resp.\ $Y$). For every $n \geq 0$, we set
$$F_n:= \{ (c,p) \in (X,o) \wr Y \mid c(q)=o \ \forall q \notin B_n, c(q) \in A_n \ \forall q \in B_n, p \in B_n\}.$$
Loosely speaking, $F_n$ is the ``wreath product $A_n \wr B_n$'' of the F{\o}lner subsets. Clearly, $|F_n| = |B_n| \cdot |A_n|^{|B_n|}$. Also, we have
$$\begin{array}{lcl} \partial F_n & = & \displaystyle \bigsqcup\limits_{p \in B_n} \{ (c,p) \mid c(q)=o \ \forall q \notin B_n, c(q) \in A_n \ \forall q \in B_n\backslash \{p\}, c(p) \in \partial A_n \} \\ \\ & & \sqcup \{ (c,p) \mid c(q)=o \ \forall q \notin B_n, c(q) \in A_n \ \forall q \in B_n, p \in \partial B_n\}, \end{array}$$
hence $|\partial F_n| = |B_n| \cdot |\partial A_n| + |\partial B_n|$. Since
$$\frac{|\partial F_n|}{|F_n|} = \frac{|\partial A_n|}{|A_n|^{|B_n|}} + \frac{|\partial B_n|}{|B_n| \cdot |A_n|^{|B_n|}} \leq \frac{|\partial A_n|}{|A_n|} + \frac{|\partial B_n|}{|B_n|} \to 0 \text{ as } n\to + \infty,$$
we conclude that $(F_n)_{n \geq 0}$ defines a F{\o}lner sequence in $(X,o) \wr Y$, and finally, as desired, that $(X,o) \wr Y$ is amenable. 
\end{proof}

\subsection{Measure-scaling quasi-isometries}\label{section:MeasureScaling}

\noindent
A natural question, which came early in the history of geometric group theory, is: are two quasi-isometric finitely generated groups necessarily biLipschitz equivalent? Or, in other words: is a quasi-isometry between two finitely generated groups always at finite distance from a bijection? For non-amenable groups, this question turns out to have a positive answer:

\begin{thm}[\cite{MR1700742}]\label{thm:BijNonQI}
Let $X,Y$ be two graphs of bounded degree. If $X$ and $Y$ are non-amenable, then every quasi-isometry $X \to Y$ lies at finite distance from a bijection. 
\end{thm}

\noindent
For amenable groups, the situation is different, and understanding when a quasi-isometry lies at finite distance from a quasi-isometry leads to the introduction of a concept that will be fundamental for us:

\begin{definition}\label{def:QuasiToOne}
Let $X,Y$ be two graphs of bounded degree and $\kappa>0$ a real number. A quasi-isometry $\varphi : X \to Y$ is \emph{quasi-$\kappa$-to-one} if there exists a constant $C >0$ such that
$$\left| |\varphi^{-1}(S)| - \kappa |S| \right| \leq C \cdot | \partial S|$$
for every finite subset $S \subset V(Y)$. 
\end{definition}

\noindent
It is worth noticing that a quasi-isometry between two non-amenable graphs of bounded degree is quasi-$\kappa$-to-one for every $\kappa>0$ (see Exercise~\ref{exo:QuasiToOne}), so our definition is relevant only for amenable graphs. Also, notice that, given an integer $n \geq 1$, a quasi-isometry that is $n$-to-one is quasi-$n$-to-one. In particular, a bijection is quasi-one-to-one. In fact, it turns out that a quasi-isometry lies at finite distance from a bijection if and only if it is quasi-one-to-one. More generally, the following statement provides a more intuitive interpretation of quasi-$\kappa$-to-one quasi-isometries when $\kappa$ is a rational number.

\begin{prop}[\cite{MR4419103}]\label{prop:kappa}
Let $X,Y$ be two graphs of bounded degree and $m,n \geq 1$ two integers. A quasi-isometry $f:X\to Y$ is quasi-$(m/n)$-to-one if and only if there exist a partition $\mathcal{P}_X$ (resp. $\mathcal{P}_Y$) of $X$ (resp. of $Y$) with uniformly bounded pieces of size $m$ (resp. $n$) and a bijection $\psi:\mathcal{P}_X\to \mathcal{P}_Y$ such that $f$ is at bounded distance from a map $g : X \to Y$ satisfying $g(P) \subset \psi(P)$ for every $P \in \mathcal{P}_X$.
\end{prop}

\noindent
A difficulty with Definition~\ref{def:QuasiToOne} is that it involves every finite subset of a given graph. Our next lemma shows that, actually, it suffices to verify the definition only for thick subsets. Recall that, given an $R \geq 0$, a set of vertices is \emph{$R$-thick} if it can be written as a union of balls of radii $R$. 

\begin{lemma}\label{lem:QuasiThick}
Let $X,Y$ be two graphs of bounded degree and $\kappa>0$ a real number. A quasi-isometry $\varphi : X \to Y$ is quasi-$\kappa$-to-one if and only if there exist $C>0$ and $R \geq 0$ such that 
$$\left| |\varphi^{-1}(S)| - \kappa  |S| \right| \leq C \cdot | \partial S|$$ 
for every finite $R$-thick subset $S \subset V(Y)$.
\end{lemma}

\noindent
During the proof, the following elementary observation will be needed:

\begin{fact}\label{fact:Boundary}
Let $X$ be a graph of bounded degree. For every $R \geq 0$, there exists $C>0$ such that $|S^{+R}\backslash S| \leq C \cdot |\partial S|$ for every finite $S \subset V(X)$. 
\end{fact}

\noindent
The proof is left as an exercise. 

\begin{proof}[Proof of Lemma~\ref{lem:QuasiThick}.]
Assume that there exist $C>0$ and $R \geq 0$ such that
$$\left| |\varphi^{-1}(S)| - \kappa |S| \right| \leq C \cdot | \partial S|$$
for every finite subset $S \subset V(Y)$. Given an arbitrary subset $S \subset V(Y)$, clearly $S^{+R}$ is $R$-thick. We have
$$\begin{array}{lcl}\left| |\varphi^{-1}(S)| - \kappa |S| \right| & \leq & \left| | \varphi^{-1}(S)|- |\varphi^{-1}(S^{+R})| \right| + \left| |\varphi^{-1}(S^{+R})| - \kappa |S^{+R}| \right| + \kappa \left| |S^{+R}| - |S| \right| \\ \\ & \leq & \left| \varphi^{-1} \left( S^{+R} \backslash S \right) \right| + C \cdot |\partial S^{+R}| + \kappa \left| S^{+R} \backslash S \right|.\end{array}$$
Notice that, since $\varphi$ is a quasi-isometry, there exists some $N \geq 0$ such that the preimage of a single vertex under $\varphi$ always has size $\leq N$, which implies that $|\varphi^{-1}(S^{+R} \backslash S)| \leq N \cdot |S^{+R} \backslash S|$. It is also clear that $\partial S^{+R} \subset S^{+R+1} \backslash S$, hence
$$\left| |\varphi^{-1}(S)| - \kappa |S| \right|  \leq (N+ \kappa) \cdot \left| S^{+R} \backslash S \right| + C \cdot \left| S^{+R+1} \backslash S \right| \leq C' \cdot |\partial S|,$$
where the constant $C' >0$ is given by Fact~\ref{fact:Boundary}. We conclude that $\varphi$ is quasi-$\kappa$-to-one, as desired. 
\end{proof}

\noindent
We conclude this subsection with the following technical criterion, which we record for future use:

\begin{cor}\label{cor:QuasiCriterion}
Let $X,Y$ be two graphs of bounded degree, $\kappa>0$ a real number, and $\varphi : X \to Y$ a quasi-isometry. If there exist $C >0$ and $K_0 \geq 0$ such that
$$\left| |\varphi(S)^{+K}| - \kappa |S| \right| \leq C \cdot |\partial S|$$
for every $K \geq K_0$ and every finite subset $S \subset V(X)$, then $\varphi$ is quasi-$\kappa^{-1}$-to-one.
\end{cor}

\begin{proof}
Fix a constant $R \geq 0$ such that every ball of radius $R$ in $Y$ intersects the image of $\varphi$, and a constant $K \geq K_0, 2R$. For every $R$-thick subset $T \subset V(Y)$, we have
\begin{equation}\label{eq:ForQuasiKappa}
\left| \varphi( \varphi^{-1}(T))^{+K} - \kappa | \varphi^{-1}(T)| \right| \leq C \cdot | \partial \varphi^{-1}(T)|.
\end{equation}
We will deduce from this inequality that $\varphi$ is quasi-$\kappa^{-1}$-to-one. For this, we need two preliminary observations.

\begin{claim}\label{claim:ForQuasiKappaOne}
There exists a constant $C_1>0$ such that $|\partial \varphi^{-1}(T)| \leq C_1 \cdot |\partial T|$.
\end{claim}

\noindent
A vertex belongs to $\partial \varphi^{-1}(T)$ if precisely when it is not sent in $T$ by $\varphi$ but it is adjacent to some vertex that is sent in $T$. Consequently, such a vertex must be sent at distance $\leq A$ from $T$, for some constant $A >0$ depending only on the parameters of $\varphi$, but  not in $T$. In other words, $\partial \varphi^{-1}(T) \subset \varphi^{-1}(T^{+A}\backslash T)$. Since $\varphi$ is a quasi-isometry and that $X$ has bounded degree, there exists a constant $B \geq 0$ such that the pre-image of a vertex un $\varphi$ always has size $\leq B$. Hence
$$|\partial \varphi^{-1}(T)| \leq |\varphi^{-1}(T^{+A} \backslash T)| \leq B \cdot | T^{+A}\backslash A|.$$
The desired conclusion the follows from Fact~\ref{fact:Boundary}.

\begin{claim}\label{claim:ForQuasiKappaTwo}
There exists $C_2 \geq 0$ such that $|T| \leq |\varphi(\varphi^{-1}(T))^{+K}| \leq |T| + C_2 \cdot |\partial T|$.
\end{claim}

\noindent
First, notice that $\varphi(\varphi^{-1}(T))^{+K}$ contains $T$, which justifies our first inequality. Indeed, since $T$ is $R$-thick, we can write $T=\bigcup_{t \in T_0} B(t,R)$ for some $T_0 \subset T$. Fix a $t \in T_0$. Because we chose $R$ sufficiently large, $B(t,R)$ intersects $\mathrm{Im}(\varphi)$. It follows that $B(t,R)$ is contained in $\varphi(\varphi^{-1}(T))^{+K} = (T \cap \mathrm{Im}(\varphi))^{+K}$ as $K \geq 2R$. Hence $T \subset \varphi(\varphi^{-1}(T))^{+K}$, as desired.

\medskip \noindent
Next, because $\varphi (\varphi^{-1}(T)) \subset T \subset \varphi(\varphi^{-1}(T))^{+K} \subset T^K$, we have
$$\left| |\varphi(\varphi^{-1}(T))^{+K}| - |T| \right| \leq \left| \varphi(\varphi^{-1}(T)) \backslash T \right| \leq \left| T^{+K} \backslash T \right| \leq C_2 \cdot |\partial T|$$
for some constant $C_2>0$, as justified by Fact~\ref{fact:Boundary}. This concludes the proof of Claim~\ref{claim:ForQuasiKappaTwo}.

\medskip \noindent
Thanks to Claims~\ref{claim:ForQuasiKappaOne} and~\ref{claim:ForQuasiKappaTwo}, we deduce from Equations~\ref{eq:ForQuasiKappa} that
$$\begin{array}{lcl}\left| \kappa |\varphi^{-1}(T)| -  |T| \right| & \leq &  \left|\varphi(\varphi^{-1}(T))^{+K}-  \kappa |\varphi^{-1}(T)| \right| + \left| \varphi(\varphi^{-1}(T))^{+K} - |T| \right| \\ \\ & \leq & C \cdot |\partial \varphi^{-1}(T)| + C_2 \cdot |\partial T| \leq (C_1+C_2) |\partial T|. \end{array}$$
We conclude from Lemma~\ref{lem:QuasiThick} that $\varphi$ is quasi-$\kappa^{-1}$-to-one, as desired. 
\end{proof}

\subsection{Lamplighters over amenable graphs}

\noindent
So far, we know that quasi-isometries between the lamplighter graphs we are interested in are aptolic (Theorem~\ref{thm:AptolicQI}). Now, we want to exploit the aptolic structure of our quasi-isometries in order to find restrictions on quasi-isometric lamplighter graphs. Our main result in this direction is the following statement:

\begin{prop}\label{prop:AptolicQILamp}
Let $X,Y$ be two graphs of bounded degree and $n,m \geq 2$ two integers. If there exists an aptolic quasi-isometry $(\alpha,\beta) : \mathcal{L}_n(X) \to \mathcal{L}_m(Y)$, then $n$ and $m$ have the same prime divisors. Moreover, if $X$ and $Y$ are amenable, then $n$ and $m$ are powers of a common number, say $n=q^a$ and $m=q^b$, and $\beta$ is quasi-$\frac{b}{a}$-to-one. 
\end{prop}

\noindent
The proof of our proposition will be based on the following observation:

\begin{lemma}\label{lem:InclusionAlpha}
Let $X,Y$ be two graphs, $n,m \geq 2$ two integers, and $(\alpha,\beta) : \mathcal{L}_n(X) \to \mathcal{L}_m(Y)$  an aptolic quasi-isometry. There exists a constant $K \geq 0$ such that
$$\alpha \left( c + \mathbb{Z}_n^{S} \right) \subset \alpha(c) + \mathbb{Z}_m^{\beta(S)^{+K}}$$
for all $c \in \mathbb{Z}_n^{(X)}$ and $S \subset V(X)$ finite.
\end{lemma}

\begin{proof}
We start by verifying the statement when $S$ is reduced to a single vertex, the general case following by induction over the cardinality of $S$. So let $v \in V(X)$ be a vertex and $c \in \mathbb{Z}_n^{(X)}$ a colouring. Notice that 
$$\{ (f,v) \in \mathcal{L}_n(X) \mid f \in c+ \mathbb{Z}_n^{\{v\}} \}$$
is contained in the ball centred at $(c,v)$ of radius $1$, so it has to be sent by $(\alpha,\beta)$ in some ball centred at $(\alpha(c), \beta(v))$ of radius $K$, where $K$ only depends on the parameters of our quasi-isometry. Notice that, if $(f,p)$ belongs to such a ball, then $f$ cannot differ from $\alpha(c)$ outside the ball $\{\beta(v)\}^{+K}$. The desired inclusion
$$\alpha \left( c + \mathbb{Z}_n^{\{v\}} \right) \subset \alpha(c) + \mathbb{Z}_m^{\{\beta(v)\}^{+K}}$$
follows. Then, if $S \subset V(X)$ is an arbitrary finite subset, which we can assume to have cardinality $\geq 2$, we fix a vertex $s \in S$. We have
$$\begin{array}{lcl} \alpha \left(  c+ \mathbb{Z}_n^S \right) & = & \alpha \left( c+ \mathbb{Z}_n^{\{s\}} + \mathbb{Z}_n^{S \backslash\{s\}} \right) \subset \alpha \left( c+ \mathbb{Z}_n^{\{s\}} \right) + \mathbb{Z}_m^{\beta(S \backslash \{s\})^{+K}} \\ \\ & \subset & \alpha (c) +  \mathbb{Z}_n^{\{\beta(s)\}^{+K}}  + \mathbb{Z}_m^{\beta(S \backslash \{s\})^{+K}} = \alpha(c) + \mathbb{Z}_n^{\beta(S)^{+K}}\end{array}$$
where the first inclusion is given by induction and second induction by the case already treated above. 
\end{proof}

\begin{proof}[Proof of Proposition~\ref{prop:AptolicQILamp}.]
Fix a quasi-inverse $\bar{\beta} : Y \to X$ of $\beta$ and let $K_1 \geq 0$ (resp.\ $K_2 \geq 0$) denote the constant given by Lemma~\ref{lem:InclusionAlpha} for the aptolic quasi-isometry $(\alpha,\beta)$ (resp.\ $(\alpha^{-1}, \bar{\beta})$). 

\begin{claim}\label{claim:Arithmetic}
For all finite $S \subset V(X)$ and $K \geq K_1$, $n^{|S|}$ divides $m^{|\beta(S)^{+K}|}$. Similarly, $m^{|T|}$ divides $n^{|\bar{\beta}(T)^{+K_2}|}$ for all finite $T \subset V(Y)$ and $K \geq K_2$.
\end{claim}

\noindent
The two assertions being symmetric, we only prove the first one. So let $S \subset V(X)$ be a finite subset and $K \geq K_1$. Set
$$E:= \alpha^{-1} \left( \alpha(0) + \mathbb{Z}_m^{\beta(S)^{+K} }\right).$$
It follows from Lemma~\ref{lem:InclusionAlpha} that
$$\alpha \left( E + \mathbb{Z}_n^S \right) \subset \alpha(E) + \mathbb{Z}_m^{\beta(S)^{K}} = \alpha(0) + \mathbb{Z}_m^{\beta(S)^{+K}} + \mathbb{Z}_m^{\beta(S)^{K}} = \alpha(E).$$
hence $E = E + \mathbb{Z}_n^S$. In other words, $E$ is $\mathbb{Z}_n^S$-invariant. By decomposing $E$ as a disjoint unions of orbits under $\mathbb{Z}_n^S$, we deduce that $|\mathbb{Z}_n^S|=n^{|S|}$ divides $|E|= |\mathbb{Z}_m^{\beta(S)^{+K}}| = m^{|\beta(S)^{+K_1}|}$, concluding the proof of Claim~\ref{claim:Arithmetic}.

\medskip \noindent
Notice that Claim~\ref{claim:Arithmetic} immediately implies the first assertion in our corollary, namely $n$ and $m$ have the same prime divisors. From now on, assume that $X$ and $Y$ are amenable. 

\medskip \noindent
Given an arbitrary finite subset $S \subset V(X)$ and a $K \geq K_1,K_2$ sufficiently large compared to the parameters of $\beta$, we know from Claim~\ref{claim:Arithmetic} that 
$$n^{|S|} \text{ divides } m^{|\beta(S)^{+K}|}, \text{ which divides } n^{|\bar{\beta}(\beta(S)^{+K})^{+K}|},$$
or equivalently
$$|S| \cdot \mathrm{val}_p(n) \leq |\beta(S)^{+K}| \cdot \mathrm{val}_p(m) \leq |\bar{\beta}( \beta(S)^{+K})^{+K}| \cdot \mathrm{val}_p(n)$$
for every prime divisor $p$ of $n$ and $m$, where $\mathrm{val}_p(\cdot)$ denotes the $p$-adic value of an integer (i.e.\ the largest $k$ for which $p^k$ divides the integer under consideration). Notice that, because we chose $K$ sufficiently large compared to the parameters of $\beta$, we have:

\begin{claim}
There exists a constant $R \geq 0$ that does not depend on $S$ such that
$$S \subset \bar{\beta}(\beta(S)^{+K})^{+K} \subset S^{+R}.$$
\end{claim}

\noindent
The claim is straightforward, and its proof is left to the interested reader as an exercise. As consequence, it follows from Fact~\ref{fact:Boundary} that there exists a constant $C>0$ that does not depend on $S$ such that
$$|\bar{\beta}( \beta(S)^{+K})^{+K}| \leq |S| + C \cdot |\partial S|.$$
Therefore, our previous inequality gives
$$|S| \cdot \mathrm{val}_p(n) \leq |\beta(S)^{+K}| \cdot \mathrm{val}_p(m) \leq (|S| + C \cdot |\partial S|) \cdot \mathrm{val}_p(n),$$
which we can rewrite as
\begin{equation}\label{equation:Valuation}
\frac{\mathrm{val}_p(n)}{\mathrm{val}_p(m)}  \leq \frac{|\beta(S)^{+K}|}{|S|} \leq \left( 1 + C \cdot \frac{|\partial S|}{|S|} \right) \cdot \frac{\mathrm{val}_p(n)}{\mathrm{val}_p(m)}.
\end{equation}
Now, taking for $S$ a F{\o}lner sequence $(F_n)_{n \geq 0}$ in $X$, we deduce from these inequalities that the sequence $(|\beta(F_n)^{+K_1}|/|F_n|)_{n \geq 0}$ converges to $\mathrm{val}_p(n)/ \mathrm{val}_p(m)$. As a consequence, the limit $\mathrm{val}_p(n)/ \mathrm{val}_p(m)$ does not depend on $p$, which happens precisely when $n$ and $m$ are powers of a common number, say $n=q^a$ and $m=q^b$. 

\medskip \noindent
The fact that $\beta$ is quasi-$\frac{b}{a}$-to-one follows immediately from Equation~\ref{equation:Valuation} and Corollary~\ref{cor:QuasiCriterion}. 
\end{proof}

\noindent
Proposition~\ref{prop:AptolicQILamp} provides necessary conditions for two lamplighter graphs to be quasi-isometric. We conclude this section by proving that, in the amenable case, these conditions are also sufficient. More precisely:

\begin{prop}\label{prop:AptoQI}
Let $X,Y$ be two amenable graphs and $q^a, q^b$ two powers of a common number. If there exists a quasi-$\frac{b}{a}$-to-one quasi-isometry $X \to Y$, then there exists an aptolic quasi-isometry $\mathcal{L}_{q^a}(X) \to \mathcal{L}_{q^b}(Y)$. 
\end{prop}

\noindent
Our proof will be based on the following characterisation of aptolic quasi-isometries:

\begin{lemma}\label{lem:AptoQI}
Let $n,m \geq 2$ be two integers, $X,Y$ two unbounded graphs with bounded degree, $\alpha : \mathbb{Z}_n^{(X)} \to \mathbb{Z}_m^{(Y)}$ and $\beta : X \to Y$ two maps. Then
$$q : \left\{ \begin{array}{ccc} \mathcal{L}_n(X) & \to & \mathcal{L}_m(Y) \\ (c,p) & \mapsto & (\alpha(c),\beta(p)) \end{array} \right.$$
is an aptolic quasi-isometry if the following conditions hold:
\begin{itemize}
	\item[(i)] $\alpha$ is a bijection;
	\item[(ii)] $\beta$ is a quasi-isometry;
	\item[(iii)] there exists $Q \geq 0$ such that, for all colourings $c_1,c_2 \in \mathbb{Z}_n^{(X)}$, the Hausdorff distance between $\beta(\mathrm{supp}(c_1-c_2))$ and $\mathrm{supp}(\alpha(c_1)-\alpha(c_2))$ is at most $Q$. 
\end{itemize}
\end{lemma}

\begin{proof}
Let $C,K \geq 0$ be such that $\beta$ is a $(C,K)$-quasi-isometry and such that there exists a $(C,K)$-quasi-isometry $\bar{\beta}$ with $\beta \circ \bar{\beta}$, $\bar{\beta} \circ \beta$ within $K$ from identities. Set
$$\bar{q} : \left\{ \begin{array}{ccc} \mathcal{L}_m(Y) & \to & \mathcal{L}_n(X) \\ (c,p) & \mapsto & \left( \alpha^{-1}(c) , \bar{\beta}(p) \right) \end{array} \right.,$$
and observe that 
$$q \circ \bar{q} : (c,p) \mapsto (c, \beta \circ \bar{\beta}(p)) \text{ and } \bar{q} \circ q : (c,p) \mapsto (c, \bar{\beta}\circ \beta(p))$$
are at distance $\leq K$ from identities. 

\medskip \noindent
Since $Y$ had bounded degree, then so does $\mathcal{L}_m(Y)$. This implies that there exists $L\geq 1$ such that there exists a path of length at most $L$ starting and ending at any vertex $x$, and visiting all vertices of $B(x,Q)$.

\medskip \noindent
Let $(c_1,p_1),(c_2,p_2) \in \mathcal{L}_n(X)$ be two points. Fix a path $\zeta$ of minimal length that starts from $p_1$, visits all the points in $\mathrm{supp}(c_1-c_2)$, and ends at $p_2$. Let $\xi \subset \mathcal{L}_m(Y)$ denote a concatenation of geodesics connecting any two consecutive points along $\beta(\zeta)$. Notice that $\xi$ has length $\leq (C+K) \mathrm{length}(\zeta)$ according to $(ii)$. By construction, $\xi$ starts from $\beta(p_1)$, visits all the points in $\beta(\mathrm{supp}(c_1-c_2))$, and ends at $\beta(p_2)$. Note that there exists a path $\eta$ of length $\leq L \cdot \mathrm{length}(\xi)$ that starts from $\beta(p_1)$, visits all the points in the $Q$-neighbourhood of $\beta(\mathrm{supp}(c_1-c_2))$, and ends at $\beta(p_2)$. Because $\mathrm{supp}(\alpha(c_1)-\alpha(c_2))$ lies in the $Q$-neighbourhood of $\beta(\mathrm{supp}(c_1-c_2))$ according to $(iii)$, it follows that
$$\begin{array}{lcl} d(q(c_1,p_1),q(c_2,p_2)) & = & d((\alpha(c_1),\beta(p_1)), (\alpha(c_2),\beta(p_2))) \leq \mathrm{length}(\eta) \leq L \cdot \mathrm{length}(\xi) \\ \\ & \leq & L(C+K) \cdot \mathrm{length}(\zeta) = L(C+K) \cdot d((c_1,p_1),(c_2,p_2)). \end{array}$$
Observe that $\bar{q}$ also satisfies $(i)-(iii)$. For $(i)$ and $(ii)$, it is clear. For $(iii)$, we know that, for all colourings $c_1,c_2 \in \mathbb{Z}_m^{(Y)}$, the Hausdorff distance between $\beta(\mathrm{supp}(\alpha^{-1}(c_1)-\alpha^{-1}(c_2)))$ and $\mathrm{supp}(c_1-c_2)$ is at most $Q$. So the Hausdorff distance between $\bar{\beta} \circ \beta(\mathrm{supp}(\alpha^{-1}(c_1)-\alpha^{-1}(c_2)))$ and $\bar{\beta}(\mathrm{supp}(c_1-c_2))$ is at most $(C+K)Q$. But the Hausdorff distance between $\bar{\beta} \circ \beta(\mathrm{supp}(\alpha^{-1}(c_1)-\alpha^{-1}(c_2)))$ and $\mathrm{supp}(\alpha^{-1}(c_1)-\alpha^{-1}(c_2))$ is at most $K$, so we conclude that the Hausdorff distance between  $\bar{\beta}(\mathrm{supp}(c_1-c_2))$ and $\mathrm{supp}(\alpha^{-1}(c_1)-\alpha^{-1}(c_2))$ is at most $(C+K)Q+K$, as desired. Therefore, by reproducing the previous argument, we show that 
$$d(\bar{q}(c_1,p_1),\bar{q}(c_2,p_2)) \leq M (C+K) \cdot d((c_1,p_1),(c_2,p_2))$$
for all $(c_1,p_1),(c_2,p_2) \in \mathcal{L}_m(Y)$, where $M$ denotes the length of one path that visits all the points in a ball of radius $(C+K)Q+K$ in $\mathcal{L}_n(X)$ and that both starts and ends at the centre. We deduce from the previous two centred inequalities that
$$\begin{array}{lcl} d(q(c_1,p_1),q(c_2,p_2)) & \geq & \displaystyle \frac{1}{M(C+K)} d( \bar{q} \circ q (c_1,p_1), \bar{q}\circ q(c_2,p_2)) \\ \\ & \geq & \displaystyle \frac{1}{M(C+K)} d((c_1,p_1),(c_2,p_2)) - \frac{2K}{M(C+K)} \end{array}$$
for all $(c_1,p_1),(c_2,p_2) \in \mathcal{L}_n(X)$; and that
$$\begin{array}{lcl} d(\bar{q}(c_1,p_1), \bar{q}(c_2,p_2)) & \geq & \displaystyle \frac{1}{L(C+K)} d(q \circ \bar{q} (c_1,p_1), q\circ \bar{q}(c_2,p_2)) \\ \\ & \geq & \displaystyle \frac{1}{L(C+K)} d((c_1,p_1),(c_2,p_2)) - \frac{2K}{L(C+K)}. \end{array}$$
Thus, $q$ is a quasi-isometry with $\bar{q}$ as a quasi-inverse, proving that $q$ is an aptolic quasi-isometry.
\end{proof}

\begin{proof}[Proof of Proposition~\ref{prop:AptoQI}.]
According to Proposition~\ref{prop:kappa}, there exist a partition $\mathcal{P}$ (resp. $\mathcal{Q}$) of $X$ (resp. of $Y$) with uniformly bounded pieces of size $b$ (resp. $a$), a bijection $\psi:\mathcal{P}\to \mathcal{Q}$, and a quasi-isometry $\beta : X \to Y$ satisfying $\beta(P) \subset \psi(P)$ for every $P \in \mathcal{P}$. Fix a bijection $\sigma : \mathbb{Z}_n^{b} \to \mathbb{Z}_m^{a}$ satisfying $\sigma(0)=0$, and define a bijection $\alpha : \mathbb{Z}_n^{(X)} \to \mathbb{Z}_m^{(Y)}$ in such a way that $\alpha$ sends $\mathcal{L}(P)$ to $\mathcal{L}(\psi(P))$ through $\sigma$ for every $P \in \mathcal{P}$. We claim that
$$q : (c,p) \mapsto (\alpha(c), \beta(p)), \ (c,p) \in \mathcal{L}_n(X)$$
is the quasi-isometry we are looking for. Let $(c_1,p_1),(c_2,p_2) \in \mathcal{L}_n(X)$ be two points. Let $P_1, \ldots, P_\ell$ denote the pieces of $\mathcal{P}$ containing points in $\mathrm{supp}(c_1-c_2)$. By construction, $\psi(P_1),\ldots, \psi(P_\ell)$ are the pieces of $\mathcal{Q}$ containing points in $\mathrm{supp}(\alpha(c_1)-\alpha(c_2))$. Because the pieces of $\mathcal{Q}$ are uniformly bounded, the Hausdorff distance between $\mathrm{supp}(\alpha(c_1)-\alpha(c_2))$ and $\psi(P_1) \cup \cdots \cup \psi(P_\ell)$ is finite. We also know by construction that $\beta(\mathrm{supp}(c_1-c_2))$ lies in $\psi(P_1) \cup \cdots \cup \psi(P_\ell)$ and has a point in each $\psi(P_1),\ldots, \psi(P_\ell)$. Once again because the pieces of $\mathcal{Q}$ are uniformly bounded, we deduce that the Hausdorff dimension between $\mathrm{supp}(\alpha(c_1)-\alpha(c_2))$ and $\beta(\mathrm{supp}(c_1-c_2))$ is bounded (by a bound that does not depend on $c_1,c_2$ but only on $\mathcal{P},\mathcal{Q}$). We conclude from Lemma~\ref{prop:AptoQI} that $q$ is an aptolic quasi-isometry, as desired. 
\end{proof}

\subsection{Lamplighters over non-amenable graphs}\label{section:LampNonAmenable}

\noindent
We know from Proposition~\ref{prop:AptolicQILamp} that, if two lamplighter graphs $\mathcal{L}_n(X)$ and $\mathcal{L}_m(Y)$ are quasi-isometric, then $X,Y$ must be quasi-isometric and $n,m$ must have the same prime divisors. In this section, our goal is to show that, for non-amenable graphs, these conditions actually characterise whether two lamplighter graphs are quasi-isometric. More precisely, we want prove that:

\begin{prop}[\cite{MR4794592}]\label{prop:QInonamenable}
Let $X$ be a graph of bounded degree and $n,m \geq 2$ two integers. If $X$ is non-amenable and if $n,m$ have the same prime divisors, then there exists an aptolic quasi-isometry between $\mathcal{L}_n(X)$ and $\mathcal{L}_m(X)$.
\end{prop}

\noindent
We emphasize that, in this statement, we do not assume that $X$ is pancylindrical. For instance, $X$ may be a (bushy) tree. 

\medskip \noindent
First, let us illustrate the construction we use in order to prove Proposition \ref{prop:QInonamenable} by explaining why the lamplighter groups $\mathbb{Z}_6 \wr \mathbb{F}_2$ and $\mathbb{Z}_{24} \wr \mathbb{F}_2$ are quasi-isometric, where the free group $\mathbb{F}_2$ will be thought of as the $4$-regular tree. The first trick is to replace $\mathbb{Z}_6 \wr \mathbb{F}_2$ (resp. $\mathbb{Z}_{24} \wr \mathbb{F}_2$) with $(\mathbb{Z}_3 \oplus \mathbb{Z}_2) \wr \mathbb{F}_2$ (resp. $(\mathbb{Z}_{3} \oplus \mathbb{Z}_2^3) \wr \mathbb{F}_2$). Loosely speaking, we split each lamp into two half-lamps. Formally, each colouring $c : \mathbb{F}_2 \to \mathbb{Z}_6$ (resp. $c : \mathbb{F}_2 \to \mathbb{Z}_{24}$) becomes the sum $c_1 \oplus c_2$ of two colourings $c_1 : \mathbb{F}_2 \to \mathbb{Z}_3$ and $c_2 : \mathbb{F}_2 \to \mathbb{Z}_2$ (resp. $c_2 : \mathbb{F}_2 \to \mathbb{Z}_2^3$). The second trick is to notice that, given a point at infinity $\xi \in \partial \mathbb{F}_2$, one can associate a colouring $\bar{c} : \mathbb{F}_2 \to \mathbb{Z}_2^3$ to any colouring $c : \mathbb{F}_2 \to \mathbb{Z}_2$ in the following way: for every point $p \in \mathbb{F}_2$, we define $\bar{c}(p) \in \mathbb{Z}_2^3$ thanks to the three digits in $\mathbb{Z}_2$ provided by the values taken by $c$ at the three points $p_1,p_2,p_3 \in \mathbb{F}_2$ that are separated from $\xi$ by $p$, i.e. $\bar{c}(p):= (c(p_1),c(p_2),c(p_3))$. 
\begin{center}
\begin{tabular}{|c|c|} \hline
\includegraphics[trim={0 0 15cm 0},clip,width=0.45\linewidth]{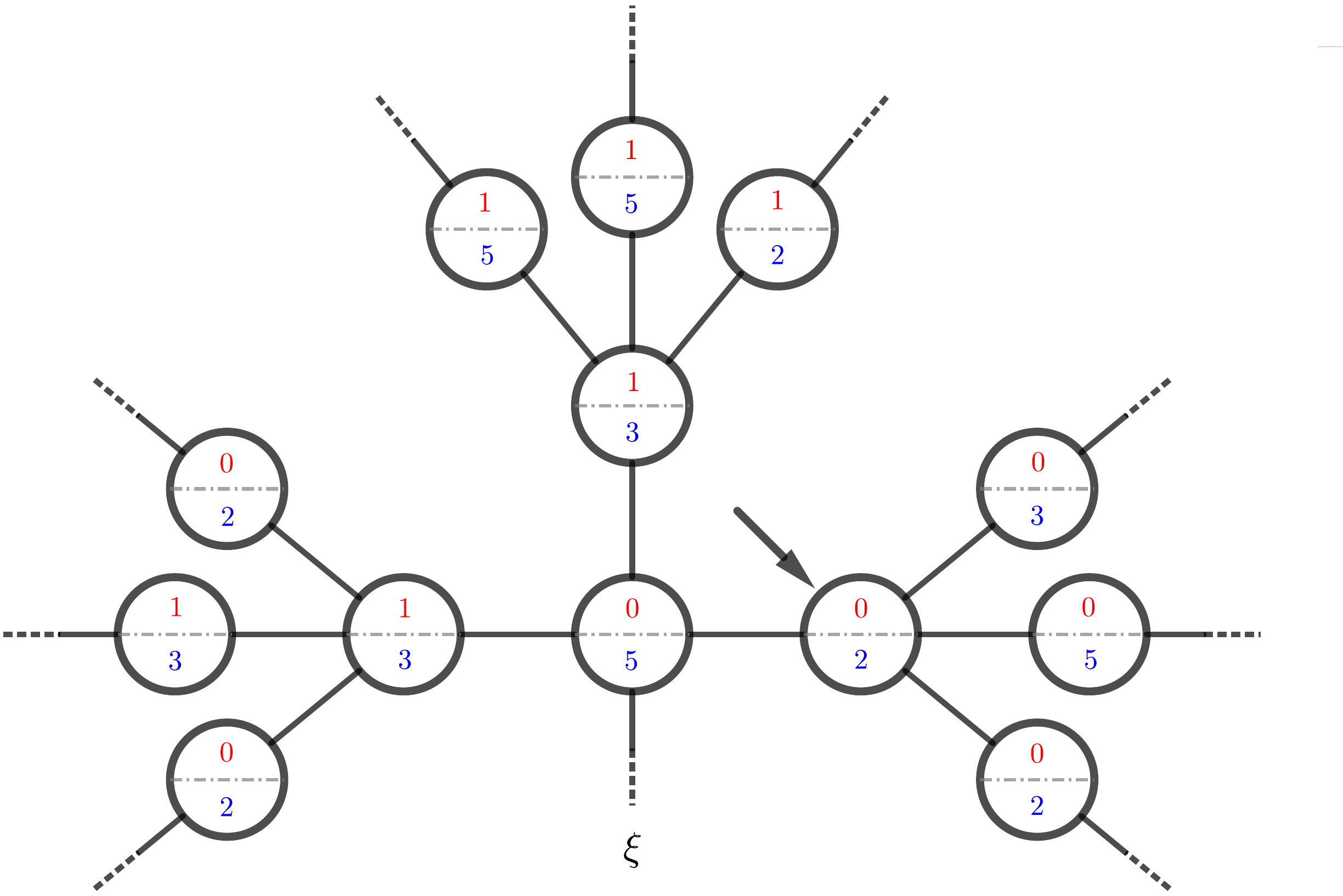} &
\includegraphics[trim={0 0 15cm 0},clip,width=0.45\linewidth]{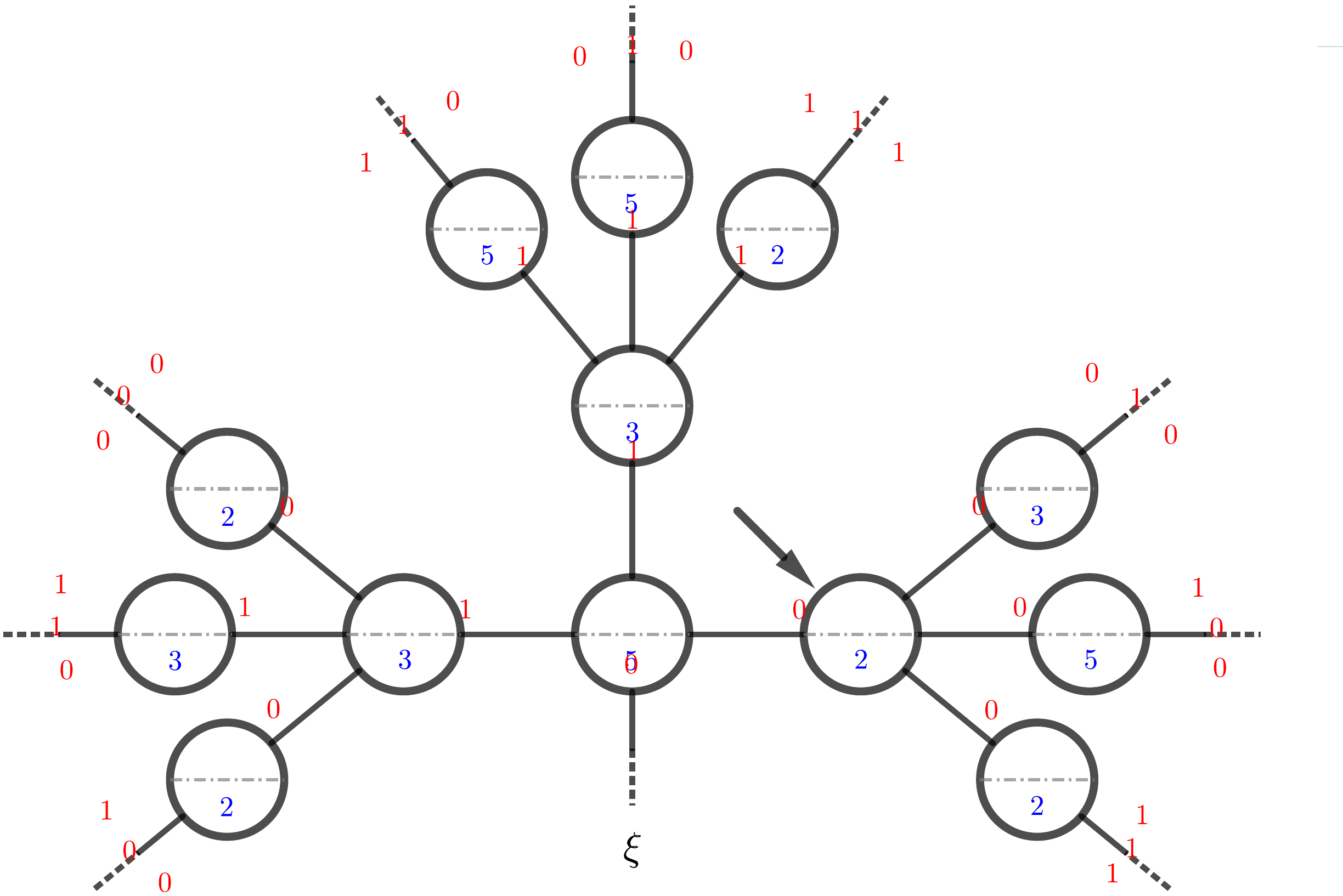} \\ \hline
\includegraphics[trim={0 0 15cm 0},clip,width=0.45\linewidth]{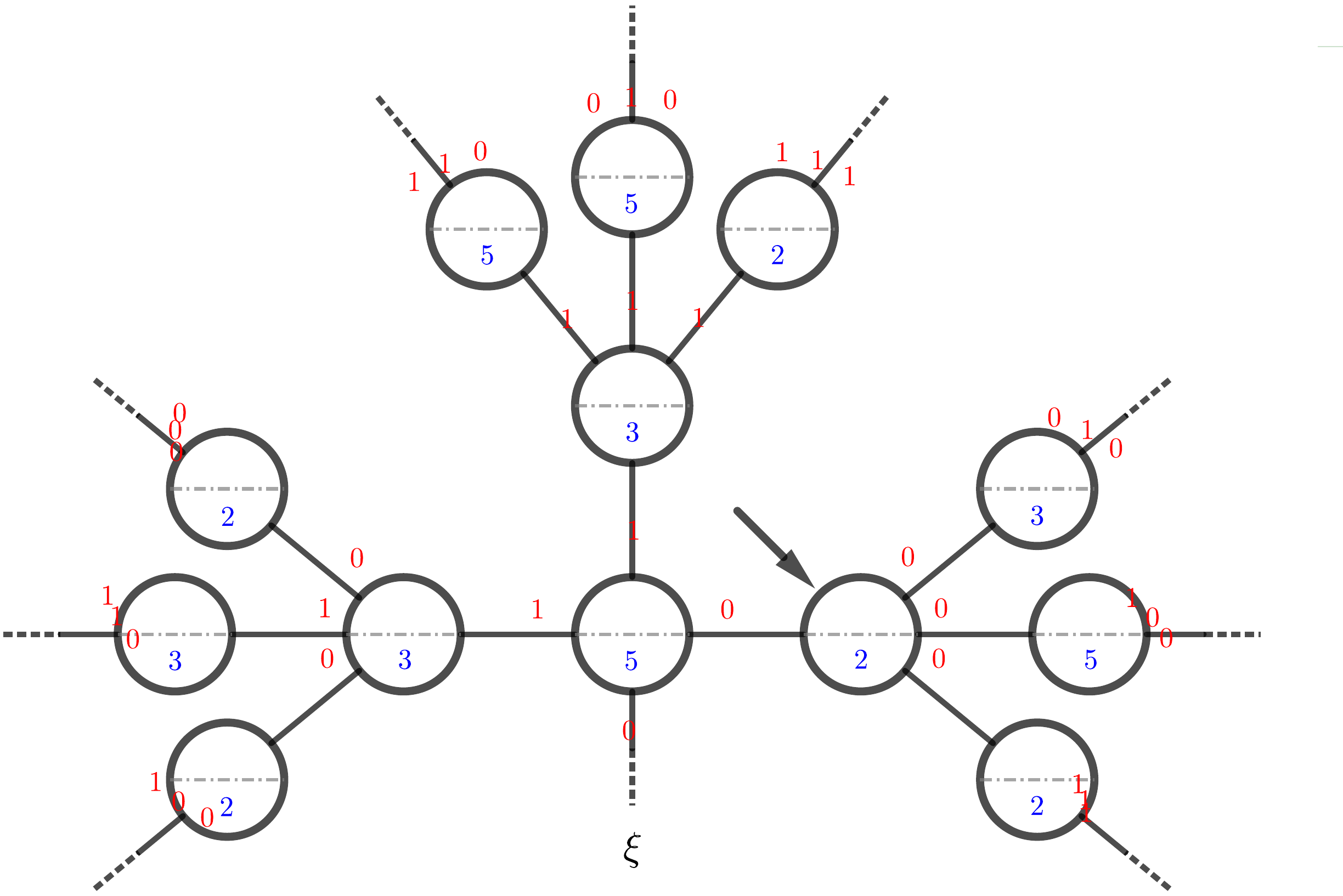} &
\includegraphics[trim={0 0 15cm 0},clip,width=0.45\linewidth]{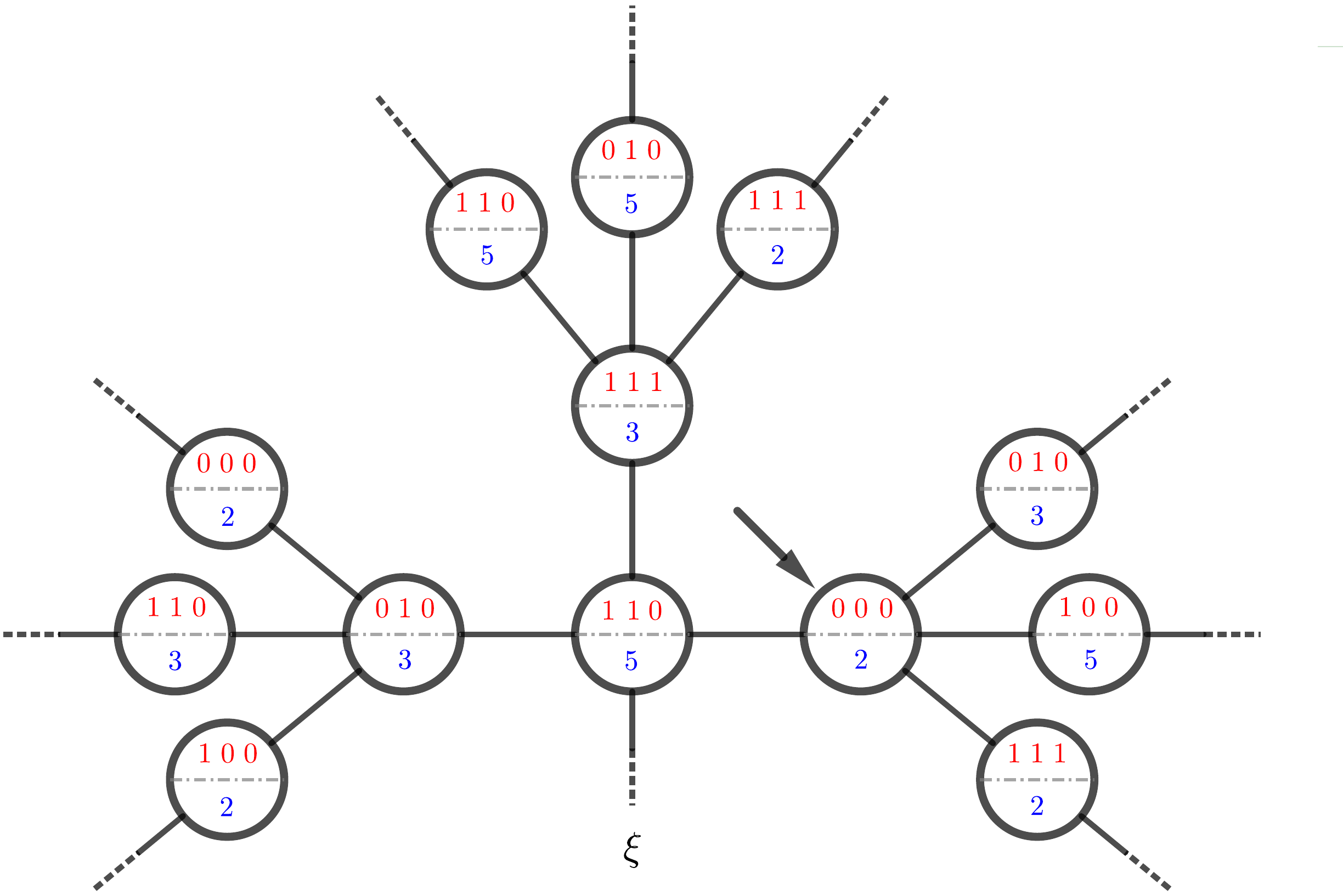} \\ \hline
\end{tabular}
\end{center}
Now, we define a map $(\mathbb{Z}_3 \oplus \mathbb{Z}_6) \wr \mathbb{F}_2 \to (\mathbb{Z}_{3} \oplus \mathbb{Z}_2^3) \wr \mathbb{F}_2$ by modifying the second halves of the lamps thanks to the previous operation and by leaving the arrow and the first halves as they were, i.e. $(c_1 \oplus c_2, p) \mapsto (c_1 \oplus \overline{c_2},p)$. This map turns out to define a quasi-isometry because the modifications on the colourings are local. 

\medskip \noindent
The key point in the previous construction is that there exists a $3$-to-$1$ map $\mathbb{F}_2 \to \mathbb{F}_2$ that lies at finite distance from the identity, namely the map that sends every vertex to its neighbour towards $\xi$. It is worth mentioning that such maps cannot exist for amenable graphs (see Exercise~\ref{exo:AmenableNoNtoOne}). We generalise our construction for arbitrary non-amenable graphs. 

\begin{proof}[Proof of Proposition \ref{prop:QInonamenable}.]
Given two integers $m \geq 1$, $n \geq 2$, and a prime $p$, we prove that $\mathcal{L}_{mp^n}(X)$ and $\mathcal{L}_{mp}(X)$ are quasi-isometric (through an aptolic quasi-isometry). This is sufficient to deduce our proposition.

\medskip \noindent
First of all, observe that there exists an $n$-to-$1$ map $f : X \to X$ at finite distance from the identity, i.e. there exists some $C \geq 0$ such that $d(f(x),x) \leq C$ for every $x \in X$. Indeed, as a consequence of Theorem~\ref{thm:BijNonQI}, the embedding $\iota : X \hookrightarrow X \oplus \mathbb{Z}_n$ is at finite distance, say $C$, from a bijection $g : X \to X \oplus \mathbb{Z}_n$. If $p : X \oplus \mathbb{Z}_n \to X$ denotes the canonical projection, then $f:=p \circ g$ is $n$-to-$1$. Moreover,
$$d(f(x),x)=d(p(g(x)),p(\iota(x))) \leq d(g(x),\iota(x)) \leq C$$
for every $x \in X$. This proves our observation.

\medskip \noindent
From now on, we fix an enumeration of $X$ and we identify $\mathbb{Z}_{mp}$ (resp. $\mathbb{Z}_{mp^n}$) with $\mathbb{Z}_m \oplus \mathbb{Z}_p$ (resp. $\mathbb{Z}_m \oplus \mathbb{Z}_p^n$). Given a finitely supported colouring $c : X \to \mathbb{Z}_{m} \oplus \mathbb{Z}_{p}$, we construct a new finitely supported colouring $\bar{c} : X \to \mathbb{Z}_m \oplus \mathbb{Z}_p^n$ as follows. For convenience, we denote by $\pi_1$ and $\pi_2$ the projections on the first and second coordinates in both $\mathbb{Z}_m \oplus \mathbb{Z}_p$ and $\mathbb{Z}_m \oplus \mathbb{Z}_p^n$. Given an $x \in X$,
\begin{itemize}
	\item set $\pi_1(\bar{c}(x)):= \pi_1(c(x))$;
	\item enumerate $f^{-1}(x)$ as $\{x_1, \ldots, x_k\}$ by following the order of induced by our enumeration of $X$, and set $\pi_2(\bar{c}(x))=( \pi_2(c(x_1)), \ldots, \pi_2(c(x_k)))$.
\end{itemize}
We claim that
$$\Phi : \left\{ \begin{array}{ccc} \mathcal{L}_{mp}(X) & \to \mathcal{L}_{mp^n}(X) \\ (c,h) & \mapsto (\bar{c},h) \end{array} \right.$$
is a quasi-isometry. In the rest of the proof, our lamplighter graphs are endowed with diligent metrics. So fix two finitely supported colourings $c_1,c_2 : X \to \mathbb{Z}_m \oplus \mathbb{Z}_p$ and two points $k_1,k_2 \in X$. 

\medskip \noindent
Notice that, if $\bar{c}_1(x) \neq \bar{c}_2(x)$ for some $x \in X$, then either $\pi_1(c_1(x)) \neq \pi_1(c_2(x))$ (hence $x \in \mathrm{supp}(c_1-c_2)$) or $\pi_2(c_1(x')) \neq \pi_2(c_2(x'))$ (hence $x' \in \mathrm{supp}(c_1-c_2)$) for some $x' \in f^{-1}(x)$. Consequently, 
\begin{equation}\label{First}
\mathrm{supp}(\bar{c}_1-\bar{c}_2) \subset \mathrm{supp}(c_1-c_2) \cup f( \mathrm{supp}(c_1-c_2)).
\end{equation}
Next, if $c_1(x) \neq c_2(x)$ for some $x \in X$, then either $\pi_1(c_1(x)) \neq \pi_1(c_2(x))$, hence $\bar{c}_1(x) \neq \bar{c}_2(x)$; or $\pi_2(c_1(x)) \neq \pi_2(c_2(x))$, hence $\bar{c}_1(f(x)) \neq \bar{c}_2(f(x))$. Consequently,
\begin{equation}\label{Second}
\mathrm{supp}(c_1-c_2) \subset \mathrm{supp}(\bar{c}_1-\bar{c}_2) \cup f^{-1}(\mathrm{supp}(\bar{c}_1-\bar{c}_2) ).
\end{equation}
Now, fix a path $\alpha$ in $X$ that starts from $k_1$, that visits all the points in $\mathrm{supp}(c_1-c_2)$, that ends at $k_2$, and such that the length of $\alpha$ coincides with the distance between $(c_1,k_1)$ and $(c_2,k_2)$ in $\mathcal{L}_{mp}(X)$. For every point of $x \in \mathrm{supp}(c_1-c_2)$, we add to $\alpha$ a loop of length $\leq 2C$ based at $x$ and passing through $f(x)$. Thus, we obtain a new path $\alpha'$ that visits all the points in $\mathrm{supp}(c_1-c_2) \cup f( \mathrm{supp}(c_1-c_2))$ and whose length is at most 
$$\mathrm{length}(\alpha)+2C |\mathrm{supp}(c_1-c_2)| \leq (2C+1) \mathrm{length}(\alpha).$$
It follows from the inclusion (\ref{First}) that
$$d((\bar{c}_1,k_1),(\bar{c}_2,k_2)) \leq \mathrm{length}(\alpha') \leq (2C+1) d((c_1,k_1),(c_2,k_2)).$$
Next, fix a path $\beta$ in $X$ that starts from $k_1$, that visits all the points in $\mathrm{supp}(\bar{c}_1-\bar{c}_2)$, that ends at $k_2$, and such that the length of $\beta$ coincides with the distance between $(\bar{c}_1,k_1)$ and $(\bar{c}_2,k_2)$ in $\mathcal{L}_{mp^n}(X)$. For every point $x \in \mathrm{supp}(\bar{c}_1-\bar{c}_2)$ and every $x' \in f^{-1}(x)$, we add to $\beta$ a loop of length $\leq 2C$ based at $x$ and passing through $x'$. Thus, we obtain a new path $\beta'$ that visits all the points in $\mathrm{supp}(\bar{c}_1-\bar{c}_2) \cup f^{-1}(\mathrm{supp}(\bar{c}_1-\bar{c}_2) )$ and whose length is at most
$$\mathrm{length}(\beta) + 2nC |\mathrm{supp}(\bar{c}_1-\bar{c}_2)| \leq (2nC+1) \mathrm{length}(\beta).$$
It follows from the inclusion (\ref{Second}) that
$$d((c_1,k_1),(c_2,k_2)) \leq \mathrm{length}(\beta') \leq (2nC+1) d((\bar{c}_1,k_1),(\bar{c}_2,k_2)).$$
This concludes the proof that $\Phi$ is a quasi-isometry.
\end{proof}

\noindent
Thus, we have essentially proved all the cases of the quasi-isometric classification provided by Theorem~\ref{thm:AptolicQIstrong}, which we now summarise:

\begin{proof}[Proof of Theorem~\ref{thm:QIclassification}.]
Let $X,Y$ be two unbounded pancylindrical graphs of bounded degree and $n,m \geq 2$ two integers. If the lamplighter graphs $\mathcal{L}_n(X)$ and $\mathcal{L}_m(Y)$ are quasi-isometric, then there exists an aptolic quasi-isometry $\mathcal{L}_n(X) \to \mathcal{L}_m(Y)$ according to Theorem~\ref{thm:AptolicQI}. In particular, $X$ and $Y$ are quasi-isometric, so they are either both amenable or both non-amenable.

\medskip \noindent
If $X$ and $Y$ are both amenable, then we know from Proposition~\ref{prop:AptolicQILamp} that $n$ and $m$ are powers of a common number, say $n=q^a$ and $m=q^b$, and that there exists a quasi-$\frac{b}{a}$-to-one quasi-isometry $X \to Y$. Conversely, if these conditions are satisfied, we know from Proposition~\ref{prop:AptolicQILamp} that the lamplighter graphs $\mathcal{L}_n(X)$ and $\mathcal{L}_m(Y)$ must be quasi-isometric.

\medskip \noindent
If $X$ and $Y$ are both non-amenable, then we know from Proposition~\ref{prop:AptolicQILamp} that $n,m$ have the same prime divisors and that $X,Y$ are quasi-isometric. Conversely, if these conditions are satisfied, then we know from Proposition~\ref{prop:QInonamenable} that the lamplighter graphs $\mathcal{L}_n(X)$ and $\mathcal{L}_m(Y)$ must be quasi-isometric.
\end{proof}

\subsection{Applications}\label{section:ApplicationsTwo}

\noindent 
As a conclusion of our study of lamplighter graphs up to quasi-isometry, let us mention a few notable consequences of our classification. A first remark is that lamplighters behave differently on amenable and amenable graphs. Exploiting this dissymmetry, it is possible to characterise amenability through quasi-isometric lamplighter graphs:

\begin{prop}
An unbounded pancylindrical graph $X$ of bounded degree is amenable if and only if the lamplighter graphs $\mathcal{L}_6(X)$ and $\mathcal{L}_{12}(X)$ are quasi-isometric.
\end{prop}

\begin{proof}
This follows direct from Theorem~\ref{thm:QIclassification} and from the observation that $6$ and $12$ have the same prime divisors but are not powers of a common number. 
\end{proof}

\noindent
We mentioned earlier (Theorem~\ref{thm:Dymarz}) that some lamplighters over $\mathbb{Z}$ are quasi-isometric but not bijectively quasi-isometric (i.e.\ not biLipschitz equivalent). It is possible to deduce from our work other such examples, as shown by our next proposition. For instance, $\mathbb{Z}_n \wr \mathbb{Z}^2$ and $\mathbb{Z}_m \wr \mathbb{Z}^2$ are quasi-isometric if and only if $n,m$ are powers of a common number but biLipschitz equivalent if and only $n=m$. 

\begin{prop}
Let $F_1,F_2$ be two finite groups and $H$ a pancylindrical amenable group. The wreath products $F_1 \wr H$ and $F_2 \wr H$ are biLipschitz equivalent if and only if $|F_1|=|F_2|$.
\end{prop}

\begin{proof}
If $|F_1|=|F_2|$, then $F_1 \wr H$ and $F_2 \wr H$ are biLipschitz equivalent according to Proposition~\ref{prop:BiLipWreathGroups}. Conversely, assume that there exists a biLipschitz equivalence $\varphi : F_1 \wr H \to F_2 \wr H$. We know from Theorem~\ref{thm:AptolicQI} that $\varphi$ lies at finite distance from an aptolic quasi-isometry $(\alpha,\beta)$, which must be quasi-one-to-one (see Exercise~\ref{exo:QuasiToOne}). Then, we deduce from Proposition~\ref{prop:AptolicQILamp} that $|F_1|=|F_2|$. 
\end{proof}

\noindent
Finally, let us mention the following interesting rigidity phenomenon: in a wreath product $F \wr H$ with $F$ finite and $H$ pancylindrical amenable, any two isomorphic finite-index subgroups must have the same index. More generally:

\begin{prop}
Let $F$ be a non-trivial finite group and $H$ a pancylindrical amenable group. If two finite-index subgroups $K_1,K_2 \leq F \wr H$ are biLipschitz equivalent (e.g.\ isomorphic), then they have the same index in $F \wr H$. 
\end{prop}

\begin{proof}
Let $n_1$ (resp.\ $n_2$) denote the index of $K_1$ (resp.\ $K_2$) in $F \wr H$. The inclusion map $\iota_1 : K_1 \hookrightarrow F \wr H$ (resp.\ $\iota_2 : K_2 \hookrightarrow F \wr H$) induces a quasi-$n_1^{-1}$-to-one (resp.\ quasi-$n_2^{-1}$-to-one) quasi-isometry (see Exercise~\ref{exo:QuasiToOne}). Given a biLipschitz equivalence $\epsilon : K_1 \to K_2$ and a quasi-inverse $\eta_1$ of $\iota_1$, we get a quasi-$n_2^{-1}n_1$-to-one quasi-isometry $\iota_2 \epsilon \circ \eta_1 : F \wr H \to F \wr H$ (again, see Exercise~\ref{exo:QuasiToOne}). The equality $n_1=n_2$ follows from the following observation:

\begin{claim}
Every quasi-isometry $F \wr H \to F \wr H$ is quasi-one-to-one.
\end{claim}

\noindent
According to Theorem~\ref{thm:AptolicQI}, it suffices to verify the statement for an arbitrary aptolic quasi-isometry $(\alpha,\beta)$. We know from Proposition~\ref{prop:AptolicQILamp} that $\beta$ must be quasi-one-to-one. Then, it follows from Exercise~\ref{exo:AptolicQuasiToOne} that $(\alpha,\beta)$ must be quasi-one-to-one as well, concluding the proof. 
\end{proof}

\subsection{Exercises}

\begin{exo}\label{exo:AmenableQI}
Let $X$ and $Y$ be two graphs of bounded degree. Show that, if $X$ and $Y$ are quasi-isometric, then $X$ is amenable if and only if so is $Y$. 
\end{exo}

\begin{exo}\label{exo:Subsemigroups}
Let $A$ and $B$ be two finitely generated groups. Assume that $A$ is non-trivial and that $B$ contains an infinite-order element, say $t \in B$. For every $a \in A \backslash \{1\}$, prove that $ta$ and $t^2a$ generate a quasi-isometrically embedded free subsemigroup in the wreath product $A \wr B$. 
\end{exo}

\begin{exo}\label{exo:QuasiToOne}
Prove the following assertions:
\begin{enumerate}
	\item A quasi-isometry between two non-amenable graphs of bounded degree is quasi-$\kappa$-to-one for every $\kappa>0$.
	\item If a quasi-isometry between two amenable graphs of bounded degree is both quasi-$\kappa_1$-to-one and quasi-$\kappa_2$-to-one, then $\kappa_1= \kappa_2$.
	\item If a quasi-isometry between two graphs of bounded degree is quasi-$\kappa$-to-one for some $\kappa>0$, then every map at finite distance is quasi-$\kappa$-to-one.
	\item Every quasi-inverse of a quasi-$\kappa$-to-one quasi-isometry is quasi-$\kappa^{-1}$-to-one. 
	\item Let $X,Y,Z$ be three graphs of bounded degree. If two quasi-isometries $\varphi_1 : X \to Y$ and $\varphi_2 : Y \to Z$ are respectively quasi-$\kappa_1$-to-one and quasi-$\kappa_2$-to-one, then $\varphi_2 \circ \varphi_1$ is quasi-$\kappa_1\kappa_2$-to-one.
	\item Given a finitely generated group $G$ and a finite-index subgroup $H \leq G$, the inclusion map $H \hookrightarrow G$ is quasi-$[G:H]^{-1}$-to-one. 
	\item There exist quasi-isometries between finitely generated groups that are not quasi-$\kappa$-to-one for every $\kappa>0$. 
\end{enumerate}
\end{exo}

\begin{exo}\label{exo:AmenableNoNtoOne}
Let $X$ be an amenable graph of bounded degree. 
\begin{enumerate}
	\item Deduce from Exercise~\ref{exo:QuasiToOne} that there is no $n$-to-one quasi-isometry $X \to X$ that lies at finite distance from the identity. 
	\item Prove the same statement just using the definition of amenability. 
\end{enumerate}
\end{exo}

\begin{exo}\label{exo:AptolicQuasiToOne}
Let $X,Y$ be two graphs of bounded degree and $n,m \geq 2$ two integers. Let $(\alpha,\beta)$ be an aptolic quasi-isometry $\mathcal{L}_n(X) \to \mathcal{L}_m(Y)$. Assuming that $\beta$ is quasi-$\kappa$-to-one for some $\kappa>0$, show that $(\alpha,\beta)$ is also quasi-$\kappa$-to-one. 
\end{exo}

\begin{exo}
Let $n \geq 2$ be an integer and $X$ a graph that contains a ball $B(x_0,R)$ whose complement has at least two unbounded connected components, say $A,B \subset X$. Let $\delta$ denote the colouring taking the value $1$ at $x_0$ and $0$ elsewhere. Prove that the map $\mathcal{L}_n(X) \to \mathcal{L}_n(X)$ defined by
$$(c,x) \mapsto \left\{ \begin{array}{cl} (c,x) & \text{if $x \in A$ or $c_{|A} \neq 0$} \\ (c+\delta,x) & \text{if $x \notin A$ and $c_{|A}=0$} \end{array} \right.$$
is a surjective $(1,2R)$-quasi-isometry but does not lie at finite distance from an aptolic quasi-isometry. 
\end{exo}

\addcontentsline{toc}{section}{References}

\bibliographystyle{alpha}
{\footnotesize\bibliography{MiniCourseLamp}}

\Address

%\addcontentsline{toc}{section}{Index}
%
%\printindex

\end{document}